\documentclass[preprint,12pt]{elsarticle}

\usepackage{fullpage,amsmath,amsthm,amsfonts,url,amssymb,stmaryrd}
\usepackage{mathrsfs}
\usepackage{mathtools}
\usepackage{amssymb,amscd}
\usepackage{color}
\usepackage{chemarrow}
\usepackage{pictexwd,dcpic}
\usepackage{extarrows}
\usepackage{enumitem}

\makeatletter
\def\ps@pprintTitle{
 \let\@oddhead\@empty
 \let\@evenhead\@empty
 \def\@oddfoot{\centerline{\thepage}}
 \let\@evenfoot\@oddfoot}
\makeatother

\newtheorem{theorem}{Theorem}[section]
\newtheorem{corollary}[theorem]{Corollary}
\newtheorem{lemma}[theorem]{Lemma}
\newtheorem{proposition}[theorem]{Proposition}
\newtheorem{definition}[theorem]{Definition}
\newtheorem{hypothesis}[theorem]{Hypothesis}
\newtheorem{remark}[theorem]{Remark}
\newtheorem{notation}[theorem]{Notation}
\newtheorem{example}[theorem]{Example}
\newtheorem{conjecture}[theorem]{Conjecture}

\newcommand{\hooklongrightarrow}{\lhook\joinrel\longrightarrow}
\newcommand{\twoheadlongrightarrow}{\relbar\joinrel\twoheadrightarrow}
\newcommand{\us}{\upsilon}
\newcommand{\ra}{\rightarrow}
\newcommand{\lra}{\longrightarrow}

\newcommand{\ul}{\underline}

\newcommand{\hV}{\textbf V}
\newcommand{\bA}{\mathbb A}
\newcommand{\bC}{\mathbb C}
\newcommand{\F}{\mathbb F}

\newcommand{\Q}{\mathbb Q}
\newcommand{\bR}{\mathbb R}

\newcommand{\bT}{\mathbb T}
\newcommand{\Z}{\mathbb Z}

\newcommand{\bW}{\mathbb W}
\newcommand{\cN}{\mathcal N}
\newcommand{\cL}{\mathcal L}
\newcommand{\co}{\mathcal O}

\newcommand{\cR}{\mathcal R}

\newcommand{\cC}{\mathcal C}

\newcommand{\cI}{\mathcal I}
\newcommand{\cW}{\mathcal W}
\newcommand{\cT}{\mathcal T}
\newcommand{\cM}{\mathcal M}
\newcommand{\cF}{\mathcal F}

\newcommand{\cE}{\mathcal E}

\newcommand{\cP}{\mathcal P}

\newcommand{\fn}{\mathfrak n}

\newcommand{\fm}{\mathfrak{m}}
\newcommand{\ub}{\mathfrak b}
\newcommand{\fl}{\mathfrak l}
\newcommand{\fp}{\mathfrak p}

\newcommand{\ug}{\mathfrak g}

\newcommand{\ft}{\mathfrak t}

\newcommand{\sE}{\mathscr E}

\newcommand{\sD}{\mathscr D}

\DeclareMathAlphabet{\mathpzc}{OT1}{pzc}{m}{it}
\DeclareMathOperator{\sm}{\mathrm sm}

\DeclareMathOperator{\gl}{\mathfrak gl}

\DeclareMathOperator{\tr}{\mathrm tr}
\DeclareMathOperator{\GL}{\mathrm GL}
\DeclareMathOperator{\gr}{\mathrm gr}

\DeclareMathOperator{\Fil}{\mathrm Fil}
\DeclareMathOperator{\Res}{\mathrm Res}
\DeclareMathOperator{\Frac}{\mathrm Frac}

\DeclareMathOperator{\Gal}{\mathrm Gal}
\DeclareMathOperator{\Hom}{\mathrm Hom}

\DeclareMathOperator{\End}{\mathrm End}
\DeclareMathOperator{\cris}{\mathrm cris}
\DeclareMathOperator{\rig}{\mathrm rig}
\DeclareMathOperator{\an}{\mathrm an}
\DeclareMathOperator{\Spec}{\mathrm Spec}

\DeclareMathOperator{\Frob}{\mathrm Frob}
\DeclareMathOperator{\Ind}{\mathrm Ind}
\DeclareMathOperator{\unr}{\mathrm unr}

\DeclareMathOperator{\Ker}{\mathrm Ker}
\DeclareMathOperator{\pr}{\mathrm pr}

\DeclareMathOperator{\ord}{\mathrm ord}

\DeclareMathOperator{\Ext}{\mathrm Ext}

\DeclareMathOperator{\Spf}{\mathrm Spf}
\DeclareMathOperator{\Ima}{\mathrm Im}
\DeclareMathOperator{\SL}{\mathrm SL}
\DeclareMathOperator{\lalg}{\mathrm lalg}

\DeclareMathOperator{\dett}{\mathrm det}
\DeclareMathOperator{\alg}{\mathrm alg}

\DeclareMathOperator{\soc}{\mathrm soc}

\DeclareMathOperator{\Rep}{\mathrm Rep}
\DeclareMathOperator{\sss}{\mathrm ss}

\DeclareMathOperator{\red}{\mathrm red}

\DeclareMathOperator{\st}{\mathrm st}
\DeclareMathOperator{\St}{\mathrm St}

\DeclareMathOperator{\Art}{\mathrm Art}

\DeclareMathOperator{\WD}{\mathrm WD}
\DeclareMathOperator{\W}{\mathrm W}
\DeclareMathOperator{\HT}{\mathrm HT}

\DeclareMathOperator{\Tor}{\mathrm Tor}
\DeclareMathOperator{\wt}{\mathrm wt}
\DeclareMathOperator{\cfs}{\mathrm fs}

\DeclareMathOperator{\tri}{\mathrm tri}
\DeclareMathOperator{\val}{\mathrm val}
\DeclareMathOperator{\univ}{\mathrm univ}
\DeclareMathOperator{\aut}{\mathrm aut}
\DeclareMathOperator{\gal}{\mathrm gal}
\DeclareMathOperator{\Def}{\mathrm Def}
\DeclareMathOperator{\Comp}{\mathrm Comp}

\DeclareMathOperator{\cts}{\mathrm cts}
\DeclareMathOperator{\reg}{\mathrm reg}
\DeclareMathOperator{\FM}{\mathrm FM}
\DeclareMathOperator{\Ord}{\mathrm Ord}
\DeclareMathOperator{\fini}{\mathrm finite}
\DeclareMathOperator{\can}{\mathrm can}
\DeclareMathOperator{\NOrd}{\mathrm NOrd}
\DeclareMathOperator{\nord}{\mathrm nord}
\DeclareMathOperator{\fin}{\mathrm fin}
\DeclareMathOperator{\lfin}{\mathrm lfin}
\DeclareMathOperator{\Mod}{\mathrm Mod}
\DeclareMathOperator{\aug}{\mathrm aug}
\DeclareMathOperator{\pro}{\mathrm pro}
\DeclareMathOperator{\nul}{\mathrm null}

\DeclareMathOperator{\rec}{\mathrm rec}
\DeclareMathOperator{\DF}{\mathrm DF}
\DeclareMathOperator{\diag}{\mathrm diag}
\DeclareMathOperator{\glo}{\mathrm gl}
\DeclareMathOperator{\ev}{\mathrm ev}
\DeclareMathOperator{\pLL}{\mathrm pLL}
\DeclareMathOperator{\ortho}{\mathrm ortho}
\DeclareMathOperator{\fg}{\mathrm fg}

\begin{document}

\title{Higher $\cL$-invariants for $\GL_3(\Q_p)$ and local-global compatibility}

\author[CB]{Christophe Breuil}
\ead[CB]{christophe.breuil@math.u-psud.fr}
\author[YD]{Yiwen Ding}
\ead[YD]{yiwen.ding@bicmr.pku.edu.cn}
\address[CB]{L.M.O., C.N.R.S., Universit\'e Paris-Sud,
Universit\'e Paris-Saclay, 91405 Orsay, France}
\address[YD]{B.I.C.M.R., Peking University,
No.5 Yiheyuan Road Haidian District,
Beijing, P.R. China 100871}

\begin{abstract}
Let $\rho_p$ be a $3$-dimensional $p$-adic semi-stable representation of $\Gal(\overline{\Q_p}/\Q_p)$ with Hodge-Tate weights $(0,1,2)$ (up to shift) and such that $N^2\ne 0$ on $D_{\st}(\rho_p)$. When $\rho_p$ comes from an automorphic representation $\pi$ of $G(\bA_{F^+})$ (for a unitary group $G$ over a totally real field $F^+$ which is compact at infinite places and $\GL_3$ at $p$-adic places), we show under mild genericity assumptions that the associated Hecke-isotypic subspaces of the Banach spaces of $p$-adic automorphic forms on $G(\bA_{F^+}^\infty)$ of arbitrary fixed tame level contain (copies of) a unique admissible finite length locally analytic representation of $\GL_3(\Q_p)$ of the form considered in \cite{Br16} which only depends on and completely determines $\rho_p$.
\end{abstract}

\maketitle
\setcounter{tocdepth}{4}
\tableofcontents
\numberwithin{equation}{section}

\section{Introduction and notation}\label{intronota}

\noindent
Let $p$ be a prime number, $n\geq 2$ an integer, $F^+$ a totally real number field and $F$ a totally imaginary quadratic extension of $F^+$ such that all places of $F^+$ dividing $p$ split in $F$. We fix a unitary algebraic group $G$ over $F^+$ which becomes $\GL_n$ over $F$ and such that $G(F^+\otimes_{\Q}\mathbb R)$ is compact and $G$ is split at all places above $p$. We also fix a place $\wp$ of $F^+$ above $p$. Then to each $\Q_p$-algebraic irreducible (finite dimensional) representation $W^{\wp}$ of $\prod_{v\mid p,v\ne \wp}G(F_v^+)$ over a finite extension $E$ of $\Q_p$ and to each prime-to-$\wp$ level $U^{\wp}$ in $G(\bA_{F^+}^{\infty,\wp})$, one can associate the Banach space of $p$-adic automorphic forms $\widehat{S}(U^{\wp}, W^{\wp})$ (see e.g. \S~\ref{prelprel}).\\

\noindent
If $\rho:\Gal(\overline F/F)\rightarrow \GL_n(E)$ is a continuous irreducible representation and $\widetilde\wp$ is a place of $F$ above $\wp$, one can consider the associated Hecke isotypic subspace $\widehat{S}(U^{\wp}, W^{\wp})[\fm_\rho]$, which is a continuous admissible representation of $G(F^+_{\wp})\buildrel\sim\over\rightarrow \GL_n(F_{\widetilde\wp})$ over $E$, or its locally $\Q_p$-analytic vectors $\widehat{S}(U^{\wp}, W^{\wp})[\fm_\rho]^{\an}$, which is an admissible locally $\Q_p$-analytic representation of $\GL_n(F_{\widetilde\wp})$. When nonzero, these representations of $\GL_n(F_{\widetilde\wp})$ are so far only understood when $n=2$ and $F_{\widetilde\wp}=\Q_p$ (\cite{Colm10a}, \cite{Em4}, \cite{Kis10}, \cite{CS2}, \cite{Colm14}, \cite{LXZ}, \cite{DLB}, \cite{CEG+}, ...). Indeed, though these representations are expected to be very rich, many results from $\GL_2(\Q_p)$ collapse (see e.g. \cite{Pascompl}, \cite{SchrB}) and it presently seems an almost impossible task to find a way to completely describe them in general. However, it is (quite reasonably) hoped that they {\it determine} the local Galois representation $\rho_{\widetilde\wp}:=\rho\vert_{\Gal(\overline F_{\widetilde\wp}/F_{\widetilde\wp})}$ and (may-be less reasonably) hoped that they also {\it only depend on} $\rho_{\widetilde\wp}$. Note that the special case where $\rho$ is automorphic is of particular interest, since then the subspace $\widehat{S}(U^{\wp}, W^{\wp})[\fm_\rho]^{\lalg}$ of locally $\Q_p$-algebraic vectors is nonzero, given by the classical local Langlands correspondence for $\GL_n(F_{\widetilde\wp})$ tensored by $\Q_p$-algebraic representations of $\GL_n(F_{\widetilde\wp})$.\\

\noindent
The aim of this work is to consolidate the above hopes in the case of $\GL_3(\Q_p)$. Let $\St_3^\infty$ be the usual smooth Steinberg representation of $\GL_3(\Q_p)$ and $v_{\overline P_i}^\infty=(\Ind_{\overline P_i(\Q_p)}^{\GL_3(\Q_p)}1)^\infty/1$ for $i=1,2$ the two smooth generalized Steinberg representations where $\overline P_1(\Q_p)=\Big(\begin{smallmatrix} * & * & 0 \\ * & * & 0\\ * & *&* \end{smallmatrix}\Big)$ and $\overline P_2(\Q_p):=\Big(\begin{smallmatrix} * & 0 & 0 \\ *&*&*\\ * & *&*\end{smallmatrix}\Big)$. Our main result is the following.

\begin{theorem}[Corollary \ref{main}]\label{mainintro}
Assume $p\geq 5$, $n=3$, $F_{\widetilde{\wp}}=\Q_p$ and $U^{\wp}=\prod_{v\ne \wp}U_v$ with $U_v$ maximal if $v\vert p$, $v\ne \wp$. Assume moreover that:
\begin{itemize}
\item $\overline\rho$ is absolutely irreducible
\item $\widehat{S}(U^{\wp}, W^{\wp})[\fm_{\rho}]^{\lalg}\neq 0$
\item $\rho_{\widetilde{\wp}}$ is semi-stable with consecutive Hodge-Tate weights and $N^2\ne 0$ on $D_{\st}(\rho_{\widetilde{\wp}})$
\item any dimension $2$ subquotient of $\overline \rho_{\widetilde{\wp}}:=\overline{\rho}|_{\Gal(\overline F_{\widetilde\wp}/F_{\widetilde\wp})}$ is nonsplit.
\end{itemize}
\noindent
Then $\widehat{S}(U^{\wp}, W^{\wp})[\fm_{\rho}]$ contains (copies of) a {\rm unique} locally analytic representation $\Pi \otimes \chi\!\circ\!\dett$ of $\GL_3(\Q_p)$ with $\chi$ a locally algebraic character of $\Q_p^\times$ and $\Pi$ of the form:
\begin{equation}\label{pi}
\begin{gathered}
\begindc{\commdiag}[32]
 \obj(0,16)[a]{\!\!\!\!\!\!\!\!\!\!\!\!\!\!\!\!\!\!\!\!\!\!\!\!\!\!\!\!$\Pi\ \ \cong \ \ \St_3^{\infty}$}
  \obj(12,2)[b]{$C_{2,1}$}
  \obj(23,-6)[c]{$\widetilde{C}_{2,2}$}
  \obj(23,10)[f]{$v_{\overline{P}_2}^{\infty}$}
  \obj(34,2)[x]{$C_{2,3}$}
  \obj(45,10)[y]{$v_{\overline{P}_1}^{\infty}$}
  \obj(45,-6)[h]{$\widetilde{C}_{2,4}$}
  \obj(56,2)[i]{$C_{2,5}$}
  \obj(12,30)[d]{$C_{1,1}$}
  \obj(23,38)[e]{$\widetilde{C}_{1,2}$}
  \obj(23,22)[g]{$v_{\overline{P}_1}^{\infty}$}
  \obj(34,30)[u]{$C_{1,3}$}
  \obj(45,22)[v]{$v_{\overline{P}_2}^{\infty}$}
  \obj(45,38)[j]{$\widetilde{C}_{1,4}$}
  \obj(56,30)[k]{$C_{1,5}$}
  \mor{a}{b}{}[+1,\solidline]
   \mor{b}{c}{}[+1,\solidline]
      \mor{b}{f}{}[+1,\solidline]
        \mor{c}{x}{}[+1, \dashline]
             \mor{f}{x}{}[+1,\solidline]
   \mor{x}{y}{}[+1, \solidline]
   \mor{a}{d}{}[+1,\solidline]
   \mor{d}{e}{}[+1, \solidline]
\mor{e}{u}{}[+1, \dashline]
\mor{u}{v}{}[+1, \solidline]
\mor{g}{u}{}[+1, \solidline]
\mor{d}{g}{}[+1, \solidline]
\mor{h}{i}{}[+1,\dashline]
\mor{y}{i}{}[+1,\solidline]
\mor{x}{h}{}[+1,\solidline]
\mor{u}{j}{}[+1, \solidline]
\mor{j}{k}{}[+1,\dashline]
\mor{v}{k}{}[+1,\solidline]
\enddc.
\end{gathered}
\end{equation}
where the $C_{i,j}$, $\widetilde{C}_{i,j}$ are certain explicit irreducible subquotients of locally analytic principal series of $\GL_3(\Q_p)$ (see \S~\ref{allthecij} or \cite[\S~4.1]{Br16}), where $\St_3^\infty = \soc_{\GL_3(\Q_p)}\Pi$ and where $-$ (resp. the dashed line) means a nonsplit (resp. a possibly split) extension as subquotient. Moreover the representation $\Pi \otimes \chi\!\circ\!\dett$ completely determines and only depends on $\rho_{\widetilde{\wp}}$. In particular the locally analytic representation $\widehat{S}(U^{\wp}, W^{\wp})[\fm_{\rho}]^{\rm an}$ of $\GL_3(\Q_p)$, hence also the continuous representation $\widehat{S}(U^{\wp}, W^{\wp})[\fm_{\rho}]$, determine $\rho_{\widetilde{\wp}}$.
\end{theorem}

\noindent
In fact one proves the stronger result that the restriction morphism:
\begin{equation}\small\label{isointro}
\Hom_{\GL_3(\Q_p)}\big(\Pi\otimes \chi\!\circ\!\dett,\widehat{S}(U^{\wp}, W^{\wp})[\fm_{\rho}]^{\rm an}\big)\rightarrow \Hom_{\GL_3(\Q_p)}\big(\St_3^\infty\otimes \chi\!\circ\!\dett,\widehat{S}(U^{\wp}, W^{\wp})[\fm_{\rho}]^{\rm an}\big)
\end{equation}
is bijective. The third assumption in Theorem \ref{mainintro} implies that $\rho_{\widetilde{\wp}}$ is up to twist isomorphic to $\Big(\begin{smallmatrix} \varepsilon^2 & * & * \\ 0 &\varepsilon &*\\ 0 & 0&1\end{smallmatrix}\Big)$ where $\varepsilon$ is the cyclotomic character. Hence $\overline \rho_{\widetilde{\wp}}$ is up to twist isomorphic to $\Big(\begin{smallmatrix} \overline\varepsilon^2 & * & * \\ 0 &\overline\varepsilon &*\\ 0 & 0&1\end{smallmatrix}\Big)$, and the fourth assumption means that we require the two $*$ above the diagonal in $\overline \rho_{\widetilde{\wp}}$ to be nonzero, a kind of assumption which already appears in the $\GL_2(\Q_p)$ case (see e.g. \cite[Thm.~1.2.1]{Em4}).\\

\noindent
Without assuming $\overline\rho$ absolutely irreducible, consecutive Hodge-Tate weights and the above condition on $\overline \rho_{\widetilde{\wp}}$, {\it but} assuming $F^+=\Q$, $\rho$ absolutely irreducible and a slightly unpleasant condition on $\widehat{S}(U^{\wp}, W^{\wp})[\fm_{\rho}]^{\lalg}$ (see \cite[Rem.~6.2.2(ii)]{Br16}), it was proven in \cite[Thm.~6.2.1]{Br16} that $\widehat{S}(U^{\wp}, W^{\wp})[\fm_{\rho}]^{\rm an}$ contains (copies of) a unique locally analytic representation which has the same form as (\ref{pi}). However, nothing more was known of its possible link to $\rho_{\widetilde{\wp}}$. So the main novelty in Theorem \ref{mainintro} is that the $\GL_3(\Q_p)$-representation $\Pi\otimes \chi\!\circ\!\dett$ contains {\it exactly} the same information as the $\Gal(\overline{\Q_p}/\Q_p)$-representation $\rho_{\widetilde{\wp}}$. Note however that $\Pi\otimes \chi\!\circ\!\dett$ is presumably only a small part of the representation $\widehat{S}(U^{\wp}, W^{\wp})[\fm_{\rho}]^{\rm an}$. For instance one could push a little bit further the methods of this paper to prove that $\widehat{S}(U^{\wp}, W^{\wp})[\fm_{\rho}]^{\an}$ as in Theorem \ref{mainintro} in fact contains (copies of) a representation of the form $\widetilde\Pi\otimes \chi\!\circ\!\dett$ with:
\begin{equation}\label{piwithan}
\begin{gathered}
 \begindc{\commdiag}[32]
 \obj(0,10)[a]{\!\!\!\!\!\!\!\!\!\!\!\!\!\!\!\!\!\!\!\!\!\!\!\!\!\!\!\!$\widetilde\Pi\ \ \cong \ \ \St_3^{\an}$}
  \obj(12,2)[b]{$v_{\overline P_2}^{\an}$}
  \obj(42,2)[x]{$\big(\Ind_{\overline B(\Q_p)}^{\GL_3(\Q_p)}1\otimes \varepsilon^{-1}\otimes \varepsilon\big)^{\an}$}
  \obj(72,2)[y]{$v_{\overline P_1}^{\an}$}
  \obj(12,18)[e]{$v_{\overline P_1}^{\an}$}
  \obj(42,18)[u]{$\big(\Ind_{\overline B(\Q_p)}^{\GL_3(\Q_p)}\varepsilon^{-1}\otimes \varepsilon \otimes 1\big)^{\an}$}
  \obj(72,18)[v]{$v_{\overline P_2}^{\an}$}
  \mor{a}{b}{}[+1,\solidline]
  \mor{b}{x}{}[+1, \solidline]
  \mor{x}{y}{}[+1, \solidline]
  \mor{a}{e}{}[+1, \solidline]
  \mor{e}{u}{}[+1, \solidline]
  \mor{u}{v}{}[+1, \solidline]
\enddc
\end{gathered}
\end{equation}
which still determines and only depends on $\rho_{\widetilde{\wp}}$. In (\ref{piwithan}), we denote by $\St_3^{\an}$, resp. $v_{\overline P_i}^{\an}$, the locally analytic Steinberg, resp. generalized Steinberg, and by $(\Ind_{\overline B(\Q_p)}^{\GL_3(\Q_p)} \cdot)^{\an}$ the locally analytic principal series from lower triangular matrices. In fact the subrepresentation of $\Pi\otimes \chi\circ\dett$ without the constituents $\widetilde C_{i,4}$, $C_{i,5}$ ($i=1,2$) in Theorem \ref{mainintro} can be seen as the ``edge'' of the representation $\widetilde\Pi\otimes \chi\!\circ\!\dett$. But even adding those constituents to $\widetilde\Pi\otimes \chi\!\circ\!\dett$ (or more precisely $(\widetilde\Pi\otimes \chi\!\circ\!\dett)^{\oplus d}$ where $d:=\dim_E\Hom_{\GL_3(\Q_p)}(\St_3^\infty\otimes \chi\!\circ\!\dett,\widehat{S}(U^{\wp}, W^{\wp})[\fm_{\rho}]^{\rm an})$), we are presumably still far from the full representation $\widehat{S}(U^{\wp}, W^{\wp})[\fm_{\rho}]^{\an}$.\\

\noindent
Theorem \ref{mainintro} (in its stronger form as above) is in fact a special case of a conjecture in arbitrary (distinct) Hodge-Tate weights. In \S~\ref{allthecij}, we show that one can associate to $\rho_{\widetilde{\wp}}$, assumed semi-stable with $N^2\ne 0$ on $D_{\st}(\rho_{\widetilde{\wp}})$ and sufficiently generic (we explain this below, any $\rho_{\widetilde{\wp}}$ as in Theorem \ref{mainintro} is sufficiently generic), a locally analytic representation $\Pi(\rho_{\widetilde{\wp}})=\Pi \otimes \chi\circ\dett$ of $\GL_3(\Q_p)$ containing the same information as $\rho_{\widetilde{\wp}}$ where $\Pi$ has the same form as (\ref{pi}) but replacing $\St_3^\infty$, $v_{\overline P_i}^\infty$ by $\St_3^\infty(\lambda)=:\St_3^\infty\otimes_EL(\lambda)$, $v_{\overline P_i}^\infty(\lambda):=v_{\overline P_i}^\infty\otimes_EL(\lambda)$. Here $L(\lambda)$ is the algebraic representation of $\GL_3(\Q_p)$ of highest weight $\lambda=k_1\geq k_2\geq k_3$ where $k_1>k_2-1>k_3-2$ are the Hodge-Tate weights of $\rho_{\widetilde{\wp}}$. We conjecture the following statement.

\begin{conjecture}[Conjecture \ref{THEconjecture}]\label{conjectureintro}
Assume $n=3$, $F_{\widetilde{\wp}}=\Q_p$ and:
\begin{itemize}
\item $\rho$ absolutely irreducible
\item $\widehat{S}(U^{\wp}, W^{\wp})[\fm_{\rho}]^{\lalg}\neq 0$
\item $\rho_{\widetilde{\wp}}$ semi-stable with $N^2\ne 0$ on $D_{\st}(\rho_{\widetilde{\wp}})$ and sufficiently generic.
\end{itemize}
Then the following restriction morphism is bijective:
\begin{equation*}
\Hom_{\GL_3(\Q_p)}\!\big(\Pi(\rho_{\widetilde{\wp}}),\widehat{S}(U^{\wp}, W^{\wp})[\fm_{\rho}]^{\an}\big)
\! \xlongrightarrow{\sim} \!\Hom_{\GL_3(\Q_p)}\!\big(\!\St_3^\infty\!\otimes_EL(\lambda)\otimes \chi\circ\dett, \widehat{S}(U^{\wp}, W^{\wp})[\fm_{\rho}]\big).
\end{equation*}
\end{conjecture}

\noindent
We now sketch the proof of Theorem \ref{mainintro}.\\

\noindent
The preliminary step, which is purely local and holds for arbitrary distinct Hodge-Tate weights, is the definition of $\Pi(\rho_{\widetilde{\wp}})$. Since $N^2\ne 0$, the $(\varphi,\Gamma)$-module $D:=D_{\rig}(\rho_{\widetilde{\wp}})$ over the Robba ring with $E$-coefficients $\cR_E$ can be uniquely written as $\cR_E(\delta_1)-\cR_E(\delta_2)-\cR_E(\delta_3)$ for some locally algebraic characters $\delta_i:\Q_p^\times\rightarrow E^\times$ (where, as usual, $\cR_E(\delta_1)$ is a submodule, $\cR_E(\delta_3)$ a quotient and $-$ means a nonsplit extension). We assume that the triangulation $(\cR_E(\delta_1),\cR_E(\delta_2),\cR_E(\delta_3))$ is noncritical, equivalently the Hodge-Tate weight of $\delta_i$ is $k_i-(i-1)$. Twisting $D_{\rig}(\rho_{\widetilde{\wp}})$ if necessary (and twisting $\Pi(\rho_{\widetilde{\wp}})$ accordingly), we can assume $\delta_1=x^{k_1}$, $\delta_2=x^{k_2}\varepsilon^{-1}$ and $\delta_3=x^{k_3}\varepsilon^{-2}$ (note that $D$ is not \'etale anymore if $k_1\ne 0$, but this won't be a problem). By the recipe for $\GL_2(\Q_p)$, one can associate to $D_1^2:=\cR_E(\delta_1)-\cR_E(\delta_2)$ and $D_2^3:=\cR_E(\delta_2)-\cR_E(\delta_3)$ locally analytic representations $\pi_{1,2}$ and $\pi_{2,3}$ of $\GL_2(\Q_p)$. Then the representations:
\begin{equation*}
 \begindc{\commdiag}[32]
 \obj(0,10)[a]{$\St_3^{\infty}(\lambda)$}
 \obj(14,10)[b]{$C_{1,1}$}
 \obj(26,4)[c]{$v_{\overline P_1}^\infty(\lambda)$}
 \obj(26,16)[d]{$\widetilde{C}_{1,2}$}
 \obj(38, 10)[e]{$C_{1,3}$,}
   \mor{a}{b}{}[+1,\solidline]
   \mor{b}{c}{}[+1,\solidline]
   \mor{b}{d}{}[+1,\solidline]
   \mor{c}{e}{}[+1,\solidline]
   \mor{d}{e}{}[+1, \dashline]
 \enddc \ \ \ \ \ \ \ \ \
 \begindc{\commdiag}[32]
 \obj(0,10)[a]{$\St_3^{\infty}(\lambda)$}
 \obj(14,10)[b]{$C_{2,1}$}
 \obj(26,4)[c]{$v_{\overline P_1}^\infty(\lambda)$}
 \obj(26,16)[d]{$\widetilde{C}_{2,2}$}
 \obj(38, 10)[e]{$C_{2,3}$,}
   \mor{a}{b}{}[+1,\solidline]
   \mor{b}{c}{}[+1,\solidline]
   \mor{b}{d}{}[+1,\solidline]
   \mor{c}{e}{}[+1,\solidline]
   \mor{d}{e}{}[+1, \dashline]
 \enddc
 \end{equation*}
can be defined as subquotients of the locally analytic parabolic inductions $(\Ind_{\overline P_1(\Q_p)}^{\GL_3(\Q_p)}\pi_{1,2}\otimes \delta_3\varepsilon^2)^{\an}$ and $(\Ind_{\overline P_2(\Q_p)}^{\GL_3(\Q_p)}\delta_1\otimes(\pi_{2,3}\otimes \varepsilon\circ \dett_{\GL_2}))^{\an}$ respectively, see \S~\ref{sec: hL-pI}. Note that these two representations (together) contain what we call the two ``simple'' $\cL$-invariants of $\rho_{\widetilde{\wp}}$ (given by the Hodge filtration on the $2$-dimensional filtered $(\varphi,N)$-modules associated to $D_1^2$ and $D_2^3$). We consider the two following representations (see \S~\ref{linvariantgl3} where they are denoted $\Pi^1(\lambda,\psi)^+$ and $\Pi^2(\lambda,\psi)^+$):
$$\ \ \ \ \ \ \ \ \ \
 \begindc{\commdiag}[32]
 \obj(0,10)[a]{$\!\!\!\!\!\!\!\!\!\!\!\!\!\!\!\!\!\!\!\!\!\!\!\!\Pi^{1}\ \ :=\ \ \St_3^{\infty}(\lambda)$}
  \obj(12,2)[b]{$C_{2,1}$}
  \obj(12,18)[c]{$C_{1,1}$}
  \obj(24,10)[d]{$v_{\overline P_1}^\infty(\lambda)$}
  \obj(36,18)[e]{$C_{1,3}$}
  \obj(24,26)[f]{$\widetilde{C}_{1,2}$}
  \mor{a}{b}{}[+1,\solidline]
   \mor{a}{c}{}[+1,\solidline]
   \mor{c}{d}{}[+1,\solidline]
   \mor{c}{f}{}[+1,\solidline]
   \mor{d}{e}{}[+1, \solidline]
   \mor{f}{e}{}[+1, \dashline]
 \enddc \ \ \ \ \ \ \ \ \ \ \ \ \ \ \ \ \ \ \
 \begindc{\commdiag}[32]
 \obj(0,10)[a]{$\!\!\!\!\!\!\!\!\!\!\!\!\!\!\!\!\!\!\!\!\!\!\!\!\Pi^{2}\ \ :=\ \ \St_3^{\infty}(\lambda)$}
  \obj(12,2)[b]{$C_{2,1}$}
  \obj(12,18)[c]{$C_{1,1}$}
  \obj(26,10)[d]{$v_{\overline P_2}^\infty(\lambda)$}
  \obj(26,-6)[f]{$\widetilde{C}_{2,2}$}
  \obj(40,2)[e]{$C_{2,3}.$}
  \mor{a}{b}{}[+1,\solidline]
   \mor{a}{c}{}[+1,\solidline]
   \mor{b}{d}{}[+1,\solidline]
   \mor{d}{e}{}[+1, \solidline]
    \mor{b}{f}{}[+1, \solidline]
     \mor{f}{e}{}[+1, \dashline]
 \enddc
$$
We say that $D$ is sufficiently generic if there are canonical isomorphisms (induced by Colmez's functor \cite{Colm10a}):
\begin{equation}\small\label{ext1intro}
\Ext^1_{{\GL_2}(\Q_p)}(\pi_{1,2},\pi_{1,2})\buildrel\sim\over\longrightarrow \Ext^1_{(\varphi,\Gamma)}(D_1^2, D_1^2)\ {\rm and}\ \Ext^1_{{\GL_2}(\Q_p)}(\pi_{2,3},\pi_{2,3})\buildrel\sim\over\longrightarrow \Ext^1_{(\varphi,\Gamma)}(D_2^3, D_2^3)
\end{equation}
satisfying the properties of Hypothesis \ref{hypo: hL-pLL0} in the text. We prove in Lemma \ref{lem: hL-etale}, Proposition \ref{prop: hL-gl2pLL} and Proposition \ref{prop: hL-gl2pLL2} that such isomorphisms are true under mild genericity assumptions on the $(\varphi,\Gamma)$-modules $D_1^2$ and $D_2^3$. Note that we couldn't find these isomorphisms in the literature (though we suspect they might be known), so we provided our own proofs, see e.g. the proof of Proposition \ref{prop: hL-gl2pLL2} in the appendix, where we go through the Galois side and use deformation theory, which forces the aforementioned mild genericity assumptions. Using these isomorphisms, we then prove that there are canonical perfect pairings of $3$-dimensional $E$-vector spaces (see Theorem \ref{thm: hL-L3}):
\begin{eqnarray}
 \Ext^1_{(\varphi,\Gamma)}\big(\cR_E(\delta_3),D_1^2\big) &\times &\Ext^1_{{\GL_3}(\Q_p)}\big(v_{\overline{P}_2}^{\infty}(\lambda),\Pi^1\big) \ \longrightarrow \ E \label{keypairing}\\
 \Ext^1_{(\varphi,\Gamma)}\big(D_2^3,\cR_E(\delta_1)\big) &\times &\Ext^1_{{\GL_3}(\Q_p)}\big(v_{\overline{P}_1}^{\infty}(\lambda),\Pi^2\big) \ \longrightarrow \ E. \label{keypairing2}
\end{eqnarray}
For instance (\ref{keypairing}) comes from a perfect pairing $\Ext^1_{(\varphi,\Gamma)}(\cR_E(\delta_3),D_1^2\big) \times \Ext^1_{(\varphi,\Gamma)}\!\big(D_1^2,\cR_E(\delta_2))\! \rightarrow E$ and an isomorphism $\Ext^1_{{\GL_3}(\Q_p)}(v_{\overline{P}_2}^{\infty}(\lambda),\Pi^1)\buildrel\sim\over\rightarrow \Ext^1_{(\varphi,\Gamma)}(D_1^2,\cR_E(\delta_2))$ induced by (the first isomorphism in) (\ref{ext1intro}) and locally analytic parabolic induction (see (\ref{equ: hL-L3i})). The $(\varphi,\Gamma)$-module $D$ gives an $E$-line in the left hand side of both (\ref{keypairing}), (\ref{keypairing2}), hence its orthogonal space gives a $2$-dimensional subspace of $\Ext^1_{{\GL_3}(\Q_p)}(v_{\overline{P}_2}^{\infty}(\lambda),\Pi^1)$ and a $2$-dimensional subspace of $\Ext^1_{{\GL_3}(\Q_p)}(v_{\overline{P}_1}^{\infty}(\lambda),\Pi^2)$. Choosing a basis of each subspace and amalgamating as much as possible the four corresponding extensions produces a well-defined locally analytic representation of the form (see (\ref{sansc5})):
$$\ \ \ \ \ \ \ \ \ \ \ \ \ \
 \begin{gathered}
\begindc{\commdiag}[32]
 \obj(0,16)[a]{\!\!\!\!\!\!\!\!\!\!\!\!\!\!\!\!\!\!\!\!\!\!\!\!\!\!\!\!$\Pi(D)^-\ \ \cong \ \ \St_3^{\infty}(\lambda)$}
  \obj(12,2)[b]{$C_{2,1}$}
  \obj(23,-6)[c]{$\widetilde{C}_{2,2}$}
  \obj(23,10)[f]{$v_{\overline{P}_2}^{\infty}(\lambda)$}
  \obj(34,2)[x]{$C_{2,3}$}
  \obj(49,2)[y]{$v_{\overline{P}_1}^{\infty}(\lambda)$}
  \obj(12,30)[d]{$C_{1,1}$}
  \obj(23,38)[e]{$\widetilde{C}_{1,2}$}
  \obj(23,22)[g]{$v_{\overline{P}_1}^{\infty}(\lambda)$}
  \obj(34,30)[u]{$C_{1,3}$}
  \obj(49,30)[v]{$v_{\overline{P}_2}^{\infty}(\lambda)$}
  \mor{a}{b}{}[+1,\solidline]
   \mor{b}{c}{}[+1,\solidline]
      \mor{b}{f}{}[+1,\solidline]
        \mor{c}{x}{}[+1, \dashline]
             \mor{f}{x}{}[+1,\solidline]
   \mor{x}{y}{}[+1, \solidline]
   \mor{a}{d}{}[+1,\solidline]
   \mor{d}{e}{}[+1, \solidline]
\mor{e}{u}{}[+1, \dashline]
\mor{u}{v}{}[+1, \solidline]
\mor{g}{u}{}[+1, \solidline]
\mor{d}{g}{}[+1, \solidline] \enddc
\end{gathered}
$$
which turns out to determine and only depend on $D$. Then results of \cite{Br16} show that there is a unique way to add constituents $\widetilde{C}_{1,4}$, $\widetilde{C}_{2,4}$, $C_{1,5}$, $C_{2,5}$ on the right so that the resulting representation $\Pi(D)=\Pi(\rho_{\widetilde{\wp}})$ contains $\Pi(D)^-$ and has the same form as (\ref{pi}) (see (\ref{piDplus})).\\

\noindent
We now assume $k_1=k_2=k_3$ and recall that $\rho_{\widetilde{\wp}}$ is then upper triangular. The strategy of the proof of Theorem \ref{mainintro} is the same as that of \cite{Ding3}, \cite{Ding4} when $n=2$ and $F_{\widetilde{\wp}}$ is arbitrary, and is entirely based on infinitesimal deformations. Very roughly, we replace the diagonal torus $\GL_1\times \GL_1$ in the arguments of \cite{Ding3} by the two Levi $L_{\overline P_1}=\GL_2\times \GL_1$ and $L_{\overline P_2}=\GL_1\times \GL_2$, and we deal with the $\GL_2$-factors using the $p$-adic local Langlands correspondence for $\GL_2(\Q_p)$.\\

\noindent
Following Emerton's local-global compatibility work for $\GL_2(\Q_p)$ (\cite{Em4}), we first study the localized modules $\Ord_{P_i}(\widehat{S}(U^{\wp}, W^{\wp})_{\overline\rho})$, $i=1,2$ where $\Ord_{P_i}$ is Emerton's ordinary functor (\cite{EOrd1}, \cite{EOrd2}) with respect to the parabolic subgroup $P_i(\Q_p)$ of $\GL_3(\Q_p)$ opposite to $\overline P_i(\Q_p)$. We show that $\Ord_{P_i}(\widehat{S}(U^{\wp}, W^{\wp})_{\overline\rho})$ is a faithful module over a certain $p$-adic localized Hecke algebra $\widetilde{\bT}(U^{\wp})_{\overline{\rho}}^{P_i-\ord}$ (see Lemma \ref{lem: Pord-reduce}) and using the $p$-adic local Langlands correspondence for $\GL_2(\Q_p)$ over deformation rings (as in \cite{Kis10} or \cite{Pas13}, see also the appendix), we define a continuous admissible representation $\pi^{\otimes}_{P_i}(U^{\wp})$ of $L_{P_i}(\Q_p)$ over $\widetilde{\bT}(U^{\wp})_{\overline{\rho}}^{P_i-\ord}$ (see (\ref{equ: lg1-piup})) and a canonical ``evaluation'' morphism:
$$X_{P_i}(U^{\wp}) \otimes_{\widetilde{\bT}(U^{\wp})_{\overline{\rho}}^{P_i-\ord}}\pi^{\otimes}_{P_i}(U^{\wp})\longrightarrow \Ord_{P_i}(\widehat{S}(U^{\wp}, W^{\wp})_{\overline\rho})$$
where $X_{P_i}(U^{\wp})$ is the $\widetilde{\bT}(U^{\wp})_{\overline{\rho}}^{P_i-\ord}$-module $\Hom_{\widetilde{\bT}(U^{\wp})_{\overline{\rho}}^{P_i-\ord}[L_{P_i}(\Q_p)]}^{\cts}\!(\pi^{\otimes}_{P_i}\!(U^{\wp}), \!\Ord_{P_i}\!(\widehat{S}(U^{\wp}, \!\bW^{\wp})_{\overline{\rho}}))$, $\bW^{\wp}$ being an invariant $\co_E$-lattice in the algebraic representation $W^\wp$ (see (\ref{xpup})).\\

\noindent
Twisting if necessary, we can assume $k_1=k_2=k_3=0$. We want to prove that the restriction morphism (\ref{isointro}) (with $\chi=1$ now) is bijective. Injectivity is not difficult, the hard part is surjectivity. Let $w$ be a nonzero vector in the subspace $D^\perp$ of $\Ext^1_{{\GL_3}(\Q_p)}(v_{\overline{P}_{3-i}}^{\infty},\Pi^i)$ orthogonal to $D$ under the pairings (\ref{keypairing}), (\ref{keypairing2}) and denote by $\Pi^w$ the corresponding extension $\Pi^i-v_{\overline{P}_{3-i}}^{\infty}$. It is enough to prove that the following restriction morphism is surjective for $i=1,2$ and any such $w$:
\begin{equation}\small\label{isointro2}
\Hom_{\GL_3(\Q_p)}\big(\Pi^w,\widehat{S}(U^{\wp}, W^{\wp})[\fm_{\rho}]^{\rm an}\big)\longrightarrow \Hom_{\GL_3(\Q_p)}\big(\St_3^\infty,\widehat{S}(U^{\wp}, W^{\wp})[\fm_{\rho}]^{\rm an}\big).
\end{equation}

\noindent
We now assume $i=1$, the case $i=2$ being symmetric. Taking ordinary parts induces an isomorphism (see the first isomorphism in (\ref{equ: lg2-isomos})):
\begin{equation}\footnotesize\label{firstisointro}
\Hom_{\GL_3(\Q_p)}\big(\St_3^\infty, \widehat{S}(U^{\wp}, W^{\wp})_{\overline{\rho}}[\fm_{\rho}]^{\an}\big)
\xlongrightarrow{\sim} \Hom_{L_{P_1}(\Q_p)}\big(\St_2^\infty \boxtimes 1,(\Ord_{P_1}(\widehat{S}(U^{\wp}, W^{\wp})_{\overline{\rho}}[\fm_{\rho}]))^{\an}\big)
\end{equation}
where $\boxtimes$ is the exterior tensor product ($\GL_2(\Q_p)$ acting on the left and $\Q_p^\times$ on the right). By a variation/generalization of the arguments in the $\GL_2(\Q_p)$-case, we prove that the restriction induces an isomorphism (see Corollary \ref{coro: lg2-simpL}):
\begin{multline}\label{secondisointro}
\Hom_{L_{P_1}(\Q_p)}\big(\pi_{1,2}\boxtimes 1, (\Ord_{P_1}(\widehat{S}(U^{\wp}\!, W^{\wp})_{\overline{\rho}}[\fm_{\rho}]))^{\an}\big)\\
\xlongrightarrow{\sim} \Hom_{L_{P_1}(\Q_p)}\big(\St_2^\infty \!\boxtimes 1,(\Ord_{P_1} (\widehat{S}(U^{\wp}, W^{\wp})_{\overline{\rho}}[\fm_{\rho}]))^{\an}\big).
\end{multline}
Note that the isomorphism (\ref{secondisointro}) involves the ``simple'' $\cL$-invariant contained in $D_1^2$ and is thus already nontrivial.\\

\noindent
Denote by $V_\rho$ the tangent space of $\Spec (\widetilde{\bT}(U^{\wp})_{\overline{\rho}}^{P_1-\ord}[1/p])$ at the closed point associated to the Galois representation $\rho$. Going through Galois deformation rings, one can prove that there is a canonical morphism of $E$-vector spaces $d\omega_{1,\rho}^+:V_\rho\longrightarrow \Ext^1_{(\varphi,\Gamma)}(D_1^{2},D_1^{2})$ such that the image of the composition $d\omega_{1,\rho}^+:V_\rho\rightarrow \Ext^1_{(\varphi,\Gamma)}(D_1^{2},D_1^{2})\twoheadrightarrow \Ext^1_{(\varphi,\Gamma)}(D_1^{2},\cR_E(\delta_2))$ is exactly $D^\perp$ (see Proposition \ref{prop: lg2-L-inv}). The proof of this statement is based on two main ingredients. The first one (see Theorem \ref{thm: hL-hL}) says that any extension $D_1^2-D_1^2$ which is contained as a $(\varphi, \Gamma)$-submodule in an extension $D-D$ is sent (after a suitable twist) to an element of $D^\perp$ via $\Ext^1_{(\varphi,\Gamma)}(D_1^{2},D_1^{2})\twoheadrightarrow \Ext^1_{(\varphi,\Gamma)}(D_1^{2},\cR_E(\delta_2))$ (the analogue of this statement in dimension $2$ was first observed by Greenberg and Stevens \cite[Thm.~3.14]{GS93}, see also \cite{Colm10}). It gives that the image is contained in $D^\perp$. The second ingredient is a lower bound on $\dim_EV_\rho$ (see Proposition \ref{prop: lg1-dim}) which implies that the image must be all of $D^\perp$.\\

\noindent
The vector $w$ is thus the image of a vector $v\in V_\rho$ via the above surjection $d\omega_{1,\rho}^+:V_\rho \twoheadrightarrow D^\perp$, and by definition of $V_\rho$, $v$ is an $E[\epsilon]/\epsilon^2$-valued point of $\Spec (\widetilde{\bT}(U^{\wp})_{\overline{\rho}}^{P_1-\ord})$. Denote by $\cI_v$ the corresponding ideal of $\widetilde{\bT}(U^{\wp})_{\overline{\rho}}^{P_1-\ord}$, by a generalization of arguments due to Chenevier (\cite{Che11}), one can prove that the $E[\epsilon]/\epsilon^2$-module $X_{P_1}(U^\wp)[\cI_v][1/p]$ of vectors in $X_{P_1}(U^\wp)[1/p]$ killed by $\cI_v$ is free of finite rank (see (\ref{equ: lg2-tgfree})). This implies that any $L_{P_1}(\Q_p)$-equivariant morphism $\pi_{1,2}\boxtimes 1\rightarrow (\Ord_{P_1}(\widehat{S}(U^{\wp}, W^{\wp})_{\overline{\rho}}[\fm_{\rho}]))^{\an}$ extends to an $E[\epsilon]/\epsilon^2$-linear and $L_{P_1}(\Q_p)$-equivariant morphism $\widetilde \pi_{1,2}\boxtimes_{E[\epsilon]/\epsilon^2} \widetilde 1\rightarrow (\Ord_{P_1}(\widehat{S}(U^{\wp}, W^{\wp})_{\overline{\rho}}[\cI_v]))^{\an}$ where $\widetilde \pi_{1,2}\boxtimes_{E[\epsilon]/\epsilon^2} \widetilde 1:=\pi^{\otimes}_{P_1}(U^{\wp})\otimes_{\widetilde{\bT}(U^{\wp})_{\overline{\rho}}^{P_1-\ord}}(\widetilde{\bT}(U^{\wp})_{\overline{\rho}}^{P_1-\ord}/\cI_v)[1/p]$. Note that $\widetilde \pi_{1,2}$ (resp. $\widetilde 1$) is an extension of $\pi_{1,2}$ (resp. $1$) by itself. By the adjunction formula for $\Ord_{P_1}$, we obtain a $\GL_3(\Q_p)$-equivariant morphism:
\begin{equation}\label{indp1}
\big(\Ind_{\overline P_1(\Q_p)}^{\GL_3(\Q_p)}\widetilde \pi_{1,2}\boxtimes_{E[\epsilon]/\epsilon^2} \widetilde 1\big)^{\an}\longrightarrow \widehat{S}(U^{\wp}, W^{\wp})_{\overline{\rho}}[\cI_v]^{\an}.
\end{equation}
The representation $\Pi^w$ is a multiplicity free subquotient of $(\Ind_{\overline P_1(\Q_p)}^{\GL_3(\Q_p)}\widetilde \pi_{1,2}\boxtimes_{E[\epsilon]/\epsilon^2} \widetilde 1)^{\an}$, and one can prove that (\ref{indp1}) induces a $\GL_3(\Q_p)$-equivariant morphism:
$$\Pi^w\longrightarrow \widehat{S}(U^{\wp}, W^{\wp})_{\overline{\rho}}[\fm_\rho]^{\an}\subseteq \widehat{S}(U^{\wp}, W^{\wp})_{\overline{\rho}}[\cI_v]^{\an}$$
which restricts to the unique morphism $\St_3^\infty\rightarrow  \widehat{S}(U^{\wp}, W^{\wp})_{\overline{\rho}}[\fm_\rho]^{\an}$ corresponding to $\pi_{1,2}\boxtimes 1\rightarrow (\Ord_{P_1}(\widehat{S}(U^{\wp}, W^{\wp})_{\overline{\rho}}[\fm_{\rho}]))^{\an}$ via (\ref{firstisointro}) and (\ref{secondisointro}) (see the proof of Theorem \ref{thm: lg2-GL3}). This proves the surjectivity of (\ref{isointro2}) (for $i=1$) and finishes the proof of Theorem \ref{mainintro}.\\

\noindent
This work raises several questions. For instance one can ask for a more explicit (local) construction of the $\GL_3(\Q_p)$-representation $\Pi(\rho_{\wp})$, and in particular try to relate the two ``branches'' in (\ref{pi}) to the filtered $(\varphi,N)$-module of $\rho_{\wp}$ along the lines of \cite[Conj.~6.1.2]{Br16}. Though there is so far no construction of (analogues of) $\Pi(\rho_{\wp})$ for $n\geq 4$, one can also still try to push further the results and methods of this paper in arbitrary dimension. Note that many of the intermediate results used in the proof of Theorem \ref{mainintro} are already proven here in a more general setting than just $\GL_3$. For instance we allow an arbitrary parabolic subgroup of $\GL_n$ in \S\S~\ref{ordinarypartfunctor}, \ref{galoisclassicallanglands}, \ref{automor}, \ref{domialgebraic} and we work in arbitrary dimension $n$ in all sections except \S\S~\ref{gl3gl3}, \ref{sec: lg2-1} and the appendix. We hope to come back to some of these questions in the future.\\

\noindent
We finally mention that some of our results in arbitrary dimension have an interest in their own. For instance Corollary \ref{coro: lg1-clas} gives new cases of classicality for certain $p$-adic automorphic forms with associated Galois representation which is de Rham at $\wp$ and Theorem \ref{thm: lg2-lg} gives a complete description (under certain assumptions) of the $P$-ordinary part of completed cohomology for a parabolic subgroup $P$ of $\GL_n$ with only $\GL_2$ or $\GL_1$ factors in its Levi subgroup, analogous to Emerton's description in the $\GL_2(\Q_p)$-case (\cite{Em4}). \\

\noindent
At the beginning of each section, the reader will find a sentence explaining its contents. We now give the main notation of the paper. In the whole text we denote by $E$ a finite extension of $\Q_p$, ${\mathcal O}_E$ its ring of integers, $\varpi_E$ a uniformizer of ${\mathcal O}_E$ and $k_E$ its residue field. Given an $E$-bilinear map $V \times W\xrightarrow{\cup} E$, for $W'\subseteq W$ we denote:
\begin{equation*}
(W')^{\perp}:=\{v\in V,\ v\cup w=0 \ \forall \ w\in W'\}.
\end{equation*}
For $L$ a finite extension of $\Q_p$, we let $\Sigma_L$ be the set of embeddings of $L$ into $E$ (equivalently into $\overline{\Q_p}$ by taking $E$ sufficiently large), $q_L:=|k_L|$ with $k_L$ the residue field of $L$, $\Gal_{L}:=\Gal(\overline L/L)$ the Galois group of $L$, $\W_L\subset \Gal_{L}$ the Weil group of $L$, and $\Gamma_L:=\mathrm{Gal}(L(\zeta_{p^n},n\geq 1)/L)$ where $(\zeta_{p^n})_{n\geq 1}$ is a compatible system of primitive $p^n$-th roots of $1$ in $\overline L$. When $L=\Q_p$ we write $\Gamma$ instead of $\Gamma_{\Q_p}$. We denote by $\varepsilon: \Gal_{L}\twoheadrightarrow \Gamma_L \ra E^{\times}$ the cyclotomic character with the convention $\HT_{\sigma}(\varepsilon)=1$ for all $\sigma\in \Sigma_L$ where $\HT_{\sigma}$ is the Hodge-Tate weight relative to the embedding $\sigma:L\hookrightarrow E$, and by $\overline\varepsilon$ its reduction modulo $p$. We normalize local class field theory by sending a uniformizer to a (lift of the) geometric Frobenius. In this way, we view characters of $\Gal_L$ as characters of $L^{\times}$ without further mention. We let $\unr(a)$ be the unramified character of $\Gal_{L}$ sending a uniformizer of $L^\times$ to $a\in E^\times$ and $|\cdot|:=\unr(q_L^{-1})$. We denote by $\val_p$ the valuation normalized by $\val_p(p)=1$.\\

\noindent
If $A$ is a finite dimensional $\Q_p$-algebra, for instance $A=E$ or $E[\epsilon]/\epsilon^2$ (the dual numbers), we write $\mathcal{R}_{A,L}$ for the Robba ring associated to $L$ with $A$-coefficients. When $L$ is fixed, we only write $\mathcal{R}_{A}$. We denote by $\Ext^i_{(\varphi,\Gamma_L)}(\cdot,\cdot)$ the extensions groups in the category of $(\varphi,\Gamma_L)$-modules over $\mathcal{R}_{E,L}$ and by $H^i_{(\varphi,\Gamma_L)}(\cdot):=\Ext^i_{(\varphi,\Gamma_L)}(\mathcal{R}_{E,L},\cdot)$ (\cite[\S~2.2.5]{Bch}, \cite{Liu07}, \cite{Colm2}). If $\delta:L^\times\rightarrow A^\times$ is a continuous character, we denote by $\mathcal{R}_{A,L}(\delta)$ the associated rank one $(\varphi,\Gamma_L)$-module (see \cite[Cons.~6.2.4]{KPX}). Thanks to local class field theory, we fix an isomorphism $\Ext^1_{(\varphi,\Gamma_L)}(\cR_E, \cR_E) \xrightarrow{\sim} \Hom(L^{\times}, E)$ where $\Hom(L^{\times}, E)$ is the $E$-vector space of continuous characters to the additive group $E$. For any continuous $\delta: L^{\times} \ra E^{\times}$, the twist by $\delta^{-1}$ induces a canonical isomorphism $\Ext^1_{(\varphi,\Gamma_L)}(\cR_E(\delta), \cR_E(\delta))\buildrel\sim\over\rightarrow \Ext^1_{(\varphi,\Gamma_L)}(\cR_E, \cR_E)$, and we deduce isomorphisms (uniformly in $\delta$):
\begin{equation}\label{equ: hL-cad}
\Ext^1_{(\varphi,\Gamma_L)}(\cR_E(\delta), \cR_E(\delta)) \xlongrightarrow{\sim} \Hom(L^{\times}, E).
\end{equation}
By \cite[Lem. 1.15]{Ding4}, the isomorphism (\ref{equ: hL-cad}) induces an isomorphism:
\begin{equation}\label{ginfty}
 \Ext^1_{g}(\cR_E(\delta), \cR_E(\delta)) \xlongrightarrow{\sim} \Hom_{\infty}(L^{\times}, E)
\end{equation}
where $\Ext^1_g$ denotes the subspace of extensions which are de Rham up to twist by characters, and $\Hom_{\infty}(L^{\times}, E)$ denotes the subspace of smooth characters. Finally, if $L=\Q_p$, we denote by $\wt(\delta)\in E$ the Sen weight of $\delta$, for instance $\wt(x^k\unr(a))=k$ for $k\in \Z$ and $a\in E^\times$.\\

\noindent
Let $G$ be the $L$-points of a reductive algebraic group over $\Q_p$, we refer without comment to \cite{ST02}, resp. \cite{ST03}, for the background on general, resp. admissible, locally $\Q_p$-analytic representations of $G$ over locally convex $E$-vector spaces, and to \cite{ST} for the background on continuous (admissible) representations of $G$ over $E$. If $V$ is a continuous representation of $G$ over $E$, we denote by $V^{\an}$ its locally $\Q_p$-analytic vectors (\cite[\S~7]{ST03}). If $V$ is a locally $\Q_p$-analytic representation of $G$ over $E$, we denote by $V^{\sm}$, resp. $V^{\lalg}$, the subrepresentation of its smooth vectors, resp. of its locally $\Q_p$-algebraic vectors (\cite[Def.~4.2.6]{Em04}). If $X$ and $Y$ are topological spaces, we denote by $\cC(X,Y)$ the set of continuous functions from $X$ to $Y$. If $P$ is the $L$-points of a parabolic subgroup of $G$ and $\pi_p$ is a continuous representation of $P$ over $E$, i.e. on a Banach vector space over $E$, we denote by:
$$(\Ind_P^G\pi_p)^{\cC^0}:=\{f:G\rightarrow \pi_P\ {\rm continuous},\ f(pg)=p(f(g))\}$$
the continuous parabolic induction endowed with the left action of $G$ by right translation on functions: $(gf)(g'):=f(gg')$. It is again a continuous representation of $G$ over $E$. Likewise, if $\pi_p$ is a locally analytic representation of $P$ on a locally convex $E$-vector space of compact type, we denote by:
$$(\Ind_P^G\pi_p)^{\an}:=\{f:G\rightarrow \pi_P\ {\rm locally}\ \Q_p{\rm -analytic},\ f(pg)=p(f(g))\}$$
the locally analytic parabolic induction endowed with the same left action of $G$. It is again a locally analytic representation of $G$ on a locally convex $E$-vector space of compact type (see e.g. \cite[Rem.~5.4]{Kohl}). If $\pi_P$ is a smooth representation of $P$ over $E$, we finally denote by $(\Ind_P^G\pi_p)^{\infty}$ the smooth parabolic induction (taking locally constant functions $f:G\rightarrow \pi_P$) endowed with the same $G$-action. We denote by $\delta_P$ the usual (smooth unramified) modulus character of $P$.\\

\noindent
If $V$, $W$ are two locally $\Q_p$-analytic representations of $G$ over $E$, we {\it define} the extension groups $\Ext^i_G(V,W)$ as in \cite[D\'ef.~3.1]{Sch11}, that is, as the extension groups $\Ext^i_{D(G,E)}(W^\vee,V^\vee)$ of their strong duals $V^\vee$, $W^\vee$ as algebraic $D(G,E)$-modules where $D(G,E)$ is the algebra of locally analytic $E$-valued distributions. If the center $Z$ of $G$ (or a subgroup $Z$ of the center of $G$) acts by the same locally analytic character on $V$ and $W$, we define the extension groups with that central character $\Ext^i_{G,Z}(V,W)$ as in \cite[(3.11)]{Sch11}, and there are then functorial morphisms $\Ext^i_{G,Z}(V,W)\rightarrow \Ext^i_{G}(V,W)$. If $V$, $W$ are smooth representations of $G$ over $E$, we denote by $\Ext^i_{G,\infty}(V,W)$ the usual extension groups in the category of smooth representations of $G$ over $E$ (see e.g. \cite[\S~2.1.3]{Da} or \cite[\S~3]{Orl}). Finally, if $(V_i)_{i=1,\cdots,r}$ are admissible locally analytic representations of $G$, the notation $V_1-V_2-V_3-\cdots-V_r$ means an admissible locally analytic representation of $G$ such that $V_1$ is a subobject, $V_2$ is a subobject of the quotient by $V_1$ etc. where each subquotient $V_i-V_{i+1}$ is a nonsplit extensions of $V_{i+1}$ by $V_i$.\\

\noindent
If $A$ is a commutative ring, $M$ an $A$ module and $I$ an ideal of $A$, we denote by $M[I]\subseteq M$ the $A$-submodule of elements killed by $I$ and by $M\{I\}:=\cup_{n\geq 1}M[I^n]$.\\

\noindent
Acknowledgement: The authors are very grateful to Y. Hu for many helpful discussions. For their answers to their questions, they also want to thank F. Herzig, Y. Hu, R. Liu, A. Minguez, C. Moeglin, Z. Qian, B. Schraen and D. Xu. Y. D. was supported by E.P.S.R.C. Grant EP/L025485/1 and by Grant No. 8102600240 from B.I.C.M.R.

\section{Higher $\cL$-invariants and deformations of $(\varphi,\Gamma)$-modules}\label{sec: hL-FM}

\noindent
In this section we define and study certain subspaces $\cL_{\FM}(D: D_1^{n-1})$ and $\ell_{\FM}(D: D_1^{n-1})$ of some Ext${}^1$ groups in the category of $(\varphi,\Gamma_L)$-modules that will be used in the next sections.\\

\noindent
We fix a finite extension $L$ of $\Q_p$ and write $\cR_E$ for $\cR_{E,L}$. Let $D$ be a trianguline $(\varphi,\Gamma_L)$-module over $\cR_E$ of arbitrary rank $n\geq 1$. We denote an arbitrary parameter of $D$ by $(\delta_1,\cdots, \delta_n)$ where the $\delta_i:L^\times\rightarrow E^\times$ are continuous characters (see e.g. \cite[\S~2.1]{BHS2}). Recall that $D$ can have several parameters, see {\it loc.cit.}

\begin{definition}\label{def: hL-spe}
We call a parameter $(\delta_1,\cdots, \delta_n)$ of $D$ special if we have:
$$\delta_{i} \delta_{i+1}^{-1}=|\cdot |\prod_{\sigma\in \Sigma_L} \sigma^{k_{\sigma,i}}\ \ \ \forall\ i\in\{1,\cdots, n-1\}$$
for some $k_{\sigma,i}\in \Z$.
\end{definition}

\noindent
We say that $(\delta_1,\cdots,\delta_n)$ as in Definition \ref{def: hL-spe} is noncritical if $k_{\sigma,i}\in \Z_{>0}$ for all $\sigma$, $i$. It follows from the proof of \cite[Prop.~2.3.4]{Bch} that a trianguline $D$ with a special noncritical parameter is de Rham up to twist. It then easily follows from Berger's equivalence of categories \cite[Thm.~A]{Berger2} that such a $D$ has only one special noncritical parameter. In the sequel when we say that $(D,(\delta_1,\cdots, \delta_n))$ is special noncritical, it means that $(\delta_1,\cdots, \delta_n)$ is the unique special noncritical parameter of $D$.\\

\noindent
We now fix a special noncritical $(D,(\delta_1,\cdots,\delta_n))$ and for $1\leq i<i'\leq n$ we denote by $D_i^{i'}$ the unique $(\varphi,\Gamma_L)$-module subquotient of $D$ of trianguline parameter $(\delta_i, \cdots, \delta_{i'})$. It is then clear that $(D_i^{i'},(\delta_i,\cdots,\delta_{i'}))$ is also a special noncritical $(\varphi,\Gamma_L)$-module.\\

\noindent
We first assume that for $i\in \{1,\cdots, n-2\}$ the extension of $\cR_E(\delta_{i+1})$ by $\cR_E(\delta_i)$ appearing as a subquotient of $D_1^{n-1}$ is nonsplit. We consider the following cup-product:
\begin{equation}\label{equ: hL-cup}
\Ext^1_{(\varphi,\Gamma_L)}(\cR_E(\delta_n), D_1^{n-1}) \times \Ext^1_{(\varphi,\Gamma_L)}(D_1^{n-1}, D_1^{n-1}) \xlongrightarrow{\cup} \Ext^2_{(\varphi,\Gamma_L)}(\cR_E(\delta_n), D_1^{n-1}).
\end{equation}

\begin{lemma}\label{equ: hL-dim}
We have $\dim_E \Ext^1_{(\varphi,\Gamma_L)}(\cR_E(\delta_n), D_1^{n-1})=(n-1)[L:\Q_p]+1$, and the surjection $D_1^{n-1}\twoheadrightarrow \cR_E(\delta_{n-1})$ induces an isomorphism:
\begin{equation*}
 \Ext^2_{(\varphi,\Gamma_L)}(\cR_E(\delta_n), D_1^{n-1})\xlongrightarrow{\sim} \Ext^2_{(\varphi,\Gamma_L)}(\cR_E(\delta_n), \cR_E(\delta_{n-1}))\cong E.
 \end{equation*}
\end{lemma}
\begin{proof}
The lemma follows easily from \cite[\S~2.2]{Na} (see also \cite{Liu07}).
\end{proof}

\noindent
By functoriality, we have the following commutative diagram of pairings:
\begin{equation}\footnotesize\label{equ: l3-diag2}
 \begin{CD}
 \Ext^1_{(\varphi,\Gamma_L)}(\cR_E(\delta_n), D_1^{n-1}) \ \ @. \times @. \Ext^1_{(\varphi,\Gamma_L)}(D_1^{n-1},D_1^{n-2}) @> \cup>> \Ext^2_{(\varphi,\Gamma_L)}(\cR_E(\delta_n), D_1^{n-2}) \\
 @| @. @V \iota VV @VVV \\
\Ext^1_{(\varphi,\Gamma_L)}(\cR_E(\delta_n), D_1^{n-1}) \ \ @. \times @. \Ext^1_{(\varphi,\Gamma_L)}(D_1^{n-1},D_1^{n-1}) @> \cup >> \Ext^2_{(\varphi,\Gamma_L)}(\cR_E(\delta_n), D_1^{n-1}) \\
 @| @. @V \kappa VV @V \sim VV \\
 \Ext^1_{(\varphi,\Gamma_L)}(\cR_E(\delta_n), D_1^{n-1}) \ \ @. \ \ \!\times \ @. \ \ \Ext^1_{(\varphi,\Gamma_L)}(D_1^{n-1}, \cR_E(\delta_{n-1})) @> \cup >> \Ext^2_{(\varphi,\Gamma_L)}(\cR_E(\delta_n), \cR_E(\delta_{n-1}))
 \end{CD}
\end{equation}
with the bottom right map being an isomorphism by Lemma \ref{equ: hL-dim}.

\begin{proposition}\label{prop-l3-cup}
Keep the above assumption and notation.\\
(1) The map $\kappa$ is surjective.\\
(2) The bottom cup-product in (\ref{equ: l3-diag2}) is a perfect pairing and we have:
$$\Ker (\kappa)=\Ext^1_{(\varphi,\Gamma_L)}(\cR_E(\delta_n), D_1^{n-1}) ^{\perp}$$
with respect to the middle cup-product in (\ref{equ: l3-diag2}).
\end{proposition}
\begin{proof}
(1) It is enough to show $\Ext^2_{(\varphi,\Gamma_L)}(D_1^{n-1}, D_1^{n-2})=0$. By d\'evissage it is enough to show that $\Ext^2_{(\varphi,\Gamma_L)}(D_1^{n-1}, \cR_E(\delta_i))=0$ for all $i\in \{1,\cdots, n-2\}$. We have a natural isomorphism:
\begin{equation*}
 \Ext^2_{(\varphi,\Gamma_L)}(D_1^{n-1}, \cR_E(\delta_i))\cong H^2_{(\varphi,\Gamma_L)}\big((D_1^{n-1})^{\vee}\otimes_{\cR_E} \cR_E(\delta_i)\big).
\end{equation*}
Together \ with \ \cite[\S~4.2]{Liu07} \ (see \ also \ \cite[Prop. 1.7(4)]{Ding4}), \ we \ are \ thus \ reduced \ to \ show $H^0_{(\varphi,\Gamma_L)}\big(D_1^{n-1}\otimes_{\cR_E} \cR_E(\delta_i^{-1} \varepsilon)\big)=0$ which follows easily from our assumption on $D_1^{n-1}$.\\

\noindent
(2) Using the natural isomorphisms:
\begin{eqnarray*}
\Ext^1_{(\varphi,\Gamma_L)}(D_1^{n-1}, \cR_E(\delta_{n-1})) &\cong &H^1_{(\varphi,\Gamma_L)}((D_1^{n-1})^{\vee} \otimes_{\cR_E} \cR_E(\delta_{n-1})),\\
 \Ext^1_{(\varphi,\Gamma_L)}(\cR_E(\delta_n), D_1^{n-1}) &\cong& H^1_{(\varphi,\Gamma_L)}(D_1^{n-1}\otimes_{\cR_E} \cR_E(\delta_n^{-1})),
\\
\Ext^2_{(\varphi,\Gamma_L)}(\cR_E(\delta_n), \cR_E(\delta_{n-1})) &\cong& H^2_{(\varphi,\Gamma_L)}(\cR_E(\delta_{n-1}\delta_n^{-1})),
\end{eqnarray*}
we are reduced to show for the first statement in (2) that the cup-product:
\begin{equation*}
H^1_{(\varphi,\Gamma_L)}((D_1^{n-1})^{\vee} \otimes_{\cR_E} \cR_E(\delta_{n-1})) \times H^1_{(\varphi,\Gamma_L)}(D_1^{n-1}\otimes_{\cR_E} \cR_E(\delta_n^{-1})) \xlongrightarrow{\cup} H^2_{(\varphi,\Gamma_L)}(\cR_E(\delta_{n-1}\delta_n^{-1}))
\end{equation*}
is a perfect pairing.
We have a commutative diagram:
\begin{equation}\footnotesize\label{equ: hL-cup2}
\begin{CD} H^1_{(\varphi,\Gamma_L)}((D_1^{n-1})^{\vee} \otimes_{\cR_E} \cR_E(\delta_{n-1}))@. \ \times @. H^1_{(\varphi,\Gamma_L)}(D_1^{n-1}\otimes_{\cR_E} \cR_E(\delta_n^{-1}))@> \cup >> H^2_{(\varphi,\Gamma_L)}(\cR_E(\delta_{n-1}\delta_n^{-1}))
\\
 @| @. @VVV @VVV \\
 H^1_{(\varphi,\Gamma_L)}((D_1^{n-1})^{\vee} \otimes_{\cR_E} \cR_E(\delta_{n-1})) @. \ \ \ \times \ \ @. H^1_{(\varphi,\Gamma_L)}(D_1^{n-1}\otimes_{\cR_E} \cR_E(\delta_{n-1}^{-1}\varepsilon)) @> \cup >> H^2_{(\varphi,\Gamma_L)}(\cR_E(\varepsilon))
\end{CD}
\end{equation}
where the two vertical maps on the right are induced by the injection $\cR_E(\delta_n^{-1})\hookrightarrow \cR_E(\delta_{n-1}^{-1} \varepsilon)$ (recall from Definition \ref{def: hL-spe} that we have $\delta_n^{-1}=\delta_{n-1}^{-1}\varepsilon \prod_{\sigma\in \Sigma_L} \sigma^{k_{\sigma,n-1}-1}$ with $k_{\sigma,n-1}-1\geq 0$). Moreover, using the same argument as in the proof of \cite[Lem. 1.13]{Ding4} (or by \cite[Lem. 4.8(i)]{BHS2} together with an easy d\'evissage argument), the vertical maps are isomorphisms. By Tate duality (see \cite[\S~4.2]{Liu07} or \cite[Prop. 1.7(4)]{Ding4}), the bottom cup-product in (\ref{equ: hL-cup2}) is a perfect pairing, hence so is the top cup-product. The first part of (2) follows.\\
By similar (and easier) arguments as in the proof of (1), we have $\Ext^2_{(\varphi,\Gamma_L)}(\cR_E(\delta_n), D_1^{n-2}) =0$. By (\ref{equ: l3-diag2}), we deduce:
\begin{equation*}
 \Ker (\kappa)\subseteq \Ext^1_{(\varphi,\Gamma_L)}(\cR_E(\delta_n), D_1^{n-1}) ^{\perp}.
\end{equation*}
However, since the bottom cup-product of (\ref{equ: l3-diag2}) is a perfect pairing and the bottom right map an isomorphism, we easily get $\Ext^1_{(\varphi,\Gamma_L)}(\cR_E(\delta_n), D_1^{n-1}) ^{\perp} \subseteq \Ker (\kappa)$, hence an equality.
\end{proof}

\noindent
The $(\varphi,\Gamma_L)$-module $D$ gives rise to a nonzero element in $\Ext^1_{(\varphi,\Gamma_L)}(\cR_E(\delta_n),D_1^{n-1})$ that we denote by $[D]$. In particular the $E$-vector subspace $E[D]$ it generates is well defined and we define (with respect to the two bottom pairings in (\ref{equ: l3-diag2})):
\begin{eqnarray*}
\cL_{\FM}(D: D_1^{n-1})&:=&(E [D])^{\perp} \subseteq \Ext^1_{(\varphi,\Gamma_L)}(D_1^{n-1},D_1^{n-1})\\
\ell_{\FM}(D: D_1^{n-1})&:=&(E [D])^{\perp} \subseteq \Ext^1_{(\varphi,\Gamma_L)}(D_1^{n-1},\cR_E(\delta_{n-1})).
 \end{eqnarray*}
By Proposition \ref{prop-l3-cup} (and the bottom right isomorphism in (\ref{equ: l3-diag2})), we deduce a short exact sequence of $E$-vector spaces:
\begin{equation}\label{sesker}
 0 \lra \Ker (\kappa) \lra \cL_{\FM}(D: D_1^{n-1}) \xlongrightarrow{\kappa} \ell_{\FM}(D: D_1^{n-1}) \lra 0.
\end{equation}

\noindent
The following corollary also follows easily from Proposition \ref{prop-l3-cup} and (\ref{sesker}).

\begin{corollary}\label{codimension}
(1) The $(\varphi,\Gamma_L)$-module $D$ (seen in $\Ext^1_{(\varphi,\Gamma_L)}(\cR_E(\delta_n),D_1^{n-1})$) is determined up to isomorphism by $D_1^{n-1}$, $\delta_n$ and $\cL_{\FM}(D: D_1^{n-1})$ (resp. and $\ell_{\FM}(D:D_1^{n-1})$).\\
(2) If $D$ (seen in $\Ext^1_{(\varphi,\Gamma_L)}(\cR_E(\delta_n), D_1^{n-1})$) is nonsplit, then $\cL_{\FM}(D: D_1^{n-1})$ (resp. $\ell_{\FM}(D: D_1^{n-1})$) is of codimension $1$ in $\Ext^1_{(\varphi,\Gamma_L)}(D_1^{n-1},D_1^{n-1})$ (resp. in $\Ext^1_{(\varphi,\Gamma_L)}(D_1^{n-1}, \cR_E(\delta_{n-1}))$).
\end{corollary}

\noindent
By functoriality we have a commutative diagram for $i<n-1$:
\begin{equation}\footnotesize\label{equ: hL-commL}
 \begin{CD}
\Ext^1_{(\varphi,\Gamma_L)}(\cR_E(\delta_n), D_1^{n-1}/D_1^i) @. \ \ \times \ \ @. \Ext^1_{(\varphi,\Gamma_L)}(D_1^{n-1}/D_1^i,\cR_E(\delta_{n-1})) @> \cup >> \Ext^2_{(\varphi,\Gamma_L)}(\cR_E(\delta_n), \cR_E(\delta_{n-1})) \\
 @A u_i AA @. @V j_i VV @| \\
  \Ext^1_{(\varphi,\Gamma_L)}(\cR_E(\delta_n), D_1^{n-1}) @.  \times @. \Ext^1_{(\varphi,\Gamma_L)}(D_1^{n-1}, \cR_E(\delta_{n-1})) @> \cup >> \Ext^2_{(\varphi,\Gamma_L)}(\cR_E(\delta_n), \cR_E(\delta_{n-1})).
  \end{CD}
\end{equation}
It is easy to deduce for $i<n-1$ from \cite[\S~4.2]{Liu07} (see also \cite[Prop. 1.7]{Ding4}):
\begin{equation*}
 \Ext^2_{(\varphi,\Gamma_L)}(\cR_E(\delta_n), D_1^i) =\Ext^2_{(\varphi,\Gamma_L)}(D_1^i, \cR_E(\delta_{n-1}))=0
\end{equation*}
and it is clear that $\Hom_{(\varphi,\Gamma_L)}(D_1^i,\cR_E(\delta_{n-1}))=\Hom_{(\varphi,\Gamma_L)}(\cR_E(\delta_n), D_1^{n-1}/D_1^i)=0$. By d\'evissage, we deduce that $u_i$ is surjective, $j_i$ is injective and $\Ker(u_i)\cong \Ext^1_{(\varphi,\Gamma_L)}(\cR_E(\delta_n), D_1^i)$. Also the two cup-products in (\ref{equ: hL-commL}) are perfect pairings by Proposition \ref{prop-l3-cup}. In particular we obtain the following lemma.

\begin{lemma}\label{lem: hL-compL}
We have in $\Ext^1_{(\varphi,\Gamma_L)}(D_1^{n-1}, \cR_E(\delta_{n-1}))$ for $i<n-1$ (via $j_i$):
\begin{equation}
\ell_{\FM}(D: D_1^{n-1})\cap \Ext^1_{(\varphi,\Gamma_L)}(D_1^{n-1}/D_1^i,\cR_E(\delta_{n-1}))=\ell_{\FM}(D/D_1^i: D_1^{n-1}/D_1^i)
\end{equation}
and with respect to the bottom pairing in (\ref{equ: hL-commL}):
$$\Ext^1_{(\varphi,\Gamma_L)}(D_1^{n-1}/D_1^i,\cR_E(\delta_{n-1}))\cong \Ext^1_{(\varphi,\Gamma_L)}(\cR_E(\delta_n), D_1^i)^{\perp}.$$
\end{lemma}

\begin{remark}\label{rem: hL-sim}
{\rm In particular, for $i=n-2$, we have a perfect pairing:
\begin{equation}\label{equ: hL-cupS}
 \Ext^1_{(\varphi,\Gamma_L)}(\cR_E(\delta_n), \cR_E(\delta_{n-1})) \times \Ext^1_{(\varphi,\Gamma_L)}(\cR_E(\delta_{n-1}), \cR_E(\delta_{n-1})) \lra E.
\end{equation}
Thanks to (\ref{equ: hL-cad}) we can thus view:
\begin{multline*}
\ell_{\FM}(D:D_1^{n-1})\cap \Ext^1_{(\varphi,\Gamma_L)}(\cR_E(\delta_{n-1}),\cR_E(\delta_{n-1}))=\ell_{\FM}(D/D_1^{n-2}:D_1^{n-1}/D_1^{n-2})
\\ =\cL_{\FM}(D/D_1^{n-2}: D_1^{n-1}/D_1^{n-2})\end{multline*}
as an $E$-vector subspace of $\Hom(L^{\times}, E)$ of codimension $\leq 1$. By \cite[Prop. 1.9]{Ding4}, the pairing (\ref{equ: hL-cupS}) induces an equality of subspaces of $\Ext^1_{(\varphi,\Gamma_L)}(\cR_E(\delta_{n-1}), \cR_E(\delta_{n-1}))$:
\begin{equation*}
  \Ext^1_{g}(\cR_E(\delta_{n-1}), \cR_E(\delta_{n-1})) \cong \Ext^1_e(\cR_E(\delta_n), \cR_E(\delta_{n-1}))^{\perp}
\end{equation*}
where $\Ext^1_e$ denotes the subspace of extensions which are crystalline up to twist by characters. In particular via (\ref{ginfty}) we have an inclusion in $\Ext^1_{(\varphi,\Gamma_L)}(\cR_E(\delta_{n-1}), \cR_E(\delta_{n-1}))$:
\begin{equation*}
 \Hom_{\infty}(L^{\times}, E)\subseteq \ell_{\FM}(D/D_1^{n-2}:D_1^{n-1}/D_1^{n-2})
\end{equation*}
if and only if $D/D_1^{n-2}$ is crystalline up to twist by characters.}
\end{remark}

\noindent
We now assume that for $i\in \{2,\cdots, n-1\}$ the extension of $\cR_E(\delta_{i+1})$ by $\cR_E(\delta_i)$ appearing as a subquotient of $D_2^{n}$ is nonsplit. Similarly to the two bottom lines of (\ref{equ: l3-diag2}) we have a commutative diagram of pairings:
\begin{equation}\small\label{equ: hL-CD3}\begin{CD}
 \Ext^1_{(\varphi,\Gamma_L)}(D_2^{n},D_2^{n}) @. \times \ @. \Ext^1_{(\varphi,\Gamma_L)}(D_2^{n},\cR_E(\delta_1)) @> \cup >> \Ext^2_{(\varphi,\Gamma_L)}(D_2^{n}, \cR_E(\delta_1)) \\
 @V\kappa' VV @. @| @V \wr VV \\
  \Ext^1_{(\varphi,\Gamma_L)}(\cR_E(\delta_2), D_2^{n}) @. \ \times \ @. \Ext^1_{(\varphi,\Gamma_L)}(D_2^{n},\cR_E(\delta_1)) @> \cup >> \Ext^2_{(\varphi,\Gamma_L)}(\cR_E(\delta_2), \cR_E(\delta_1))
  \end{CD}
\end{equation}
where the right vertical map is an isomorphism of $1$-dimensional $E$-vector spaces, the bottom cup-product is a perfect pairing, $\kappa'$ is surjective, and:
\begin{equation*}
\dim_E \Ext^1_{(\varphi,\Gamma_L)}(\cR_E(\delta_2), D_2^{n}) =\dim_E \Ext^1_{(\varphi,\Gamma_L)}(D_2^{n},\cR_E(\delta_1))=(n-1)[L:\Q_p]+1.
\end{equation*}
We define as previously the orthogonal spaces $\cL_{\FM}(D:D_2^{n})\subseteq \Ext^1_{(\varphi,\Gamma_L)}(D_2^{n},D_2^{n})$ and $\ell_{\FM}(D:D_2^{n})\subseteq \Ext^1_{(\varphi,\Gamma_L)}(\cR_E(\delta_2), D_2^{n})$ of $E[D]\subseteq \Ext^1_{(\varphi,\Gamma_L)}(D_2^{n},\cR_E(\delta_1))$ and we again have a short exact sequence:
\begin{equation*}
0 \lra \Ker (\kappa' )\lra \cL_{\FM}(D:D_2^{n}) \xlongrightarrow{\kappa'} \ell_{\FM}(D:D_2^{n}) \lra 0.
\end{equation*}
We have as in (\ref{equ: hL-commL}) a commutative diagram for $2<i$:
\begin{equation*}
 \begin{CD}
\Ext^1_{(\varphi,\Gamma_L)}(\cR_E(\delta_2), D_2^{n}) @. \ \times \ @. \Ext^1_{(\varphi,\Gamma_L)}(D_2^{n}, \cR_E(\delta_{1})) @> \cup >> \Ext^2_{(\varphi,\Gamma_L)}(\cR_E(\delta_2), \cR_E(\delta_{1})) \\
 @A j_i AA @. @V u_i VV @| \\
  \Ext^1_{(\varphi,\Gamma_L)}(\cR_E(\delta_2), D_2^i) @. \ \times \ @. \Ext^1_{(\varphi,\Gamma_L)}(D_2^i, \cR_E(\delta_1)) @> \cup >> \Ext^2_{(\varphi,\Gamma_L)}(\cR_E(\delta_2), \cR_E(\delta_{1}))
  \end{CD}
\end{equation*}
where the cup-products are perfect pairings, $j_i$ is injective and $u_i$ is surjective. Moreover as in Lemma \ref{lem: hL-compL}, we have in $\Ext^1_{(\varphi,\Gamma_L)}(\cR_E(\delta_2), D_2^{n})$ for $2<i$ (via $j_i$):
\begin{equation*}
 \ell_{\FM}(D:D_2^{n}) \cap \Ext^1_{(\varphi,\Gamma_L)}(\cR_E(\delta_2), D_2^i) =\ell_{\FM}(D_1^i: D_2^i).
\end{equation*}

\begin{theorem}\label{thm: hL-hL}
Let $\widetilde{D}_1^{n-1}$ (resp. $\widetilde{D}_2^{n}$) be a deformation of $D_1^{n-1}$ (resp. ${D}_2^{n}$) over $\cR_{E[\epsilon]/\epsilon^2}$ of rank $n-1$ (thus with $\widetilde{D}_1^{n-1}\equiv D_1^{n-1}\pmod{\epsilon}$, resp. $\widetilde{D}_2^{n}\equiv D_2^{n}\pmod{\epsilon}$). Then there exist a deformation $\widetilde{D}$ of $D$ over $\cR_{E[\epsilon]/\epsilon^2}$ and a deformation $\widetilde{\delta}_n$ (resp. $\widetilde{\delta}_1$) of $\delta_n$ (resp. $\delta_1$) over $E[\epsilon]/\epsilon^2$ such that $\widetilde{D}$ sits in an exact sequence of $(\varphi,\Gamma_L)$-modules over $\cR_{E[\epsilon]/\epsilon^2}$:
\begin{equation*}
 0 \lra \widetilde{D}_1^{n-1} \lra \widetilde{D} \lra \cR_{E[\epsilon]/\epsilon^2}(\widetilde{\delta}_n) \lra 0
\end{equation*}
\begin{equation*}
 \big(\text{resp. } 0 \lra\cR_{E[\epsilon]/\epsilon^2}(\widetilde{\delta}_1) \lra \widetilde{D} \lra\widetilde{D}_2^{n} \lra 0\big)
\end{equation*}
if and only if (with notation as for $[D]$):
\begin{equation*}
 [\widetilde{D}_1^{n-1}\otimes_{\cR_{E[\epsilon]/\epsilon^2}} \cR_{E[\epsilon]/\epsilon^2}(\widetilde{\delta}_n^{-1}\delta_n)] \in \cL_{\FM}(D:D_1^{n-1})
\end{equation*}
\begin{equation*}
(\text{resp. } [\widetilde{D}_2^{n}\otimes_{\cR_{E[\epsilon]/\epsilon^2}} \cR_{E[\epsilon]/\epsilon^2}(\widetilde{\delta}_1^{-1}\delta_1)] \in \cL_{\FM}(D:D_2^{n})).
\end{equation*}
\end{theorem}
\begin{proof}
We prove the case $D_1^{n-1}$, the proof for $D_2^{n}$ being symmetric. Replacing $\widetilde{D}$ and $\widetilde{D}_1^{n-1}$ by $\widetilde{D}\otimes_{\cR_{E[\epsilon]/\epsilon^2}} \cR_{E[\epsilon]/\epsilon^2}(\widetilde{\delta}_n^{-1}\delta_n)$ and $\widetilde{D}_1^{n-1} \otimes_{\cR_{E[\epsilon]/\epsilon^2}} \cR_{E[\epsilon]/\epsilon^2}(\widetilde{\delta}_n^{-1}\delta_n)$ respectively, we can assume $\widetilde{\delta}_n=\delta_n$. By twisting by $\cR_E(\delta_n^{-1})$, without loss of generality we can assume $\delta_n=1$. Now consider the exact sequence $0 \lra D_1^{n-1}\lra \widetilde{D}_1^{n-1} \lra D_1^{n-1}\lra 0$, taking cohomology, we get a long exact sequence:
 \begin{equation*}
  0\lra H^1_{(\varphi,\Gamma_L)}(D_1^{n-1}) \lra H^1_{(\varphi,\Gamma_L)}(\widetilde{D}_1^{n-1}) \xlongrightarrow{\pr} H^1_{(\varphi,\Gamma_L)}(D_1^{n-1}) \xlongrightarrow{c} H^2_{(\varphi,\Gamma_L)}(D_1^{n-1})
 \end{equation*}
with the map $c$ equal (up to nonzero scalars) to $\langle [\widetilde{D}_1^{n-1}],\cdot\rangle$ where $\langle \ ,\ \rangle$ is the cup product in (\ref{equ: hL-cup}) with $\delta_n=1$ (and $[\widetilde{D}_1^{n-1}]$ is seen in $\Ext^1_{(\varphi,\Gamma_L)}(D_1^{n-1}, D_1^{n-1})$). So we have $\langle [\widetilde{D}_1^{n-1}],[D]\rangle=0$ if and only if $[D]\in H^1_{(\varphi,\Gamma_L)}(D_1^{n-1})$ lies in the image of $\pr$ if and only if a deformation $\widetilde D$ of $D$ as in the statement exists. But by definition we also have $\langle [\widetilde{D}_1^{n-1}], [D]\rangle=0$ if and only if $[\widetilde{D}_1^{n-1}]\in \cL_{\FM}(D: D_1^{n-1})$. This concludes the proof.
\end{proof}

\begin{remark}
{\rm One can view Theorem \ref{thm: hL-hL} as a parabolic version of \cite[Thm. 3.14]{GS93} or \cite[Thm. 0.5]{Colm10}.}
\end{remark}

\noindent
When $n=2$, the two cases in Theorem \ref{thm: hL-hL} obviously coincide, which in particular implies the following corollary.

\begin{corollary}\label{n=2FM}
Assume $n=2$, then we have $\cL_{\FM}(D:\cR_{E}(\delta_1))=\cL_{\FM}(D:\cR_{E}(\delta_2))$ when these two vector spaces are viewed as subspaces of $ \Hom(L^{\times}, E)$ via (\ref{equ: hL-cad}).
\end{corollary}

\begin{remark}\label{compatgl2}
{\rm For any $\sigma\in \Sigma_L$, denote by $\Hom_{\sigma}(L^{\times}, E)$ the subspace of $\Hom(L^{\times}, E)$ consisting of locally $\sigma$-analytic characters on $L^{\times}$. We have $\Hom_{\infty}(L^{\times}, E) \subset \Hom_{\sigma}(L^{\times}, E)$ and $\dim_E \Hom_{\sigma}(L^{\times}, E)=2$. Let $\log_p: L^{\times} \rightarrow L$ be the unique character which restricts to the $p$-adic logarithm on $\co_L^{\times}$ and such that $\log_p(p)=0$. We see that $(\val_p,\sigma \circ \log_p)$ form a basis of $\Hom_{\sigma}(L^{\times}, E)$. Assume $n=2$, $D$ special noncritical and noncrystalline (equivalently semi-stable noncrystalline with distinct Hodge-Tate weights) and denote by $\cL_{\FM}(D)\subset \Hom(L^{\times}, E)$ the subspace of Corollary \ref{n=2FM}. Then we have $\cL_{\FM}(D)\cap \Hom_{\infty}(L^{\times}, E)=0$ and $\cL_{\FM}(D)_{\sigma}:=\cL_{\FM}(D) \cap \Hom_{\sigma}(L^{\times}, E)$ $1$-dimensional (inside $\Hom(L^{\times}, E)$). Thus for any $\sigma\in \Sigma_L$ there exists $\cL_{\sigma} \in E$ such that $\cL_{\FM}(D)_{\sigma}$ is generated by the vector $\sigma\circ \log_p - \cL_{\sigma}\val_p$. By comparing Theorem \ref{thm: hL-hL} with \cite[Thm.~1.1]{Zhang} (which generalizes a formula due to Colmez), it follows that this $\cL_{\sigma}$ is equal to Fontaine-Mazur's $\cL$-invariant obtained from the Hodge line in the $\sigma$-direct summand of the $(\varphi,N)$-filtered module associated to $D$ (with the normalization of \cite[\S~3.1]{Colm10}).}
\end{remark}

\noindent
We end this section by a quick speculation. We can call $\cL_{\FM}(D: D_1^{n-1})$ (resp. $\cL_{\FM}(D: D_2^{n})$) the (Fontaine-Mazur) $\cL$-invariants of $D$ relative to $D_1^{n-1}$ (resp. to $D_2^{n}$). A natural question in the $p$-adic Langlands program is to understand their counterpart on the automorphic side, e.g. in the setting of locally $\Q_p$-analytic representations of $\GL_n(L)$. The above results suggest that such invariants might be found in deformations of certain representations of (lower rank) Levi subgroups of $\GL_n(L)$. In the following section, we indeed succeed in finding such $\cL$-invariants in the locally analytic representations of $\GL_3(\Q_p)$ constructed in \cite{Br16} by means of the $p$-adic Langlands correspondence for $\GL_2(\Q_p)$.

\section{$\cL$-invariants for $\GL_3(\Q_p)$}\label{gl3gl3}

\noindent
In this section we use the subspaces $\cL_{\FM}(D: D_1^{n-1})$ and $\ell_{\FM}(D: D_1^{n-1})$ defined in \S~\ref{sec: hL-FM} to associate to a given $3$-dimensional semi-stable noncrystalline representation of $\Gal_{\Q_p}$ with distinct Hodge-Tate weights one of the finite length locally analytic representations of $\GL_3(\Q_p)$ constructed in \cite{Br16}.

\subsection{Preliminaries on locally analytic representations}

\noindent
We recall some useful notation and statements on locally analytic representations.\\

\noindent
We fix the $\Q_p$-points $G$ of a reductive algebraic group over $\Q_p$ (we will only use its $\Q_p$-points).

\begin{lemma}\label{lem: hL-cc}
Let $V_1$, $V_2$, $V$ be locally $\Q_p$-analytic representations of $G$ over $E$ such that $V$ is a strict extension of $V_2$ by $V_1$ in the category of locally analytic representations of $G$. Suppose $\Hom_G(V_2,V)=0$, where $\Hom_G(V_2,V)$ is the $E$-vector space of continuous $G$-equivariant morphisms, and that $V_1$, $V_2$ have the same central character $\chi$. Then $V$ has central character $\chi$.
\end{lemma}
\begin{proof}
For $z$ in the center of $G$ consider the $G$-equivariant map $V \ra V$, $v\mapsto zv-\chi(z)v$. It is easy to see this map induces a continuous $G$-equivariant morphism $V_2 \ra V$, which has to be zero. The lemma follows.
\end{proof}

\noindent
Let $V_1\hookrightarrow V_2 \hookrightarrow V$ be closed embeddings of locally $\Q_p$-analytic representations of $G$ over $E$ with central character $\chi$. Let $U$ be a strict extension of $V_1$ by $V$ and $W:=U/V_2$ (where $V_2\hookrightarrow V\hookrightarrow U$). We can then view $U$ as a representation of $G$ over $E[\epsilon]/\epsilon^2$ on which $\epsilon$ acts via $\epsilon: U\twoheadlongrightarrow V_1 \hooklongrightarrow U$. Thus the closed subrepresentation $V$ of $U$ is exactly the subspace annihilated by $\epsilon$. We also see $W$ as a representation over $E[\epsilon]/\epsilon^2$ by making $\epsilon$ act trivially, so that $U\twoheadrightarrow W$ is a surjection of $E[\epsilon]/\epsilon^2$-modules. Let $\psi :\Q_p^{\times}\ra E$ be a continuous additive character and define the character $1+\psi \epsilon : \Q_p^\times\rightarrow 1+E\epsilon\subset (E[\epsilon]/\epsilon^2)^\times$. Set $U':=U\otimes_{E[\epsilon]/\epsilon^2} (1+\psi \epsilon)\circ \dett$ and $W':=U'/V_2$ (where we still denote by $V_2$ the image of $V_2\otimes _{E[\epsilon]/\epsilon^2} (1+\psi \epsilon)\circ \dett$).

\begin{lemma}\label{lem: hL-cent3}
We have $W\cong W'$ as $G$-representations.
\end{lemma}
\begin{proof}
Let $e$ be a basis of the underlying $E[\epsilon]/\epsilon^2$-module of the representation $(1+\psi\epsilon) \circ \dett$, we have a natural $E$-linear bijection $f: U \buildrel \sim \over \longrightarrow U', \ v\mapsto v\otimes e$. For $v\in V$, we have:
\begin{equation*}
g(f(v))=g(v\otimes e)=g(v)\otimes ((1+\psi \epsilon)\circ \dett(g))e=g(v)\otimes e=f(g(v))
\end{equation*}
where the last equality follows from the fact that $g(v)\in V\hookrightarrow U$ is annihilated by $\epsilon$. Thus $f|_{V}$ induces a $G$-equivariant automorphism of $V$ if we still denote by $V$ the image of $V\otimes _{E[\epsilon]/\epsilon^2} (1+\psi \epsilon)\circ \dett$ in $U'$. We now consider the induced map (still denoted by $f$):
\begin{equation*}
 f: U/V_2 \buildrel \sim \over\longrightarrow U'/V_2.
\end{equation*}
The same argument using the fact that $W$ is killed by $\epsilon$ shows that $f$ is $G$-equivariant.
\end{proof}

\noindent
The following lemma will often be tacitly used in the sequel.

\begin{lemma}\label{tacit}
Let $Z:=(\Q_p^{\times})^r$ for some integer $r$ and $\chi$, $\chi'$ be locally analytic characters of $Z$ over $E$. Assume $\chi\neq \chi'$, then we have $\Ext_Z^i(\chi', \chi)=0$ for $i\geq 0$.
\end{lemma}
\begin{proof}
This follows from \cite[Cor.~8.8]{Kohl} together with \cite[Thm.~4.8]{Kohl} and \cite[Thm.~6.5]{Kohl}.
\end{proof}

\begin{notation}\label{not: hL-ext}
Let $V_1$, $V_2$ be admissible locally $\Q_p$-analytic representations of $G$ over $E$, $W\subseteq \Ext^1_G(V_2,V_1)$ be a finite dimensional $E$-vector subspace and $d:=\dim_E W$. Then we denote by $\sE(V_1,V_2^{\oplus d}, W)$ the extension of $V_2^{\oplus d}$ by $V_1$ naturally associated to $W$.
\end{notation}

\noindent
Explicitly, let $e_1, \cdots, e_d$ be a basis of $W$ over $E$ and denote by $\sE(V_1,V_2,e_i)\in \Ext^1_G(V_2,V_1)$ the extension corresponding to $e_i$, then we have:
\begin{equation*}
  \sE(V_1,V_2^{\oplus d},W) :=\bigoplus^{i=1,\cdots, d}_{V_1} \sE(V_1,V_2,e_i)
\end{equation*}
where the subscript $V_1$ means the amalgamate sum over $V_1$. This is an admissible locally $\Q_p$-analytic representation of $G$ over $E$ which only depends on $W$.

\subsection{$p$-adic Langlands correspondence for $\GL_2(\Q_p)$ and deformations}

\noindent
We study Ext${}^1$ groups of rank $2$ special $(\varphi,\Gamma)$-modules over $\cR_E$ and relate them to Ext${}^1$ groups of their associated locally analytic $\GL_2(\Q_p)$-representations. We prove several results on these Ext${}^1$ groups that are used in the next sections. Some statements in this section might already be known or hidden in the literature, but we provide complete proofs.

\subsubsection{Deformations of rank $2$ special $(\varphi,\Gamma)$-modules}\label{sec: hL-FM3}

\noindent
We define and study certain subspaces of Ext${}^1$ groups of $(\varphi,\Gamma)$-module over $\cR_E$ and relate them to infinitesimal deformations of rank $2$ special $(\varphi,\Gamma)$-modules.\\

\noindent
We now assume $L=\Q_p$ and let $(D, (\delta_1,\delta_2))$ be a special, noncritical and nonsplit $(\varphi,\Gamma)$-module over $\cR_E$ (see the beginning of \S~\ref{sec: hL-FM}).

\begin{lemma}\label{lem: hL-ext2}
We have $\dim_E \Ext^1_{(\varphi,\Gamma)}(D,D)=5$ and a short exact sequence:
\begin{equation}\label{equ: hL-ex0}
  0 \lra \Ext^1_{(\varphi,\Gamma)}(D,\cR_E(\delta_1)) \xlongrightarrow{\iota} \Ext^1_{(\varphi,\Gamma)}(D,D)\xlongrightarrow{\kappa} \Ext^1_{(\varphi,\Gamma)}(D, \cR_E(\delta_2)) \lra 0
\end{equation}
where $\dim_E \Ext^1_{(\varphi,\Gamma)}(D,\cR_E(\delta_1))=2$ and $\dim_E \Ext^1_{(\varphi,\Gamma)}(D, \cR_E(\delta_2)) =3$.
\end{lemma}
\begin{proof}
By the hypothesis on $D$ we have a long exact sequence:
\begin{multline}\label{equ: hL-lex}
 0 \lra \Hom_{(\varphi,\Gamma)}(D, \cR_E(\delta_1))\lra \Hom_{(\varphi,\Gamma)}(D,D) \\ \lra \Hom_{(\varphi,\Gamma)}(D, \cR_E(\delta_2)) \lra \Ext^1_{(\varphi,\Gamma)}(D,\cR_E(\delta_1)) \\ \xlongrightarrow{\iota} \Ext^1_{(\varphi,\Gamma)}(D,D)\xlongrightarrow{\kappa} \Ext^1_{(\varphi,\Gamma)}(D, \cR_E(\delta_2)) \lra \cdots
\end{multline}
By Proposition \ref{prop-l3-cup}(1), $\kappa$ is surjective, and $\iota$ is injective since the third arrow is obviously an isomorphism. By (\ref{equ: hL-cad}) we have $\dim_E \Ext^1_{(\varphi,\Gamma)}(\cR_E(\delta_2),\cR_E(\delta_2) )=2$ and using Tate duality (\cite[\S~4.2]{Liu07}) we get $\dim_E \Ext^1_{(\varphi,\Gamma)}(\cR_E(\delta_1),\cR_E(\delta_2) )=1$ and $\Ext^2_{(\varphi,\Gamma)}(\cR_E(\delta_2),\cR_E(\delta_2) )=0$ which implies $\dim_E \Ext^1_{(\varphi,\Gamma)}(D, \cR_E(\delta_2))=3$ by an obvious d\'evissage. By \cite[Thm.~4.3]{Liu07} together with \cite[\S~4.2]{Liu07} we obtain (where $D^\vee$ is the dual of $D$):
\begin{equation*}
\dim_E \Ext^1_{(\varphi,\Gamma)}(D,\cR_E(\delta_1)) = \dim_E H^1_{(\varphi,\Gamma)}(D^\vee\otimes_{\cR_E}\cR_E(\delta_1))=2.
\end{equation*}
The lemma follows.
\end{proof}

\noindent
We have $\dim_E \Hom_{(\varphi,\Gamma)}(\cR_E(\delta_1), \cR_E(\delta_1))=\dim_E \Ext^2_{(\varphi,\Gamma)}(\cR_E(\delta_2), \cR_E(\delta_1))=1$, the latter again via Tate duality \cite[\S~4.2]{Liu07}, and from {\it loc.cit.} and the proof of Lemma \ref{lem: hL-ext2} we have:
\begin{multline*}
\dim_E \Ext^1_{(\varphi,\Gamma)}(\cR_E(\delta_2), \cR_E(\delta_1))=\dim_E \Ext^1_{(\varphi,\Gamma)}(D, \cR_E(\delta_1))\\ =\dim_E \Ext^1_{(\varphi,\Gamma)}(\cR_E(\delta_1), \cR_E(\delta_1))=2.
\end{multline*}
Moreover we also have $\Ext^2_{(\varphi,\Gamma)}(D, \cR_E(\delta_1))=H^0_{(\varphi,\Gamma)}(D^\vee\otimes_{\cR_E}\cR_E(\delta_1))=0$. We deduce a long exact sequence:
\begin{multline}\label{equ: hL-lex2}
 0 \lra \Hom_{(\varphi,\Gamma)}(\cR_E(\delta_1), \cR_E(\delta_1))\lra \Ext^1_{(\varphi,\Gamma)}(\cR_E(\delta_2), \cR_E(\delta_1)) \\ \xlongrightarrow{\iota_1} \Ext^1_{(\varphi,\Gamma)}(D, \cR_E(\delta_1)) \xlongrightarrow{\kappa_1} \Ext^1_{(\varphi,\Gamma)}(\cR_E(\delta_1), \cR_E(\delta_1)) \\ \lra \Ext^2_{(\varphi,\Gamma)}(\cR_E(\delta_2), \cR_E(\delta_1))\lra 0
\end{multline}
where $\dim_E \Ima(\iota_1)=\dim_E \Ima(\kappa_1)=1$. Since $\Ext^2_{(\varphi,\Gamma)}(\cR_E(\delta_2), \cR_E(\delta_2))=H^0_{(\varphi,\Gamma)}(\cR_E(\varepsilon))=0$, we also have a short exact sequence:
\begin{multline}\label{equ: hL-lex3}
 0 \lra \Ext^1_{(\varphi,\Gamma)}(\cR_E(\delta_2), \cR_E(\delta_2)) \xlongrightarrow{\iota_2} \Ext^1_{(\varphi,\Gamma)}(D, \cR_E(\delta_2)) \\ \xlongrightarrow{\kappa_2} \Ext^1_{(\varphi,\Gamma)}(\cR_E(\delta_1), \cR_E(\delta_2)) \lra 0
\end{multline}
with $\dim_E \Ext^1_{(\varphi,\Gamma)}(\cR_E(\delta_2), \cR_E(\delta_2))=2$ and $\dim_E \Ext^1_{(\varphi,\Gamma)}(\cR_E(\delta_1), \cR_E(\delta_2)) =1$ (see the proof of Lemma \ref{lem: hL-ext2}). We denote by $\kappa_0$ the following composition:
\begin{equation}\label{equ: l3-exa2}
\kappa_0: \Ext^1_{(\varphi,\Gamma)}(D,D) \xlongrightarrow{\kappa} \Ext^1_{(\varphi,\Gamma)}(D,\cR_E(\delta_2)) \xlongrightarrow{\kappa_2} \Ext^1_{(\varphi,\Gamma)}(\cR_E(\delta_1),\cR_E(\delta_2)).
\end{equation}
In the sequel we loosely identify $\Ext^1_{(\varphi,\Gamma)}(D, D)$ with deformations $\widetilde{D}$ of $D$ over $\cR_{E[\epsilon]/\epsilon^2}$, dropping the $[\cdot]$ (this won't cause any ambiguity). We define:
$$\Ext^1_{\tri}(D,D):=\Ker(\kappa_0)\subseteq \Ext^1_{(\varphi,\Gamma)}(D, D).$$
It is then easy to check that those $\widetilde{D}$ in $\Ext^1_{\tri}(D,D)$ can be written as a (nonsplit) extension of $\cR_{E[\epsilon]/\epsilon^2}(\widetilde{\delta}_2)$ by $\cR_{E[\epsilon]/\epsilon^2}(\widetilde{\delta}_1)$ as a $(\varphi,\Gamma)$-module over $\cR_{E[\epsilon]/\epsilon^2}$ where $\widetilde{\delta}_i$ for $i\in \{1,2\}$ is a deformation of the character $\delta_i$ over $E[\epsilon]/\epsilon^2$.

\begin{lemma}\label{lem: hL-tri3}
We have $\dim_E \Ext^1_{\tri}(D,D)=4$.
\end{lemma}
\begin{proof}
It follows from the surjectivity of $\kappa$ (Lemma \ref{lem: hL-ext2}) that $\Ker(\kappa_0)$ is the inverse image (under the map $\kappa$) of $\Ker (\kappa_2)$ in $\Ext^1_{(\varphi,\Gamma)}(D,D)$. The lemma follows then from (\ref{equ: hL-lex3}) and a dimension count using the first equality in Lemma \ref{lem: hL-ext2}.
\end{proof}
\noindent
By (\ref{equ: hL-lex3}), (\ref{equ: hL-ex0}) and the proof of Lemma \ref{lem: hL-tri3}, we get a short exact sequence:
\begin{equation}\label{equ: l3-exa1}
 0 \lra \Ext^1_{(\varphi,\Gamma)}(D,\cR_E(\delta_1))\xlongrightarrow{\iota} \Ext^1_{\tri}(D, D)\\ \xlongrightarrow{\kappa} \Ext^1_{(\varphi,\Gamma)}(\cR_E(\delta_2),\cR_E(\delta_2))\lra 0.
\end{equation}
The map $\kappa$ in (\ref{equ: l3-exa1}) is given by sending $(\widetilde{D}, (\widetilde{\delta}_1, \widetilde{\delta}_2))\in \Ext^1_{\tri}(D,D)$ (with the above notation) to $\widetilde{\delta}_2\in \Ext^1_{(\varphi,\Gamma)}(\cR_E(\delta_2),\cR_E(\delta_2))$. In particular we deduce from (\ref{equ: l3-exa1}) the following lemma.

\begin{lemma}\label{lem: hL-kk}
Let $\widetilde{D}\in \Ext^1_{\tri}(D,D)$ and $(\widetilde{\delta}_1, \widetilde{\delta}_2)$ be the above trianguline parameter of $\widetilde{D}$ over $E[\epsilon]/\epsilon^2$. Then $\widetilde{D}\in \Ker (\kappa)$ if and only if $\widetilde{\delta}_2=\delta_2$.
\end{lemma}

\noindent
Let $D_1:=D_1^1=\cR_E(\delta_1)\subset D$ (notation of \S~\ref{sec: hL-FM}), identifying $\Ext^1_{(\varphi,\Gamma)}(D_1, D_1)$ with $\Hom(\Q_p^{\times}, E)$ by (\ref{equ: hL-cad}) we view $\cL_{\FM}(D: D_1)\subseteq \Ext^1_{(\varphi,\Gamma)}(D_1, D_1)$ (see \S~\ref{sec: hL-FM}) as an $E$-vector subspace of $\Hom(\Q_p^{\times}, E)$. Since $D$ is assumed to be nonsplit, $\cL_{\FM}(D:D_1)$ is one dimensional by Corollary \ref{codimension}(2). The following formula (sometimes called a Colmez-Greenberg-Stevens formula) is a special case of Theorem \ref{thm: hL-hL} (via the identification (\ref{equ: hL-cad})).

\begin{corollary}\label{coro: hL-tri}
Let $\widetilde{D}\in \Ext^1_{\tri}(D,D)$ and $(\widetilde{\delta}_1, \widetilde{\delta}_2)$ its above trianguline parameter. Let $\psi\in \Hom(\Q_p^{\times}, E)$ such that $\widetilde{\delta}_2\widetilde{\delta}_1^{-1}=\delta_2\delta_1^{-1}(1+\psi \epsilon)$, then $\psi\in \cL_{\FM}(D: D_1)$.
\end{corollary}
\noindent
Likewise one checks that the composition:
\begin{equation}\label{equ: hL-kappa1'}
\Ker(\kappa) \xlongrightarrow{\iota^{-1}} \Ext^1_{(\varphi,\Gamma)}(D, \cR_E(\delta_1)) \xlongrightarrow{\kappa_1} \Ext^1_{(\varphi,\Gamma)}(\cR_E(\delta_1), \cR_E(\delta_1))
\end{equation}
(see (\ref{equ: hL-ex0}) for $\iota$ and (\ref{equ: hL-lex2}) for $\kappa_1$) is given by sending $(\widetilde{D}, (\widetilde{\delta}_1, \delta_2))\in \Ker(\kappa)$ (cf. Lemma \ref{lem: hL-kk}) to $\widetilde{\delta}_1$. Hence, by Corollary \ref{coro: hL-tri} and $\dim_E \Ima(\kappa_1)=1$, (\ref{equ: hL-kappa1'}) has image equal to $\cL_{\FM}(D:D_1)$. Denote by $\iota_0$ the following composition:
\begin{equation*}
\iota_0: \Ext^1_{(\varphi,\Gamma)}(\cR_E(\delta_2), \cR_E(\delta_1)) \xlongrightarrow{\iota_1} \Ext^1_{(\varphi,\Gamma)}(D, \cR_E(\delta_1)) \xlongrightarrow{\iota} \Ext^1_{\tri}(D,D)
\end{equation*}
(see (\ref{equ: hL-lex2}) for $\iota_1$). By Lemma \ref{lem: hL-ext2} and $\dim_E \Ima(\iota_1)=1$ we see that $\Ima(\iota_0)$ is a one dimensional subspace of $\Ker(\kappa)$. From (\ref{equ: hL-lex2}) and (the discussion after) (\ref{equ: hL-kappa1'}), we deduce a short exact sequence:
\begin{equation}\label{equ: hL-kap}
 0 \lra \Ima(\iota_0) \lra \Ker(\kappa) \xlongrightarrow{\kappa_1} \cL_{\FM}(D: D_1) \lra 0.
\end{equation}
In particular, $\Ima(\iota_0)$ is generated by $(\widetilde{D}, \widetilde{\delta}_1, \widetilde{\delta}_2)\in \Ext^1_{\tri}(D, D)$ with $\widetilde{\delta}_1=\delta_1$ and $\widetilde{\delta}_2=\delta_2$.\\

\noindent
We denote by $\Ext^1_{(\varphi,\Gamma),Z}(D,D)$ the $E$-vector subspace of $\Ext^1_{(\varphi,\Gamma)}(D,D)$ consisting of $(\varphi,\Gamma)$-modules $\widetilde{D}$ over $\cR_{E[\epsilon]/\epsilon^2}$ such that $\wedge^2_{\cR_{E[\epsilon]/\epsilon^2}} \widetilde{D}\cong \cR_{E[\epsilon]/\epsilon^2}(\delta_1\delta_2)$ (i.e. with ``constant'' determinant), and by $\Ext^1_g(D, D)$ the $E$-vector subspace of $\Ext^1_{(\varphi,\Gamma)}(D,D)$ consisting of $\widetilde{D}$ such that $\widetilde{D} \otimes_{\cR_E}\cR_E(\delta_1^{-1})$ is de Rham.

\begin{lemma}\label{lem: hL-cenG}
We have $\dim_E \Ext^1_{(\varphi,\Gamma),Z}(D,D)=3$.
\end{lemma}
\begin{proof}
We have a natural exact sequence:
\begin{equation}\label{equ: hL-cent}
 0 \lra \Ext^1_{(\varphi,\Gamma),Z}(D,D) \lra \Ext^1_{(\varphi,\Gamma)}(D, D) \lra \Hom(\Q_p^{\times}, E)
\end{equation}
where the last map sends $\widetilde{D}\in \Ext^1_{(\varphi,\Gamma)}(D, D)$ to $\psi'$ with $\psi'$ satisfying:
$$\wedge^2_{\cR_{E[\epsilon]/\epsilon^2}} \widetilde{D}\cong \cR_{E[\epsilon]/\epsilon^2}(\delta_1\delta_2(1+\psi'\epsilon)).$$
On the other hand, we have an injection $j: \Hom(\Q_p^{\times}, E) \hookrightarrow \Ext^1_{(\varphi,\Gamma)}(D,D)$, $\psi \mapsto D\otimes_E \cR_{E[\epsilon]/\epsilon^2}(1+(\psi/2)\epsilon)$ and it is clear that $\Ima(j)$ gives a section of the last map of (\ref{equ: hL-cent}). Hence the latter is surjective and the lemma follows from the first equality in Lemma \ref{lem: hL-ext2}.
\end{proof}

\begin{lemma}\label{2dimltri}
We have $\dim_E \big(\Ext^1_{(\varphi,\Gamma),Z}(D,D) \cap \Ext^1_{\tri}(D,D)\big)=2$.
\end{lemma}
\begin{proof}
From \ Lemma \ref{lem: hL-ext2}, \ Lemma \ref{lem: hL-tri3} \ and \ Lemma \ref{lem: hL-cenG} \ it \ is \ sufficient \ to \ show \ that $\Ext^1_{(\varphi,\Gamma),Z}(D,D)$ is not contained in $\Ext^1_{\tri}(D,D)$. However with the notation of the proof of Lemma \ref{lem: hL-cenG}, we have $\Ima(j)\cap \Ext^1_{(\varphi,\Gamma),Z}(D,D)=0$ inside $\Ext^1_{(\varphi,\Gamma)}(D, D)$ and it is clear that $\Ima(j)\subseteq \Ext^1_{\tri}(D,D)$. If $\Ext^1_{(\varphi,\Gamma),Z}(D,D)\subseteq \Ext^1_{\tri}(D,D)$, this would imply $\dim_E \Ext^1_{\tri}(D,D)\geq \dim_E \Ima(j)+\dim_E \Ext^1_{(\varphi,\Gamma),Z}(D,D)=2+3=5$, contradicting Lemma \ref{lem: hL-tri3}.
\end{proof}

\begin{lemma}\label{lem: hL-dRg}
(1) We have $\Ext^1_{g}(D,D)\subset\Ext^1_{\tri}(D,D)$.\\
\noindent
(2) For $\widetilde{D} \in \Ext^1_{\tri}(D,D)$ of trianguline parameter $\widetilde{\delta}_1=\delta_1(1+\psi_1\epsilon)$, $\widetilde{\delta}_2=\delta_2(1+\psi_2\epsilon)$, we have $\widetilde{D}\in \Ext^1_g(D,D)$ if and only if $\psi_i\in \Hom_{\infty}(\Q_p^{\times}, E)$ for $i=1,2$.\\
\noindent
(3) We have:
\begin{equation}\dim_E \Ext^1_g(D,D)=\begin{cases}
  2 & \text{\!if\ \ } \cL_{\FM}(D:D_1)\neq \Hom_{\infty}(\Q_p^{\times}, E) \\
  3 & \text{\!if\ \ } \cL_{\FM}(D:D_1)= \Hom_{\infty}(\Q_p^{\times}, E).
 \end{cases}
 \end{equation}
\end{lemma}
\begin{proof}
(1) Twisting by a character, we can (and do) assume that the $\delta_i$ for $i=1,2$ are locally algebraic (see Definition \ref{def: hL-spe}). Since $\wt(\delta_2\delta_1^{-1})\in \Z_{<0}$ we have $\Ext^1_{g}(\cR_E(\delta_1),\cR_E(\delta_2))=H^1_{g}(\cR_E(\delta_2\delta_1^{-1}))=0$ and we deduce from (\ref{equ: l3-exa2})) (since being de Rham is preserved by taking subquotients) $\Ext^1_g(D,D)\subseteq \Ker(\kappa_0)=\Ext^1_{\tri}(D,D)$.\\
\noindent
(2) We know $\cR_{E[\epsilon]/\epsilon^2}(\delta_i (1+\psi_i\epsilon))$ is de Rham if and only if $\psi_i \in \Hom_{\infty}(\Q_p^{\times}, E)$ (see e.g. \cite[Rem. 2.2(2)]{Ding4}). The ``only if" part follows. For $i\in \{1,2\}$ let $\psi_i \in \Hom_{\infty}(\Q_p^{\times}, E)$, $\widetilde{\delta}_i:=\delta_i(1+\psi_i\epsilon)$ and $\widetilde D\in \Ext^1_{(\varphi,\Gamma)}(\cR_{E[\epsilon]/\epsilon^2}(\widetilde{\delta}_2), \cR_{E[\epsilon]/\epsilon^2}(\widetilde{\delta}_1))\subset \Ext^1_{\tri}(D,D)$. Since $\cR_{E[\epsilon]/\epsilon^2}(\widetilde{\delta_2})$ is de Rham, we are reduced to show that:
\begin{equation*}
D\otimes_{\cR_{E[\epsilon]/\epsilon^2}} \cR_{E[\epsilon]/\epsilon^2}\big(\widetilde{\delta_2}^{-1}\big) \in H^1_{(\varphi,\Gamma)}\big(\cR_{E[\epsilon]/\epsilon^2}(\widetilde{\delta}_1\widetilde{\delta}_2^{-1})\big)
\end{equation*}
is de Rham. However, since $\wt(\delta_1\delta_2^{-1})\in \Z_{>0}$ and $\cR_{E[\epsilon]/\epsilon^2}(\widetilde{\delta}_1\widetilde{\delta}_2^{-1})$ is de Rham, we know (e.g. by \cite[Lem.~1.11]{Ding4}) that any element in $H^1_{(\varphi,\Gamma)}\big(\cR_{E[\epsilon]/\epsilon^2}(\widetilde{\delta}_1\widetilde{\delta}_2^{-1})\big)$ is de Rham. The ``if" part follows.\\
\noindent
(3) The exact sequence (\ref{equ: hL-cent}) induces a short exact sequence:
\begin{equation}\label{seg}
 0 \lra \Ext^1_{(\varphi,\Gamma),Z}(D,D) \cap \Ext^1_g(D,D) \lra \Ext^1_g(D,D) \lra \Hom_{\infty}(\Q_p^{\times}, E)\lra 0
\end{equation}
where the last map is surjective since the map $j$ in the proof of Lemma \ref{lem: hL-cenG} induces an injection $\Hom_{\infty}(\Q_p^{\times}, E)\hookrightarrow \Ext^1_g(D,D)$. We have $\widetilde{D}\in \Ext^1_{(\varphi,\Gamma),Z}(D,D) \cap \Ext^1_g(D,D)$ if and only if $\psi_1,\psi_2\in \Hom_{\infty}(\Q_p^{\times}, E)$ and $\psi_1+\psi_2=0$ (for $\psi_i$ as in (2)).
Moreover, for any $\widetilde{D}\in \Ext^1_{\tri}(D,D)$ we have $\psi_1-\psi_2\in \cL_{\FM}(D:D_1)$ by Corollary \ref{coro: hL-tri}. If $\cL_{\FM}(D:D_1)\neq \Hom_{\infty}(\Q_p^{\times}, E)$, we see this implies $\psi_i=0$, hence $\Ext^1_{(\varphi,\Gamma),Z}(D,D)\cap \Ext^1_g(D,D)=\Ima(\iota_0)$ is one dimensional by the sentences before and after (\ref{equ: hL-kap}). If $\cL_{\FM}(D:D_1)=\Hom_{\infty}(\Q_p^{\times}, E)$, we have $\Ext^1_{(\varphi,\Gamma),Z}(D,D)\cap \Ext^1_g(D,D)=\Ext^1_{(\varphi,\Gamma),Z}(D,D)\cap \Ext^1_{\tri}(D,D)$ by (2) since $\psi_1+\psi_2=0$ and $\psi_1-\psi_2\in \Hom_{\infty}(\Q_p^{\times}, E)$ is equivalent to $\psi_1,\psi_2\in \Hom_{\infty}(\Q_p^{\times}, E)$, hence $\Ext^1_{(\varphi,\Gamma),Z}(D,D)\cap \Ext^1_g(D,D)$ is $2$-dimensional by Lemma \ref{2dimltri}. The result then follows from (\ref{seg}).
\end{proof}

\noindent
Now fix $k\in \Z_{\geq 1}$, set $\delta_3:=\delta_2x^{-k} |\cdot|^{-1}$ and consider the special case of the pairing (\ref{equ: hL-cup}):
\begin{equation}\label{equ: hL-cup3}
\Ext^1_{(\varphi,\Gamma_L)}(\cR_E(\delta_3), D) \times \Ext^1_{(\varphi,\Gamma_L)}(D, D) \xlongrightarrow{\cup} \Ext^2_{(\varphi,\Gamma_L)}(\cR_E(\delta_3), D)\simeq E.
\end{equation}
Recall the map $\Ext^1_{(\varphi,\Gamma)}(\cR_E(\delta_3), \cR_E(\delta_1))\rightarrow \Ext^1_{(\varphi,\Gamma_L)}(\cR_E(\delta_3), D)$ is injective from our assumptions on the $\delta_i$, $i\in \{1,2,3\}$.

\begin{lemma}
We have $\Ext^1_{\tri}(D,D)=\Ext^1_{(\varphi,\Gamma)}(\cR_E(\delta_3), \cR_E(\delta_1))^{\perp}$ in (\ref{equ: hL-cup3}) and a commutative diagram:
\begin{equation}\label{equ: hL-comtri}
   \begin{CD}
  \Ext^1_{(\varphi,\Gamma)}(\cR_E(\delta_3), \cR_E(\delta_2)) @.\ \ \times \ \ @. \Ext^1_{(\varphi,\Gamma)}(\cR_E(\delta_2),\cR_E(\delta_2)) @> \cup >> E \\
   @| @. @A \kappa AA @ |\\
  \Ext^1_{(\varphi,\Gamma)}\big(\cR_E(\delta_3), \cR_E(\delta_2)\big) @. \ \times @. \!\!\!\!\!\!\!\!\!\!\!\!\Ext^1_{\tri}\big(D,D\big) @> \cup >> E\\
  @AAA @. @VVV @| \\
 \ \ \Ext^1_{(\varphi,\Gamma)}\big(\cR_E(\delta_3), D\big) @. \ \times @. \!\!\!\!\!\!\!\!\!\!\!\!\Ext^1_{(\varphi,\Gamma)}\big(D,D\big)@> \cup >> E.
  \end{CD}
\end{equation}
\end{lemma}
\begin{proof}
The top squares of (\ref{equ: hL-comtri}) are induced from the bottom squares of (\ref{equ: l3-diag2}). Recall $\Ext^1_{\tri}\big(D,D\big)=\kappa^{-1}\big(\Ext^1_{(\varphi,\Gamma)}(\cR_E(\delta_2),\cR_E(\delta_2))\big)\subseteq \Ext^1_{(\varphi,\Gamma)}\big(D,D\big)$ (see the proof of Lem\-ma \ref{lem: hL-tri3}). Replacing the middle objects in (\ref{equ: hL-commL}) (for $\delta_n=\delta_3$, $\delta_{n-1}=\delta_2$, $D_1^{n-1}/D_1^i=D/D_1=\cR_E(\delta_2)$ and $D_1^{n-1}=D$) by their preimage under the map $\kappa:\Ext^1_{(\varphi,\Gamma)}(D,D)\rightarrow \Ext^1_{(\varphi,\Gamma)}(D, \cR_E(\delta_2))$, we obtain the bottom squares of (\ref{equ: hL-comtri}). This gives the commutativity. Together with the second part of Lemma \ref{lem: hL-compL}, the first statement also follows.
\end{proof}

\subsubsection{Deformations of $\GL_2(\Q_p)$-representations in special cases}\label{sec: hL-ext1}

\noindent
We study certain Ext${}^1$ groups in the category of locally analytic representations of $\GL_2(\Q_p)$.\\

\noindent
For an integral weight $\mu$ of $\GL_2(\Q_p)$, we denote by $\delta_{\mu}$ the algebraic character of the diagonal torus $T(\Q_p)$ of weight $\mu$. We fix $\lambda=(k_1,k_2)\in \Z^{2}$ a dominant weight of $\GL_2(\Q_p)$ with respect to the Borel subgroup $B(\Q_p)$ of upper triangular matrices (i.e. $k_1\geq k_2$), and denote by $L(\lambda)$ the associated algebraic representation of $\GL_2(\Q_p)$ over $E$. If $s$ is the nontrivial element of the Weyl group of ${\rm GL}_2$, we have $s\cdot \lambda=(k_2-1,k_1+1)$ (dot action with respect to $B(\Q_p)$). We denote by $\St_2^{\infty}$ be the usual smooth Steinberg representation of $\GL_2(\Q_p)$ over $E$ and set:
\begin{equation*}
 I(\lambda):=\big(\Ind_{\overline{B}(\Q_p)}^{\GL_2(\Q_p)}\delta_{\lambda}\big)^{\an}, \ \ I(s\cdot \lambda):=\big(\Ind_{\overline{B}(\Q_p)}^{\GL_2(\Q_p)}\delta_{s\cdot \lambda}\big)^{\an}
\end{equation*}
where $\overline{B}(\Q_p)$ is the subgroup of lower triangular matrices. Then $I(\lambda)$ has the form $I(\lambda)\cong L(\lambda)- \St_2^{\infty}(\lambda) -I(s\cdot \lambda)$ (recall $-$ denotes a nonsplit extension), $\St_2^{\infty}(\lambda):=\St_2\otimes_E L(\lambda)$ and where the subrepresentation $L(\lambda)-\St_2^{\infty}(\lambda)$ is isomorphic to $i(\lambda):=(\Ind_{\overline{B}(\Q_p)}^{\GL_2(\Q_p)} 1)^{\infty}\otimes_E L(\lambda)$. We denote by $\St_2^{\an}(\lambda):=I(\lambda)/L(\lambda)=\St_2^{\infty}(\lambda) -I(s\cdot \lambda)$ and set:
\begin{eqnarray*}
 \widetilde{I}(\lambda)&:=&\big(\Ind_{\overline{B}(\Q_p)}^{\GL_2(\Q_p)} \delta_{\lambda}(|\cdot|^{-1}\otimes |\cdot|)\big)^{\an}\\
 \widetilde{I}(s\cdot \lambda)&:=&\big(\Ind_{\overline{B}(\Q_p)}^{\GL_2(\Q_p)} \delta_{s\cdot\lambda}(|\cdot|^{-1}\otimes |\cdot|)\big)^{\an}.
\end{eqnarray*}
Then $\widetilde{I}(\lambda)$ has the form $\St_2^{\infty}(\lambda)-L(\lambda)-\widetilde{I}(s\cdot \lambda)$ where the subrepresentation $\St_2^{\infty}(\lambda) -L(\lambda)$ is isomorphic to $\widetilde{i}(\lambda):=(\Ind_{\overline{B}(\Q_p)}^{\GL_2(\Q_p)} |\cdot|^{-1} \otimes |\cdot|)^{\infty}\otimes_E L(\lambda)$. If $V$ is a locally analytic representation of $\GL_2(\Q_p)$, we define the locally analytic homology groups $H_i(\overline N(\Q_p),V)$ as in \cite[Def.~2.7]{Kohl} where $\overline N(\Q_p)$ is the unipotent radical of $\overline B(\Q_p)$.

\begin{lemma}
We have the following isomorphisms:
\begin{eqnarray}\label{equ: hL-N1}
   H_i(\overline{N}(\Q_p), L(\lambda))&\cong &\begin{cases}
  \delta_{\lambda}& i=0 \\
  \delta_{s\cdot \lambda} & i=1 \\
  0 & i\geq 2
 \end{cases}\\
\label{equ: hL-N2}
  H_i(\overline{N}(\Q_p), \St_2^{\infty}(\lambda))&\cong &\begin{cases}
 \delta_{\lambda}(|\cdot|^{-1}\otimes |\cdot|) & i=0 \\
 \delta_{s\cdot \lambda}(|\cdot|^{-1}\otimes |\cdot|)&i=1\\
 0& i\geq 2
\end{cases}\\
\label{equ: hL-N3}
   H_i(\overline{N}(\Q_p), I(s\cdot \lambda))&\cong &\begin{cases}
  \delta_{s\cdot \lambda} & i=0 \\
  \delta_{\lambda} (|\cdot|^{-1}\otimes |\cdot|)& i=1 \\
  0 & i\geq 2
 \end{cases}\\
\label{equ: hL-N4}
   H_i(\overline{N}(\Q_p), \widetilde{I}(s\cdot \lambda))&\cong &\begin{cases}
  \delta_{s\cdot \lambda}(|\cdot|^{-1}\otimes |\cdot|) & i=0 \\
  \delta_{\lambda} & i=1 \\
  0 & i\geq 2.
 \end{cases}
\end{eqnarray}
\end{lemma}
\begin{proof}
The isomorphisms (\ref{equ: hL-N1}) and (\ref{equ: hL-N2}) follow from results on classical Jacquet module together with \cite[(4.41) \& Thm. 4.10]{Sch11}. The isomorphisms (\ref{equ: hL-N3}) and (\ref{equ: hL-N4}) follow from \cite[Thm. 8.13]{Kohl} and \cite[Thm. 3.15]{Sch11}.
\end{proof}

\noindent
The following statement is not new, we include a proof for the reader's convenience.

\begin{theorem}\label{thm: hL-lGL2}
We have natural isomorphisms:
\begin{equation}\small\label{equ: hL-sim}
\Hom(\Q_p^{\times}, E)\xlongrightarrow{\sim} \Ext^1_{\GL_2(\Q_p)}(L(\lambda), I(\lambda)/L(\lambda))\xlongleftarrow{\sim}\Ext^1_{\GL_2(\Q_p)}(\widetilde{I}(\lambda)/\St_2^{\infty}(\lambda), I(\lambda)/L(\lambda)).
\end{equation}
\end{theorem}
\begin{proof}
The first isomorphism follows from \cite[\S~2.1]{Br04}, but we include a proof. By \cite[(4.38)]{Sch11} and (\ref{equ: hL-N1}), we have isomorphisms (where $Z(\Q_p)\cong \Q_p^{\times}$ is the center of $\GL_2(\Q_p)$, that we often shorten into $Z$):
\begin{equation*}
\Hom(T(\Q_p)/Z(\Q_p), E) \xlongrightarrow{\sim} \Ext^1_{T(\Q_p),Z}(\delta_{\lambda}, \delta_{\lambda}) \xlongrightarrow{\sim} \Ext^1_{\GL_2(\Q_p),Z}(L(\lambda), I(\lambda)).
\end{equation*}
By \cite[Cor. 4.8]{Sch11}, we have $\Ext^1_{\GL_2(\Q_p),Z}(L(\lambda), L(\lambda))=\Ext^2_{\GL_2(\Q_p),Z}(L(\lambda), L(\lambda))=0$. By d\'evissage, Lemma \ref{lem: hL-cc} and \cite[Lem.~2.1.1]{Br16}, we have then:
\begin{equation}\footnotesize\label{equ: hL-ecc}
\Ext^1_{\GL_2(\Q_p),Z}(L(\lambda), I(\lambda)) \xlongrightarrow{\sim} \Ext^1_{\GL_2(\Q_p),Z}(L(\lambda), I(\lambda)/L(\lambda)) \cong \Ext^1_{\GL_2(\Q_p)}(L(\lambda), I(\lambda)/L(\lambda)).
\end{equation}
The first isomorphism follows. By \cite[Prop. 3.1.6]{Br16} we have $\Ext^i_{\GL_2(\Q_p)}(\widetilde{I}(s\cdot \lambda), \St_2^{\infty}(\lambda))=0$ for $i=1,2$. By \cite[(4.37)]{Sch11} and (\ref{equ: hL-N4}), we have $\Ext^i_{\GL_2(\Q_p)}(\widetilde{I}(s\cdot \lambda), I(s\cdot \lambda))=0$ for $i\in \Z_{\geq 0}$. By d\'evissage, we deduce $\Ext^i_{\GL_2(\Q_p)}(\widetilde{I}(s\cdot \lambda), I(\lambda)/L(\lambda))=0$ for $i=1, 2$. The second isomorphism then follows by d\'evissage again.
\end{proof}

\noindent
By \cite[(4.37) \& (4.38)]{Sch11}, we have a commutative diagram (with the notation of the last proof):
\begin{equation}\label{diag: hL-GL2}
\begin{CD}
\Hom(T(\Q_p)/Z(\Q_p), E) @> \sim>> \Ext^1_{T(\Q_p),Z}(\delta_{\lambda}, \delta_{\lambda}) @> \sim>> \Ext^1_{\GL_2(\Q_p),Z}(L(\lambda), I(\lambda))\\
@VVV @VVV @VVV \\
\Hom(T(\Q_p), E) @> \sim>> \Ext^1_{T(\Q_p)}(\delta_{\lambda}, \delta_{\lambda}) @> \sim>> \Ext^1_{\GL_2(\Q_p)}(L(\lambda), I(\lambda)).
\end{CD}
\end{equation}
Consider the short exact sequence:
\begin{equation}\footnotesize
0 \lra \Ext^1_{\GL_2(\Q_p)}(L(\lambda), L(\lambda)) \lra \Ext^1_{\GL_2(\Q_p)}(L(\lambda), I(\lambda)) \lra \Ext^1_{\GL_2(\Q_p)}(L(\lambda), I(\lambda)/L(\lambda))\longrightarrow 0
\end{equation}
where the last map is surjective by (\ref{equ: hL-ecc}). Contrary to $\Ext^1_{\GL_2(\Q_p),Z}(L(\lambda), L(\lambda))=0$, we have $\Ext^1_{\GL_2(\Q_p)}(L(\lambda), L(\lambda))\ne 0$. More precisely, we have a commutative diagram:
\begin{equation*}
\begin{CD}
 \Hom(Z(\Q_p), E) @> \sim >> \Ext^1_{\GL_2(\Q_p)}(L(\lambda), L(\lambda))\\
 @VVV @VVV \\
\Hom(T(\Q_p), E) @> \sim>> \Ext^1_{\GL_2(\Q_p)}(L(\lambda), I(\lambda))
\end{CD}
\end{equation*}
where the bottom horizontal map is the composition of the bottom line in (\ref{diag: hL-GL2}) and where the left vertical map is given by $\psi\mapsto \psi\circ \dett$. In particular, the natural surjective map
\begin{equation}\label{equ: hL-L2}
\Hom(T(\Q_p), E) \twoheadlongrightarrow \Ext^1_{\GL_2(\Q_p)}(L(\lambda), I(\lambda)/L(\lambda))
\end{equation}
is zero on $\Hom(Z(\Q_p),E)$.\\

\noindent
We can make this map more explicit (e.g. by unwinding the spectral sequence \cite[(4.37)]{Sch11}). Let $\psi\in \Hom(\Q_p^{\times}, E)$ and choose $\psi_i\in \Hom(\Q_p^{\times}, E)$ for $i=1,2$ such that $\psi_1-\psi_2=\psi$. Let $\sigma(\psi_1,\psi_2)$ be the following two dimensional representation of $T(\Q_p)$:
\begin{equation*}
 \sigma(\psi_1,\psi_2)\begin{pmatrix}
  a & 0 \\ 0 & d
 \end{pmatrix}:=\begin{pmatrix}
  1 & \psi_1(a)+\psi_2(d)\\0 & 1
 \end{pmatrix}
\end{equation*}
and consider the natural exact sequence:
\begin{equation*}
 0\longrightarrow I(\lambda) \longrightarrow \big(\Ind_{\overline{B}(\Q_p)}^{\GL_2(\Q_p)} \delta_{\lambda}\otimes_E \sigma(\psi_1, \psi_2)\big)^{\an} \xlongrightarrow{\pr} I(\lambda) \longrightarrow 0.
\end{equation*}
Then the locally analytic representation:
\begin{equation}\label{equ: hL-upsi}
\pi(\lambda, \psi)^-:=\pr^{-1}(L(\lambda))/L(\lambda)\ \cong \ \St_2^{\an}(\lambda)-L(\lambda)
\end{equation}
only depends on $\psi$ and not on the choice of $\psi_1$ and $\psi_2$, and the map (\ref{equ: hL-L2}) is given by sending $\psi$ to $\pi(\lambda, \psi)^-$ (seen as an element of $\Ext^1_{\GL_2(\Q_p)}(L(\lambda), I(\lambda)/L(\lambda))$). Moreover:
\begin{equation}\label{equ: hL-wpsi}
\pi(\lambda,\psi_1,\psi_2)^-:=\pr^{-1}(i(\lambda))/L(\lambda)
\end{equation}
actually depends on (and is determined by) both $\psi_1$ and $\psi_2$. By \cite[(4.37)]{Sch11} and (\ref{equ: hL-N1}), (\ref{equ: hL-N2}), we have $\Hom(T(\Q_p),E)\buildrel\sim\over\longrightarrow \Ext^1_{T(\Q_p)}(\delta_{\lambda}, \delta_{\lambda}) \xlongrightarrow{\sim} \Ext^1_{\GL_2(\Q_p)}(i(\lambda), I(\lambda))$ and the composition is given by mapping $(\psi_1, \psi_2)$ to $\pr^{-1}(i(\lambda))$.
By \cite[Prop. 15]{Orl} and the same argument as in the proof of \cite[Cor. 2]{Orl} (see also \cite{Da}), we have:
\begin{equation}\label{equ: hL-smooth}
 \Ext^i_{\GL_2(\Q_p),\infty}(i(0), 1)=0,\ \ i\in \Z_{\geq 0}.
\end{equation}
By a version without central character of the spectral sequence \cite[(4.27)]{Sch11} (which follows exactly by the same argument), we deduce from (\ref{equ: hL-smooth}) $\Ext^i_{\GL_2(\Q_p)}(i(\lambda), L(\lambda))=0$ for $i\in \Z_{\geq 0}$ and in particular that the natural push-forward map:
\begin{equation*}
 \Ext^1_{\GL_2(\Q_p)}(i(\lambda), I(\lambda)) \longrightarrow \Ext^1_{\GL_2(\Q_p)}(i(\lambda), I(\lambda)/L(\lambda))
\end{equation*}
is a bijection. Putting the above maps together, we obtain:
\begin{equation*}
 \Hom(T(\Q_p),E)\xlongrightarrow{\sim} \Ext^1_{\GL_2(\Q_p)}(i(\lambda), I(\lambda)/L(\lambda)), \ (\psi_1, \psi_2)\longmapsto \pi(\lambda,\psi_1, \psi_2)^-.
\end{equation*}

\noindent
Let $0\neq \psi\in \Hom(\Q_p^{\times}, E)$ and define:
\begin{equation}\label{pipsi}
\pi(\lambda, \psi)\in \Ext^1_{\GL_2(\Q_p)}\big(\widetilde{I}(\lambda)/\St_2^{\infty}(\lambda), I(\lambda)/L(\lambda)\big)
\end{equation}
to be the preimage of $\pi(\lambda, \psi)^-$ in (\ref{equ: hL-upsi}) via the second isomorphism of (\ref{equ: hL-sim}) (we should write $[\pi(\lambda, \psi)]$ to denote some element of the above $\Ext^1$ associated to the representation $\pi(\lambda, \psi)$, but as in \S~\ref{sec: hL-FM3} we drop the $[\cdot]$, which won't cause any ambiguity). So one has $\pi(\lambda,\psi)\simeq I(\lambda)/L(\lambda)-(L(\lambda)-\widetilde I(s\cdot \lambda))$ and we recall the irreducible constituents of $I(\lambda)/L(\lambda)$ are $\St_2^{\infty}(\lambda)$ and $I(s\cdot\lambda)$. We now study the extension groups $\Ext^1_{\GL_2(\Q_p)}(\pi(\lambda, \psi), \pi(\lambda, \psi))$ and $\Ext^1_{\GL_2(\Q_p),Z}(\pi(\lambda, \psi), \pi(\lambda, \psi))$. \ Note \ first \ that \ by \ \cite[Lem.~2.1.1]{Br16} \ one \ can \ identify $\Ext^1_{\GL_2(\Q_p)}(\pi(\lambda, \psi), \pi(\lambda, \psi))$ with deformations $\widetilde{\pi}$ of $\pi(\lambda, \psi)$ over $E[\epsilon]/\epsilon^2$. Let $\chi_{\lambda}:=\delta_{\lambda}|_{Z(\Q_p)}$, which is the central character of $\pi(\lambda, \psi)$.

\begin{lemma}\label{lem: hL-cent2}
For any $\widetilde{\pi}\in \Ext^1_{\GL_2(\Q_p)}(\pi(\lambda, \psi), \pi(\lambda, \psi))$, there exists a unique lifting $\widetilde{\chi}_{\lambda}: \Q_p^{\times} \ra (E[\epsilon]/\epsilon^2)^{\times}$ of $\chi_{\lambda}$ such that $Z(\Q_p)$ acts on $\widetilde{\pi}$ via $\widetilde{\chi}_\lambda$.
\end{lemma}
\begin{proof}
 For $v\in \widetilde{\pi}$, we have $(z-\chi_{\lambda}(z))v\in \pi(\lambda, \psi)$ and thus $(z-\chi_{\lambda}(z))^2 v=0$ for all $z\in Z(\Q_p)$. The map $v\mapsto (z-\chi_{\lambda}(z))v$ induces a morphism from $\pi(\lambda, \psi)$ (as quotient of $\widetilde{\pi}$) to $\pi(\lambda, \psi)$ (as subobject of $\widetilde{\pi}$) which is $\GL_2(\Q_p)$-equivariant since $z$ is in the center. But any such endomorphism of $\pi(\lambda, \psi)$ is a scalar by \cite[\S~3.4]{DosSch} since all (absolutely) irreducible constituents of $\pi(\lambda, \psi)$ are distinct. So for any $v\in \widetilde{\pi}$ and $z\in Z(\Q_p)$, we have $(z-\chi_{\lambda}(z))v\in E(\epsilon v)$, and hence there exists $\widetilde{\chi}_{\lambda}: Z(\Q_p)\ra (E[\epsilon]/\epsilon^2)^{\times}$ (which \emph{a priori} depends on $v$) such that $z v=\widetilde{\chi}_{\lambda}(z) v$ for all $z\in Z(\Q_p)$. We fix a $v$ which is not in $\pi(\lambda, \psi)=\epsilon \widetilde{\pi}$ and define:
 \begin{equation*}
 \pi(\widetilde{\chi}_{\lambda}):=\big\{w\in \widetilde{\pi},\ (z-\widetilde{\chi}_{\lambda}(z))w=0 \ \forall \ z\in Z(\Q_p)\big\}
 \end{equation*}
which is a $\GL_2(\Q_p)$-subrepresentation of $\widetilde{\pi}$ strictly containing $\pi(\lambda, \psi)=\epsilon \widetilde{\pi}$. Thus we have $\St_2^{\infty}(\lambda) \subseteq \pi(\widetilde{\chi}_{\lambda})/\pi(\lambda, \psi)$ since $\soc_{\GL_2(\Q_p)} \pi(\lambda, \psi)\cong\St_2^{\infty}(\lambda)$ and $\pi(\widetilde{\chi}_{\lambda})/\pi(\lambda, \psi)$ is a nonzero subrepresentation of $\widetilde\pi/\epsilon \widetilde{\pi}\simeq \pi(\lambda, \psi)$. We need to prove $\pi(\widetilde{\chi}_{\lambda})= \widetilde{\pi}$. If $\pi(\widetilde{\chi}_{\lambda})\neq \widetilde{\pi}$, then there exists another lifting $\widetilde{\chi}_{\lambda}'\neq \widetilde \chi_{\lambda}$ such that $\pi(\widetilde{\chi}_{\lambda}')$ strictly contains $\pi(\lambda, \psi)$, and hence $\St_2^{\infty}(\lambda)\subseteq \pi(\widetilde{\chi}_{\lambda}')/\pi(\lambda, \psi)$. This implies $Z(\Q_p)$ acts on the subextension $V_1\subset \widetilde \pi$ of $\St_2^{\infty}(\lambda)$ by $\pi(\lambda, \psi)=\epsilon \widetilde{\pi}$ via $\chi_{\lambda}$. However, by Lemma \ref{lem: hL-cc}, this then implies $Z(\Q_p)$ acts on the whole $\widetilde{\pi}$ by the character $\chi_{\lambda}$, a contradiction. Hence $\pi(\widetilde{\chi}_{\lambda})=\widetilde{\pi}$.
\end{proof}

\begin{lemma}\label{Lem: hL-cent4}
We have a short exact sequence:
\begin{equation*}
0 \lra \Ext^1_{\GL_2(\Q_p),Z}(\pi(\lambda, \psi), \pi(\lambda, \psi)) \lra \Ext^1_{\GL_2(\Q_p)}(\pi(\lambda, \psi), \pi(\lambda,\psi)) \xlongrightarrow{\pr} \Hom(\Q_p^{\times}, E)\lra 0
\end{equation*}
where $\pr$ sends $\widetilde{\pi}$ to $(\widetilde{\chi}_\lambda\chi_\lambda^{-1}-1)/\epsilon$ where $\widetilde{\chi}_\lambda$ is the central character of $\widetilde{\pi}$ given by Lemma \ref{lem: hL-cent2}.
\end{lemma}
\begin{proof}
It is sufficient to prove $\pr$ is surjective. However, it is easy to check that $\psi'\mapsto \pi(\lambda, \psi)\otimes_E (1+\frac{\psi'}{2}\epsilon) \circ \dett$ gives a section of the map $\pr$. The lemma follows.
\end{proof}

\begin{lemma}\label{lem: hL-ext1}
(1) We have $\Ext^i_{\GL_2(\Q_p)}(I(s\cdot \lambda), \pi(\lambda, \psi))=\Ext^i_{\GL_2(\Q_p)}(\widetilde{I}(s\cdot \lambda), \pi(\lambda, \psi))=0$ for all $i\in \Z_{\geq 0}$.\\

\noindent
(2) We have:
$$\begin{array}{l}
\dim_E \Ext^1_{\GL_2(\Q_p),Z}(\St_2^{\infty}(\lambda), \pi(\lambda, \psi))=2,\ \dim_E \Ext^1_{\GL_2(\Q_p)}(\St_2^{\infty}(\lambda), \pi(\lambda, \psi))=4\\
\dim_E \Ext^1_{\GL_2(\Q_p),Z}(L(\lambda), \pi(\lambda, \psi))=\dim_E \Ext^1_{\GL_2(\Q_p)}(L(\lambda), \pi(\lambda, \psi))=1\\
\dim_E \!\Ext^1_{\GL_2(\Q_p),Z}(\pi(\lambda, \psi)/\!\St_2^{\infty}(\lambda), \pi(\lambda, \psi))\!=\!\dim_E \!\Ext^1_{\GL_2(\Q_p)}(\pi(\lambda, \psi)/\!\St_2^{\infty}(\lambda), \pi(\lambda, \psi))\!=\!1.
\end{array}$$

\noindent
(3) We have exact sequences:
\begin{multline}\label{equ: hL-stv2}
  0 \lra \Ext^1_{\GL_2(\Q_p),Z}(\pi(\lambda, \psi)/\St_2^{\infty}(\lambda), \pi(\lambda, \psi)) \lra \Ext^1_{\GL_2(\Q_p),Z}(\pi(\lambda, \psi), \pi(\lambda, \psi)) \\ \lra \Ext^1_{\GL_2(\Q_p),Z}(\St_2^{\infty}(\lambda), \pi(\lambda, \psi))
\end{multline}
\begin{multline}\label{equ: hL-stv3}
  0 \lra \Ext^1_{\GL_2(\Q_p)}(\pi(\lambda, \psi)/\St_2^{\infty}(\lambda), \pi(\lambda, \psi)) \xlongrightarrow{\iota_0} \Ext^1_{\GL_2(\Q_p)}(\pi(\lambda, \psi), \pi(\lambda, \psi)) \\ \xlongrightarrow{\kappa_1} \Ext^1_{\GL_2(\Q_p)}(\St_2^{\infty}(\lambda), \pi(\lambda, \psi))
\end{multline}
with $\dim_E \Ext^1_{\GL_2(\Q_p),Z}(\pi(\lambda, \psi), \pi(\lambda, \psi))\leq 3$ and $\dim_E \Ext^1_{\GL_2(\Q_p)}(\pi(\lambda, \psi), \pi(\lambda, \psi))\leq 5$.\\

\noindent
(4) We have exact sequences (see (\ref{equ: hL-upsi}) for $\pi(\lambda, \psi)^-$):
\begin{multline}\label{equ: hL-iot1}
 0\lra \Ext^1_{\GL_2(\Q_p)}(\St_2^{\infty}(\lambda), \pi(\lambda, \psi)^-)\xlongrightarrow{\iota_1} \Ext^1_{\GL_2(\Q_p)}(\St_2^{\infty}(\lambda), \pi(\lambda, \psi)) \\
  \lra \Ext^1_{\GL_2(\Q_p)}(\St_2^{\infty}(\lambda), \widetilde{I}(s\cdot \lambda))\lra 0
\end{multline}
\begin{multline}\label{equ: hL-pi_-}
 0 \lra \Ext^1_{\GL_2(\Q_p)}(\St_2^{\infty}(\lambda), \St_2^{\infty}(\lambda)) \lra \Ext^1_{\GL_2(\Q_p)}(\St_2^{\infty}(\lambda), \pi(\lambda, \psi)^-)\\ \lra \Ext^1_{\GL_2(\Q_p)}(\St_2^{\infty}(\lambda), \pi(\lambda, \psi)^-/\St_2^{\infty}(\lambda)) \lra 0
\end{multline}
with \ \ $\dim_E \Ext^1_{\GL_2(\Q_p)}(\St_2^{\infty}(\lambda), \St_2^{\infty}(\lambda))=2$, \ \ $\dim_E \Ext^1_{\GL_2(\Q_p)}(\St_2^{\infty}(\lambda), \widetilde{I}(s\cdot \lambda))=1$ \ \ and \ \ $\dim_E \Ext^1_{\GL_2(\Q_p)}(\St_2^{\infty}(\lambda), L(\lambda))=\dim_E \Ext^1_{\GL_2(\Q_p)}(\St_2^{\infty}(\lambda), \pi(\lambda, \psi)^-/\St_2^{\infty}(\lambda))=1$.
\end{lemma}
\begin{proof}
In this proof, we write $\Ext^{i}$ (resp. $\Ext^{i}_{Z}$) for $\Ext^{i}_{\GL_2(\Q_p)}$ (resp. $\Ext^{i}_{\GL_2(\Q_p),Z}$).\\
(1) We prove the case of $I(s\cdot \lambda)$, the proof for $\widetilde{I}(s\cdot \lambda)$ being parallel. By \cite[(4.37)]{Sch11}, (\ref{equ: hL-N3}) and Lemma \ref{tacit}, we have
 $\Ext^{i}(I(s\cdot \lambda), I(\lambda))=0$ for all $i\in \Z_{\geq 0}$. By \cite[Cor. 4.3]{Sch11}, we have:
\begin{equation*}
  \Ext^i(I(s\cdot \lambda), L(\lambda))\cong \Ext^{i-1}\big(L(\lambda)^\vee, (\Ind_{\overline{B}(\Q_p)}^{\GL_2(\Q_p)}\delta_{\mu}\delta_B)^{\an}\big)
\end{equation*}
where $\mu$ is the weight such that $L(\lambda)^\vee\cong L(\mu)$. However, by \cite[(4.37)]{Sch11}, (\ref{equ: hL-N1}) (with $\lambda$ replaced by $\mu$) and Lemma \ref{tacit}, we have $\Ext^{i}(L(\lambda)^\vee, (\Ind_{\overline{B}(\Q_p)}^{\GL_2(\Q_p)}\delta_{\mu}\delta_B)^{\an})=0$ for $i\in \Z_{\geq 0}$, hence $\Ext^i(I(s\cdot \lambda), L(\lambda))=0$ for all $i\in \Z_{\geq 0}$. By d\'evissage, we deduce:
\begin{equation}\label{equ: hL-nul1}
 \Ext^i(I(s\cdot \lambda), I(\lambda)/L(\lambda))=0, \ \forall \ i\geq 0.
\end{equation}
By \cite[(4.37)]{Sch11} and (\ref{equ: hL-N3}) (+ Lemma \ref{tacit}), we also have:
\begin{equation}\label{equ: hL-e1}
 \Ext^i(I(s\cdot \lambda), \widetilde{I}(s\cdot \lambda))=0, \ \forall \ i\geq 0
\end{equation}
and by d\'evissage we deduce $\Ext^i(I(s\cdot \lambda), L(\lambda)-\widetilde{I}(s\cdot \lambda))=0$ for $i\geq 0$. Together with (\ref{equ: hL-nul1}) this implies (again by d\'evissage) $\Ext^i(I(s\cdot \lambda), \pi(\lambda, \psi))=0$ for all $i\geq 0$. This concludes the proof of (1).\\

\noindent
(2) By \cite[Cor. 4.8]{Sch11}, we have $\Ext^i_{Z}(\St_2^{\infty}(\lambda), \St_2^{\infty}(\lambda))=0$ for $i=1, 2$. Consider the following map:
\begin{equation}\label{equ: hL-Stdet}
\Hom(\Q_p^{\times}, E)\lra \Ext^1(\St_2^{\infty}(\lambda), \St_2^{\infty}(\lambda)), \ \psi'\longmapsto \St_2^{\infty}(\lambda)\otimes_E (1+\psi'\epsilon)\circ \dett.
\end{equation}
It is straightforward to see this map is injective. We claim it is also surjective. For any nonsplit extension $\widetilde{\pi}\in\Ext^1(\St_2^{\infty}(\lambda), \St_2^{\infty}(\lambda))$ (which we view as an representation of $\GL_2(\Q_p)$ over $E[\epsilon]/\epsilon^2$), let $\psi'\in \Hom(\Q_p^{\times}, E)$ be such that the central character of $\widetilde{\pi}$ is given by $\chi_{\lambda}(1+\psi'\epsilon)$ (argue as in Lemma \ref{lem: hL-cent2} for the latter, though this is simpler here). Then the representation $\widetilde{\pi}\otimes_{E[\epsilon]/\epsilon^2} (1-(\psi'/2)\epsilon)\circ \dett$ has central character $\chi_{\lambda}$ and hence is isomorphic to $\St_2^{\infty}(\lambda)^{\oplus 2}\cong\St_2^{\infty}(\lambda) \otimes_E E[\epsilon]/\epsilon^2 $ since $\Ext^1_{Z}(\St_2^{\infty}(\lambda), \St_2^{\infty}(\lambda))=0$ (\cite[Cor.~4.8]{Sch11}). So we have:
\begin{eqnarray*}
\widetilde{\pi}&\cong &\big(\widetilde{\pi}\otimes_{E[\epsilon]/\epsilon^2} (1-(\psi'/2)\epsilon)\circ \dett \big)\otimes_{E[\epsilon]/\epsilon^2} (1+(\psi'/2)\epsilon)\circ \dett \\
&\cong & \big(\St_2^{\infty}(\lambda)\otimes_E E[\epsilon]/\epsilon^2\big) \otimes_{E[\epsilon]/\epsilon^2} (1+(\psi'/2)\epsilon)\circ \dett\\
&\cong & \St_2^{\infty}(\lambda) \otimes_E (1+(\psi'/2)\epsilon)\circ \dett.
\end{eqnarray*}
Thus $\dim_E \Ext^1(\St_2^{\infty}(\lambda), \St_2^{\infty}(\lambda))=2$. By \cite[(4.38)]{Sch11} and (\ref{equ: hL-N2}), we have:
\begin{eqnarray*}
\Ext^1_Z(\St_2^{\infty}(\lambda), \widetilde{I}(\lambda))&\cong &\Ext^1_{T(\Q_p),Z}(\delta_\lambda(|\cdot|^{-1}\otimes |\cdot|), \delta_\lambda(|\cdot|^{-1}\otimes |\cdot|))\\
\Ext^i_Z(\St_2^{\infty}(\lambda), I(s\cdot \lambda))&=&0 \ \ \forall \ i\in \Z_{\geq 0}.
\end{eqnarray*}
 Putting these together we deduce by d\'evissage:
\begin{multline}\label{equ: hL-stv4}
\Ext^1(\St_2^{\infty}(\lambda), \pi(\lambda, \psi)/\St_2^{\infty}(\lambda))\cong  \Ext^1_Z(\St_2^{\infty}(\lambda),\pi(\lambda, \psi)/\St_2^{\infty}(\lambda))\\ \cong \Ext^1_Z(\St_2^{\infty}(\lambda), L(\lambda)-\widetilde{I}(s\cdot \lambda))\cong \Ext^1_Z(\St_2^{\infty}(\lambda), \widetilde{I}(\lambda))\cong \Hom(\Q_p^{\times}, E)
\end{multline}
where, for the second last isomorphism, we use the exact sequence:
\begin{equation*}
  0 \lra \St_2^{\infty}(\lambda) \lra \widetilde{I}(\lambda) \lra (L(\lambda)-\widetilde{I}(s\cdot \lambda))\lra 0
\end{equation*}
together with $\Ext^i_Z(\St_2^{\infty}(\lambda), \St_2^{\infty}(\lambda))=0$, $i=1,2$. Likewise we have:
\begin{multline}\label{equ: hL-stv}
0=\Ext^1_Z(\St_2^{\infty}(\lambda), \St_2^{\infty}(\lambda)) \lra \Ext^1_Z(\St_2^{\infty}(\lambda), \pi(\lambda, \psi)) \\ \lra  \Ext^1_Z(\St_2^{\infty}(\lambda),\pi(\lambda, \psi)/\St_2^{\infty}(\lambda))\lra \Ext^2_Z(\St_2^{\infty}(\lambda), \St_2^{\infty}(\lambda))=0
\end{multline}
from which together with (\ref{equ: hL-stv4}) we deduce $\dim_E \Ext^1_Z(\St_2^{\infty}(\lambda), \pi(\lambda, \psi))=2$. Similarly:
\begin{multline*}
  0 \lra \Ext^1(\St_2^{\infty}(\lambda), \St_2^{\infty}(\lambda)) \lra \Ext^1(\St_2^{\infty}(\lambda), \pi(\lambda, \psi)) \\ \lra  \Ext^1(\St_2^{\infty}(\lambda), \pi(\lambda, \psi)/\St_2^{\infty}(\lambda)) \lra 0
\end{multline*}
where the last map is surjective by (\ref{equ: hL-stv}) and the first isomorphism of (\ref{equ: hL-stv4}). By the above dimension computations we deduce $\dim_E\Ext^1(\St_2^{\infty}(\lambda), \pi(\lambda, \psi))=4$. To prove the remaining equalities in (2), we only need to prove $\dim_E \Ext^1_Z(L(\lambda), \pi(\lambda, \psi))=1$ since the other equalities follow easily from (1) and Lemma \ref{lem: hL-cc}. By \cite[(4.38)]{Sch11} and (\ref{equ: hL-N1}), we have $\Ext^i_Z(L(\lambda), \widetilde{I}(s\cdot \lambda))=0$ for $i\geq 0$ and by \cite[Cor. 4.8]{Sch11}, we have $\Ext^1_Z(L(\lambda), L(\lambda))=0$. By d\'evissage, we deduce then $\Ext^1_Z(L(\lambda), \widetilde{I}(\lambda)/\St_2^{\infty}(\lambda))=0$. By the exact sequence:
\begin{multline}\label{equ: hL-1Vpsi}
 0 \lra \Hom(L(\lambda), \widetilde{I}(\lambda)/\St_2^{\infty}(\lambda)) \lra \Ext^1_Z(L(\lambda), I(\lambda)/L(\lambda)) \lra \Ext^1_Z(L(\lambda), \pi(\lambda, \psi)) \\ \lra \Ext^1_Z(L(\lambda), \widetilde{I}(\lambda)/\St_2^{\infty}(\lambda)) =0,
\end{multline}
Theorem \ref{thm: hL-lGL2} and an easy dimension count, we get $\dim_E \Ext^1_Z(L(\lambda), \pi(\lambda, \psi))=1$.\\

\noindent
(3) This follows easily from (1), (2) and Lemma \ref{Lem: hL-cent4}.\\

\noindent
(4) To get the exact sequences, it is sufficient to prove that the maps:
\begin{eqnarray*}
 \Ext^1(\St_2^{\infty}(\lambda), \pi(\lambda, \psi)) &\lra &\Ext^1(\St_2^{\infty}(\lambda), \widetilde{I}(s\cdot \lambda))\\
\Ext^1(\St_2^{\infty}(\lambda), \pi(\lambda, \psi)^-) &\lra &\Ext^1(\St_2^{\infty}(\lambda), \pi(\lambda, \psi)^-/\St_2^{\infty}(\lambda))\end{eqnarray*}
are surjective and it is sufficient to prove they are surjective with $\Ext^1$ replaced by $\Ext^1_Z$ (since the vector spaces on the right hand side do not change). The second one follows easily from $\Ext^2_Z(\St_2^{\infty}(\lambda),\St_2^{\infty}(\lambda))=0$ (see the proof of (2) above). By \cite[(4.38)]{Sch11} and (\ref{equ: hL-N2}), we have $\dim_E \Ext^1_Z(\St_2^{\infty}(\lambda), \widetilde{I}(s\cdot \lambda))=1$ and $\Ext^1_Z(\St_2^{\infty}(\lambda), {I}(s\cdot \lambda))=0$, and by \cite[Cor.~4.8]{Sch11} we have $\dim_E \Ext^1_Z(\St_2^{\infty}(\lambda), L(\lambda))=1$. The last two equalities imply by d\'evissage $\dim_E\Ext^1_Z(\St_2^{\infty}(\lambda), \pi(\lambda, \psi)^-/\St_2^{\infty}(\lambda))\leq 1$. The first, together with $\dim_E \Ext^1_{Z}(\St_2^{\infty}(\lambda), \pi(\lambda, \psi))=2$ in (2), imply the surjectivity of $\Ext^1_Z(\St_2^{\infty}(\lambda), \pi(\lambda, \psi)) \ra \Ext^1_Z(\St_2^{\infty}(\lambda), \widetilde{I}(s\cdot \lambda))$, and then $\dim_E\Ext^1_Z(\St_2^{\infty}(\lambda), \pi(\lambda, \psi)^-/\St_2^{\infty}(\lambda))=1$. We have seen $\dim_E \Ext^1(\St_2^{\infty}(\lambda), \St_2^{\infty}(\lambda))=2$ in the proof of (2), the rest of (4) follows from lemma \ref{lem: hL-cc}.
\end{proof}

\begin{remark}\label{rem: hL-sp}
{\rm It follows from (\ref{equ: hL-1Vpsi}) and (2) that, if $\psi'\notin E\psi\subset \Hom(\Q_p^\times,E)$, the image of $\pi(\lambda, \psi')^-$, seen as an element of $\Ext^1(L(\lambda), I(\lambda)/L(\lambda))$, in $\Ext^1(L(\lambda), \pi(\lambda, \psi))$ is the unique nonsplit extension $V$ of $L(\lambda)$ by $\pi(\lambda, \psi)$. Moreover $V$ contains the unique extension $V_0$ of $L(\lambda)^{\oplus 2}$ by $I(\lambda)/L(\lambda)=\St_2^{\infty}(\lambda)-I(s\cdot\lambda)$ with socle $\St_2^{\infty}(\lambda)$ and we have $(V_0)^{\lalg}\cong V^{\lalg}\cong \St_2^{\infty}(\lambda)-L(\lambda)\cong \widetilde{i}(\lambda)$. By Lemma \ref{lem: hL-ext1}(1) and (\ref{equ: hL-stv3}), we have:
\begin{multline*}
\Ext^1_{\GL_2(\Q_p)}(L(\lambda), \pi(\lambda, \psi))\cong \Ext^1_{\GL_2(\Q_p)}(\pi(\lambda, \psi)/\St_2^{\infty}(\lambda), \pi(\lambda, \psi))\\
\hooklongrightarrow \Ext^1_{\GL_2(\Q_p)}(\pi(\lambda, \psi), \pi(\lambda, \psi))
\end{multline*}
and we let $\widetilde{\pi}$ be the image of $V\in \Ext^1_{\GL_2(\Q_p)}(L(\lambda), \pi(\lambda, \psi))$ via the above injection. It is not difficult then to deduce:
\begin{equation}\label{equ: hL-lalg}
 \widetilde{\pi}^{\lalg}\cong
 \begin{cases} \St_2^{\infty}(\lambda)^{\oplus 2} & \psi \text{ not smooth}\\
  \St_2^{\infty}(\lambda) \oplus \widetilde{i}(\lambda) & \psi \text{ smooth}
  \end{cases}
\end{equation}
and that, if $\psi$ is smooth, the map $\widetilde{\pi}^{\lalg} \ra \pi(\lambda, \psi)^{\lalg}\cong \widetilde{i}(\lambda)$ induced by $\widetilde{\pi} \twoheadrightarrow \pi(\lambda, \psi)$ is nonzero but not surjective.}
\end{remark}

\noindent
We now make the following hypotheses, which will be proved (under some mild technical assumption) in Proposition \ref{prop: hL-gl2pLL} below.

\begin{hypothesis}\label{hypo: hL-dim}
(1) Any representation in $\Ext^1_{\GL_2(\Q_p)}(\pi(\lambda, \psi), \pi(\lambda, \psi))$ is very strongly admissible in the sense of \cite[Def.~0.12]{Em2} (which implies $\pi(\lambda,\psi)$ itself is very strongly admissible).\\
\noindent
(2) We have $\dim_E \Ext^1_{\GL_2(\Q_p),Z}(\pi(\lambda, \psi), \pi(\lambda, \psi))=3$ and $\dim_E \!\Ext^1_{\GL_2(\Q_p)}(\pi(\lambda, \psi), \pi(\lambda, \psi))\!=5$. In particular (by Lemma \ref{lem: hL-ext1}(2)), the last maps in (\ref{equ: hL-stv2}) and (\ref{equ: hL-stv3}) are surjective.
\end{hypothesis}

\noindent
Denote by $\Hom(T(\Q_p),E)_{\psi}$ the subspace of $\Hom(T(\Q_p),E)$ generated by those $(\psi_1, \psi_2)\in \Hom(\Q_p^\times,E)^2$ such that $\psi_1-\psi_2\in E\psi$. For a locally analytic character $\delta: T(\Q_p) \ra E^{\times}$, denote by $\Ext^1_{T(\Q_p)}(\delta, \delta)_\psi\subseteq \Ext^1_{T(\Q_p)}(\delta, \delta)$ the $E$-vector subspace corresponding to $\Hom(T(\Q_p),E)_{\psi}$ via the natural bijection $\Ext^1_{T(\Q_p)}(\delta, \delta)\cong\Hom(T(\Q_p),E)$. Denoting by $J_B$ the Jacquet-Emerton functor relative to the Borel subgroup $B$ (where the $T(\Q_p)$-action is normalized as in \cite{Em11}), we have since $\psi\neq 0$:
\begin{equation}\label{equ: hL-Jac}
 J_B(\pi(\lambda, \psi))=\begin{cases}
 J_B(\St_2^{\infty}(\lambda))\cong \delta_{\lambda} (|\cdot|\otimes |\cdot|^{-1}) & \psi \text{ not smooth} \\
 J_B(\St_2^{\infty}(\lambda)) \oplus J_B(L(\lambda))\cong\delta_{\lambda} (|\cdot|\otimes |\cdot|^{-1}) \oplus \delta_{\lambda} & \psi \text{ smooth}.
 \end{cases}
\end{equation}
It is clear that the right hand side is contained in the left hand side. Since $\pi(\lambda, \psi)$ is very strongly admissible, it is not difficult to prove they are equal using \cite[Thm. 4.3]{Br13II} together with the left exactness of $J_B$ and \cite[Ex.~5.1.9]{Em2}.

\begin{lemma}\label{lem: hL-tri1}
(1) Let $V\in \Ext^1_{\GL_2(\Q_p)}(\St_2^{\infty}(\lambda), \pi(\lambda, \psi))$, then $J_B(V)\neq J_B(\pi(\lambda, \psi))$ if and only if $V$ lies in the image of $\Ext^1_{\GL_2(\Q_p)}(\St_2^{\infty}(\lambda), \pi(\lambda, \psi)^-)$.\\
\noindent
(2) The functor $J_B$ induces a bijection:
\begin{equation}\label{equ: hL-jac2}
 \Ext^1_{\GL_2(\Q_p)}(\St_2^{\infty}(\lambda), \pi(\lambda, \psi)^-) \xlongrightarrow{\sim} \Ext^1_{T(\Q_p)}\big(\delta_{\lambda}(|\cdot|\otimes |\cdot|^{-1}), \delta_{\lambda}(|\cdot|\otimes |\cdot|^{-1})\big)_\psi.
\end{equation}
\end{lemma}
\begin{proof}
In this proof we write $\chi:=\delta_{\lambda}(|\cdot|\otimes |\cdot|^{-1})$ for simplicity.\\

\noindent
(1) We first prove the ``only if" part, and for that we can assume that $V$ is nonsplit. If $J_B(V)\neq J_B(\pi(\lambda, \psi))$, then by (\ref{equ: hL-Jac}) we see that $J_B(V)$ is isomorphic to an extension of $\chi$ by $J_B(\pi(\lambda, \psi))$ and that there exists an extension $\widetilde{\chi}$ of $\chi$ by $\chi$ such that $j_1: \widetilde{\chi}\hookrightarrow J_B(V)$ (recall $\Ext^1_{T(\Q_p)}(\chi, \delta_{\lambda})=0$). Denote by $\widetilde{\delta}_{\lambda}:=\widetilde{\chi}\otimes_E (|\cdot|^{-1}\otimes |\cdot|)$, which is thus isomorphic to an extension of $\delta_{\lambda}$ by $\delta_{\lambda}$.
One can check (e.g. by the proof of \cite[Lem. 4.11]{Ding17}) that the morphism $j_1$ is balanced in the sense of \cite[Def. 0.8]{Em11}. From Hypothesis \ref{hypo: hL-dim} (both (1) and (2) are needed), we deduce that $V$ is very strongly admissible. By \cite[Thm. 0.13]{Em2}, the map $j_1$ then induces a $\GL_2(\Q_p)$-equivariant map:
\begin{equation*}
j_2: I_{\overline{B}(\Q_p)}^{\GL_2(\Q_p)}\widetilde{\delta}_{\lambda} \lra V
\end{equation*}
such that the morphism $j_1$ can be recovered from $j_2$ by applying the functor $J_B(\cdot)$ and where $I_{\overline{B}(\Q_p)}^{\GL_2(\Q_p)}\widetilde{\delta}_{\lambda}$ denotes the closed subrepresentation of $\big(\Ind_{\overline{B}(\Q_p)}^{\GL_2(\Q_p)}\widetilde{\delta}_{\lambda}\big)^{\an}$ generated by $\widetilde{\chi}$ via the natural embedding (see \cite[Lem.~0.3]{Em2} for details):
\begin{equation*}
\widetilde{\chi}\hooklongrightarrow J_B\big(\big(\Ind_{\overline{B}(\Q_p)}^{\GL_2(\Q_p)}\widetilde{\delta}_{\lambda}\big)^{\an}\big) \hooklongrightarrow \big(\Ind_{\overline{B}(\Q_p)}^{\GL_2(\Q_p)}\widetilde{\delta}_{\lambda}\big)^{\an}.
\end{equation*}
We have $\soc_{\GL_2(\Q_p)} \Ima(j_2) \buildrel\sim\over\ra \soc_{\GL_2(\Q_p)} V \cong\St_2^{\infty}(\lambda)$ (as $V$ is nonsplit). This implies $\St_2^{\infty}(\lambda)$ has multiplicity $2$ in the irreducible constituents of $\Ima(j_2)$, since otherwise we would have $\Ima(j_2)\subseteq \pi(\lambda, \psi)$ and thus $\widetilde{\chi}\hookrightarrow J_B(\Ima(j_2))\subseteq J_B(\pi(\lambda, \psi))$ which is a contradiction. By the exact sequence (\ref{equ: hL-iot1}) together with the fact that $\widetilde{I}(s\cdot \lambda)$ is not an irreducible constituent of $I_{\overline{B}(\Q_p)}^{\GL_2(\Q_p)}\widetilde{\delta}_{\lambda}$ (since it is not an irreducible constituent of $(\Ind_{\overline{B}(\Q_p)}^{\GL_2(\Q_p)}\widetilde{\delta}_{\lambda})^{\an}$), we obtain that $V$ comes from an element in $\Ext^1(\St_2^{\infty}(\lambda), \pi(\lambda, \psi)^-)$.\\
\noindent
We prove the ``if" part. For $\psi'\in\Hom(\Q_p^{\times}, E)$, let $U(\psi'):=\St_2^{\infty}(\lambda)\otimes_E (1+\psi'\epsilon)\circ \dett$, hence $J_B(U(\psi'))\cong\chi\otimes_E (1+\psi'\epsilon)\circ \dett$. In particular, taking $J_B$ induces a bijection by (\ref{equ: hL-Stdet}):
\begin{equation*}
\Ext^1_{\GL_2(\Q_p)}(\St_2^{\infty}(\lambda), \St_2^{\infty}(\lambda)) \xlongrightarrow{\sim} \Hom(Z(\Q_p), E) \big(\hooklongrightarrow \Hom(T(\Q_p),E)\big).
\end{equation*}
Denote by $W(\psi')\in \Ext^1_{\GL_2(\Q_p)}(\St_2^{\infty}(\lambda), \pi(\lambda, \psi)^-)$ the image of $U(\psi')$ via the injection in (\ref{equ: hL-pi_-}) (so $U(\psi')\subset W(\psi')$), then by left exactness of $J_B$ we have:
\begin{equation}\label{equ: hL-jac3}
\chi\otimes_E (1+\psi'\epsilon)\circ \dett \hooklongrightarrow J_B(W(\psi')).
\end{equation}
Now let $\Psi=(\psi_1, \psi_2) \in \Hom(T(\Q_p), E)_{\psi}\setminus \Hom(Z(\Q_p),E)$ (i.e. $\psi_1\neq \psi_2$ and $\psi_1-\psi_2\in E\psi$) and consider the representation $\pi(\lambda, \psi_1, \psi_2)^-$ in (\ref{equ: hL-wpsi}). We know $\pi(\lambda, \psi)^-\subseteq \pi(\lambda, \psi_1, \psi_2)^-$ and thus $\pi(\lambda, \psi_1, \psi_2)^-$ gives a \emph{nonsplit} extension of $\St_2^{\infty}(\lambda)$ by $\pi(\lambda, \psi)^-$ (since the quotient $i(\lambda)$ is nonsplit). Note that by construction $\pi(\lambda, \psi_1, \psi_2)^-$ is a subquotient of $W(\Psi):=(\Ind_{\overline{B}}^{\GL_2} \delta_{\lambda}(1+\Psi \epsilon))^{\an}$ and that we have a natural injection $\chi\otimes_E (1+\Psi \epsilon)\hookrightarrow J_B(W(\Psi))$ (cf. \cite[Lem. 0.3]{Em2}). Moreover $\pi(\lambda, \psi_1, \psi_2)^-\hookrightarrow W(\Psi)/L(\lambda)$ and neither $J_B(L(\lambda))$ nor $J_B((W(\Psi)/L(\lambda))/\pi(\lambda, \psi_1, \psi_2)^-)=J_B(I(s\cdot\lambda))$ contains $\chi$ as a subquotient (the latter by \cite[Ex.~5.1.9]{Em2}). By left exactness of $J_B$ we deduce:
\begin{equation}\label{equ: hL-jac4}
\chi\otimes_E (1+\Psi\epsilon)\hooklongrightarrow J_B(\pi(\lambda, \psi_1, \psi_2)^-).
\end{equation}
From Lemma \ref{lem: hL-ext1}(2)\&(4) we deduce $\dim_E \Ext^1_{\GL_2(\Q_p)}(\St_2^{\infty}(\lambda), \pi(\lambda, \psi)^-)=3$ and we let $\Pi$ be the unique extension of $\St_2^{\infty}(\lambda)^{\oplus 3}$ by $\pi(\lambda, \psi)^-$ with $\soc_{\GL_2(\Q_p)} \Pi\cong \St_2^{\infty}(\lambda)$. The above discussion implies $J_B(\Pi)$ contains the unique extension of $\chi^{\oplus 3}$ by $\chi$ with socle $\chi$ attached to the $3$-dimensional space $\Ext^1_{T(\Q_p)}(\chi, \chi)_\psi$. Indeed, let $\{\psi_1'\circ \dett, \psi_2' \circ \dett, \Psi_3:=\Psi\}$ be a basis of the $3$-dimensional space $\Hom(T(\Q_p),E)_{\psi}\cong \Ext^1_{T(\Q_p)}(\chi, \chi)_\psi$ where $\{\psi_1', \psi_2'\}$ is a basis of $\Hom(\Q_p^{\times},E)$ and $\Psi=(\psi_1, \psi_2)$ is as after (\ref{equ: hL-jac3}), then we have by (\ref{equ: hL-pi_-}) again:
\begin{equation*}
\Pi \cong W(\psi_1')\oplus_{\pi(\lambda,\psi)^-} W(\psi_2') \oplus_{\pi(\lambda, \psi)^-} \pi(\lambda, \psi_1, \psi_2)^-.
\end{equation*}
By (\ref{equ: hL-jac3}), (\ref{equ: hL-jac4}), left exactness of $J_B$ and (\ref{equ: hL-Jac}), we deduce that applying $J_B$ to $\Pi\twoheadrightarrow \St_2^{\infty}(\lambda)^{\oplus 3}$ induces a surjective map:
 \begin{equation}\label{equ: hL-jac5}
  J_B(\Pi) \twoheadrightarrow J_B(\St_2^{\infty}(\lambda)^{\oplus 3})\cong \chi^{\oplus 3},
 \end{equation}
from which we deduce (together with (\ref{equ: hL-Jac}) and the left exactness of $J_B$) that for any $U\in \Ext^1(\St_2^{\infty}(\lambda), \pi(\lambda, \psi)^-)$, we have $J_B(U)\neq J_B(\pi(\lambda, \psi))$.\\

\noindent
(2) By the proof of (1) (see (\ref{equ: hL-jac3}), (\ref{equ: hL-jac4}), (\ref{equ: hL-jac5})) together with (\ref{equ: hL-Jac}) and $\Ext^1_{T(\Q_p)}(\chi, \delta_{\lambda})=0$, we see that taking $J_B$ induces a map:
\begin{equation*}
 \Ext^1_{\GL_2(\Q_p)}(\St_2^{\infty}(\lambda), \pi(\lambda, \psi)^-) \lra \Ext^1_{T(\Q_p)}\big(\chi, J_B(\pi(\lambda, \psi)^-)\big)\cong \Ext^1_{T(\Q_p)}(\chi, \chi)
\end{equation*}
which induces an isomorphism $\Ext^1_{\GL_2(\Q_p)}(\St_2^{\infty}(\lambda), \pi(\lambda, \psi)^-) \buildrel\sim\over \ra \Ext^1_{T(\Q_p)}(\chi, \chi)_\psi$. This finishes the proof.
\end{proof}

\begin{remark}\label{rem: hL-Jac}
{\rm From the proof of Lemma \ref{lem: hL-tri1}, we can explicitly describe the inverse of (\ref{equ: hL-jac2}) as follows. Let $\Psi\in \Hom(T(\Q_p), E)_{\psi}$, define:
$$\widetilde{\chi}=\delta_{\lambda}(|\cdot|\otimes |\cdot|^{-1})(1+\Psi\epsilon)\in \Ext^1_{T(\Q_p)}\big(\delta_{\lambda}(|\cdot|\otimes |\cdot|^{-1}), \delta_{\lambda}(|\cdot|\otimes |\cdot|^{-1})\big)_\psi$$
and consider the short exact sequence:
\begin{equation*}
 0 \lra I(\lambda) \lra \big(\Ind_{\overline{B}(\Q_p)}^{\GL_2(\Q_p)}\delta_{\lambda}(1+\Psi \epsilon)\big)^{\an} \xlongrightarrow{\pr} I(\lambda) \lra 0.
 \end{equation*}
If $\Psi=(\psi', \psi')$, i.e. $\Psi\in \Hom(Z(\Q_p), E)$, then $\pr^{-1}(i(\lambda))/L(\lambda)$ has a subrepresentation isomorphic to $\St_2^{\infty}(\lambda)\otimes_E (1+\psi'\epsilon)\circ \dett$, and the inverse image of $\widetilde{\chi}$ in (\ref{equ: hL-jac2}) is then given by the push-forward of this representation via $\St_2^{\infty}(\lambda) \hookrightarrow \pi(\lambda, \psi)^-$. If $\Psi\notin \Hom(Z(\Q_p),E)$, the inverse image of $\Psi$ is then isomorphic to $\pr^{-1}(i(\lambda))/L(\lambda)$.}
\end{remark}

\noindent
We now denote by $\Ext^1_{\tri}(\pi(\lambda, \psi), \pi(\lambda, \psi))$ the kernel of the composition:
\begin{multline}\label{equ: hL-kappa0}
\kappa_0: \Ext^1_{\GL_2(\Q_p)}(\pi(\lambda, \psi), \pi(\lambda, \psi)) \xlongrightarrow{\kappa_1} \Ext^1_{\GL_2(\Q_p)}(\St_2^{\infty}(\lambda), \pi(\lambda, \psi)) \longrightarrow \\
\Ext^1_{\GL_2(\Q_p)}(\St_2^{\infty}(\lambda), \widetilde{I}(s\cdot \lambda))
\end{multline}
with $\kappa_1$ as in (\ref{equ: hL-stv3}). In particular, by (\ref{equ: hL-stv3}) we have $\Ima(\iota_0)\subset \Ext^1_{\tri}(\pi(\lambda, \psi), \pi(\lambda, \psi))$.

\begin{lemma}\label{dimtri}
(1) We have $\dim_E \Ext^1_{\tri}(\pi(\lambda, \psi), \pi(\lambda, \psi))=4$.\\
\noindent
(2) For $\widetilde{\pi}\in \Ext^1_{\GL_2(\Q_p)}(\pi(\lambda, \psi), \pi(\lambda, \psi))$, we have $\widetilde{\pi}\in \Ext^1_{\tri}(\pi(\lambda, \psi), \pi(\lambda, \psi))$ if and only if $\delta_{\lambda}(|\cdot|\otimes |\cdot|^{-1})$ appears with multiplicity $2$ in $J_B(\widetilde{\pi})$.\\
\noindent
(3) We have a natural short exact sequence:
\begin{multline}\label{equ: hL-ex1}
0 \lra \Ext^1_{\GL_2(\Q_p)}(\pi(\lambda, \psi)/\St_2^{\infty}(\lambda), \pi(\lambda, \psi)) \xlongrightarrow{\iota_0}\Ext^1_{\tri}(\pi(\lambda, \psi), \pi(\lambda, \psi)) \\ \lra \Ext^1_{T(\Q_p)}\big(\delta_{\lambda}(|\cdot|\otimes |\cdot|^{-1}), \delta_{\lambda}(|\cdot|\otimes |\cdot|^{-1})\big)_\psi \lra 0.
\end{multline}
\end{lemma}
\begin{proof}
By Hypothesis \ref{hypo: hL-dim}(2) and Lemma \ref{lem: hL-ext1}(4) (see in particular (\ref{equ: hL-iot1})), there is a natural exact sequence:
\begin{equation}\label{equ: hL-tria}
0 \lra \Ext^1(\pi(\lambda, \psi)/\St_2^{\infty}(\lambda), \pi(\lambda, \psi)) \xlongrightarrow{\iota_0}\Ext^1_{\tri}(\pi(\lambda, \psi), \pi(\lambda, \psi)) \lra \Ima(\iota_1) \ra 0.
\end{equation}
(1) follows by Lemma \ref{lem: hL-ext1}(2), (\ref{equ: hL-iot1}) and a dimension count. Together with (the proof of) Lemma \ref{lem: hL-tri1}(1), left exactness of $J_B$ and (\ref{equ: hL-Jac}), we easily deduce (2) and (3), where the third map of (\ref{equ: hL-ex1}) is given by:
\begin{multline}\label{equ: hL-jac7}
\Ext^1_{\tri}(\pi(\lambda, \psi), \pi(\lambda, \psi)) \twoheadrightarrow \Ima(\iota_1) \xlongrightarrow[\sim]{\iota_1^{-1}} \Ext^1_{\GL_2(\Q_p)}(\St_2^{\infty}(\lambda), \pi(\lambda, \psi)^-) \\
\xlongrightarrow[\sim]{(\ref{equ: hL-jac2})} \Ext^1_{T(\Q_p)}\big(\delta_{\lambda}(|\cdot|\otimes |\cdot|^{-1}), \delta_{\lambda}(|\cdot|\otimes |\cdot|^{-1})\big)_\psi.
\end{multline}
\end{proof}

\begin{remark}\label{rem: hL-tri}
{\rm By Lemma \ref{lem: hL-tri1}(2) and its proof, for any $\widetilde{\pi}\in \Ext^1_{\tri}(\pi(\lambda, \psi), \pi(\lambda, \psi))$, the composition (\ref{equ: hL-jac7}) sends $\widetilde{\pi}$ to the (unique) deformation $\widetilde{\chi}\in \Ext^1_{T(\Q_p)}\big(\delta_{\lambda}(|\cdot|\otimes |\cdot|^{-1}), \delta_{\lambda}(|\cdot|\otimes |\cdot|^{-1})\big)_\psi$ such that $\widetilde{\chi} \hookrightarrow J_B(\widetilde{\pi})$.}
\end{remark}

\noindent
We denote by $\kappa$ the following composition:
\begin{multline}\label{equ: hL-kp1aut}
 \kappa: \Ext^1_{\tri}(\pi(\lambda, \psi), \pi(\lambda, \psi)) \xlongrightarrow{(\ref{equ: hL-jac7})} \Ext^1_{T(\Q_p)}\big(\delta_{\lambda}(|\cdot|\otimes |\cdot|^{-1}), \delta_{\lambda}(|\cdot|\otimes |\cdot|^{-1})\big)_\psi\\ \cong \Hom(T(\Q_p), E)_{\psi} \xlongrightarrow{\pr_2} \Hom(\Q_p^{\times}, E)
\end{multline}
where the last map sends $\Psi=(\psi_1, \psi_2)$ to $\psi_2$ (and hence is surjective). From the exact sequence:
\begin{equation*}
0 \lra E \psi \lra \Hom(T(\Q_p), E)_{\psi} \xlongrightarrow{\pr_2} \Hom(\Q_p^{\times}, E) \lra 0
\end{equation*}
(where the injection is $\psi\mapsto \Psi=(\psi,0)$), we obtain with (\ref{equ: hL-ex1}) an exact sequence (compare with (\ref{equ: hL-kap})):
\begin{equation}\label{equ: hL-kappa1}
0 \lra \Ext^1_{\GL_2(\Q_p)}(\pi(\lambda, \psi)/\St_2^{\infty}(\lambda), \pi(\lambda, \psi)) \xlongrightarrow{\iota_0} \Ker(\kappa) \lra E \psi \lra 0.
\end{equation}

\begin{lemma}\label{lem: hL-kp1}
(1) We have $\dim_E \Ker(\kappa)=2$, and $\widetilde{\pi}\in \Ker (\kappa)$ if and only if $\kappa_1(\widetilde{\pi})\in E \iota_1(\pi(\lambda, \psi,0)^-)$ where $\kappa_1$ is as in (\ref{equ: hL-stv3}) and $\iota_1$ as in (\ref{equ: hL-iot1}).\\
\noindent
(2) We have $\Ext^1_{\GL_2(\Q_p),Z}(\pi(\lambda, \psi), \pi(\lambda, \psi)) \cap \Ker (\kappa)=\Ima(\iota_0)$ (where the intersection is in $\Ext^1_{\GL_2(\Q_p)}(\pi(\lambda, \psi), \pi(\lambda, \psi))$), and it is a $1$-dimensional $E$-vector space.
\end{lemma}
\begin{proof}
(1) The first statement follows from (\ref{equ: hL-kappa1}) and Lemma \ref{lem: hL-ext1}(2). By (\ref{equ: hL-jac4}) applied to $\Psi=(\psi,0)$ and Remark \ref{rem: hL-tri}, the ``if" part follows. However, it is straightforward from (\ref{equ: hL-stv3}) and Lemma \ref{lem: hL-ext1}(2) that $\kappa_1^{-1}(E \iota_1([\pi(\lambda, \psi,0)^-]))$ is also $2$-dimensional. The ``only if" part follows.\\
\noindent
(2) The direction $\supseteq$ is clear from the definitions and Lemma \ref{lem: hL-ext1}(3). By (1), it is thus sufficient to show that if $\kappa_1([\widetilde{\pi}])\ne 0$, i.e. $\kappa_1([\widetilde{\pi}])\in E^{\times} \iota_1([\pi(\lambda, \psi,0)^-])$, then $\widetilde{\pi}$ does not have central character $\chi_{\lambda}$ (which is the central character of $\pi(\lambda, \psi)$), and it is enough to show that $\pi(\lambda, \psi,0)^-$ does not have central character $\chi_{\lambda}$. By the construction following Theorem \ref{thm: hL-lGL2} and by Lemma \ref{lem: hL-cc} (applied first to the extension $W$ of $V_1=\pi(\lambda, \psi,0)^-$ by $V_2=L(\lambda)$ inside $V:=(\Ind_{\overline{B}(\Q_p)}^{\GL_2(\Q_p)} \delta_{\lambda}\otimes_E \sigma(\psi, 0))^{\an}$, then to $V_1=W$, $V_2=I(\lambda)/L(\lambda)$), if $\pi(\lambda, \psi,0)^-$ has central character $\chi_{\lambda}$, so does $(\Ind_{\overline{B}(\Q_p)}^{\GL_2(\Q_p)} \delta_{\lambda}\otimes_E \sigma(\psi, 0))^{\an}$, a contradiction.
\end{proof}

\noindent
We denote by $\Ext^1_g(\pi(\lambda, \psi), \pi(\lambda, \psi))$ the $E$-vector subspace of $\Ext^1_{\GL_2(\Q_p)}(\pi(\lambda, \psi), \pi(\lambda, \psi))$ generated by those $\widetilde{\pi}$ such that $\widetilde{\pi}^{\lalg}\neq \pi(\lambda, \psi)^{\lalg}$.

\begin{lemma}\label{lem: hL-dR}(1)
We have ($\iota_0$ as in (\ref{equ: hL-stv3})):
\begin{equation*}
\Ima(\iota_0) \subseteq \Ext^1_g(\pi(\lambda, \psi), \pi(\lambda, \psi))\subseteq \Ext^1_{\tri}(\pi(\lambda, \psi), \pi(\lambda, \psi)).\end{equation*}
\noindent
(2) The exact sequence (\ref{equ: hL-ex1}) induces an exact sequence:
\begin{multline}\label{equ: hL-ex2}
 0 \lra \Ext^1_{\GL_2(\Q_p)}(\pi(\lambda, \psi)/\St_2^{\infty}(\lambda), \pi(\lambda, \psi)) \xlongrightarrow{\iota_0} \Ext^1_g(\pi(\lambda, \psi), \pi(\lambda, \psi))\\ \lra \Hom_{\infty}(T(\Q_p), E)_{\psi} \lra 0
\end{multline}
 where we have identified $\Ext^1_{T(\Q_p)}\big(\delta_{\lambda}(|\cdot|\otimes |\cdot|^{-1}), \delta_{\lambda}(|\cdot|\otimes |\cdot|^{-1})\big)_\psi$ with $\Hom(T(\Q_p),E)_{\psi}$ and $\Hom_{\infty}(T(\Q_p),E)_{\psi}$ is the subspace of smooth characters in $\Hom(T(\Q_p),E)_{\psi}$. In particular:
\begin{equation*}
  \dim_E \Ext^1_g(\pi(\lambda, \psi),\pi(\lambda, \psi))=\begin{cases}
   2 & \psi \text{ non smooth}\\
   3 & \psi \text{ smooth}.
  \end{cases}
\end{equation*}
\end{lemma}
\begin{proof}
(1) It is easy to see $\Ima(\iota_0)\subseteq \Ext^1_g(\pi(\lambda, \psi),\pi(\lambda, \psi))$. Since $\soc_{\GL_2(\Q_p)} \pi(\lambda, \psi)\cong \St_2^{\infty}(\lambda)$, for any $\widetilde{\pi}\in \Ext^1_g(\pi(\lambda, \psi), \pi(\lambda, \psi))$, its image $\kappa_1(\widetilde{\pi})$ in $\Ext^1_{\GL_2(\Q_p)}(\St_2^{\infty}(\lambda), \pi(\lambda, \psi))$ (see (\ref{equ: hL-stv3})) in fact lies in the image of:
\begin{equation*}
\Ext^1_{\GL_2(\Q_p)}(\St_2^{\infty}(\lambda), \pi(\lambda, \psi)^{\lalg}) \lra \Ext^1_{\GL_2(\Q_p)}(\St_2^{\infty}(\lambda), \pi(\lambda, \psi)).
\end{equation*}
This easily implies $\kappa_0(\widetilde{\pi})=0$ ($\kappa_0$ as in (\ref{equ: hL-kappa0})), and (1) follows.\\

\noindent
(2) Let $\widetilde{\pi}\in \Ext^1_g(\pi(\lambda, \psi), \pi(\lambda, \psi))$. By (1) and Remark \ref{rem: hL-tri}, we know that there exists $\Psi\in \Hom(T(\Q_p),E)_{\psi}$ such that
\begin{equation} \label{equ: hL-jac6}\delta_{\lambda}(|\cdot|\otimes |\cdot|^{-1})\otimes_E (1+\Psi\epsilon)\hooklongrightarrow J_B(\widetilde{\pi}).
\end{equation}
Moreover, the natural surjection $\widetilde{\pi}\twoheadrightarrow \pi(\lambda,\psi)$ induces a nonzero map $\widetilde{\pi}^{\lalg}/\pi(\lambda,\psi)^{\lalg} \ra \pi(\lambda,\psi)^{\lalg}$, and hence we have $\St_2^{\infty}(\lambda)\hookrightarrow \widetilde{\pi}^{\lalg}/\pi(\lambda, \psi)^{\lalg}$ (since $\soc_{\GL_2(\Q_p)} \pi(\lambda, \psi)^{\lalg}\cong \St_2^{\infty}(\lambda)$). Thus $\St_2^{\infty}(\lambda)$ is not an irreducible constituent of $\widetilde{\pi}/\widetilde{\pi}^{\lalg}$, from which we see (together with the left exactness of $J_B$ and \cite[Ex.~5.1.9]{Em2}) that the map (\ref{equ: hL-jac6}) must have image in the subspace $J_B(\widetilde{\pi}^{\lalg})$. However, $J_B(\widetilde{\pi}^{\lalg})$ is locally algebraic since so is $\widetilde{\pi}^{\lalg}$, which implies $\Psi\in \Hom_{\infty}(T(\Q_p),E)\cap \Hom(T(\Q_p),E)_{\psi}=\Hom_{\infty}(T(\Q_p),E)_{\psi}$.\\
\noindent
By (1), the sequence (\ref{equ: hL-ex1}) hence induces (\ref{equ: hL-ex2}), except for the surjectivity on the right. By Lemma \ref{lem: hL-ext1}(2) and an easy dimension count, it is enough to prove this surjectivity.
However, by the construction in Remark \ref{rem: hL-Jac}, if $\Psi$ in Remark \ref{rem: hL-Jac} is smooth, then we see that the inverse image $\widetilde{\pi}_1$ of $\Psi$ in (\ref{equ: hL-jac2}) has extra locally algebraic vectors than $(\pi(\lambda, \psi)^-)^{\lalg}=\pi(\lambda, \psi)^{\lalg}$. Let $\widetilde{\pi}\in \Ext^1_{\GL_2(\Q_p)}(\pi(\lambda, \psi), \pi(\lambda, \psi))$ such that $\kappa_1(\widetilde{\pi})=\iota_1(\widetilde{\pi}_1)$, it is easy to see that we have an injection $\widetilde{\pi}_1\subset \widetilde{\pi}$, hence $\widetilde{\pi}\in\Ext^1_g(\pi(\lambda, \psi), \pi(\lambda, \psi))$, and that $\widetilde{\pi}$ is sent (up to nonzero scalars) to $\Psi$ via (\ref{equ: hL-jac7}) (use Remark \ref{rem: hL-tri} and that $\widetilde{\pi}_1$ is sent to $\Psi$ via (\ref{equ: hL-jac2})). This concludes the proof.
\end{proof}

\noindent
Finally, for any locally algebraic character $\delta: \Q_p^{\times} \ra E^{\times}$, it is obvious that all the above results hold if we twist all the representations of $\GL_2(\Q_p)$ by $\delta \circ \dett$.

\subsubsection{$p$-adic correspondence for $\GL_2(\Q_p)$ and deformations}\label{sec: hL-pLL}

\noindent
We relate the Ext${}^1$ groups of \S~\ref{sec: hL-FM3} to those of \S~\ref{sec: hL-ext1} via the local $p$-adic correspondence for $\GL_2(\Q_p)$.\\

\noindent
We keep all the previous notation. For $k\in \Z_{>0}$ and $0\neq \psi\in \Hom(\Q_p^{\times}, E)$, we denote by $D(k, \psi)\in \Ext^1_{(\varphi,\Gamma)}(\cR_E, \cR_E(|\cdot|x^k))$ the unique (nonsplit) extension up to isomorphism such that:
\begin{equation*}
(ED(k, \psi))^{\perp}=E \psi\in \Ext^1_{(\varphi,\Gamma)}(\cR_E(|\cdot|x^k), \cR_E(|\cdot|x^k))\buildrel (\ref{equ: hL-cad}) \over \cong \Hom(\Q_p^{\times}, E)
\end{equation*}
for the perfect pairing given by the cup-product:
$$\Ext^1_{(\varphi,\Gamma)}(\cR_E, \cR_E(|\cdot|x^k)) \times \Ext^1_{(\varphi,\Gamma)}(\cR_E(|\cdot|x^k),\cR_E(|\cdot|x^k)) \lra E.$$
For $\lambda=(k_1,k_2)\in \Z^{2}$ with $k_1>k_2$, we denote by $D(\lambda, \psi):=D(k_1-k_2,\psi)\otimes_{\cR_E} \cR_E(x^{k_2})$ and $\lambda^{\flat} :=(k_1, k_2+1)$. For $\alpha\in E^{\times}$, we set:
\begin{equation}\label{dalpha}
D(\alpha, \lambda, \psi):=D(\lambda, \psi)\otimes_{\cR_E} \cR_E(\unr(\alpha)).
\end{equation}
We also make the following hypotheses.

\begin{hypothesis}\label{hypo: hL-pLL0}
(1) There exists a natural isomorphism:
\begin{equation}\label{equ: hL-pLL0}
\pLL: \Ext^1_{(\varphi,\Gamma)}(D(p,\lambda, \psi), D(p,\lambda, \psi)) \xlongrightarrow{\sim} \Ext^1_{\GL_2(\Q_p)}(\pi(\lambda^{\flat}, \psi), \pi(\lambda^{\flat}, \psi))
\end{equation}
and any representation in $\Ext^1_{\GL_2(\Q_p)}(\pi(\lambda, \psi), \pi(\lambda, \psi))$ is very strongly admissible.\\
\noindent
(2) The isomorphism (\ref{equ: hL-pLL0}) induces an isomorphism:
\begin{equation}\label{equ: hL-pLLtri}
\Ext^1_{\tri}\big(D(p,\lambda, \psi),D(p,\lambda, \psi)\big)\xlongrightarrow{\sim} \Ext^1_{\tri}\big(\pi(\lambda^{\flat}, \psi), \pi(\lambda^{\flat}, \psi)\big).
\end{equation}
(3) For $\widetilde{D}\in \Ext^1_{\tri}(D(p,\lambda, \psi),D(p,\lambda, \psi))$ and $(\psi_1, \psi_2)\in \Hom(\Q_p^{\times},E)^2$ such that:
$$\big(x^{k_1}(1+\psi_1\epsilon), |\cdot|^{-1}x^{k_2}(1+\psi_2\epsilon)\big)$$
is a trianguline parameter of $\widetilde{D}$ (see \S~\ref{sec: hL-FM3}), if $\widetilde{\pi}\in \Ext^1_{\tri}\big(\pi(\lambda^{\flat}, \psi), \pi(\lambda^{\flat}, \psi)\big)$ is the image of $\widetilde{D}$ via the isomorphism (\ref{equ: hL-pLLtri}), we have an embedding:
\begin{equation*}
 \delta_{\lambda^{\flat}}(|\cdot|\otimes |\cdot|^{-1})(1+\Psi\epsilon) \hooklongrightarrow J_B(\widetilde{\pi})
 \end{equation*}
where $\Psi:=(\psi_1,\psi_2)\in \Hom(T(\Q_p),E)$.
\end{hypothesis}

\begin{remark}\label{remhypo}
{\rm (1) By Lemma \ref{lem: hL-ext2} and Lemma \ref{lem: hL-ext1}(3), Hypothesis \ref{hypo: hL-pLL0}(1) implies Hypothesis \ref{hypo: hL-dim}.\\
\noindent
(2) Under mild hypothesis, we will show Hypothesis \ref{hypo: hL-pLL0} in Proposition \ref{prop: hL-gl2pLL} and Proposition \ref{prop: hL-gl2pLL2} below using some deformation theory combined with Colmez's functor. The isomorphism (\ref{equ: hL-pLL0}) should also induce a bijection:
\begin{equation}\label{equ: hL-pLLcent}
 \Ext^1_{(\varphi,\Gamma), Z}(D(p,\lambda, \psi), D(p,\lambda, \psi)) \xlongrightarrow{\sim} \Ext^1_{\GL_2(\Q_p),Z}(\pi(\lambda^{\flat}, \psi), \pi(\lambda^{\flat}, \psi)),
\end{equation}
but we won't need this property in the paper.}
\end{remark}

\begin{lemma}\label{lem: hL-kappag}
Assuming Hypothesis \ref{hypo: hL-pLL0}, then (\ref{equ: hL-pLL0}) induces isomorphisms:
\begin{eqnarray}
\label{equ: hL-pLLdR} \Ext^1_g\big(D(p,\lambda, \psi),D(p,\lambda, \psi)\big)&\xlongrightarrow{\sim}& \Ext^1_g\big(\pi(\lambda^{\flat}, \psi), \pi(\lambda^{\flat}, \psi)\big)\\
\label{equ: hL-kp1L} \Ker (\kappa^{\gal}) &\xlongrightarrow{\sim}& \Ker (\kappa^{\aut})
\end{eqnarray}
where we denote by $\kappa^{\gal}$ the morphism $\kappa$ in (\ref{equ: hL-lex}) and by $\kappa^{\aut}$ the morphism $\kappa$ in (\ref{equ: hL-kp1aut}).
\end{lemma}
\begin{proof}
For $\widetilde{D}\in \Ext^1_{(\varphi,\Gamma)}(D(\lambda,\psi), D(\lambda,\psi))$ it follows from Lemma \ref{lem: hL-dRg} that we have $\widetilde{D}\in \Ext^1_{g}(D(\lambda,\psi), D(\lambda,\psi))$ if and only if $\widetilde{D}$ is trianguline and the trianguline parameter of $\widetilde{D}$ is locally algebraic. Together with Remark \ref{rem: hL-tri}, Lemma \ref{lem: hL-dR}(2) and Hypothesis \ref{hypo: hL-pLL0}(2)\&(3), the first isomorphism follows. The second follows from Lemma \ref{lem: hL-kk} together with Remark \ref{rem: hL-tri}, (\ref{equ: hL-kp1aut}) and Hypothesis \ref{hypo: hL-pLL0}(2)\&(3).
\end{proof}

\noindent
The following lemma is a trivial consequence of the Colmez-Fontaine theorem and of the main result of \cite{Berger2}.

\begin{lemma}\label{lem: hL-etale}
Let $\alpha\in E^{\times}$ such that $\val_p(\alpha)=\frac{1-(k_1+k_2)}{2}$, then $D(\alpha, \lambda, \psi)$ is \'etale, i.e. $D(\alpha, \lambda, \psi)\cong D_{\rig}(\rho)$ for a $2$-dimensional continuous representation $\rho$ of $\Gal_{\Q_p}$ over $E$.
\end{lemma}

\noindent
If $\alpha'\in E^{\times}$ is such that $D(\alpha', \lambda, \psi)\cong D_{\rig}(\rho')$ is also \'etale, then $\alpha^{-1}\alpha'\in \co_E^{\times}$ and $\rho'\cong \rho\otimes_E \unr(\alpha'\alpha^{-1})$, hence $\rho$ as in Lemma \ref{lem: hL-etale} is unique up to twist by characters. Let $\rho$ be as in Lemma \ref{lem: hL-etale} (for a choice of $\alpha$) and denote by $\widehat{\pi}(\rho)$ the continuous Banach representation of $\GL_2(\Q_p)$ over $E$ attached to $\rho$ via the local $p$-adic Langlands correspondence for $\GL_2(\Q_p)$ (\cite{Colm10a}). Then we have using Remark \ref{compatgl2} together with \cite{LXZ}:
$$\widehat{\pi}(\rho)^{\an}\cong \pi(p^{-1}\alpha, \lambda^{\flat}, \psi):=\pi(\lambda^{\flat}, \psi)\otimes_E \unr(p^{-1}\alpha) \circ \dett.$$

\begin{proposition}\label{prop: hL-gl2pLL}
Assume $\rho$ admits an invariant lattice such that its mod $\varpi_E$ reduction $\overline\rho$ satisfies (\ref{hypo: hL-modp}) (in the appendix), then Hypothesis \ref{hypo: hL-pLL0}(1) (hence Hypothesis \ref{hypo: hL-dim} by Remark \ref{remhypo}(1)) is true.
\end{proposition}
\begin{proof}
Let $\alpha\in E^{\times}$ such that $D(\alpha, \lambda, \psi)\cong D_{\rig}(\rho)$. By Corollary \ref{coro: hL-pLL1}, Colmez's functor $\hV_{\varepsilon^{-1}}$ (see \S~\ref{notprel}) induces a surjection:
\begin{equation}\label{equ: hL-pLLb}
\Ext^1_{\GL_2(\Q_p)}\big(\widehat{\pi}(\rho), \widehat{\pi}(\rho)\big) \twoheadlongrightarrow \Ext^1_{(\varphi,\Gamma)}\big(D(\alpha, \lambda, \psi), D(\alpha, \lambda, \psi)\big)
\end{equation}
where the $\Ext^1$ on the left is in the category of admissible {\it unitary} Banach representations of $\GL_2(\Q_p)$ (recall unitary means that there exists a unit ball preserved by $\GL_2(\Q_p)$). By the exactness of locally ($\Q_p$-)analytic vectors (\cite[Thm.~7.1]{ST03}), we have a morphism:
\begin{equation}\label{equ: hL-pLLbis}
\Ext^1_{\GL_2(\Q_p)}\big(\widehat{\pi}(\rho), \widehat{\pi}(\rho)\big)\longrightarrow \Ext^1_{\GL_2(\Q_p)}\big(\widehat{\pi}(\rho)^{\an}, \widehat{\pi}(\rho)^{\an}\big)
\end{equation}
which we claim is injective. Indeed, assume there is a continuous $\GL_2(\Q_p)$-equivariant section $\widehat{\pi}(\rho)^{\an}\hookrightarrow \widetilde \pi^{\an}\subset \widetilde\pi$ for $\widetilde \pi\in \Ext^1_{\GL_2(\Q_p)}(\widehat{\pi}(\rho), \widehat{\pi}(\rho))$. Then, using that the universal unitary completion of $\widehat{\pi}(\rho)^{\an} \cong \pi(p^{-1}\alpha, \lambda^{\flat}, \psi)$ is isomorphic to $\widehat{\pi}(\rho)$ (by \cite{CD}) together with the universal property of this universal completion and the exactness in \cite[Thm.~7.1]{ST03}, we easily deduce that the above continuous injection $\widehat{\pi}(\rho)^{\an}\hookrightarrow \widetilde\pi$ canonically factors through a continuous injection $\widehat{\pi}(\rho)\hookrightarrow \widetilde\pi$ which provides a section to $\widetilde\pi\twoheadrightarrow \widehat{\pi}(\rho)$. However, by Lemma \ref{lem: hL-ext2} we have $\dim_E \Ext^1_{(\varphi,\Gamma)}(D(\alpha,\lambda, \psi), D(\alpha,\lambda, \psi))=5$, and by Lemma \ref{lem: hL-ext1}(3) (and twisting by $\unr(p^{-1}\alpha)\circ \dett$) we have $\dim_E \Ext^1_{\GL_2(\Q_p)}(\widehat{\pi}(\rho)^{\an}, \widehat{\pi}(\rho)^{\an})\leq 5$. Thus both (\ref{equ: hL-pLLbis}) and (\ref{equ: hL-pLLb}) are bijective. The composition of (\ref{equ: hL-pLLb}) with the inverse of (\ref{equ: hL-pLLbis}) gives an isomorphism:
\begin{equation}\label{equ: hL-pLLa}
\Ext^1_{(\varphi,\Gamma)}\big(D(\alpha, \lambda, \psi), D(\alpha, \lambda, \psi)\big) \xlongrightarrow{\sim} \Ext^1_{\GL_2(\Q_p)}\big(\widehat{\pi}(\rho)^{\an}, \widehat{\pi}(\rho)^{\an}\big).
\end{equation}
Twisting by $\cR_E(\unr(p\alpha^{-1}))$ on the left hand side and by $\unr(p\alpha^{-1})\circ \dett$ on the right hand side, we deduce an isomorphism:
\begin{equation}\label{equ: hL-pLL5}
\pLL: \Ext^1_{(\varphi,\Gamma)}\big(D(p,\lambda, \psi), D(p,\lambda, \psi)\big) \xlongrightarrow{\sim} \Ext^1_{\GL_2(\Q_p)}(\pi(\lambda^{\flat}, \psi), \pi(\lambda^{\flat}, \psi)).
\end{equation}
The first part of Hypothesis \ref{hypo: hL-pLL0}(1) follows.\\
\noindent
From the bijectivity of (\ref{equ: hL-pLLbis}), we see any element in $\Ext^1_{\GL_2(\Q_p)}(\pi(p^{-1}\alpha, \lambda^{\flat}, \psi), \pi(p^{-1}\alpha, \lambda^{\flat}, \psi))$ is isomorphic to the locally analytic vectors of an extension of $\widehat{\pi}(\rho)$ by $\widehat{\pi}(\rho)$ (in the category of admissible unitary Banach representations of $\GL_2(\Q_p)$) and in particular is very strongly admissible. Twisting by $\unr(p\alpha^{-1})\circ \dett$, we deduce any element in $\Ext^1_{\GL_2(\Q_p)}(\pi(\lambda^{\flat}, \psi), \pi(\lambda^{\flat}, \psi))$ is also very strongly admissible, which is the second part of Hypothesis \ref{hypo: hL-pLL0}(1).
\end{proof}

\begin{remark}\label{rem: hL-pLL2}
{\rm (1) Keeping the assumptions of Proposition \ref{prop: hL-gl2pLL}, by the same argument together with a version with fixed central character of (\ref{equ: hL-pLL}) (see (\ref{equ: GL2-centC})), we can show that (\ref{equ: hL-pLL5}) induces an isomorphism as in (\ref{equ: hL-pLLcent}).\\
(2) Assume $\End_{\Gal_{\Q_p}}(\overline{\rho})=k_E$, any element $t$ in the left hand side set of (\ref{equ: hL-pLLa}) gives rise to an ideal $\cI_t\subseteq R_{\overline{\rho}}$ with $R_{\overline{\rho}}/\cI_t\cong \co_E[\epsilon]/\epsilon^2$ ($R_{\overline{\rho}}$ is the universal deformation ring of $\overline{\rho}$, see \S~\ref{sec: ord-galoisdef}). With the notation of \S~\ref{sec: app-def}, the map (\ref{equ: hL-pLLa}) then sends $t$ to $((\pi^{\univ}(\overline{\rho}) \otimes_{R_{\overline{\rho}}} R_{\overline{\rho}}/\cI_t)\otimes_{\co_E} E)^{\an}$.}
\end{remark}

\noindent
The following proposition is presumably not new, but we couldn't find the precise statement in the existing literature. We provide a complete proof in \S~\ref{sec: app-tri}.

\begin{proposition}\label{prop: hL-gl2pLL2}
Keep \ the \ assumptions \ of \ Proposition \ \ref{prop: hL-gl2pLL} \ and \ assume \ moreover $\End_{\Gal_{\Q_p}}(\overline{\rho})=k_E$, and $p\geq 5$ if $\overline{\rho}$ is nongeneric (see just before Proposition \ref{prop: gl2-tang} for this terminology). Then Hypothesis \ref{hypo: hL-pLL0}(2) and Hypothesis \ref{hypo: hL-pLL0}(3) are true. Consequently, the statements in Lemma \ref{lem: hL-kappag} also hold.
\end{proposition}

\begin{remark}
{\rm Assume $\psi$ smooth, let $\widetilde{\pi}\in \Ext^1_g(\pi(\lambda^{\flat}, \psi), \pi(\lambda^{\flat}, \psi))$ as in Remark \ref{rem: hL-sp} (with $\lambda$ replaced by $\lambda^{\flat}$), and let $\widetilde{D}\in \Ext^1_{g}(D(p,\lambda,\psi), D(p,\lambda,\psi))$ the inverse image of $\widetilde{\pi}$ via the isomorphism (\ref{equ: hL-pLLdR}). By Remark \ref{rem: hL-sp} (see in particular (\ref{equ: hL-lalg})), the existence of $\widetilde{D}$ confirms the discussion in \cite[Rem. 1.6(a)]{Dos15}.}
\end{remark}

\subsection{$\cL$-invariants for $\GL_3(\Q_p)$}\label{allthecij}

\noindent
We use the previous results for $\GL_2(\Q_p)$ and the results of \S~\ref{sec: hL-FM} to associate to a $3$-dimensional semi-stable representation of $\Gal_{\Q_p}$ with $N^2\ne 0$ and distinct Hodge-Tate weights one of the finite length locally analytic representations of $\GL_3(\Q_p)$ constructed in \cite{Br16}.

\subsubsection{Notation and preliminaries}\label{sec: hL-GL30}

\noindent
We define some useful locally analytic representations of $\GL_3(\Q_p)$.\\

\noindent
We now switch to $\GL_3(\Q_p)$ and we let $B(\Q_p)$ (resp. $\overline{B}(\Q_p)$) be the Borel subgroup of upper (resp. lower) triangular matrices, $T(\Q_p)$ the diagonal torus and $N(\Q_p)$ (resp. $\overline{N}(\Q_p)$) the unipotent radical of $B(\Q_p)$ (resp. $\overline{B}(\Q_p)$). We set:
$$P_1(\Q_p):=\begin{pmatrix} * & * & * \\ * & * & *\\0 & 0&* \end{pmatrix},\ \ P_2(\Q_p):=\begin{pmatrix} * & * & * \\0&*&*\\0 & *&*\end{pmatrix}.$$
For $i\in \{1,2\}$ we denote by $L_i(\Q_p)$ the Levi subgroup of $P_i(\Q_p)$ containing $T(\Q_p)$, $N_i(\Q_p)$ the unipotent radical of $P_i(\Q_p)$, $\overline{P}_i(\Q_p)$ the parabolic subgroup opposite to $P_i(\Q_p)$ and $\overline{N}_i(\Q_p)$ the unipotent radical of $\overline{P}_i$. Finally we let $\ug$, $\ub$, $\ft$, $\fn$, $\fp_i$, $\fl_i$, $\fn_i$, $\overline{\fn}_i$ the respective $\Q_p$-Lie algebras.\\

\noindent
We fix $\lambda=(k_1,k_2,k_3)$ a dominant integral weight of $\ft$ with respect to the Borel subgroup $B$, i.e. $k_1\geq k_2\geq k_3$. We let $L(\lambda)$ (resp. $L_i(\lambda)$ for $i\in \{1,2\}$) be the algebraic representation of $\GL_3(\Q_p)$ (resp. of $L_i(\Q_p)$) of highest weight $\lambda$ and $\delta_{\lambda}$ be the algebraic character of $T(\Q_p)$ of weight $\lambda$. To lighten notation we set:
\begin{eqnarray*}
 I_{\overline{B}}^{\GL_3}(\lambda)&:=&\big(\Ind_{\overline{B}(\Q_p)}^{\GL_3(\Q_p)} \delta_{\lambda}\big)^{\an} \\
 I_{\overline{P}_i}^{\GL_3}(\lambda)&:=& \big(\Ind_{\overline{P}_i(\Q_p)}^{\GL_3(\Q_p)} L_i(\lambda)\big)^{\an} \\
 i_{\overline{B}}^{\GL_3}(\lambda) &:=&\big(\Ind_{\overline{B}(\Q_p)}^{\GL_3(\Q_p)} 1\big)^{\infty} \otimes_E L(\lambda) \\
 i_{\overline{P}_i}^{\GL_3}(\lambda) &:=&\big(\Ind_{\overline{P}_i(\Q_p)}^{\GL_3(\Q_p)} 1\big)^{\infty} \otimes_E L(\lambda).
\end{eqnarray*}
We also set:
\begin{eqnarray*}
\St_3^{\an}(\lambda)&:=&I_{\overline{B}}^{\GL_3}(\lambda)/\sum_{i=1,2} I_{\overline{P}_i}^{\GL_3}(\lambda)\\
\St_3^{\infty}(\lambda)&:=&i_{\overline{B}}^{\GL_3}(\lambda)/\sum_{i=1,2} i_{\overline{P}_i}^{\GL_3}(\lambda)\\
v_{\overline{P}_i}^{\an}(\lambda)&:=&I_{\overline{P}_i}^{\GL_3}(\lambda)/L(\lambda)\\
v_{\overline{P}_i}^{\infty}(\lambda)&:=&i_{\overline{P}_i}^{\GL_3}(\lambda)/L(\lambda).
\end{eqnarray*}
We have $\St_3^{\infty}(\lambda) \xrightarrow{\sim} \St_3^{\an}(\lambda)^{\lalg}$, $v_{\overline{P}_i}^{\infty}(\lambda)\xrightarrow{\sim}v_{\overline{P}_i}^{\an}(\lambda)^{\lalg}$ and long exact sequences (cf. \cite[Prop. 5.4]{Sch11}):
\begin{eqnarray}\label{equ: hL-resol}
0 \lra L(\lambda)\lra I_{\overline{P}_1}^{\GL_3}(\lambda) \oplus I_{\overline{P}_2}^{\GL_3}(\lambda) \lra I_{\overline{B}}^{\GL_3}(\lambda) \lra \St_3^{\an}(\lambda)\lra 0&&\\
\nonumber
0 \lra L(\lambda) \lra i_{\overline{P}_1}^{\GL_3}(\lambda) \oplus i_{\overline{P}_2}^{\GL_3}(\lambda) \lra i_{\overline{B}}^{\GL_3}(\lambda) \lra \St_3^{\infty}(\lambda) \lra 0.&&
\end{eqnarray}
For an integral weight $\mu$, we denote by $\overline{L}(\mu)$ the unique simple quotient of the Verma module $\overline{M}(\mu):=\text{U}(\ug)\otimes_{\text{U}(\overline{\ub})} \mu$. Note that we have $\overline{L}(-\lambda)\cong L(\lambda)'$. We use without comment the theory of \cite{OS}, see e.g. \cite[\S~2]{Br13I} for a summary. We often write $\GL_3$, $\overline{P}_i$, $Z$ (= the center of $\GL_3$) instead of $\GL_3(\Q_p)$, $\overline{P}_i(\Q_p)$, $Z(\Q_p)$ etc.\\

\noindent
We now give several useful short exact sequences of admissible locally analytic representations of $\GL_3(\Q_p)$ over $E$. For $i=1,2$, we have a \emph{nonsplit} exact sequence:
\begin{equation}\label{equ: hL-gSt}
 0 \lra v_{\overline{P}_i}^{\infty}(\lambda) \lra v_{\overline{P}_i}^{\an}(\lambda) \lra \cF_{\overline{P}_i}^{\GL_3}(\overline{L}(-s_j\cdot \lambda), 1) \lra 0
\end{equation}
where $j\neq i$ and $s_i$ denotes the simple reflection corresponding to the simple root of $L_i(\Q_p)$. Indeed, by \cite[Lem.~5.3.1]{Br16}, the theory of \cite{OS} and \cite[Cor. 2.5]{Br13I}, we have a nonsplit exact sequence:
\begin{equation*}
 0 \lra i_{\overline{P}_i}^{\GL_3}(\lambda) \lra I_{\overline{P}_i}^{\GL_3}(\lambda) \lra \cF_{\overline{P}_i}^{\GL_3}(\overline{L}(-s_j\cdot \lambda), 1)\lra 0,
\end{equation*}
which together with the fact $\Ext^1_{{\GL_3}(\Q_p)}( \cF_{\overline{P}_i}^{\GL_3}(\overline{L}(-s_j\cdot \lambda), 1), L(\lambda))=0$ (cf. \cite[Cor. 4.3]{Sch11}) implies that (\ref{equ: hL-gSt}) is nonsplit by a straightforward d\'evissage. We let $\lambda_{1,2}:=(k_1,k_2)$ (which is thus a dominant weight for $\GL_2(\Q_p)$ as in \S~\ref{sec: hL-ext1}), it is easy to see that we have a commutative diagram (where we write $\GL_3$ for $\GL_3(\Q_p)$ etc.):
\begin{equation*}\begin{CD}
\big(\Ind_{\overline{P}_1}^{\GL_3}\big((\Ind_{\overline{B}\cap L_1}^{L_{1}}1)^{\infty}\otimes_E L_1(\lambda)\big)\big)^{\an} @>>> \big(\Ind_{\overline{P}_1}^{\GL_3} \St_2^{\infty}(\lambda_{1,2}) \otimes x^{k_3}\big)^{\an} \\
@VVV @VVV \\
\big(\Ind_{\overline{P}_1}^{\GL_3}\big((\Ind_{\overline{B}\cap L_1}^{L_{1}} \delta_{\lambda})^{\an}\big)\big)^{\an} @>>> \big(\Ind_{\overline{P}_1}^{\GL_3} \St_2^{\an}(\lambda_{1,2}) \otimes x^{k_3}\big)^{\an}
\end{CD}
\end{equation*}
where all the vertical maps are injective and all the horizontal maps are surjective. Using the exactness and transitivity of parabolic induction, the bottom surjection induces an isomorphism $I_{\overline{B}}^{\GL_3}(\lambda)/I_{\overline{P}_1}^{\GL_3}(\lambda) \xlongrightarrow{\sim} (\Ind_{\overline{P}_1}^{\GL_3} \St_2^{\an}(\lambda_{1,2}) \otimes x^{k_3})^{\an}$. Together with (\ref{equ: hL-resol}), we deduce an exact sequence:
\begin{equation}\label{equ: hL-gSt1}
 0 \lra v_{\overline{P}_2}^{\an}(\lambda) \lra \big(\Ind_{\overline{P}_1}^{\GL_3} \St_2^{\an}(\lambda_{1,2}) \otimes x^{k_3}\big)^{\an} \lra \St_3^{\an}(\lambda)\lra 0.
\end{equation}
By the theory of \cite{OS} and \cite[Cor. 2.5]{Br13I}, we have a nonsplit exact sequence:
\begin{multline}\label{equ: hL-St1}
 0 \lra\big(\Ind_{\overline{P}_1}^{\GL_3} \St_2^{\infty}\otimes 1\big)^{\infty}\otimes_E L(\lambda) \lra \big(\Ind_{\overline{P}_1}^{\GL_3} \St_2^{\infty}(\lambda_{1,2}) \otimes x^{k_3}\big)^{\an} \\ \lra \cF_{\overline{P}_1}^{\GL_3}(\overline{L}(-s_2\cdot \lambda), \St_2^{\infty}\otimes 1)\lra 0.
\end{multline}
We also have another exact sequence (see e.g. \cite[(53)]{Br16}):
\begin{equation}\label{exactinfty}
 0 \lra v_{\overline{P}_2}^{\infty}(\lambda) \lra \big(\Ind_{\overline{P}_1}^{\GL_3} \St_2^{\infty}\otimes 1\big)^{\infty}\otimes_E L(\lambda) \lra \St_3^{\infty}(\lambda) \lra 0
\end{equation}
and from (\ref{equ: hL-gSt1}), (\ref{equ: hL-St1}), (\ref{equ: hL-gSt}) and (\ref{exactinfty}) we easily deduce that in $(\Ind_{\overline{P}_1}^{\GL_3} \St_2^{\an}(\lambda_{1,2}) \otimes x^{k_3})^{\an}$ we have:
\begin{equation}\label{inter}
\big(\Ind_{\overline{P}_1}^{\GL_3} \St_2^{\infty}(\lambda_{1,2}) \otimes x^{k_3}\big)^{\an} \cap v_{\overline{P}_2}^{\an}(\lambda) = v_{\overline{P}_2}^{\infty}(\lambda).
\end{equation}
Let $C_{2,1}:=\cF_{\overline{P}_1}^{\GL_3}(\overline{L}(-s_2\cdot \lambda), \St_2^{\infty}\otimes 1)$ and:
\begin{equation*}
S_{1,0}:=\big(\Ind_{\overline{P}_1}^{\GL_3} \St_2^{\infty}(\lambda_{1,2}) \otimes x^{k_3}\big)^{\an} /v_{\overline{P}_2}^{\infty}(\lambda)
\end{equation*}
which is a subrepresentation of $\St_3^{\an}(\lambda)$ by (\ref{equ: hL-gSt1}) and (\ref{inter}) and sits in an exact sequence by (\ref{equ: hL-St1}) and (\ref{exactinfty}):
\begin{equation}\label{S10}
 0 \lra \St_3^{\infty}(\lambda) \lra S_{1,0} \lra C_{2,1} \lra 0.
\end{equation}
We claim the latter is nonsplit. Indeed, as in the proof of \cite[Prop. 4.6.1]{Br16}, we have $\Ext^1_{{\GL_3}(\Q_p)}(C_{2,1}, v_{\overline{P}_2}^{\infty}(\lambda))=0$. Together with the fact that (\ref{equ: hL-St1}) is nonsplit, the claim follows by a straightforward d\'evissa\-ge. By replacing $P_1$ by $P_2$ and $s_2$ by $s_1$, we define in the same way $C_{1,1}$ as $C_{2,1}$ and $S_{2,0}$ as $S_{1,0}$, and we have similar results for $C_{1,1}$ and $S_{2,0}$. In particular, we have $\soc_{{\GL_3}(\Q_p)} S_{i,0}\cong \St_3^{\infty}(\lambda)$.\\

\noindent
In the sequel, we define several locally analytic representations $C_{i,j}$ and $S_{i,j}$ of $\GL_3(\Q_p)$ for $i\in \{1,2\}$ and $j\in \{0,1,2,3\}$, these representations being such that $C_{i,0}=\St_{3}^{\infty}(\lambda)$ for $i\in \{1,2\}$ and $C_{i,j}\hookrightarrow \soc_{{\GL_3}(\Q_p)} S_{i,j}$ for all $i,j$.

\subsubsection{Simple $\cL$-invariants}\label{sec: hL-sim3}

\noindent
We recall some facts on simple $\cL$-invariants.\\

\noindent
We keep all the previous notation.

\begin{lemma}\label{lem: hL-sim}
Let $i,j\in\{1,2\}$, $i\neq j$.\\
\noindent
(1) We have $\Ext^1_{{\GL_3}(\Q_p)}(v_{\overline{P}_i}^{\infty}(\lambda), \St_{3}^{\an}(\lambda)/S_{j,0})=0$ and an isomorphism:
\begin{equation*}
\Ext^1_{{\GL_3}(\Q_p)}(v_{\overline{P}_i}^{\infty}(\lambda), S_{j,0}) \xlongrightarrow{\sim} \Ext^1_{{\GL_3}(\Q_p)}(v_{\overline{P}_i}^{\infty}(\lambda), \St_{3}^{\an}(\lambda)).
\end{equation*}
(2) We have $\Ext^1_{{\GL_3}(\Q_p)}(L(\lambda), \St_3^{\an}(\lambda))=0$ and an isomorphism:
\begin{equation*}
 \Ext^1_{{\GL_3}(\Q_p)}(v_{\overline{P}_i}^{\infty}(\lambda), \St_{3}^{\an}(\lambda)) \xlongrightarrow{\sim} \Ext^1_{{\GL_3}(\Q_p)}(i_{\overline{P}_i}^{\GL_3}(\lambda), \St_{3}^{\an}(\lambda)).
\end{equation*}
\end{lemma}
\begin{proof}
In each case, the isomorphism follows from the first equality by an obvious d\'evissage.\\
\noindent
(1) It is enough to prove $\Ext^1_{{\GL_3}(\Q_p)}(v_{\overline{P}_i}^{\infty}(\lambda), C)=0$ for all the irreducible constituents $C$ of $\St_{3}^{\an}(\lambda)/S_{j,0}$. By the theory \cite{OS}, we know that $C$ is of the form $\cF_{\overline{P}_w}^{\GL_3}(\overline{L}(-w\cdot \lambda), \pi^{\infty})$ were $w$ is a nontrivial element of the Weyl group distinct from $s_i$ (since we mod out by $S_{j,0}$), $P_w\subset \GL_3$ is the maximal parabolic subgroup containing $B$ such that $w\cdot \lambda$ is dominant for $L_{P_w}$ (with respect to $B\cap L_{P_w}$) and $\pi^\infty$ is a smooth irreducible representation of $L_{P_w}(\Q_p)$ over $E$. In particular $C$ is a subrepresentation of $(\Ind_{\overline{P}_w(\Q_p)}^{\GL_3(\Q_p)} {L}(w\cdot \lambda)_{L_{P_w}}\otimes_E\pi^\infty)^{\an}$ where ${L}(w\cdot \lambda)_{L_{P_w}}$ is the irreducible algebraic representation of $L_{P_w}$ of highest weight $w\cdot \lambda$ with respect to $B\cap L_{P_w}$. If $w\ne s_j$, i.e. $w$ has length $>1$, then it easily follows from \cite[(4.37)]{Sch11} together with \cite[Thm.~4.10]{Sch11} that we have:
$$\Ext^1_{{\GL_3}(\Q_p)}\big(v_{\overline{P}_i}^{\infty}(\lambda), (\Ind_{\overline{P}_w(\Q_p)}^{\GL_3(\Q_p)} {L}(w\cdot \lambda)_{L_{P_w}}\otimes_E\pi^\infty)^{\an}\big)=0$$
(note that the separateness assumption in {\it loc.cit.} holds since $\Sigma:=v_{\overline{P}_i}^{\infty}(\lambda)$ is locally algebraic (with the notation of {\it loc.cit.})). This implies $\Ext^1_{{\GL_3}(\Q_p)}(v_{\overline{P}_i}^{\infty}(\lambda), C)=0$. If $w=s_j$, then we have $\pi^\infty=\St_{2}^{\infty}\otimes 1$ if $i=1$ or $\pi^\infty=1\otimes \St_{2}^{\infty}$ if $i=2$, and $\Ext^1_{{\GL_3}(\Q_p)}(v_{\overline{P}_i}^{\infty}(\lambda), C)=0$ (via Lemma \ref{lem: hL-cc}) is then one of the cases of \cite[(4.45)]{Sch11} (or its symmetric).\\
\noindent
(2) This follows directly from \cite[Prop.~5.6]{Sch11} and Lemma \ref{lem: hL-cc}.
\end{proof}

\noindent
Let $\Psi=(\psi_1, \psi_2, \psi_3)\in \Hom(T(\Q_p), E)$ (with obvious notation) and consider the exact sequence:
\begin{equation}\label{equ: hL-sim1}
 0 \lra I_{\overline{B}}^{\GL_3}(\lambda) \lra \big(\Ind_{\overline{B}(\Q_p)}^{\GL_3(\Q_p)} \delta_{\lambda}(1+\Psi \epsilon)\big)^{\an} \xlongrightarrow{\pr}I_{\overline{B}}^{\GL_3}(\lambda)\lra 0.
\end{equation}
For $i=1, 2$, we see that $\pr^{-1}(i_{\overline{P_i}}^{\GL_3}(\lambda))/\sum_{j=1,2} I_{\overline{P_j}}^{\GL_3}(\lambda)$ is by construction an extension of $i_{\overline{P}_i}^{\GL_3}(\lambda)$ by $\St_{3}^{\an}(\lambda)$. By Lemma \ref{lem: hL-sim}, it comes from a unique extension $\Pi^i(\lambda, \Psi)_0$ of $v_{\overline{P}_i}^{\infty}(\lambda)$ by $S_{i,0}$. If $\Psi$ is smooth (i.e. all $\psi_j$ are smooth, $j\in \{1,2,3\}$), by considering the following exact sequence (which is then ``contained'' in (\ref{equ: hL-sim1})):
\begin{equation*}
 0 \lra i_{\overline{B}}^{\GL_3}(\lambda) \lra \big(\Ind_{\overline{B}(\Q_p)}^{\GL_3(\Q_p)}(1+\Psi \epsilon)\big)^{\infty} \otimes_E L(\lambda) \xlongrightarrow{\pr'}i_{\overline{B}}^{\GL_3}(\lambda)\lra 0,
\end{equation*}
we see that $\Pi^i(\lambda,\Psi)_0$ then comes via the embedding $\St_3^{\infty}(\lambda)\hookrightarrow S_{i,0}$ from a (unique) locally algebraic extension of $v_{\overline{P}_i}^{\infty}(\lambda)$ by $\St_3^{\infty}(\lambda)$.

\begin{proposition}\label{prop: hL-sim}
For $i\in \{1,2\}$ the extension $\Pi^i(\lambda,\Psi)^-\in \Ext^1_{{\GL_3}(\Q_p)}\big(v_{\overline{P}_i}^{\infty}(\lambda), S_{i,0}\big)$ is split if and only if $\psi_i=\psi_{i+1}$, i.e. $\Psi\in \Hom(Z_{L_i}(\Q_p), E)$ where $Z_{L_i}(\Q_p)$ is the center of $L_i(\Q_p)$. Moreover, we have a commutative diagram:
\begin{equation}\label{equ: hL-sim0}
\begin{CD}
\Hom_{\infty}(\Q_p^{\times}, E) @> \sim >>\Ext^1_{{\GL_3}(\Q_p)}(v_{\overline{P}_i}^{\infty}(\lambda), \St_3^{\infty}(\lambda))\\
@VVV @VVV \\
\Hom(\Q_p^{\times}, E) @> \sim >>\Ext^1_{{\GL_3}(\Q_p)}(v_{\overline{P}_i}^{\infty}(\lambda), S_{j,0})
\end{CD}
\end{equation}
where the vertical maps are the natural injections, the bottom horizontal map is given by the composition of $\Hom(\Q_p^{\times}, E)\buildrel\sim\over\rightarrow \Hom(T(\Q_p),E)/\Hom(Z_{L_i}(\Q_p),E)$ with $\Psi\mapsto \Pi^i(\lambda, \Psi)_0$, and the top horizontal map is induced by the bottom map.
\end{proposition}
\begin{proof}
See \cite[Thm.~2.17\ \&\ Rem.~2.18(ii)]{Ding17}.
\end{proof}

\noindent
We now let $\delta_1:=x^{k_1}$, $\delta_2:=|\cdot|^{-1}x^{k_2-1}$, $\delta_3:=|\cdot|^{-2}x^{k_3-2}$ and identify $\Ext^1_{(\varphi,\Gamma)}(\cR_E(\delta_i), \cR_E(\delta_i))$ with $\Hom(\Q_p^{\times}, E)$ by (\ref{equ: hL-cad}).

\begin{corollary}\label{coro: hL-simL}
For $i,j\in\{1,2\}$, $i\neq j$, we have a natural perfect pairing:
\begin{equation*}
\Ext^1_{{\GL_3}(\Q_p)}(v_{\overline{P}_i}^{\infty}(\lambda), S_{j,0}) \times \Ext^1_{(\varphi,\Gamma)}(\cR_E(\delta_{i+1}), \cR_E(\delta_i)) \xlongrightarrow{\cup_1} E,
\end{equation*}
and the same holds with $S_{j,0}$ replaced by $\St^{\an}_3(\lambda)$ (for $i\in \{1,2\}$). Moreover, the one dimensional subspace $\Ext^1_e(\cR_E(\delta_{i+1}), \cR_E(\delta_i))$ of crystalline extensions is exactly annihilated by the subspace $\Ext^1_{{\GL_3}(\Q_p)}(v_{\overline{P}_i}^{\infty}(\lambda), \St_3^{\infty}(\lambda))$.
\end{corollary}
\begin{proof}
By Proposition \ref{prop: hL-sim} (together with the above identification) and Proposition \ref{prop-l3-cup}(2) (applied to $\cR_E(\delta_n)=\cR_E(\delta_{i+1})$, $D_1^{n-1}=\cR_E(\delta_i)$ for $i=1,2$), we obtain the perfect pairing of the statement. By Lemma \ref{lem: hL-sim}(1), we have a similar perfect pairing with $S_{j,0}$ replaced by $\St^{\an}_3(\lambda)$. The last part follows then from (\ref{equ: hL-sim0}) and the discussion in Remark \ref{rem: hL-sim}.
\end{proof}

\subsubsection{Parabolic inductions}\label{sec: hL-pI}

\noindent
We study the locally analytic representation $(\Ind_{\overline{P}_1(\Q_p)}^{\GL_3(\Q_p)}\pi(\lambda_{1,2}, \psi)\otimes x^{k_3})^{\an}$ (cf. \S~\ref{sec: hL-ext1}) and some of its subquotients.\\

\noindent
We keep the previous notation and fix $0\neq \psi\in \Hom(\Q_p^{\times}, E)$. For a locally analytic representation $V$ of $\GL_2(\Q_p)$ over $E$ we use the notation:
\begin{equation*}
 I_{\overline{P}_1}^{\GL_3}(V,k_3):=\big(\Ind_{\overline{P}_1(\Q_p)}^{\GL_3(\Q_p)} V\otimes x^{k_3}\big)^{\an}.
\end{equation*}
We have studied the subrepresentation $I_{\overline{P}_1}^{\GL_3}(\St_2^{\infty}(\lambda_{1,2}), k_3)$ in \S~\ref{sec: hL-GL30}. Exactness of parabolic induction gives the isomorphism (recalling that $s$ is the unique nontrivial element in the Weyl group of $\GL_2$):
\begin{equation*}
I_{\overline{P}_1}^{\GL_3}(I(s\cdot \lambda_{1,2}),k_3)\cong I_{\overline{P}_1}^{\GL_3}(\St_2^{\an}(\lambda_{1,2}),k_3)/I_{\overline{P}_1}^{\GL_3}(\St_2^{\infty}(\lambda_{1,2}), k_3).
\end{equation*}
From (\ref{inter}) and (\ref{equ: hL-gSt}) (for $i=2$) we deduce an injection $\cF_{\overline{P}_2}^{\GL_3}(\overline{L}(-s_1\cdot \lambda), 1)\hookrightarrow I_{\overline{P}_1}^{\GL_3}(I(s\cdot\lambda_{1,2}),k_3)$ and together with (\ref{equ: hL-gSt1}) an isomorphism:
\begin{equation*}
 S_{1,1}:= I_{\overline{P}_1}^{\GL_3}(I(s\cdot \lambda_{1,2}),k_3)/\cF_{\overline{P}_2}^{\GL_3}(\overline{L}(-s_1\cdot \lambda), 1) \xlongrightarrow{\sim} \St_3^{\an}(\lambda)/S_{1,0}.
\end{equation*}
Since $C_{1,1}=\cF_{\overline{P}_2}^{\GL_3}(\overline{L}(-s_1\cdot \lambda), \St_2^{\infty}\otimes 1)\hookrightarrow \St_3^{\an}(\lambda)/\St_3^{\infty}(\lambda)$ and $C_{1,1}$ is not an irreducible constituent of $S_{1,0}$ by (\ref{S10}), we have a commutative diagram:
\begin{equation*}\begin{CD}
S_{2,0} @>>> \St_3^{\an}(\lambda) \\
@VVV @VVV \\
 C_{1,1} @>>>  S_{1,1}
 \end{CD}
\end{equation*}
where the vertical maps are the natural surjections and the horizontal maps are injections. From the theory of \cite{OS}, one moreover easily deduces that the irreducible constituents of $S_{1,1}/C_{1,1}$ are:
\begin{multline}\label{equ: hL-ic}
\Big\{ \cF_{\overline{P}_2}^{\GL_3}(\overline{L}(-s_1s_2\cdot \lambda), 1), \ \cF_{\overline{P}_2}^{\GL_3}(\overline{L}(-s_1s_2\cdot \lambda), 1\otimes \St_2^{\infty}), \ \cF_{\overline{P}_1}^{\GL_3}(\overline{L}(-s_2s_1\cdot \lambda),1), \\ \cF_{\overline{P}_1}^{\GL_3}(\overline{L}(-s_2s_1\cdot \lambda), \St_2^{\infty}\otimes 1), \ \cF_{\overline{B}}^{\GL_3}(\overline{L}(-s_1s_2s_1 \cdot \lambda), 1)\Big\},
\end{multline}
all of them occurring with multiplicity one. Since $\pi(\lambda_{1,2}, \psi)^-\cong \St_2^{\an}(\lambda_{1,2})-L(\lambda_{1,2})$ (see (\ref{equ: hL-upsi})), we have an exact sequence:
\begin{equation}\label{Ipi-}
 0 \lra I_{\overline{P}_1}^{\GL_3}(\St_2^{\an}(\lambda_{1,2}), k_3) \lra I_{\overline{P}_1}^{\GL_3}(\pi(\lambda_{1,2}, \psi)^-,k_3) \xlongrightarrow{\pr} I_{\overline{P}_1}^{\GL_3}(L(\lambda_{1,2}), k_3) \lra 0
\end{equation}
where $I_{\overline{P}_1}^{\GL_3}(L(\lambda_{1,2}), k_3)\cong I_{\overline{P}_1}^{\GL_3}(\lambda)$. Denote by:
\begin{equation*}
S_{1,2}:=v_{\overline{P}_1}^{\an}(\lambda), \ \ C_{1,2}:=v_{\overline{P}_1}^{\infty}(\lambda)\cong \soc_{{\GL_3}(\Q_p)} S_{1,2}.
 \end{equation*}
Since $I_{\overline{P}_1}^{\GL_3}(\St_2^{\an}(\lambda_{1,2}), k_3)/v_{\overline{P}_2}^{\an}(\lambda)\cong \St_3^{\an}(\lambda)$ (see (\ref{equ: hL-gSt1})) and $L(\lambda)\hookrightarrow I_{\overline{P}_1}^{\GL_3}(\lambda)$, it follows from (\ref{Ipi-}) together with Lemma \ref{lem: hL-sim}(2) that we have an injection:
\begin{equation*}
L(\lambda)\hooklongrightarrow I_{\overline{P}_1}^{\GL_3}(\pi(\lambda_{1,2}, \psi)^-,k_3)/v_{\overline{P}_2}^{\an}(\lambda),
\end{equation*}
and \ we \ let \ $\widetilde\Pi^1(\lambda, \psi)^-$ \ be \ the \ cokernel, \ which \ is \ thus \ isomorphic \ to \ an \ extension \ of $I_{\overline{P}_1}^{\GL_3}(\lambda)/L(\lambda)\cong v_{\overline{P}_1}^{\an}(\lambda)$ by $\St_3^{\an}(\lambda)$. Finally we denote by $T_2$, resp. $\overline{B}_2$, the diagonal torus, resp. the lower triangular matrices, of $\GL_2$.

\begin{lemma}\label{lem: hL-sim2}
We have a commutative diagram:
\begin{equation*}
  \begin{CD}
   0 @>>> S_{2,0} @>>> \Pi^1(\lambda, \psi)_0 @>>> v_{\overline{P}_1}^{\infty}(\lambda) @>>> 0 \\
   @. @VVV @VVV @VVV @. \\
   0 @>>> \St_3^{\an}(\lambda) @>>> \widetilde\Pi^1(\lambda, \psi)^- @>>> v_{\overline{P}_1}^{\an}(\lambda) @>>> 0
  \end{CD}
\end{equation*}
where $\Pi^1(\lambda, \psi)_0$ denotes the image of $\psi$ via the bottom isomorphism of (\ref{equ: hL-sim0}).
\end{lemma}
\begin{proof}
Let $\Psi_1:=(\psi_1, \psi_2)\in \Hom(T_2(\Q_p), E)$ and $\Psi:=(\psi_1,\psi_2,0)\in \Hom(T(\Q_p),E)$ with $0\neq \psi_1-\psi_2\in E \psi$. We have (by the transitivity of parabolic inductions):
\begin{multline*}
I_{\overline{P}_1}^{\GL_3}(\pi(\lambda_{1,2}, \psi)^-,k_3) \hooklongrightarrow I_{\overline{P}_1}^{\GL_3}\big((\Ind_{\overline{B}_2(\Q_p)}^{\GL_2(\Q_p)} \delta_{\lambda_{1,2}}(1+\Psi_1\epsilon))^{\an}/L(\lambda_{1,2}), k_3\big) \\ \cong \big(\Ind_{\overline{B}(\Q_p)}^{\GL_3(\Q_p)} \delta_{\lambda}(1+\Psi\epsilon)\big)^{\an}/ I_{\overline{P}_1}^{\GL_3} (\lambda)
\end{multline*}
which induces an injection by (\ref{equ: hL-resol}) together with Lemma \ref{lem: hL-sim}(2):
\begin{equation*}
 \widetilde\Pi^1(\lambda, \psi)^- \hooklongrightarrow W:=\Big(\big(\Ind_{\overline{B}(\Q_p)}^{{\GL_3}(\Q_p)} \delta_{\lambda}(1+\Psi\epsilon)\big)^{\an}/ \sum_{i=1,2}I_{\overline{P}_i}^{\GL_3} (\lambda)\Big)/L(\lambda).
\end{equation*}
By Proposition \ref{prop: hL-sim} and the discussion above it, $W$ contains $\Pi^1(\lambda, \psi)_0$ as subrepresentation, and it is easy to see that the injection $\Pi^1(\lambda, \psi)_0\hookrightarrow W$ factors through $\widetilde\Pi^1(\lambda, \psi)^-$ (e.g. by comparing the irreducible constituents). The lemma follows.
\end{proof}
\noindent We set $\widetilde{C}_{1,2}:=\cF_{\overline{P}_1}^{\GL_3}(\overline{L}(-s_2s_1\cdot \lambda),1)$. By \cite[Prop. 4.2.1 (ii)]{Br16} and the proof of \cite[Lem. 4.4.1]{Br16}, we know that there exits a unique (up to isomorphism) non-split extension $C_{1,1}-\widetilde{C}_{1,2}$, and it is a subrepresentation of $S_{1,1}$. Using the formula in \cite[\S~5.2]{Br16} and (\cite[(4.37)]{Sch11}), it is not difficult to show:
\begin{equation*}
  \Ext^1_{\GL_3(\Q_p)}\big(\widetilde{C}_{1,2}, \cF_{\overline{P}_1}^{\GL_3}(\overline{M}(-s_2\cdot \lambda),\St_2^{\infty}\otimes1)\big)=0,
\end{equation*}
and hence (by d\'evissage) $\Ext^1_{\GL_3(\Q_p)}\big(\widetilde{C}_{1,2}, C_{2,1}\big)=0$. We deduce that $\St_3^{\an}(\lambda)$ (which is of the form $S_{1,0}-S_{1,1}$) has a unique subrepresentation of the form $\St_3^{\infty}(\lambda)-C_{1,1}-\widetilde{C}_{1,2}$, containing $S_{2,0}$. Denote by $\Pi^1(\lambda, \psi)^-$ the push-forward of $\Pi^1(\lambda, \psi)_0$ via $S_{2,0}\hookrightarrow \St_3^{\infty}(\lambda)-C_{1,1}-\widetilde{C}_{1,2}$, which, by Lemma \ref{lem: hL-sim2}, is a subrepresentation of $\widetilde{\Pi}^1(\lambda, \psi)^-$.

\begin{remark}
{\rm If $\psi$ is not smooth then $\Pi^1(\lambda, \psi)^-$ has the form:
\begin{equation*}
 \begindc{\commdiag}[32]
 \obj(0,10)[a]{$\St_3^{\infty}(\lambda)$}
 \obj(14,10)[b]{$C_{1,1}$}
 \obj(26,4)[c]{$C_{1,2}$}
 \obj(26,16)[d]{$\widetilde{C}_{1,2}$}
   \mor{a}{b}{}[+1,\solidline]
   \mor{b}{c}{}[+1,\solidline]
   \mor{b}{d}{}[+1,\solidline]
 \enddc
 \end{equation*}
 whereas if $\psi$ is smooth it has the form:
\begin{equation*}
 \begindc{\commdiag}[32]
 \obj(0,10)[a]{$\St_3^{\infty}(\lambda)$}
  \obj(12,18)[c]{$C_{1,1}$}
  \obj(24,18)[b]{$\widetilde{C}_{1,2}$}
  \obj(14,10)[d]{$C_{1,2}.$}
   \mor{a}{c}{}[+1,\solidline]
   \mor{a}{d}{}[+1,\solidline]
   \mor{c}{b}{}[+1,\solidline]
 \enddc
 \end{equation*}
In all cases $\widetilde\Pi^1(\lambda, \psi)^-$ has the form $S_{1,0} - S_{1,1} - S_{1,2} \cong \St_3^{\an}(\lambda) - S_{1,2}$.}
\end{remark}

\noindent
We now set:
\begin{multline*}
S_{1,3}:=I_{\overline{P}_1}^{\GL_3}(\widetilde{I}(s\cdot \lambda_{1,2}), k_3) \cong \big(\Ind_{\overline{B}(\Q_p)}^{{\GL_3}(\Q_p)} \delta_{s_1\cdot \lambda}(|\cdot|^{-1}\otimes |\cdot|\otimes 1)\big)^{\an}\\ \cong \cF_{\overline{B}}^{\GL_3}\big(\overline{M}(-s_1\cdot \lambda), |\cdot|^{-1}\otimes |\cdot|\otimes 1\big)
\end{multline*}
\begin{multline*}
 C_{1,3}:=\cF_{\overline{B}}^{\GL_3}\big(\overline{L}(-s_1\cdot \lambda), |\cdot|^{-1}\otimes |\cdot|\otimes 1\big) \\ \cong \cF_{\overline{P}_2}^{\GL_3}\big(\overline{L}(-s_1\cdot \lambda), |\cdot|^{-1}\otimes (\Ind_{\overline{B}_2(\Q_p)}^{\GL_2(\Q_p)}|\cdot|\otimes 1)^{\infty}\big) \cong \soc_{{\GL_3}(\Q_p)} S_{1,3},
\end{multline*}
where the last isomorphism follows from \cite[Cor. 2.5]{Br13I}. The irreducible constituents of $S_{1,3}/C_{1,3}$ are (from \cite{OS}):
\begin{multline}\label{equ: hL-ic2}
\Big\{\cF_{\overline{P}_2}^{\GL_3}\big(\overline{L}(-s_1s_2\cdot \lambda), |\cdot|^{-1}\otimes (\Ind_{\overline{B}_2(\Q_p)}^{\GL_2(\Q_p)} |\cdot |\otimes 1)^{\infty}\big), \ \cF_{\overline{P}_1}^{\GL_3}\big(\overline{L}(-s_2s_1\cdot \lambda), \St_2^{\infty}\otimes 1\big), \\ \cF_{\overline{P}_1}^{\GL_3}\big(\overline{L}(-s_2s_1\cdot \lambda), 1\big), \ \cF_{\overline{B}}^{\GL_3}\big(\overline{L}(-s_2s_1s_2\cdot \lambda), |\cdot|^{-1}\otimes |\cdot|\otimes 1\big)\Big\},
\end{multline}
all of them occurring with multiplicity one.

\begin{lemma}\label{lem: hL-C13}
The natural map:
 \begin{equation}\label{equ: hL-E2}
  \Ext^1_{{\GL_3}(\Q_p)}(C_{1,3}, \Pi^1(\lambda, \psi)^-) \lra \Ext^1_{{\GL_3}(\Q_p)}(C_{1,3}, \widetilde\Pi^1(\lambda, \psi)^-)
 \end{equation}
is an isomorphism of $1$-dimensional vector spaces.
\end{lemma}
\begin{proof}
(a) By \cite[Prop. 4.6.1]{Br16}, we have
\begin{equation*}
  \Ext^1_{\GL_3(\Q_p)}(C_{1,3}, \Pi^1(\lambda, \psi)^-) \xlongrightarrow{\sim} \Ext^1_{\GL_3(\Q_p)}(C_{1,3}, \Pi^1(\lambda, \psi)^-/\St_3^{\infty}(\lambda)).
\end{equation*}
By \cite[Prop. 4.4.2 \& Prop. 4.2.1(i)]{Br16} (resp. by \cite[Lem. 4.4.1 \& Prop. 4.2.1(i)]{Br16}), we deduce:
\begin{equation*}
  \dim_E  \Ext^1_{\GL_3(\Q_p)}(C_{1,3}, \Pi^1(\lambda, \psi)^-/\St_3^{\infty}(\lambda))=1
\end{equation*}
in the case where $\psi$ is not smooth (resp. in the case where $\psi$ is smooth).

\noindent
(b) Since $\Hom_{{\GL_3}(\Q_p)}(C_{1,3}, \widetilde\Pi^1(\lambda, \psi)^-/\Pi^1(\lambda, \psi)^-)=0$, we see (\ref{equ: hL-E2}) is injective, and it is sufficient to prove $\Ext^1_{{\GL_3}(\Q_p)}(C_{1,3},C)=0$ for any irreducible constituent of $\widetilde\Pi^1(\lambda, \psi)^-/\Pi^1(\lambda, \psi)^-$. By Step 3 of \cite[Prop. 4.4.2]{Br16}, it is left to show $\Ext^1_{\GL_3(\Q_p)}(C_{1,3},C_{2,1})=0$. However, using \cite[Cor. 5.3.2(ii)\ \&\ Lem. 5.3.3]{Br16} and (\cite[(4.37)]{Sch11}), one can show:
\begin{equation*}
  \Ext^1_{{\GL_3}(\Q_p)}(C_{1,3}, \cF_{\overline{P}_1}^{\GL_3}(\overline{M}(-s_2\cdot \lambda), \St_2^{\infty}\otimes 1))=0
\end{equation*}
and hence $\Ext^1_{\GL_3(\Q_p)}(C_{1,3},C_{2,1})=0$. The lemma follows.
\end{proof}

\noindent
Now consider the exact sequence (see (\ref{pipsi})):
\begin{equation}\label{exactforpush}
 0 \lra I_{\overline{P}_1}^{\GL_3}(\pi(\lambda_{1,2}, \psi)^-, x^{k_3}) \lra I_{\overline{P}_1}^{\GL_3}(\pi(\lambda_{1,2}, \psi), x^{k_3}) \xlongrightarrow{\pr} S_{1,3} \lra 0.
\end{equation}
The push-forward of $\pr^{-1}(C_{1,3})$ via $I_{\overline{P}_1}^{\GL_3}(\pi(\lambda_{1,2}, \psi)^-, x^{k_3}) \twoheadrightarrow \widetilde\Pi^1(\lambda, \psi)^-$ gives an extension of $C_{1,3}$ by $\widetilde\Pi^1(\lambda, \psi)^-$, which by Lemma \ref{lem: hL-C13} comes from an extension of $C_{1,3}$ by $\Pi^1(\lambda, \psi)^-$ denoted by $\Pi^1(\lambda, \psi)$.

\begin{lemma}
The extension $\Pi^1(\lambda, \psi)\in \Ext^1_{{\GL_3}(\Q_p)}(C_{1,3}, \Pi^1(\lambda, \psi)^-)$ is nonsplit.
\end{lemma}
\begin{proof}The lemma follows from Step 2 of the proof of \cite[Prop. 4.4.2]{Br16}.
\end{proof}

\begin{remark}\label{formofrepr}
{\rm (1) If $\psi$ is not smooth then $\Pi^1(\lambda, \psi)$ has the form:
\begin{equation}\label{ctilde}
 \begindc{\commdiag}[32]
 \obj(0,10)[a]{$\St_3^{\infty}(\lambda)$}
 \obj(14,10)[b]{$C_{1,1}$}
 \obj(26,4)[c]{$C_{1,2}$}
 \obj(26,16)[d]{$\widetilde{C}_{1,2}$}
 \obj(38, 10)[e]{$C_{1,3}$,}
   \mor{a}{b}{}[+1,\solidline]
   \mor{b}{c}{}[+1,\solidline]
   \mor{b}{d}{}[+1,\solidline]
   \mor{c}{e}{}[+1,\solidline]
   \mor{d}{e}{}[+1, \dashline]
 \enddc
 \end{equation} whereas if $\psi$ is smooth it has the form:
\begin{equation*}
 \begindc{\commdiag}[32]
 \obj(0,10)[a]{$\St_3^{\infty}(\lambda)$}
  \obj(12,18)[c]{$C_{1,1}$}
   \obj(24,18)[b]{$\widetilde{C}_{1,2}$}
  \obj(14,10)[d]{$C_{1,2}$}
  \obj(26,10)[e]{$C_{1,3}.$}
   \mor{a}{c}{}[+1,\solidline]
   \mor{a}{d}{}[+1,\solidline]
   \mor{d}{e}{}[+1, \solidline]
   \mor{c}{b}{}[+1, \solidline]
 \enddc
 \end{equation*}

\noindent (2) One can actually show that the subquotient $\widetilde{C}_{1,2}-C_{1,3}$ in (\ref{ctilde}) is also non-split (see \cite[Rk.~4.4.3(ii)]{Br16}). But we don't need this fact in the paper.}
\end{remark}

\noindent
Denote by $\widetilde\Pi^1(\lambda, \psi)$ the push-forward of (\ref{exactforpush}) along $I_{\overline{P}_1}^{\GL_3}(\pi(\lambda_{1,2}, \psi)^-, k_3) \twoheadrightarrow \widetilde\Pi^1(\lambda, \psi)^-$, which thus has the following form by Lemma \ref{lem: hL-sim2}:
\begin{equation}\label{sigmatilde1}
\widetilde\Pi^1(\lambda, \psi)\cong S_{1,0} - S_{1,1} - S_{1,2} - S_{1,3}\cong \St_3^{\an}(\lambda) - S_{1,2} - S_{1,3}
\end{equation}
and contains $\Pi^1(\lambda, \psi)$ by Lemma \ref{lem: hL-C13}.\\

\noindent
We define $C_{2, i}$, $S_{2,i}$ for $i\in\{1,2,3\}$, $\widetilde{C}_{2,2}$, $\Pi^2(\lambda, \psi)_0$, $\Pi^2(\lambda, \psi)^-$, $\widetilde{\Pi}^2(\lambda, \psi)^-$, $\Pi^2(\lambda, \psi)$ and $\widetilde\Pi^2(\lambda, \psi)$ in a similar way be replacing $\overline{P}_1$ by $\overline{P}_2$ (and modifying everything accordingly, e.g. $I(s\cdot\lambda_{1,2})\otimes x^{k_3}$ is replaced by $x^{k_1}\otimes I(s\cdot\lambda_{2,3})$ with $\lambda_{2,3}:=(k_2,k_3)$ etc.). In particular all these representations are subquotients of:
\begin{equation*}
 \big(\Ind_{\overline{P}_2(\Q_p)}^{{\GL_3}(\Q_p)} x^{k_1} \otimes \pi(\lambda_{2,3}, \psi)\big)^{\an}
\end{equation*}
and all the above results have their symmetric version with $\overline{P}_1$ replaced by $\overline{P}_2$.

\subsubsection{$\cL$-invariants}\label{linvariantgl3}

\noindent
We associate a finite length locally analytic representation of $\GL_3(\Q_p)$ to a $3$-dimensional semi-stable representation of $\Gal_{\Q_p}$ with $N^2\ne 0$ and distinct Hodge-Tate weights.\\

\noindent
We keep the notation of the previous sections (in particular we have fixed $\lambda=(k_1,k_2,k_3)$ and $0\neq \psi\in \Hom(\Q_p^{\times}, E)$). From the constructions of $\Pi^1(\lambda, \psi)$ and $\widetilde\Pi^1(\lambda, \psi)$ (and from Lemma \ref{lem: hL-C13}), it is not difficult to see that one has an injection:
\begin{equation*}
\Pi^1(\lambda, \psi)^+:=\Pi^1(\lambda, \psi)\oplus_{\St_3^{\infty}(\lambda)} S_{1,0} \ \hookrightarrow \ \widetilde\Pi^1(\lambda, \psi).
 \end{equation*}
From Remark \ref{formofrepr}, we see that $\Pi^1(\lambda, \psi)^+$ has the following form ($\psi$ not smooth on the left, $\psi$ smooth on the right):
$$
 \begindc{\commdiag}[32]
 \obj(0,10)[a]{$\St_3^{\infty}(\lambda)$}
  \obj(12,2)[b]{$C_{2,1}$}
  \obj(12,18)[c]{$C_{1,1}$}
  \obj(24,10)[d]{$C_{1,2}$}
  \obj(36,18)[e]{$C_{1,3}$}
  \obj(24,26)[f]{$\widetilde{C}_{1,2}$}
  \mor{a}{b}{}[+1,\solidline]
   \mor{a}{c}{}[+1,\solidline]
   \mor{c}{d}{}[+1,\solidline]
   \mor{c}{f}{}[+1,\solidline]
   \mor{d}{e}{}[+1, \solidline]
   \mor{f}{e}{}[+1, \dashline]
 \enddc\ \ \ \ \ \ \ \ \ \ \ \ \ \ \ \ \ \
 \begindc{\commdiag}[32]
 \obj(0,10)[a]{$\St_3^{\infty}(\lambda)$}
  \obj(12,2)[b]{$C_{2,1}$.}
  \obj(12,18)[c]{$C_{1,1}$}
  \obj(14,10)[d]{$C_{1,2}$}
  \obj(25,10)[e]{$C_{1,3}$}
  \obj(23,18)[f]{$\widetilde{C}_{1,2}$}
  \mor{a}{b}{}[+1,\solidline]
   \mor{a}{c}{}[+1,\solidline]
   \mor{a}{d}{}[+1,\solidline]
   \mor{d}{e}{}[+1, \solidline]
     \mor{c}{f}{}[+1, \solidline]
 \enddc
$$

\begin{lemma}\label{lem: hL-L3}
(1) The natural map:
\begin{equation}\label{equ: hL-L3a}
\Ext^1_{{\GL_3}(\Q_p)}\big(v_{\overline{P}_2}^{\infty}(\lambda), \Pi^1(\lambda, \psi)^+\big) \longrightarrow \Ext^1_{{\GL_3}(\Q_p)}\big(v_{\overline{P}_2}^{\infty}(\lambda), \widetilde\Pi^1(\lambda,\psi)\big)
\end{equation}
is an isomorphism.

\noindent
(2) We have an exact sequence:
\begin{multline}\label{equ: hL-L3b}
 0 \lra \Ext^1_{{\GL_3}(\Q_p)}\big(v_{\overline{P}_2}^{\infty}(\lambda), \St_3^{\infty}(\lambda)\big) \lra \Ext^1_{{\GL_3}(\Q_p)}\big(v_{\overline{P}_2}^{\infty}(\lambda), \Pi^1(\lambda, \psi)^+\big)\\ \lra \Ext^1_{{\GL_3}(\Q_p)}\big(v_{\overline{P}_2}^{\infty}(\lambda), \Pi^1(\lambda, \psi)/\St_3^{\infty}(\lambda)\big)\oplus \Ext^1_{{\GL_3}(\Q_p)}\big(v_{\overline{P}_2}^{\infty}(\lambda), C_{2,1}\big) \lra 0
\end{multline}
where:
\begin{eqnarray*}
 &&\dim_E \Ext^1_{{\GL_3}(\Q_p)}\big(v_{\overline{P}_2}^{\infty}(\lambda), \St_3^{\infty}(\lambda)\big)=1\\
 && \dim_E \Ext^1_{{\GL_3}(\Q_p)}\big(v_{\overline{P}_2}^{\infty}(\lambda), \Pi^1(\lambda, \psi)^+\big)=3\\
 && \dim_E \Ext^1_{{\GL_3}(\Q_p)}\big(v_{\overline{P}_2}^{\infty}(\lambda), \Pi^1(\lambda, \psi)/\St_3^{\infty}(\lambda)\big)=1\\
 && \dim_E \Ext^1_{{\GL_3}(\Q_p)}\big(v_{\overline{P}_2}^{\infty}(\lambda), C_{2,1}\big) = 1.
\end{eqnarray*}
\end{lemma}
\begin{proof}
(1) It is easy to see $\Hom_{{\GL_3}(\Q_p)}(v_{\overline{P}_2}^{\infty}(\lambda), \widetilde\Pi^1(\lambda, \psi)/\Pi^1(\lambda, \psi)^+)=0$, and thus (\ref{equ: hL-L3a}) is injective. It is sufficient to show $\Ext^1_{{\GL_3}(\Q_p)}(v_{\overline{P}_2}^{\infty}(\lambda), \widetilde\Pi^1(\lambda, \psi)/\Pi^1(\lambda, \psi)^+)=0$. From \cite[(4.37)~\&~(4.41)]{Sch11} and \cite[Prop. 4.10]{Sch11}, we easily deduce that for any irreducible representation $W$ in the union (\ref{equ: hL-ic}) $\cup$ (\ref{equ: hL-ic2}) we have $\Ext^1_{{\GL_3}(\Q_p)}(v_{\overline{P}_2}^{\infty}(\lambda), W)=0$. As in Step 4 of the proof of \cite[Prop. 4.3.1]{Br16}, we also have $\Ext^1_{{\GL_3}(\Q_p)}(v_{\overline{P}_2}^{\infty}(\lambda), \cF_{\overline{P}_1}^{\GL_3}(\overline{L}(-s_2\cdot \lambda), 1))=0$. Since the irreducible constituents of $\widetilde\Pi^1(\lambda, \psi)/\Pi^1(\lambda, \psi)^+$ are exactly given by the representations in the set (\ref{equ: hL-ic}) $\cup$ (\ref{equ: hL-ic2}) $\cup$ $\{\cF_{\overline{P}_1}^{\GL_3}(\overline{L}(-s_2\cdot \lambda), 1)\}$, the result follows by d\'evissage.\\
\noindent
(2) First note that by Lemma \ref{lem: hL-cc} the extension groups in (\ref{equ: hL-L3b}) do not change if $\Ext^1_{{\GL_3}(\Q_p)}$ is replaced by $\Ext^1_{{\GL_3}(\Q_p),Z}$. By \cite[Cor. 4.8]{Sch11}, we have $\Ext^2_{{\GL_3}(\Q_p),Z}(v_{\overline{P}_2}^{\infty}(\lambda), \St_3^{\infty}(\lambda))=0$, from which we easily deduce (\ref{equ: hL-L3b}). By \emph{loc.cit.} we also have:
$$\dim_E \Ext^1_{{\GL_3}(\Q_p),Z}(v_{\overline{P}_2}^{\infty}(\lambda), \St_3^{\infty}(\lambda))=\dim_E \Ext^1_{{\GL_3}(\Q_p)}(v_{\overline{P}_2}^{\infty}(\lambda), \St_3^{\infty}(\lambda))=1.$$
It follows from (\ref{equ: hL-sim0}) and $\Ext^2_{{\GL_3}(\Q_p),Z}(v_{\overline{P}_2}^{\infty}(\lambda), \St_3^{\infty}(\lambda))=0$ that we have:
\begin{equation*}
 \dim_E \Ext^1_{{\GL_3}(\Q_p)}\big(v_{\overline{P}_2}^{\infty}(\lambda), C_{2,1}\big) =\dim_E \Ext^1_{{\GL_3}(\Q_p)}\big(v_{\overline{P}_2}^{\infty}(\lambda), S_{1,0}/\St_3^{\infty}(\lambda)\big) = 1.
\end{equation*}
From Remark \ref{formofrepr} and \cite[Prop. 4.2.2 (ii) \& Prop. 4.2.3 (ii)]{Br16} we easily deduce:
\begin{equation*}
\Ext^1_{{\GL_3}(\Q_p)}\big(v_{\overline{P}_2}^{\infty}(\lambda), \Pi^1(\lambda, \psi)/\St_3^{\infty}(\lambda)\big)\xlongrightarrow{\sim}\Ext^1_{{\GL_3}(\Q_p)}\big(v_{\overline{P}_2}^{\infty}(\lambda), C_{1,2}-C_{1,3}\big) .
\end{equation*}
By \cite[Prop. 4.3.1~\&~Prop.~4.2.1 (i)]{Br16} the latter is one dimensional. This concludes the proof.
\end{proof}

\noindent
By \cite[(4.38)]{Sch11}, we have a spectral sequence: \footnote{Actually, to apply \cite[(4.38)]{Sch11}, one needs to show that the (dual of the) $\overline{P}_1$-representation $\pi(\lambda_{1,2}, \psi)\otimes x^{k_3}$ satisfies the condition (FIN) of \cite[\S~6]{St-dual}. However, any irreducible constituent of $\pi(\lambda_{1,2}, \psi)\otimes x^{k_3}$ is either locally algebraic or isomorphic to a locally analytic principal series, and hence satisfies the condition (FIN) (see the discussion in the beginning of \cite[\S~4.4]{Sch11} for the locally algebraic case, and the discussion before Step $1$ in the proof of \cite[Prop.~4.3.1]{Br16} for the case of principal series). One deduces then that the dual of $\pi(\lambda_{1,2}, \psi)\otimes x^{k_3}$ also satisfies (FIN).}
\begin{multline}\label{equ: hL-spe}
\Ext^i_{L_1(\Q_p), Z}\big(H_j(\overline{N}_1(\Q_p), v_{\overline{P}_2}^{\infty}(\lambda)), \pi(\lambda_{1,2}, \psi)\otimes x^{k_3}\big)\\ \Rightarrow \Ext^{i+j}_{{\GL_3}(\Q_p), Z}\big(v_{\overline{P}_2}^{\infty}(\lambda), I_{\overline{P}_1}^{\GL_3}(\pi(\lambda_{1,2}, \psi), k_3)\big).
\end{multline}
From \cite[(4.41)~\&~(4.42)]{Sch11} and the discussion after \cite[(52)]{Br16} we have (with obvious notation):
\begin{equation*}
H_i(\overline{N}_1(\Q_p),v_{\overline{P}_2}^{\infty}(\lambda))\cong \Big(\bigoplus_{\substack{\lg w=i\\ \text{$w\!\cdot\! \lambda$ \!is\! $B\!\cap\! L_1$\!-dominant}}} L_1(w\cdot \lambda)\Big) \otimes \big((\St_2^{\infty}\otimes 1) \oplus (|\cdot|^{-1}\circ \dett\otimes |\cdot |^{2})\big).
\end{equation*}
For all $w$ with $w\cdot \lambda$ dominant with respect to $B(\Q_p)\cap L_1(\Q_p)$ we have by considering the action of the center of $L_1(\Q_p)$:
\begin{eqnarray*}
\Hom_{L_1(\Q_p)}\big(L_1(w\cdot \lambda)\otimes_E(|\cdot|^{-1}\circ \dett\otimes |\cdot |^{2}), \pi(\lambda_{1,2}, \psi)\otimes x^{k_3}\big)=0 \\
\Ext^1_{L_1(\Q_p)}\big(L_1(w\cdot \lambda)\otimes_E(|\cdot|^{-1}\circ \dett\otimes |\cdot |^{2}), \pi(\lambda_{1,2}, \psi)\otimes x^{k_3}\big)=0
 \end{eqnarray*}
and it is easy to see from the above formula:
\begin{equation*}
 \Hom_{L_1(\Q_p)}\big(H_1(\overline{N}_1(\Q_p), v_{\overline{P}_2}^{\infty}(\lambda)), \pi(\lambda_{1,2}, \psi)\otimes x^{k_3}\big)=0.
\end{equation*}
Thus we deduce from (\ref{equ: hL-spe}) an isomorphism:
\begin{multline}\label{equ: hL-L3c}
 \Ext^1_{L_1(\Q_p), Z}\big(\St_2^{\infty}(\lambda_{1,2}) \otimes x^{k_3}, \pi(\lambda_{1,2}, \psi)\otimes x^{k_3}\big)\\ \xlongrightarrow{\sim} \Ext^1_{{\GL_3}(\Q_p), Z}\big(v_{\overline{P}_2}^{\infty}(\lambda), I_{\overline{P}_1}^{\GL_3}(\pi(\lambda_{1,2}, \psi), k_3)\big).
\end{multline}

\noindent
Denote by $W$ be the kernel of $I_{\overline{P}_1}^{\GL_3}(\pi(\lambda_{1,2}, \psi), k_3) \twoheadrightarrow \widetilde\Pi^1(\lambda, \psi)$, which (by the definition of $\widetilde\Pi^1(\lambda, \psi)$ and $\widetilde\Pi^1(\lambda, \psi)^-$) is an extension of $L(\lambda)$ by $v_{\overline{P}_2}^{\an}(\lambda)$. By \cite[Cor. 2.13]{Ding17}, we have $\Ext^i_{{\GL_3}(\Q_p), Z}(v_{\overline{P}_2}^{\infty}(\lambda), I_{\overline{P}_2}^{\GL_3}(\lambda))=0$ for all $i\geq 0$ and:
\begin{equation*}
 \Ext^i_{{\GL_3}(\Q_p), Z}\big(v_{\overline{P}_2}^{\infty}(\lambda), L(\lambda)\big)=\begin{cases}
  E & {\rm if}\ i=1 \\
  0 & \text{otherwise}.
 \end{cases}
\end{equation*}
By d\'evissage (recall $ I_{\overline{P}_2}^{\GL_3}(\lambda)\cong L(\lambda)-v_{\overline{P}_2}^{\an}(\lambda)$, see \S~\ref{sec: hL-GL30}), we get:
\begin{equation}
 \Ext^i_{{\GL_3}(\Q_p),Z}\big(v_{\overline{P}_2}^{\infty}(\lambda), v_{\overline{P}_2}^{\an}(\lambda)\big)=\begin{cases}
  E & {\rm if}\ i=0 \\
  0 & \text{otherwise}.
 \end{cases}
\end{equation}
Again by d\'evissage, we deduce $\Ext^2_{{\GL_3}(\Q_p), Z}(v_{\overline{P}_2}^{\infty}(\lambda), W) =0$ and an isomorphism:
$$\Ext^1_{{\GL_3}(\Q_p), Z}(v_{\overline{P}_2}^{\infty}\big(\lambda), W\big) \xlongrightarrow{\sim} \Ext^1_{{\GL_3}(\Q_p), Z}\big(v_{\overline{P}_2}^{\infty}(\lambda), L(\lambda)\big)$$
of $1$-dimensional $E$-vector spaces (with \cite[Cor.~4.8]{Sch11} for the dimension). From the former equality we obtain an exact sequence:
\begin{multline}\label{equ: hL-L3d}
 0 \lra \Ext^1_{{\GL_3}(\Q_p),Z}\big(v_{\overline{P}_2}^{\infty}(\lambda), W\big) \lra \Ext^1_{{\GL_3}(\Q_p),Z}\big(v_{\overline{P}_2}^{\infty}(\lambda), I_{\overline{P}_1}^{\GL_3}(\pi(\lambda_{1,2}, \psi), k_3) \big) \\ \lra \Ext^1_{{\GL_3}(\Q_p), Z}\big(v_{\overline{P}_2}^{\infty}(\lambda), \widetilde\Pi^1(\lambda, \psi)\big)
  \lra 0.
\end{multline}
Together with Lemma \ref{lem: hL-L3}, (\ref{equ: hL-L3c}) and a dimension count we obtain:
\begin{multline}\label{dimen=4}
\dim_E \Ext^1_{L_1(\Q_p), Z}\big(\St_2^{\infty}(\lambda_{1,2})\otimes x^{k_3}, \pi(\lambda_{1,2}, \psi)\otimes x^{k_3}\big)\\ =\dim_E \Ext^1_{{\GL_3}(\Q_p),Z}\big(v_{\overline{P}_2}^{\infty}(\lambda), I_{\overline{P}_1}^{\GL_3}(\pi(\lambda_{1,2}, \psi), k_3)\big)=4.
\end{multline}
By (\ref{equ: hL-L3c}) and (\ref{equ: hL-L3d}), we have a natural surjection:
\begin{equation}\label{equ: hL-L3h}
\Ext^1_{L_1(\Q_p), Z}\big(\St_2^{\infty}(\lambda_{1,2}) \otimes x^{k_3}, \pi(\lambda_{1,2}, \psi)\otimes x^{k_3}\big) \twoheadlongrightarrow \Ext^1_{{\GL_3}(\Q_p), Z}\big(v_{\overline{P}_2}^{\infty}(\lambda), \widetilde\Pi^1(\lambda, \psi)\big).
\end{equation}
Similarly, we have natural maps (without fixing the central character of $\GL_3(\Q_p)$ and using (\ref{equ: hL-St1}) and (\ref{exactinfty})):
\begin{multline}\label{equ: hL-L3e}
 \Ext^1_{L_1(\Q_p)}\big(\St_2^{\infty}(\lambda_{1,2})\otimes x^{k_3}, \pi(\lambda_{1,2}, \psi)\otimes x^{k_3}\big) \longrightarrow \Ext^1_{{\GL_3}(\Q_p)}\big(v_{\overline{P}_2}^{\infty}(\lambda), I_{\overline{P}_1}^{\GL_3}(\pi(\lambda_{1,2}, \psi), k_3)\big) \\ \lra \Ext^1_{{\GL_3}(\Q_p)}\big(v_{\overline{P}_2}^{\infty}(\lambda), \widetilde\Pi^1(\lambda, \psi)\big)
\end{multline}
whose composition is surjective by (\ref{equ: hL-L3h}) and the isomorphism (Lemma \ref{lem: hL-cc}):
\begin{equation*}
\Ext^1_{{\GL_3}(\Q_p), Z}\big(v_{\overline{P}_2}^{\infty}(\lambda), \widetilde\Pi^1(\lambda, \psi)\big) \xlongrightarrow{\sim} \Ext^1_{{\GL_3}(\Q_p)}\big(v_{\overline{P}_2}^{\infty}(\lambda), \widetilde\Pi^1(\lambda, \psi)\big).
\end{equation*}

\begin{remark}\label{rem: hL-L3}
{\rm We can describe (\ref{equ: hL-L3e}) (and similarly for (\ref{equ: hL-L3c}) and (\ref{equ: hL-L3h})) in the following explicit way. For any $\widetilde{\pi}\in \Ext^1_{L_1(\Q_p)}(\St_2^{\infty}(\lambda_{1,2}) \otimes x^{k_3}, \pi(\lambda_{1,2}, \psi)\otimes x^{k_3})$, the parabolic induction $(\Ind_{\overline{P}_1(\Q_p)}^{{\GL_3}(\Q_p)} \widetilde{\pi})^{\an}$ lies in an exact sequence:
\begin{equation}\small\label{equ: hL-L3g}
 0 \lra I_{\overline{P}_1}^{\GL_3}(\pi(\lambda_{1,2}, \psi), k_3) \lra \big(\Ind_{\overline{P}_1(\Q_p)}^{{\GL_3}(\Q_p)} \widetilde{\pi}\big)^{\an} \xlongrightarrow{\pr} \big(\Ind_{\overline{P}_1(\Q_p)}^{{\GL_3}(\Q_p)} \St_2^{\infty}(\lambda_{1,2})\otimes x^{k_3}\big)^{\an} \lra 0.
\end{equation}
Then the first map of (\ref{equ: hL-L3e}) is given by sending $\widetilde{\pi}$ to $\pr^{-1}(v_{\overline{P}_2}^{\infty}(\lambda))$ and the second map is given by quotienting by the subspace $W$. In particular the composition sends $\widetilde{\pi}$ to $\pr^{-1}(v_{\overline{P}_2}^{\infty}(\lambda))/W$.}
\end{remark}

\noindent
Consider the following composition:
\begin{multline}\label{equ: hL-L3f}
 \Ext^1_{\GL_2(\Q_p)}\big(\St_2^{\infty}(\lambda_{1,2}), \pi(\lambda_{1,2}, \psi)\big) \lra \Ext^1_{L_1(\Q_p)}\big(\St_2^{\infty}(\lambda_{1,2})\otimes x^{k_3}, \pi(\lambda_{1,2}, \psi)\otimes x^{k_3}\big) \\ \xlongrightarrow{\text{(\ref{equ: hL-L3e})}} \Ext^1_{{\GL_3}(\Q_p)}\big(v_{\overline{P}_2}^{\infty}(\lambda), \widetilde\Pi^1(\lambda, \psi)\big)
\end{multline}
where the first map sends $\widetilde{\pi}$ to $\widetilde{\pi}\otimes x^{k_3}$.

\begin{lemma}\label{lem: hL-L3b}
(1) The composition (\ref{equ: hL-L3f}) is surjective.\\
(2) The kernel of the composition (\ref{equ: hL-L3f}) is $1$-dimensional and is generated by $\iota_1(\pi(\lambda, \psi,0)^-)$ (see (\ref{equ: hL-iot1}) and (\ref{equ: hL-wpsi})).
\end{lemma}
\begin{proof}
(1) For any $\widetilde{\pi}\in \Ext^1_{L_1(\Q_p), Z}(\St_2^{\infty}(\lambda_{1,2})\otimes x^{k_3}, \pi(\lambda_{1,2}, \psi)\otimes x^{k_3})$, we can view $\widetilde{\pi}$ as a representation of $L_1(\Q_p)$ over $E[\epsilon]/\epsilon^2$ by making $\epsilon$ act as the composition (unique up to nonzero scalars):
\begin{equation*}
\widetilde{\pi} \twoheadlongrightarrow \St_2^{\infty}(\lambda_{1,2})\otimes x^{k_3} \hooklongrightarrow \pi(\lambda_{1,2}, \psi)\otimes x^{k_3} \hooklongrightarrow \widetilde{\pi}.
\end{equation*}
Let $Z_2:=\Q_p^{\times} \hookrightarrow L_1(\Q_p)\cong \GL_2(\Q_p)\times \Q_p^{\times}, \ a\mapsto (1,a)$, which acts on $\widetilde{\pi}$ by a character $\widetilde{\chi}$ of $\Q_p^{\times}$ over $E[\epsilon]/\epsilon^2$ (by the same argument as in the proof of Lemma \ref{lem: hL-cent2}). Consider $\widetilde{\pi}':=\widetilde{\pi}\otimes_{E[\epsilon]/\epsilon^2} (\widetilde{\chi}^{-1}\circ \dett)$, on which $Z_2$ acts thus by $x^{k_3}$. So there exists $\widetilde{\pi}'_0\in \Ext^1_{\GL_2(\Q_p)}\big(\St_2^{\infty}(\lambda_{1,2}), \pi(\lambda_{1,2}, \psi)\big)$ such that $\widetilde{\pi}'\cong \widetilde{\pi}'_0\otimes x^{k_3}$ (``external'' tensor product). However, by Lemma \ref{lem: hL-cent3}, Remark \ref{rem: hL-L3} and the fact that:
\begin{equation*}
\big(\Ind_{\overline{P}_1(\Q_p)}^{\GL_3(\Q_p)} \widetilde{\pi}'\big)^{\an}\cong \big(\Ind_{\overline{P}_1(\Q_p)}^{\GL_3(\Q_p)} \widetilde{\pi}\big)^{\an}\otimes_{E[\epsilon]/\epsilon^2} \widetilde{\chi}^{-1}\circ \dett,
\end{equation*}
we see that the image of $\widetilde{\pi}$ via (\ref{equ: hL-L3h}) is isomorphic to the image of $\widetilde{\pi}'_0$ via (\ref{equ: hL-L3f}). Since (\ref{equ: hL-L3h}) is surjective, so is
 (\ref{equ: hL-L3f}).

\noindent
(2) Since (\ref{equ: hL-L3f}) is surjective, by counting dimensions using Lemma \ref{lem: hL-L3} and (\ref{dimen=4}) we see that the kernel of (\ref{equ: hL-L3f}) is one dimensional. It is thus sufficient to prove $\iota_1(\pi(\lambda, \psi,0)^-)$ is sent to zero. Let $\Psi_{1,2}:=(\psi,0)$ and $\Psi:=(\psi,0,0)$. By construction (cf. (\ref{equ: hL-wpsi})), $\pi(\lambda, \psi,0)^-$ is a subquotient of $(\Ind_{\overline{B}_2(\Q_p)}^{\GL_2(\Q_p)} \delta_{\lambda_{1,2}}(1+\Psi_{1,2}\epsilon))^{\an}$, and thus $(\Ind_{\overline{P}_1(\Q_p)}^{{\GL_3}(\Q_p)} \pi(\lambda,\psi, 0)^-\otimes x^{k_3})^{\an}$ is a subquotient of $(\Ind_{\overline{B}(\Q_p)}^{{\GL_3}(\Q_p)} \delta_{\lambda} (1+\Psi\epsilon))^{\an}$. However, from the first part of Proposition \ref{prop: hL-sim} and Lemma \ref{lem: hL-sim}(1), we deduce (see Proposition \ref{prop: hL-sim} for $\Pi^2(\lambda, \Psi)_0$):
\begin{equation*}
v_{\overline{P}_2}^{\infty}(\lambda) \hooklongrightarrow \Pi^2(\lambda, \Psi)_0\hooklongrightarrow \big(\Ind_{\overline{B}(\Q_p)}^{{\GL_3}(\Q_p)} \delta_{\lambda} (1+\Psi\epsilon)\big)/\sum_{i=1,2} I_{\overline{P}_i}^{\GL_3}(\lambda).
 \end{equation*}
In particular the image of $\iota_1(\pi(\lambda, \psi,0)^-)$ via (\ref{equ: hL-L3f}) contains $v_{\overline{P}_2}^{\infty}(\lambda)$ as a subrepresentation, hence the associated extension is split. This concludes the proof.
\end{proof}

\noindent
We now can prove the main result of the section. We let $\lambda^{\sharp}:=(k_1, k_2-1, k_3-2)$, $\lambda_{1,2}^{\sharp}:=(k_1, k_2-1)$, $\lambda_{2,3}^{\sharp}:=(k_2-1, k_3-2)$ and $D_1^2:=D(p, \lambda_{1,2}^{\sharp}, \psi)$ (see (\ref{dalpha}), the notation $D_1^2$ is for (future) compatibility with the notation at the beginning of \S~\ref{sec: hL-FM}).

\begin{theorem}\label{thm: hL-L3}
Assume Hypothesis \ref{hypo: hL-pLL0} for $D_1^2$. The cup product (\ref{equ: hL-cup}) together with the isomorphisms:
\begin{multline*}
\Ext^1_{(\varphi,\Gamma)}\big(D_1^2, D_1^2\big)\cong \Ext^1_{(\varphi,\Gamma)}\big(D(p,\lambda_{1,2}^{\sharp}, \psi),D(p,\lambda_{1,2}^{\sharp}, \psi)\big)\\
\xlongrightarrow[\sim]{(\ref{equ: hL-pLL0})} \Ext^1_{\GL_2(\Q_p)}\big(\pi(\lambda_{1,2}, \psi), \pi(\lambda_{1,2}, \psi)\big)
\end{multline*}
induce a perfect pairing of $3$-dimensional $E$-vector spaces:
\begin{equation}\label{equ: hL-L3}
 \Ext^1_{(\varphi,\Gamma)}\big(\cR_E(x^{k_3-2}|\cdot|^{-2}), D_1^2\big) \times \Ext^1_{{\GL_3}(\Q_p)}\big(v_{\overline{P}_2}^{\infty}(\lambda),\Pi\big) \xlongrightarrow{\cup} E
\end{equation}
with $\Pi= \widetilde\Pi^1(\lambda, \psi)$ or $\Pi^1(\lambda, \psi)^+$.
\end{theorem}
\begin{proof}
The dimension $3$ comes from Lemma \ref{lem: hL-L3}(2). We have morphisms (see (\ref{equ: hL-stv3}) for $\kappa_1$):
\begin{multline}\label{equ: hL-L3p}
 \Ext^1_{\GL_2(\Q_p)}\big(\pi(\lambda_{1,2}, \psi), \pi(\lambda_{1,2}, \psi)\big)
\xlongrightarrow{\kappa_1} \Ext^1_{\GL_2(\Q_p)}\big(\St_2^{\infty}(\lambda_{1,2}), \pi(\lambda_{1,2}, \psi)\big)\\
 \xlongrightarrow{(\ref{equ: hL-L3f})}\Ext^1_{{\GL_3}(\Q_p)}\big(v_{\overline{P}_2}^{\infty}(\lambda), \widetilde\Pi^1(\lambda, \psi)\big)
 \cong \Ext^1_{{\GL_3}(\Q_p)}\big(v_{\overline{P}_2}^{\infty}(\lambda), \Pi^1(\lambda, \psi)^+\big)
\end{multline}
where the first morphism is the surjection in (\ref{equ: hL-stv3}) (it is surjective by Remark \ref{remhypo}(1)) and the last isomorphism is Lemma \ref{lem: hL-L3}(1). By Lemma \ref{lem: hL-L3b}(2) and Lemma \ref{lem: hL-kp1}(1), we obtain that the kernel of the composition in (\ref{equ: hL-L3p}) is equal to $\Ker(\kappa^{\aut})$ where we use the notation of Lemma \ref{lem: hL-kappag}. Note that this composition is surjective by Lemma \ref{lem: hL-L3b}(1) (and the surjectivity of $\kappa_1$). Now consider:
\begin{multline}\label{equ: hL-L3i}
\Ext^1_{(\varphi,\Gamma)}\big(D_1^2, D_1^2\big) \xlongrightarrow{\sim} \Ext^1_{(\varphi,\Gamma)}\big(D(p,\lambda_{1,2}^{\sharp}, \psi),D(p,\lambda_{1,2}^{\sharp}, \psi)\big) \\ \xlongrightarrow[\sim]{(\ref{equ: hL-pLL0})} \Ext^1_{\GL_2(\Q_p)}\big(\pi(\lambda_{1,2}, \psi), \pi(\lambda_{1,2}, \psi)\big) \buildrel {(\ref{equ: hL-L3p})}\over \twoheadlongrightarrow \Ext^1_{{\GL_3}(\Q_p)}\big(v_{\overline{P}_2}^{\infty}(\lambda), \Pi^1(\lambda, \psi)^+\big).
\end{multline}
By (\ref{equ: hL-kp1L}), the kernel of the composition in (\ref{equ: hL-L3i}) is thus isomorphic to $\Ker(\kappa^{\gal})$. Since this composition is moreover surjective, the theorem then follows from Proposition \ref{prop-l3-cup} (where $\kappa$ there is denoted $\kappa^{\gal}$ here).
\end{proof}

\noindent
We let $\delta_1:=x^{k_1}$, $\delta_2:=x^{k_2-1} |\cdot|^{-1}$, and $\delta_3:=x^{k_3-2}|\cdot|^{-2}$.

\begin{proposition}\label{prop: hL-compL}
Assume Hypothesis \ref{hypo: hL-pLL0} for $D_1^2$. We have a commutative diagram:
 \begin{equation}\label{equ: hL-compL}
   \begin{CD}\Ext^1_{{\GL_3}(\Q_p)}\big(v_{\overline{P}_2}^{\infty}(\lambda), \St_3^{\an}(\lambda)\big) @. \ \times \ \ @. \Ext^1_{(\varphi,\Gamma)}\big(\cR_E(\delta_3), \cR_E(\delta_2)\big) @> \cup_1 >> E\\
 @VVV @. @A u_1 AA @ | \\
 \Ext^1_{{\GL_3}(\Q_p)}\big(v_{\overline{P}_2}^{\infty}(\lambda),\widetilde\Pi^1(\lambda, \psi)\big) @. \!\times @. \!\!\!\!\!\!\!\!\Ext^1_{(\varphi,\Gamma)}\big(\cR_E(\delta_3), D_1^2\big) @> \cup >> E\\
 \end{CD}
 \end{equation}
where the left vertical map is the natural injection, the middle vertical map is the natural surjection, the bottom (perfect) pairing is the one in Theorem \ref{thm: hL-L3} and the top (perfect) pairing is the one in Corollary \ref{coro: hL-simL}. The same holds with $(\St_3^{\an}(\lambda),\widetilde\Pi^1(\lambda, \psi))$ replaced by $(S_{1,0},\Pi^1(\lambda, \psi)^+)$.
\end{proposition}
\begin{proof}
(a) We first show that the composition:
\begin{multline}\label{equ: hL-trivL}
\Ext^1_{\tri}\big(\pi(\lambda_{1,2},\psi), \pi(\lambda_{1,2}, \psi)\big) \hooklongrightarrow \Ext^1_{\GL_2(\Q_p)}\big(\pi(\lambda_{1,2},\psi), \pi(\lambda_{1,2}, \psi)\big) \\ \twoheadlongrightarrow \Ext^1_{\GL_2(\Q_p)}\big(\St_2^{\infty}(\lambda_{1,2}), \pi(\lambda_{1,2}, \psi)\big) \buildrel{(\ref{equ: hL-L3f})} \over \twoheadlongrightarrow \Ext^1_{{\GL_3}(\Q_p)}\big(v_{\overline{P}_2}^{\infty}(\lambda),\widetilde\Pi^1(\lambda, \psi)\big)
\end{multline}
factors through:
\begin{equation}\footnotesize\label{equ: hL-L3k}
\Ext^1_{\tri}\big(\pi(\lambda_{1,2}, \psi), \pi(\lambda_{1,2}, \psi)\big) \longrightarrow \Ext^1_{{\GL_3}(\Q_p)}\big(v_{\overline{P}_2}^{\infty}(\lambda), \St_3^{\an}(\lambda)\big)\hooklongrightarrow \Ext^1_{{\GL_3}(\Q_p)}\big(v_{\overline{P}_2}^{\infty}(\lambda),\widetilde\Pi^1(\lambda, \psi)\big).
\end{equation}
By (\ref{equ: hL-tria}), the composition of the first two maps in (\ref{equ: hL-trivL}) has image equal to $\Ima(\iota_1)$ (cf. (\ref{equ: hL-iot1})). It is thus sufficient to show that the composition:
\begin{multline}\label{equ: hL-L3j}
 \Ext^1_{\GL_2(\Q_p)}\big(\St_2^{\infty}(\lambda_{1,2}), \pi(\lambda_{1,2}, \psi)^-\big) \xlongrightarrow{\iota_1} \Ext^1_{\GL_2(\Q_p)}\big(\St_2^{\infty}(\lambda_{1,2}), \pi(\lambda_{1,2}, \psi)\big)\\ \xlongrightarrow{(\ref{equ: hL-L3f})} \Ext^1_{{\GL_3}(\Q_p)}\big(v_{\overline{P}_2}^{\infty}(\lambda),\widetilde\Pi^1(\lambda, \psi)\big)
\end{multline}
factors through:
\begin{equation}\footnotesize\label{equ: hL-L3l}
\Ext^1_{\GL_2(\Q_p)}\big(\St_2^{\infty}(\lambda_{1,2}), \pi(\lambda_{1,2}, \psi)^-\big) \longrightarrow \Ext^1_{{\GL_3}(\Q_p)}\big(v_{\overline{P}_2}^{\infty}(\lambda), \St_3^{\an}(\lambda)\big)\hooklongrightarrow \Ext^1_{{\GL_3}(\Q_p)}\big(v_{\overline{P}_2}^{\infty}(\lambda),\widetilde\Pi^1(\lambda, \psi)\big).
\end{equation}
By the construction in Remark \ref{rem: hL-L3}, it is easy to see that any element in the image of (\ref{equ: hL-L3j}) comes by push-forward from a certain extension of $v_{\overline{P}_2}^{\infty}(\lambda)$ by $\widetilde\Pi^1(\lambda, \psi)^-$. By the proof of Lemma \ref{lem: hL-L3}(1), one has:
\begin{equation*}
 \Ext^1_{{\GL_3}(\Q_p)}\big(v_{\overline{P}_2}^{\infty}(\lambda), \St_3^{\an}(\lambda)\big) \xlongrightarrow{\sim} \Ext^1_{{\GL_3}(\Q_p)}\big(v_{\overline{P}_2}^{\infty}(\lambda), \widetilde\Pi^1(\lambda, \psi)^-\big).
\end{equation*}
We deduce that the map (\ref{equ: hL-L3j}) factors through $\Ext^1_{{\GL_3}(\Q_p)}(v_{\overline{P}_2}^{\infty}(\lambda), \St_3^{\an}(\lambda))$.\\
\noindent
(b) We prove the map $\Ext^1_{\tri}(\pi(\lambda_{1,2}, \psi), \pi(\lambda_{1,2}, \psi)) \twoheadlongrightarrow \Ext^1_{{\GL_3}(\Q_p)}(v_{\overline{P}_2}^{\infty}(\lambda), \St_3^{\an}(\lambda))$ is surjective.\\
\noindent
The composition of the last two maps in (\ref{equ: hL-trivL}) is equal to (\ref{equ: hL-L3p}) and has kernel equal to $\Ker (\kappa^{\aut})$ by \ the \ proof \ of \ Theorem \ \ref{thm: hL-L3}. \ From \ (\ref{equ: hL-kp1aut}) \ we \ have \ $\Ker (\kappa^{\aut})\subseteq \Ext^1_{\tri}(\pi(\lambda_{1,2}, \psi), \pi(\lambda_{1,2}, \psi))$, so the kernel of the composition in (\ref{equ: hL-trivL}) is $\Ker (\kappa^{\aut})$. From Lemma \ref{lem: hL-kp1}(1), we get that the kernel of the composition in (\ref{equ: hL-L3k}) is (also) $\Ker (\kappa^{\aut})$ and is $2$-dimensional. \ From \ Lemma \ \ref{lem: hL-sim}(1) \ and \ Proposition \ \ref{prop: hL-sim} \ we \ deduce \ $\dim_E\Ext^1_{{\GL_3}(\Q_p)}(v_{\overline{P}_2}^{\infty}(\lambda), \St_3^{\an}(\lambda))=2$. Together with Lemma \ref{dimtri}(1) and a dimension count, we obtain that the first map in (\ref{equ: hL-L3k}) is surjective. From the proof of (a), it follows that the first map in (\ref{equ: hL-L3l}) is also surjective. In summary, we have a natural commutative diagram:
\begin{equation}\label{equ: hL-CD}\begin{CD}
   \Ext^1_{\tri}\big(\pi(\lambda_{1,2}, \psi), \pi(\lambda_{1,2}, \psi)\big) @>>> \Ext^1_{{\GL_3}(\Q_p)}\big(v_{\overline{P}_2}^{\infty}(\lambda), \St_3^{\an}(\lambda)\big) \\
   @VVV @VVV \\
    \Ext^1_{\GL_2(\Q_p)}\big(\pi(\lambda_{1,2}, \psi), \pi(\lambda_{1,2}, \psi)\big) @>>> \Ext^1_{{\GL_3}(\Q_p)}\big(v_{\overline{P}_2}^{\infty}(\lambda), \widetilde\Pi^1(\lambda,\psi)\big)
    \end{CD}
\end{equation}
where the horizontal maps are surjective and the vertical maps are injective.\\
\noindent
(c) By the discussion in (a), the morphism (\ref{equ: hL-L3l}) can be constructed in a similar way as in Remark \ref{rem: hL-L3}. In particular, for $\widetilde{\pi}\in \Ext^1_{\GL_2(\Q_p)}(\St_2^{\infty}(\lambda_{1,2}), \pi(\lambda_{1,2}, \psi)^-) $, its image in $\Ext^1_{{\GL_3}(\Q_p)}(v_{\overline{P}_2}^{\infty}(\lambda), \St_3^{\an}(\lambda))$ is a subquotient of $(\Ind_{\overline{P}_1(\Q_p)}^{\GL_3(\Q_p)} \widetilde{\pi}\otimes x^{k_3})^{\an}$. By the transitivity of parabolic inductions, one can check the following diagram commutes:
\begin{equation*}\begin{CD}
\Hom(T_2(\Q_p), E)_{\psi} @> \pr_2 >> \Hom(\Q_p^{\times}, E) \\
@VVV @VVV\\
  \Ext^1_{\GL_2(\Q_p)}\big(\St_2^{\infty}(\lambda_{1,2}), \pi(\lambda_{1,2}, \psi)^-\big) @> (\ref{equ: hL-L3l}) >> \Ext^1_{{\GL_3}(\Q_p)}\big(v_{\overline{P}_2}^{\infty}(\lambda), \St_3^{\an}(\lambda)\big)
  \end{CD}
\end{equation*}
where the left vertical map is given by the inverse of (\ref{equ: hL-jac2}) (see Remark \ref{rem: hL-Jac} for its construction), and the right vertical map is given as in Proposition \ref{prop: hL-sim} (see the discussion above Proposition \ref{prop: hL-sim} for its construction). We deduce that the following diagram commutes (see (\ref{equ: hL-kp1aut}) for $\kappa^{\aut}=\kappa$ and recall (\ref{equ: hL-L3k}) comes from (\ref{equ: hL-L3l}) by the proof of (a)):
\begin{equation}\label{equ: hL-l3m}
 \begin{CD}  \Ext^1_{\tri}\big(\pi(\lambda_{1,2}, \psi), \pi(\lambda_{1,2}, \psi)\big)@>(\ref{equ: hL-L3k}) >> \Ext^1_{{\GL_3}(\Q_p)}\big(v_{\overline{P}_2}^{\infty}(\lambda), \St_3^{\an}(\lambda)\big) \\
@V \kappa^{\aut} VV @V \wr VV\\
 \Hom(\Q_p^{\times}, E) @ >\sim >> \Hom(\Q_p^{\times}, E)
 \end{CD}
\end{equation}
where the right vertical map is the inverse of the bottom horizontal map in (\ref{equ: hL-sim0}) (via Lemma \ref{lem: hL-sim}(1)).\\
\noindent
(d) By Hypothesis \ref{hypo: hL-pLL0}(2)\&(3), the bottom squares of (\ref{equ: hL-comtri}) induce a commutative diagram:
\begin{equation}\label{equ: hL-L3o}
   \begin{CD}
   \ \ \ \ \ \Ext^1_{\tri}\big(\pi(\lambda_{1,2}, \psi), \pi(\lambda_{1,2}, \psi)\big) @. \ \ \times \ \ @. \Ext^1_{(\varphi,\Gamma)}\big(\cR_E(\delta_3), \cR_E(\delta_2)\big) @> \cup >> E\\
 @VVV @. @AAA @| \\
 \Ext^1_{\GL_2(\Q_p)}\big(\pi(\lambda_{1,2}, \psi), \pi(\lambda_{1,2}, \psi)\big) @. \!\ \times @. \!\!\!\!\!\!\Ext^1_{(\varphi,\Gamma)}\big(\cR_E(\delta_3), D_1^2\big) @> \cup >> E.\\
 \end{CD}
\end{equation}
and the top squares of (\ref{equ: hL-comtri}) induce another commutative diagram:
\begin{equation}\label{equ: hL-comma}
 \begin{CD}\Ext^1_{\tri}\big(\pi(\lambda_{1,2}, \psi), \pi(\lambda_{1,2}, \psi)\big) @. \!\times \!@. \ \Ext^1_{(\varphi,\Gamma)}\big(\cR_E(\delta_3), \cR_E(\delta_2)\big) @> \cup >> E\\
 @V \kappa^{\aut} VV @. @| @| \\
\ \Ext^1_{(\varphi,\Gamma)}(\cR_E(\delta_2),\cR_E(\delta_2)) \!@. \ \ \times \ \ @. \Ext^1_{(\varphi,\Gamma)}\big(\cR_E(\delta_3), \cR_E(\delta_2)\big) @> \cup_1 >> E\\
 \end{CD}
\end{equation}
where we identify $\Hom(\Q_p^{\times}, E)$ with $\Ext^1_{(\varphi,\Gamma)}(\cR_E(\delta_2),\cR_E(\delta_2))$ (see (\ref{equ: hL-cad})).\\
\noindent
(e) We finally prove the proposition. By (\ref{equ: hL-CD}), (\ref{equ: hL-L3o}) and Theorem \ref{thm: hL-L3}, we deduce a commutative diagram as in (\ref{equ: hL-compL}) but with the top pairing $\cup_1$ replaced by the pairing induced by the top pairing of (\ref{equ: hL-L3o}) via the surjection (see (b)):
$$\Ext^1_{\tri}\big(\pi(\lambda_{1,2}, \psi), \pi(\lambda_{1,2}, \psi)\big) \twoheadlongrightarrow \Ext^1_{{\GL_3}(\Q_p)}\big(v_{\overline{P}_2}^{\infty}(\lambda), \St_3^{\an}(\lambda)\big).$$
However, by (\ref{equ: hL-comma}) and (\ref{equ: hL-l3m}), we see these two pairings actually coincide. This concludes the proof.
\end{proof}

\noindent
We fix a {\it nonsplit} extension $D\in \Ext^1_{(\varphi,\Gamma)}\big(\cR_E(\delta_3), D_1^2\big)$ and we let (assuming Hypothesis \ref{hypo: hL-pLL0} for $D_1^2$):
\begin{equation}\label{callaut}
\cL_{\aut}(D:D_1^2)\subseteq \Ext^1_{{\GL_3}(\Q_p)}\big(v_{\overline{P}_2}^{\infty}(\lambda),\widetilde\Pi^1(\lambda, \psi)\big)\cong \Ext^1_{{\GL_3}(\Q_p)}\big(v_{\overline{P}_2}^{\infty}(\lambda),\Pi^1(\lambda, \psi)^+\big)
\end{equation}
be the $2$-dimensional $E$-vector subspace annihilated by $D$ via (\ref{equ: hL-L3}).

\begin{remark}\label{rk3.47}
{\rm By Theorem \ref{thm: hL-L3} and its proof, the composition (\ref{equ: hL-L3i}) actually induces an isomorphism $ \Ext^1_{(\varphi,\Gamma)}(D_1^2, \cR_E(\delta_2)) \xlongrightarrow{\sim}  \Ext^1_{{\GL_3}(\Q_p)}(v_{\overline{P}_2}^{\infty}(\lambda),\Pi)$ with $\Pi=\widetilde\Pi^1(\lambda, \psi)$ or $\Pi^1(\lambda, \psi)^+$\!. Moreover, for a nonsplit $D$ in $\Ext^1_{(\varphi,\Gamma)}(\cR_E(\delta_3), D_1^2)$ as above and from the definitions of $\ell_{\FM}(D: D_1^2)$ and $\cL_{\aut}(D:D_1^2)$, this isomorphism induces an isomorphism:
\begin{equation}\label{rem: hL-LgalLaut}
\ell_{\FM}(D: D_1^2) \xlongrightarrow{\sim} \cL_{\aut}(D:D_1^2)
\end{equation}
since both are annihilated by $D$ via the corresponding pairing.}
\end{remark}

\noindent
We define (cf. Notation \ref{not: hL-ext}):
\begin{eqnarray}
 \widetilde{\Pi}^1(D)^-&:=&\sE\big(\widetilde\Pi^1(\lambda, \psi), v_{\overline{P}_2}^{\infty}(\lambda)^{\oplus 2}, \cL_{\aut}(D:D_1^2)\big)\label{equ: L3-tildPi} \\
 \Pi^1(D)^-&:=& \sE\big(\Pi^1(\lambda, \psi)^+, v_{\overline{P}_2}^{\infty}(\lambda)^{\oplus 2}, \cL_{\aut}(D:D_1^2)\big).\label{equ: L3-Pi}
\end{eqnarray}
It follows from the perfect pairing (\ref{equ: hL-L3}) and Lemma \ref{lem: hL-L3}(1) that $D$ is {\it determined} by the subspace $\cL_{\aut}(D:D_1^2)$, hence by $\widetilde{\Pi}^1(D)^-$ and $\Pi^1(D)^-$. Let:
$$\cL_{\aut}(D: D_1^2)_0:=\Ext^1_{{\GL_3}(\Q_p)}\big(v_{\overline{P}_2}^{\infty}(\lambda), \St_3^{\an}(\lambda)\big)\cap \cL_{\aut}(D:D_1^2)$$
which we also view as a subspace of $\Ext^1_{{\GL_3}(\Q_p)}\big(v_{\overline{P}_2}^{\infty}(\lambda), S_{2,0}\big)$ by Lemma \ref{lem: hL-sim}(1). By Proposition \ref{prop: hL-compL}, we have via the pairing $\cup_1$ ($u_1$ as in Proposition \ref{prop: hL-compL} and identifying $D$ with its corresponding extension):
\begin{equation}\label{equ: hL-ort}
 \cL_{\aut}(D:D_1^2)_0 =(E u_1(D))^{\perp}.
\end{equation}

\noindent
We assume now that the extension $u_1(D)$ of $\cR_E(\delta_3)$ by $\cR_E(\delta_2)$ inside $D$ is nonsplit. As $\cup_1$ is perfect (cf. (\ref{equ: hL-compL})) this implies $\dim_E \cL_{\aut}(D: D_1^2)_0=1$. We define (cf. Notation \ref{not: hL-ext}):
\begin{eqnarray*}
 \widetilde{\Pi}^1(D)^-_2&:=&\sE\big(\St_3^{\an}(\lambda), v_{\overline{P}_2}^{\infty}(\lambda), \cL_{\aut}(D:D_1^2)_0\big)\\
 \Pi^1(D)^-_2&:=&\sE\big(S_{2,0}, v_{\overline{P}_2}^{\infty}(\lambda), \cL_{\aut}(D:D_1^2)_0\big).
\end{eqnarray*}
We have injections $\widetilde{\Pi}^1(D)^-_2\hookrightarrow \widetilde{\Pi}^1(D)^-$ and $\Pi^1(D)^-_2 \hookrightarrow \Pi^1(D)^-$.

\begin{remark}\label{rem: hL-sim2}
{\rm Identifying $\Ext^1_{(\varphi,\Gamma)}(\cR_E(\delta_2), \cR_E(\delta_2))$ with $\Hom(\Q_p^{\times}, E)$, the vector space $(Eu_1(D))^{\perp}$ via the top perfect pairing $\cup_1$ of (\ref{equ: hL-comtri}) is thus a one dimensional subspace of $\Hom(\Q_p^{\times}, E)$, and we let $\psi'$ be a basis. By Corollary \ref{coro: hL-simL}, we have $\Pi^1(D)^-_2\cong \Pi^2(\lambda, \psi')_0$ where we denote by $\Pi^2(\lambda, \psi')_0$ the image of $\psi'$ via the bottom bijection of (\ref{equ: hL-sim0}).}
\end{remark}

\begin{proposition}\label{prop: hL-L3}
Assume Hypothesis \ref{hypo: hL-pLL0} for $D_1^2$ and $N^2\neq 0$ on the filtered $(\varphi,N)$-module associated to $D$ (\cite[Thm.~A]{Berger2}, and note that the latter implies that $u_1(D)$ is nonsplit and $\psi$ is non smooth). Then there exists a unique subrepresentation:
$$\Pi^1(D)^{-}_1\in \Ext^1_{{\GL_3}(\Q_p)}\big(v_{\overline{P}_2}^{\infty}(\lambda), \Pi^1(\lambda, \psi)\big)\setminus \Ext^1_{{\GL_3}(\Q_p)}\big(v_{\overline{P}_2}^{\infty}(\lambda), \St_3^{\infty}(\lambda)\big)$$
of $\Pi^1(D)^-$ such that $\Pi^1(D)^-\cong \Pi^1(D)_1^-\oplus_{\St_3^{\infty}(\lambda)} \Pi^1(D)_2^-$. In particular, $\Pi^1(D)^-$ has the following form:
\begin{equation*}
 \begindc{\commdiag}[32]
 \obj(0,10)[a]{$\St_3^{\infty}(\lambda)$}
  \obj(12,2)[b]{$C_{2,1}$}
  \obj(25,2)[c]{$v_{\overline{P}_2}^{\infty}(\lambda).$}
  \obj(12,18)[d]{$C_{1,1}$}
  \obj(23,10)[e]{$C_{1,2}$}
  \obj(34,18)[f]{$C_{1,3}$}
  \obj(47,10)[g]{$v_{\overline{P}_2}^{\infty}(\lambda)$}
  \obj(23, 26)[h]{$\widetilde{C}_{1,2}$}
  \mor{a}{b}{}[+1,\solidline]
   \mor{b}{c}{}[+1,\solidline]
   \mor{a}{d}{}[+1,\solidline]
   \mor{d}{e}{}[+1, \solidline]
   \mor{e}{f}{}[+1, \solidline]
   \mor{f}{g}{}[+2, \solidline]
   \mor{d}{h}{}[+2, \solidline]
   \mor{h}{f}{}[+1, \dashline]
 \enddc
\end{equation*}
\end{proposition}
\begin{proof}
Considering the surjection in (\ref{equ: hL-L3b}):
\begin{multline*}
\pr: \Ext^1_{{\GL_3}(\Q_p)}\big(v_{\overline{P}_2}^{\infty}(\lambda), \Pi^1(\lambda, \psi)^+\big)\\ \xlongrightarrow{(\pr_1, \pr_2)} \Ext^1_{{\GL_3}(\Q_p)}\big(v_{\overline{P}_2}^{\infty}(\lambda), \Pi^1(\lambda, \psi)/\St_3^{\infty}(\lambda)\big)\oplus \Ext^1_{{\GL_3}(\Q_p)}\big(v_{\overline{P}_2}^{\infty}(\lambda), C_{2,1}\big),
\end{multline*}
we see with Lemma \ref{lem: hL-L3}(2), Remark \ref{formofrepr} and the form of $\Pi^1(\lambda, \psi)^+$ at the beginning of \S~\ref{linvariantgl3} that we have:
\begin{eqnarray*}
 \Ker(\pr)&\cong& \Ext^1_{{\GL_3}(\Q_p)}\big(v_{\overline{P}_2}^{\infty}(\lambda), \St_3^{\infty}(\lambda)\big) \\
 \Ker(\pr_1) & \cong& \Ext^1_{{\GL_3}(\Q_p)}\big(v_{\overline{P}_2}^{\infty}(\lambda), S_{1,0}\big) \\
 \Ker(\pr_2) & \cong & \Ext^1_{{\GL_3}(\Q_p)}\big(v_{\overline{P}_2}^{\infty}(\lambda), \Pi^1(\lambda, \psi)\big)
\end{eqnarray*}
and it follows from Lemma \ref{lem: hL-L3}(2) that the first kernel has dimension $1$ and the two others dimension $2$. We first show:
\begin{equation}\label{equ: hL-pr12}
\Ker(\pr)\cap \cL_{\aut}(D:D_1^2)=\Ker(\pr)\cap \cL_{\aut}(D:D_1^2)_0=0.
\end{equation}
The first equality is clear since by definition and Lemma \ref{lem: hL-sim}(1) we have $\cL_{\aut}(D:D_1^2)_0=\Ker(\pr_1)\cap \cL_{\aut}(D:D_1^2)$. As $N^2\neq 0$, the quotient $u_1(D)$ (as an extension of $\cR_E(\delta_3)$ by $\cR_E(\delta_2)$) is {\it not} crystalline, hence by the second part of Corollary \ref{coro: hL-simL}, $u_1(D)$ is {\it not} annihilated by $\Ext^1_{{\GL_3}(\Q_p)}(v_{\overline{P}_2}^{\infty}(\lambda), \St_3^{\infty}(\lambda))$. Since the latter vector space has dimension $1$, one deduces from (\ref{equ: hL-ort}):
\begin{equation}\label{equ: hL-nsm}
\Ext^1_{{\GL_3}(\Q_p)}\big(v_{\overline{P}_2}^{\infty}(\lambda), \St_3^{\infty}(\lambda)\big)\cap \cL_{\aut}(D:D_1^2)_0 =0
\end{equation}
and \ the \ second \ equality \ in \ (\ref{equ: hL-pr12}) \ follows. \ As \ $\cL_{\aut}(D:D_1^2)$ \ has \ dimension \ $2$ \ and \ $\Ext^1_{{\GL_3}(\Q_p)}(v_{\overline{P}_2}^{\infty}(\lambda), \Pi^1(\lambda, \psi)^+)$ \ dimension \ $3$, \ we \ easily \ deduce \ from \ (\ref{equ: hL-pr12}) \ and \ $\dim_E\Ker(\pr_i)=2$ that $\dim_E\Ker(\pr_i)\cap \dim_E \cL_{\aut}(D:D_1^2) =1$ for $i=1,2$. Let $\cL_{\aut}(D:D_1^2)_1:=\Ker(\pr_2)\cap \cL_{\aut}(D:D_1^2)\subseteq \Ext^1_{{\GL_3}(\Q_p)}\big(v_{\overline{P}_2}^{\infty}(\lambda), \Pi^1(\lambda, \psi)\big)$ and set (with Notation \ref{not: hL-ext}):
\begin{equation*}
 \Pi^1(D)^-_1:=\sE\big(\Pi^1(\lambda, \psi), v_{\overline{P}_2}^{\infty}(\lambda), \cL_{\aut}(D:D_1^2)_1\big).
\end{equation*}
Since we have $\cL_{\aut}(D:D_1^2)_1\oplus \cL_{\aut}(D:D_1^2)_0=\cL_{\aut}(D:D_1^2)$ (as follows from (\ref{equ: hL-pr12})) and $\Ker(\pr)\cap \cL_{\aut}(D:D_1^2)_1=0$ ({\it ibid.}), one easily checks the statements in the proposition.
\end{proof}

\noindent
Replacing $\overline{P}_1$ by $\overline{P}_2$, we define $\Pi^2(\lambda, \psi)^+:=\Pi^2(\lambda, \psi)\oplus_{\St_3^{\infty}(\lambda)} S_{2,0} \hookrightarrow \widetilde\Pi^2(\lambda, \psi)$ as for $\Pi^1(\lambda, \psi)^+$ at the beginning of \S~\ref{linvariantgl3}. All the above results have their analogue (or symmetric) version. Let $D_2^3:=D(p^2,\lambda_{2,3}^{\sharp}, \psi)$ (see the beginning of \S~\ref{sec: hL-pLL}). The following theorem is the analogue of Theorem \ref{thm: hL-L3} and Proposition \ref{prop: hL-compL}.

\begin{theorem}\label{symm}
Assume Hypothesis \ref{hypo: hL-pLL0} for $D_2^3$. The isomorphism (\ref{equ: hL-pLL0}) and (\ref{equ: hL-CD3}) induce a perfect pairing:
\begin{equation*}
   \Ext^1_{{\GL_3}(\Q_p)}\big(v_{\overline{P}_1}^{\infty}(\lambda),W\big) \times \Ext^1_{(\varphi,\Gamma)}\big(D_2^3, \cR_E(\delta_1)\big) \xlongrightarrow{\cup} E
\end{equation*}
such that the following diagram commutes:
\begin{equation*}
  \begin{CD}\Ext^1_{{\GL_3}(\Q_p)}\big(v_{\overline{P}_1}^{\infty}(\lambda), \Pi^-\big) @. \ \times \ @. \Ext^1_{(\varphi,\Gamma)}\big(\cR_E(\delta_2), \cR_E(\delta_1)\big) @> \cup_1 >> E\\
 @VVV @. @A u_1 AA @| \\
 \Ext^1_{{\GL_3}(\Q_p)}\big(v_{\overline{P}_1}^{\infty}(\lambda),\Pi\big) @. \times @. \!\!\!\!\!\!\!\!\Ext^1_{(\varphi,\Gamma)}\big(D_2^3, \cR_E(\delta_1)\big) @> \cup >> E\\
 \end{CD}
\end{equation*}
with $(\Pi^-,\Pi)=(S_{2,0},\Pi^2(\lambda, \psi)^+)$ or $(\St^{\an}_3(\lambda),\widetilde\Pi^2(\lambda, \psi))$ and where the top perfect pairing is given as in Corollary \ref{coro: hL-simL} (via Lemma \ref{lem: hL-sim}(1)).
\end{theorem}

\noindent
For $D\in \Ext^1_{(\varphi,\Gamma)}\big(D_2^3, \cR_E(\delta_1)\big)$, we define (using the symmetric version of Lemma \ref{lem: hL-L3}(1)):
\begin{equation}\footnotesize\label{aut2}
\cL_{\aut}(D:D_2^3):=(E D)^{\perp}\subseteq \Ext^1_{{\GL_3}(\Q_p)}\big(v_{\overline{P}_1}^{\infty}(\lambda),\Pi^2(\lambda, \psi)^+\big)\cong \Ext^1_{{\GL_3}(\Q_p)}\big(v_{\overline{P}_1}^{\infty}(\lambda),\widetilde\Pi^2(\lambda, \psi)\big)
\end{equation}
and likewise using Lemma \ref{lem: hL-sim}(1):
\begin{multline*}
 \cL_{\aut}(D:D_2^3)_0:=\cL_{\aut}(D:D_2^3) \cap \Ext^1_{{\GL_3}(\Q_p)}\big(v_{\overline{P}_1}^{\infty}(\lambda),S_{2,0}\big)\\ \subseteq \Ext^1_{{\GL_3}(\Q_p)}\big(v_{\overline{P}_1}^{\infty}(\lambda),S_{2,0}\big)\cong \Ext^1_{{\GL_3}(\Q_p)}\big(v_{\overline{P}_1}^{\infty}(\lambda),\St_3^{\an}(\lambda)\big),
\end{multline*}
and we have $\cL_{\aut}(D:D_2^3)_0=(Eu_1(D))^{\perp}$ via the pairing $\cup_1$ in Theorem \ref{symm}. We also define:
\begin{eqnarray}
 \Pi^2(D)^-&:=& \sE\big(\Pi^2(\lambda, \psi)^+, v_{\overline{P}_1}^{\infty}(\lambda)^{\oplus 2}, \cL_{\aut}(D:D_2^3)\big)\label{piD2}\\
 \widetilde{\Pi}^2(D)^- &:=&\sE\big(\widetilde\Pi^2(\lambda, \psi), v_{\overline{P}_1}^{\infty}(\lambda)^{\oplus 2}, \cL_{\aut}(D:D_2^3)\big) \label{pitilde2}\\
 \Pi^2(D)^-_1&:=& \sE\big(S_{2,0}, v_{\overline{P}_1}^{\infty}(\lambda), \cL_{\aut}(D:D_2^3)_0\big)\nonumber \\
 \widetilde{\Pi}^2(D)^-_1&:=& \sE\big(\St_3^{\an}(\lambda), v_{\overline{P}_1}^{\infty}(\lambda), \cL_{\aut}(D:D_2^3)_0\big)\nonumber
\end{eqnarray}
and we have $\Pi^2(D)^-_1\hookrightarrow \Pi^2(D)^-$ and $\widetilde{\Pi}^2(D)^-_1\hookrightarrow\widetilde{\Pi}^2(D)^-$. Similarly as in Proposition \ref{prop: hL-L3}, assuming Hypothesis \ref{hypo: hL-pLL0} there exists a unique representation if $N^2\neq 0$:
\begin{equation*}
 \Pi^2(D)^-_2\in \Ext^1_{{\GL_3}(\Q_p)}\big(v_{\overline{P}_1}^{\infty}(\lambda), \Pi^2(\lambda, \psi)\big)\setminus \Ext^1_{{\GL_3}(\Q_p)}\big(v_{\overline{P}_1}^{\infty}(\lambda), \St_3^{\infty}(\lambda)\big)
\end{equation*}
such that $\Pi^2(D)^-\cong \Pi^2(D)^-_1\oplus_{\St_3^{\infty}(\lambda)} \Pi^2(D)^-_2$.\\

\noindent
Now we fix $(D,(\delta_1,\delta_2,\delta_3))$ a special noncritical $(\varphi,\Gamma)$-module of rank $3$ over $\cR_E$ (see the beginning of \S~\ref{sec: hL-FM}) with $\delta_1=x^{k_1}$, $\delta_2=x^{k_2-1}|\cdot |^{-1}$, and $\delta_3= x^{k_3-2}|\cdot |^{-2}$. We assume the extension of $\cR_E(\delta_2)$ (resp. of $\cR_E(\delta_3)$) by $\cR_E(\delta_1)$ (resp. by $\cR_E(\delta_2)$) is nonsplit and we let $\psi_1$ be a basis of $\cL_{\FM}(D_1^2: \cR_E(\delta_1))\subseteq \Hom(\Q_p^{\times}, E)$ and $\psi_2$ a basis of $\cL_{\FM}(D_2^3: \cR_E(\delta_2))\subseteq \Hom(\Q_p^{\times}, E)$ (see \S~\ref{sec: hL-FM}), i.e. we have $D_1^2\cong D(p,\lambda_{1,2}^{\sharp}, \psi_1)$ and $D_2^3\cong D(p^2,\lambda_{2,3}^{\sharp}, \psi_2)$ (see the beginning of \S~\ref{sec: hL-pLL} and (\ref{dalpha})). We assume $N^2\neq 0$, which is equivalent to $\psi_i$ not smooth for $i=1,2$ and we also assume that Hypothesis \ref{hypo: hL-pLL0} holds for $D_1^2$ and $D_2^3$ (recall that under quite mild genericity assumptions this is automatic by Lemma \ref{lem: hL-etale}, Proposition \ref{prop: hL-gl2pLL} and Proposition \ref{prop: hL-gl2pLL2}). We can then associate to $D$ the above representations:
\begin{eqnarray*}
\Pi^1(D)^-&\cong &\Pi^1(D)^-_1\oplus_{\St_3^{\infty}(\lambda)} \Pi^1(D)^-_2\ \hookrightarrow \ \widetilde{\Pi}^1(D)^-\\
\Pi^2(D)^-&\cong &\Pi^2(D)^-_1\oplus_{\St_3^{\infty}(\lambda)} \Pi^2(D)^-_2\ \hookrightarrow \ \widetilde{\Pi}^2(D)^-.
\end{eqnarray*}
By the symmetric version of Lemma \ref{lem: hL-sim2}, the subrepresentation $S_{1,0}-v_{\overline{P}_1}^{\infty}(\lambda)$ of $\Pi^2(\lambda, \psi_2)\!\subseteq \Pi^2(D)^-_2\subseteq \Pi^2(D)^-$ is isomorphic to the image $\Pi^2(\lambda,\psi_2)_0$ of $\psi_2$ via the bottom map of (\ref{equ: hL-sim0}). By Lemma \ref{lem: hL-compL}, we deduce:
\begin{multline*}
\cL_{\FM}(D_2^3:\cR_E(\delta_2))=\ell_{\FM}(D_2^3:\cR_E(\delta_2))=\ell_{\FM}(D:D_1^2)\cap \Ext^1_{(\varphi,\Gamma)}(\cR_E(\delta_2), \cR_E(\delta_2))\\
=E \psi_2\subseteq \Hom(\Q_p^{\times}, E^{\times}).
\end{multline*}
Thus by Remark \ref{rem: hL-sim2}, $\Pi^1(D)^-_2$ is also isomorphic to $\Pi^2(\lambda,\psi_2)_0$. In particular, we have an injection $\Pi^1(D)^-_2\hookrightarrow \Pi^2(D)^-_2$. Similarly, we have an injection $\Pi^2(D)^-_1\hookrightarrow \Pi^1(D)^-_1$.
Denote by $\Pi^0(D)^-$ the following subrepresentation of $\Pi^1(D)^-$ and $\Pi^2(D)^-$:
\begin{equation*}
 \begindc{\commdiag}[32]
 \obj(0,10)[a]{\!\!\!\!\!\!\!\!\!\!\!\!\!\!\!\!\!\!\!\!\!\!\!\!\!\!\!\!\!\!\!\!$\Pi^0(D)^-\ \ \cong \ \ \St_3^{\infty}(\lambda)$}
  \obj(12,2)[b]{$C_{2,1}$}
  \obj(23,2)[c]{$C_{2,2}$}
  \obj(12,18)[d]{$C_{1,1}$}
  \obj(23,18)[e]{$C_{1,2}$}
  \mor{a}{b}{}[+1,\solidline]
   \mor{b}{c}{}[+1,\solidline]
   \mor{a}{d}{}[+1,\solidline]
   \mor{d}{e}{}[+1, \solidline]
 \enddc
\end{equation*}
and put $\Pi(D)^-:=\Pi^1(D)^-\oplus_{\Pi^0(D)^-} \Pi^2(D)^-$, which is thus of the following form (where $C_{1,4}\cong v_{\overline{P}_2}^{\infty}(\lambda)\cong C_{2,2}$ and $C_{2,4}\cong v_{\overline{P}_1}^{\infty}(\lambda)\cong C_{1,2}$):
\begin{equation}\label{sansc5}
\begin{gathered}
\begindc{\commdiag}[32]
 \obj(0,16)[a]{\!\!\!\!\!\!\!\!\!\!\!\!\!\!\!\!\!\!\!\!\!\!\!\!\!\!\!\!$\Pi(D)^-\ \ \cong \ \ \St_3^{\infty}(\lambda)$}
  \obj(12,2)[b]{$C_{2,1}$}
  \obj(23,-6)[c]{$\widetilde{C}_{2,2}$}
  \obj(23,10)[f]{$C_{2,2}$}
  \obj(34,2)[x]{$C_{2,3}$}
  \obj(45,2)[y]{$C_{2,4}$}
  \obj(12,30)[d]{$C_{1,1}$}
  \obj(23,38)[e]{$\widetilde{C}_{1,2}$}
  \obj(23,22)[g]{$C_{1,2}$}
  \obj(34,30)[u]{$C_{1,3}$}
  \obj(45,30)[v]{$C_{1,4}$}
  \mor{a}{b}{}[+1,\solidline]
   \mor{b}{c}{}[+1,\solidline]
      \mor{b}{f}{}[+1,\solidline]
        \mor{c}{x}{}[+1, \dashline]
             \mor{f}{x}{}[+1,\solidline]
   \mor{x}{y}{}[+1, \solidline]
   \mor{a}{d}{}[+1,\solidline]
   \mor{d}{e}{}[+1, \solidline]
\mor{e}{u}{}[+1, \dashline]
\mor{u}{v}{}[+1, \solidline]
\mor{g}{u}{}[+1, \solidline]
\mor{d}{g}{}[+1, \solidline]
\enddc.
\end{gathered}
\end{equation}
It follows from the previous results that the $(\varphi,\Gamma)$-module $D$ and the $\GL_3(\Q_p)$-representation $\Pi(D)^-$ determine each other. From the results of \cite[\S~4]{Br16} (see in particular \cite[Rem.~4.6.3]{Br16}), there is a unique locally analytic representation $\Pi(D)$ containing $\Pi(D)^-$ of the form:
\begin{equation}\label{piDplus}
\begin{gathered}
\begindc{\commdiag}[32]
 \obj(0,16)[a]{\!\!\!\!\!\!\!\!\!\!\!\!\!\!\!\!\!\!\!\!\!\!\!\!\!\!\!\!$\Pi(D)^-\ \ \cong \ \ \St_3^{\infty}(\lambda)$}
  \obj(12,2)[b]{$C_{2,1}$}
  \obj(23,-6)[c]{$\widetilde{C}_{2,2}$}
  \obj(23,10)[f]{$C_{2,2}$}
  \obj(34,2)[x]{$C_{2,3}$}
  \obj(45,10)[y]{$C_{2,4}$}
  \obj(45,-6)[h]{$\widetilde{C}_{2,4}$}
  \obj(56,2)[i]{$C_{2,5}$}
  \obj(12,30)[d]{$C_{1,1}$}
  \obj(23,38)[e]{$\widetilde{C}_{1,2}$}
  \obj(23,22)[g]{$C_{1,2}$}
  \obj(34,30)[u]{$C_{1,3}$}
  \obj(45,22)[v]{$C_{1,4}$}
  \obj(45,38)[j]{$\widetilde{C}_{1,4}$}
  \obj(56,30)[k]{$C_{1,5}$}
  \mor{a}{b}{}[+1,\solidline]
   \mor{b}{c}{}[+1,\solidline]
      \mor{b}{f}{}[+1,\solidline]
        \mor{c}{x}{}[+1, \dashline]
             \mor{f}{x}{}[+1,\solidline]
   \mor{x}{y}{}[+1, \solidline]
   \mor{a}{d}{}[+1,\solidline]
   \mor{d}{e}{}[+1, \solidline]
\mor{e}{u}{}[+1, \dashline]
\mor{u}{v}{}[+1, \solidline]
\mor{g}{u}{}[+1, \solidline]
\mor{d}{g}{}[+1, \solidline]
\mor{h}{i}{}[+1,\dashline]
\mor{y}{i}{}[+1,\solidline]
\mor{x}{h}{}[+1,\solidline]
\mor{u}{j}{}[+1, \solidline]
\mor{j}{k}{}[+1,\dashline]
\mor{v}{k}{}[+1,\solidline]
\enddc.
\end{gathered}
\end{equation}
where the irreducible constituents $C_{1,5}$, $C_{2,5}$, $\widetilde{C}_{1,4}$, $\widetilde{C}_{2,4}$ are defined in \cite[\S~4.1]{Br16}.\\

\noindent
For $\chi: \Q_p^{\times} \ra E^{\times}$ and $D':=D\otimes_{\cR_E} \cR_E(\chi)$, we finally set $\Pi(D')^-:=\Pi(D)^-\otimes \chi \circ\dett$, $\Pi(D'):=\Pi(D)\otimes \chi \circ\dett$, and if $D'\cong D_{\rig}(\rho)$ for a certain $\rho:\Gal_{\Q_p} \ra \GL_3(E)$, we set $\Pi(\rho):=\Pi(D')$. In particular, we have thus associated to any sufficiently generic semi-stable $\rho:\Gal_{\Q_p} \ra \GL_3(E)$ with distinct Hodge-Tate weights and with $N^2\ne 0$ on $D_{\st}(\rho)=(B_{\rm st}\otimes_{\Q_p}\rho)^{\Gal_{\Q_p}}$ a locally analytic representation $\Pi(\rho)$ of $\GL_3(\Q_p)$ over $E$ which has the form (\ref{piDplus}) and which only depends on and completely determines $\rho$.

\section{Ordinary part functor}\label{ordinarypartfunctor}

\noindent
In this section we give several properties of the ordinary part functor of \cite{EOrd1} and review the ordinary part of a locally algebraic representation that has an invariant lattice (\cite[\S~5.6]{Em4}).

\subsection{Notation and preliminaries}\label{sec: ord-1}

\noindent
We start with some preliminary notation.\\

\noindent
We fix finite extensions $L$ and $E$ of $\Q_p$ as in \S~\ref{intronota} and denote by $\varpi_L$ a uniformizer of $L$. We let $G$ be a connected reductive algebraic group over $L$, $B$ a Borel subgroup of $G$, $P$ a parabolic subgroup of $G$ containing $B$ with $N_P$ the unipotent radical of $P$ and $L_P$ a Levi subgroup of $P$. We let $\overline{P}$ be the parabolic subgroup of $G$ opposite to $P$, $N_{\overline P}$ its unipotent radical and $Z_{L_P}$ the center of $L_P=L_{\overline P}$. As in \cite{EOrd1}, we denote by $\Comp(\co_E)$ the category of complete noetherian local $\co_E$-algebras with finite residue field.\\

\noindent
Let $K$ be a compact open subgroup of $G(L)$, as in \cite[\S~3.3]{EOrd1} we say $K$ admits an \emph{Iwahori decomposition} (with respect to $P$ and $\overline{P}$) if the following natural map:
\begin{equation*}
 (K\cap N_{\overline{P}}(L)) \times (K\cap Z_{L_P}(L)) \times (K\cap N_P(L)) \longrightarrow K
\end{equation*}
is an isomorphism. We let $I_0 \supset I_1\supset I_2\supset \cdots \supset I_i \supset I_{i+1} \supset \cdots$ be a cofinal family of compact open subgroups of $G(L)$ such that:
\begin{itemize}
 \item $I_i$ is normal in $I_0$
 \item $I_i$ admits an Iwahori decomposition.
\end{itemize}
For $i\in \Z_{\geq 0}$, we put $N_i:=N_P(L)\cap I_i$, $L_i:=L_P(L)\cap I_i$ and $\overline{N}_i:=N_{\overline{P}}(L)\cap I_i$. For $i\geq j\geq 0$, we put $I_{i,j}:=\overline{N}_i L_j N_0$, which can be checked to be a compact open subgroup of $I_0$ such that:
\begin{equation*}
 \overline{N}_i \times L_j \times N_0 \xlongrightarrow{\sim} I_{i,j}.
\end{equation*}

\begin{remark}
{\rm For any $i\in \Z_{\geq 0}$, the subgroups $\overline{N}_i$, $L_i$, and $N_i$ of $I_0$ are normalized by $L_0$, and hence $I_{i,j}$ is normalized by $L_0$ for any $i\geq j \geq 0$. We show this for $\overline{N}_i$ (the other cases are similar). Let $z\in L_0$, we have $zN_{\overline{P}}(L)z^{-1}=N_{\overline{P}}(L)$, which together with the fact $zI_iz^{-1}=I_i$ implies $z \overline{N}_i z^{-1}=z(N_{\overline{P}}(L) \cap I_i) z^{-1}=N_{\overline{P}}(L) \cap I_i=\overline{N}_i$.}
\end{remark}

\noindent
Now we set:
\begin{equation*}
L_P^+:=\{z\in L_P(L),\ zN_0z^{-1}\subseteq N_0\}
\end{equation*}
and $Z_{L_P}^+:=L_P^+\cap Z_{L_P}(L)$. We will assume moreover the following hypothesis.

\begin{hypothesis}\label{hypo: ord}
For any $z\in Z_{L_P}^+$ and $i\in \Z_{\geq 0}$, we have $\overline{N}_i \subseteq z \overline{N}_i z^{-1}$.
\end{hypothesis}

\begin{example}\label{ex: ord-gln}
{\rm Let $G=\GL_n$, $P$ a parabolic subgroup containing the Borel subgroup $B$ of upper triangular matrices, and let $L_P\cong \GL_{n_1} \times \cdots \times \GL_{n_k}$ be the Levi subgroup of $P$ containing the diagonal subgroup $T$. Let $I_i:=\{g\in \GL_n(\co_L),\ g\equiv 1 \pmod{\varpi_L^i}\}$, we have:
\begin{equation*}
  Z_{L_P}^+=\{(a_1,\cdots, a_k)\in Z_{L_P}(L),\ \val_p(a_1)\geq \cdots \geq \val_p(a_k)\}
\end{equation*}
where $a_j\in L^{\times}$ is seen in (the center of) $\GL_{n_j}(L)$ by the diagonal map. It is straightforward to check that Hypothesis \ref{hypo: ord} is satisfied for $\{I_i\}_{i\in \Z_{\geq 0}}$.}
\end{example}

\subsection{The functor $\Ord_P$}\label{sec: ord-2}

\noindent
We review and/or prove useful results on the functor $\Ord_P$ of \cite{EOrd1}, \cite{EOrd2}.\\

\noindent
Let $A\in \Comp(\co_E)$ with $\fm_A$ the maximal ideal of $A$ and let $V$ be a smooth representation of $G(L)$ over $A$ in the sense of \cite[Def. 2.2.5]{EOrd1}. Recall we have in particular $V\cong \varinjlim_n V[\fm_A^n]$. The $A$-submodule $V^{N_0}$ of elements fixed by $N_0$ is equipped with a natural Hecke action of $L_P^+$ given by (cf. \cite[Def. 3.1.3]{EOrd1}):
\begin{equation}\label{equ: ord-Hec}
 z \cdot v:= \sum_{x\in N_0/zN_0z^{-1}} \widetilde{x} (zv)
\end{equation}
where $z\in L_P^+$, $v\in V^{N_0}$, and $\widetilde{x}$ is an arbitrary lift of $x$ in $N_0$. Note that the $A$-module $V^{N_0}$ is a smooth representation of $L_0$ over $A$. Following \cite[Def. 3.1.9]{EOrd1}, we define:
\begin{equation}\label{deford}
\Ord_P(V):= \Hom_{A[Z_{L_P}^+]}\big(A[Z_{L_P}(L)], V^{N_0}\big)_{Z_{L_P}(L)-\fini},
\end{equation}
which is called the $P$-ordinary part of $V$. Here the $A$-module $\Hom_{A[Z_{L_P}^+]}(A[Z_{L_P}(L)], V^{N_0})$ is naturally equipped with an $A$-linear action of $Z_{L_P}(L)$ given by $(z\cdot f)(x):=f(zx)$, and $(\cdot)_{Z_{L_P}(L)-\fini}$ denotes the $A$-submodule of locally $Z_{L_P}(L)$-finite elements (cf. \cite[Def. 2.3.1 (2)]{EOrd1}). By \cite[Lem. 3.1.7]{EOrd1}, $\Hom_{A[Z_{L_P}^+]}(A[Z_{L_P}(L)], V^{N_0})$ and $ \Ord_P(V)$ are smooth representations of $L_P(L)$ over $A$. By \cite[Thm. 3.3.3]{EOrd1}, if $V$ is moreover admissible (cf. \cite[Def. 2.2.9]{EOrd1}), then $\Ord_P(V)$ is a smooth admissible representation of $L_P(L)$ over $A$. As in \cite[Def. 3.1.10]{EOrd1}, we have the canonical lifting map:
\begin{equation}\label{equ: ord-can}
\iota_{\can}: \Ord_P(V) \lra V^{N_0}, \ f\mapsto f(1)
\end{equation}
which is $L_P^+$-linear, and injective if $V$ is admissible (cf. \cite[Thm. 3.3.3]{EOrd1}). We put:
\begin{equation*}
\NOrd_P(V):=\big\{v\in V^{N_0}\text{ such that there exists $z\in Z_{L_P}^+$ with $z\cdot v=0$}\big\}
\end{equation*}
which is an $A$-submodule of $V^{N_0}$ stable by $L_P^+$.

\begin{theorem}\label{thm: Ord-dirc}
Assume $V$ is an admissible representation of $G(L)$, then we have:
\begin{equation*}
\Ord_P(V) \oplus \NOrd_P(V) \xlongrightarrow{\sim} V^{N_0}
\end{equation*}
as smooth representations of $L_0$, where $\Ord_P(V)$ is sent to $V^{N_0}$ by $\iota_{\can}$.
\end{theorem}
\begin{proof}
We easily reduce to the case where $V$ is annihilated by $\fm_A^n$ for a certain $n\in \Z_{>0}$.\\

\noindent
(a) Set $V_i:=V^{I_{i,i}}$, since $V$ is smooth we have $V^{N_0}=\varinjlim_{i} V_i$. By Hypothesis \ref{hypo: ord} and \cite[Lem. 3.3.2]{EOrd1} (applied to $I_0=I_1=I_{i,i}$), we see that $V_i$ is stable by the action of $Z_{L_P}^+$. Since $V$ is admissible, $V_i$ is a finitely generated $A$-module. Let $B_i$ be the $A$-subalgebra of $\End_A(V_i)$ generated by $Z_{L_P}^+$, then $B_i$ is a finite commutative $A$-algebra. Note that $B_i$ is actually a finite $A/\fm_A^n$-algebra since $V_i$ is annihilated by $\fm_A^n$, so in particular is Artinian. For a maximal ideal $\fm$ of $B_i$, we call $\fm$ ordinary (resp. nonordinary) if ${\rm Image}(Z_{L_P}^+)\cap \fm=\emptyset$ (resp. ${\rm Image}(Z_{L_P}^+) \cap \fm\neq \emptyset$) where ${\rm Image}(Z_{L_P}^+)$ is the image of $Z_{L_P}^+$ in $\End_A(V_i)$ (or in $B_i$). Since $B_i$ is artinian we have a natural decomposition:
\begin{equation*}
 B_i \cong \prod_{\fm \text{\ ordinary}} \!\!\!(B_i)_{\fm} \ \ \times \!\!\prod_{\fm \text{\ non\ ordinary}} \!\!\!\!\!(B_i)_{\fm}\ \ =:\ \ B_{i,\ord} \times B_{i, \nord}
\end{equation*}
and another decomposition:
\begin{equation}\label{equ: ord-dec}
 V_i\ \cong \ (V_i)_{\ord}\oplus (V_i)_{\nord} \ :=\prod_{\fm \text{\ ordinary}} \!\!\!(V_i)_{\fm}\ \ \times \!\!\!\prod_{\fm \text{\ non\ ordinary}} \!\!\!\!\!\!(V_i)_{\fm}.
\end{equation}
Note that, for $v\in V_i$, we have $v\in (V_i)_{\nord}$ if and only if there exists $z\in Z_{L_P}^+$ such that $z\cdot v=0$. In particular $(V_i)_{\nord}=\NOrd_P(V)\cap V_i$. Note also that $V_i$ is stable by $L_0$ since $I_{i,i}$ is normalized by $L_0$. Since the action of $L_0$ and $Z_{L_P}^+$ commute, (\ref{equ: ord-dec}) is equivariant under the action of $L_0$.\\

\noindent
(b) For $j>i$, the natural injection $V_i \hookrightarrow V_j$ is equivariant under the action of $Z_{L_P}^+$ and $L_0$. Therefore the restriction to the subspace $V_i$ induces a surjection $\kappa_{j,i}: B_j \twoheadrightarrow B_i$ of finite $A/\fm_A^n$-algebras (it is surjective because both $A$-algebras are generated by the image of $Z_{L_P}^+$). For a maximal ideal $\fn$ of $B_i$, it is clear that $\fn$ is ordinary (resp. nonordinary) if and only if $\kappa_{j,i}^{-1}(\fn)$ is ordinary (resp. nonordinary). Thus the inclusion
$V_i \hookrightarrow V_j$ induces injections $(V_i)_{\ord}\hookrightarrow (V_j)_{\ord}$ and $(V_i)_{\nord} \hookrightarrow (V_j)_{\nord}$ which are equivariant under the action of $L_0$ and $Z_{L_P}^+$. From $(V_i)_{\nord}=\NOrd_P(V)\cap V_i$ in (a), we also see $\NOrd_P(V)\cong \varinjlim_i (V_i)_{\nord}$.\\

\noindent
(c) By \cite[Thm. 3.3.3]{EOrd1}, we have $\Ord_P(V)=\varinjlim_i \Ord_P(V)^{L_i}$ and $\iota_{\can}$ is injective. Moreover, we have:
\begin{equation}\small\label{equ: ord-fini}
 \Ord_P(V)^{L_i}=\Hom_{A[Z_{L_P}^+]}\big(A[Z_{L_P}(L)], V^{L_iN_0}\big)_{Z_{L_P}(L)-\fini}=\Hom_{A[Z_{L_P}^+]}\big(A[Z_{L_P}(L)], V^{I_{i,i}}\big)
\end{equation}
where the first equality follows by definition (recall $L_0$, and hence $L_i$, normalize $N_0$ and commute with $Z_{L_P}^+$), and the second follows by the proof of \emph{loc.cit} as we now explain. Since $V^{I_{i,i}}$ is a finitely generated $A$-module, any element in $\Hom_{A[Z_{L_P}^+]}(A[Z_{L_P}(L)], V^{I_{i,i}})$ is locally $Z_{L_P}(L)$-finite, hence we have an inclusion:
\begin{equation}\label{equ: ord-fini2}
\Hom_{A[Z_{L_P}^+]}\big(A[Z_{L_P}(L)], V^{I_{i,i}}\big)\subseteq \Hom_{A[Z_{L_P}^+]}\big(A[Z_{L_P}(L)], V^{L_iN_0}\big)_{Z_{L_P}(L)-\fini}.
\end{equation}
However, by the proof of \cite[Thm. 3.3.3]{EOrd1}, we have $\iota_{\can}(\Ord_P(V)^{L_i}) \subseteq V^{I_{i,i}}$, in other words, any element in the right hand side set of (\ref{equ: ord-fini2}) has image in $V^{I_{i,i}}$ and thus is contained in the left hand side (note that by Hypothesis \ref{hypo: ord}, the $A$-module $U$ in the proof of \cite[Thm. 3.3.3]{EOrd1} is actually equal to $V^{I_{j,j}}$ with the notation of {\it loc.cit.}).\\

\noindent
(d) Combining (\ref{equ: ord-fini}) with the isomorphism at the end of the proof of \cite[Lem. 3.1.5]{EOrd1} (applied to $U=V^{I_{i,i}}=V_i$), the map $\iota_{\can}$ induces an isomorphism $\Ord_P(V)^{L_i} \xlongrightarrow{\sim} (V_i)_{\ord}$ which is equivariant under the action of $L_0$ and $Z_{L_P}^+$. Thus we deduce $\Ord_P(V)\cong \varinjlim_i (V_i)_{\ord}$ and together with (\ref{equ: ord-dec}) and (b):
\begin{equation*}
 V^{N_0}\cong \varinjlim_i V_i \cong \varinjlim_i \big((V_i)_{\ord} \oplus (V_i)_{\nord}\big) \cong \Ord_P(V) \oplus \NOrd_P(V)
\end{equation*}
which concludes the proof.
\end{proof}

\begin{corollary}\label{coro: ordinj}
Assume $A:=\co_E/\varpi_E^n$ for some $n>0$, $V$ is an admissible representation of $G(L)$ over $A$ and $V$ is an injective object in the category of smooth representations of $I_0$ over $A$. Then $\Ord_P(V)$ is an injective object in the category of smooth representations of $L_0$ over $A$.
\end{corollary}
\begin{proof}
By the same argument as in the proof of \cite[Cor. 5.3.19]{Em4}, there exists $r>0$ such that $V$ is a direct factor of $\cC(I_0, A)^{\oplus r}$ as a representation of $I_0$ where $\cC(I_0, A)$ (= the $A$-module of continuous, hence locally constant, functions from $I_0$ to $A$ with the discrete topology on the latter) is endowed with the left action of $I_0$ by right translation. Since $I_0$ admits an Iwahori decomposition, we deduce from this that $V^{N_0}$ is a direct factor of:
\begin{equation}\label{equ: ord-N0inv}
W:=(\cC(I_0, A)^{N_0})^{\oplus r}\cong \big(\cC(\overline{N}_0, A) \otimes_{\co_E/\varpi_E^n}\cC(L_0, A)\big)^{\oplus r}
\end{equation}
where $L_0$ acts on the latter by $l(f\otimes h):=f\otimes l(h)$. By \cite[Prop. 2.1.3]{EOrd2}, $W$ is an injective object in the category of smooth representations of $L_0$ over $A$. It follows from Theorem \ref{thm: Ord-dirc} that $\Ord_P(V)$ is a direct factor of $W$, and hence also an injective object.
\end{proof}

\noindent
Let now $V$ be a $\varpi_E$-adically continuous representation of $G$ over $A$ in the sense of \cite[Def.~2.4.1]{EOrd1}. Then $V/\varpi_E^n$ is a smooth representation of $G$ over $A/\varpi_E^n$ for all $n\in \Z_{>0}$. Following \cite[Def.~3.4.1]{EOrd1}, we define:
\begin{equation}\label{ordproj}
\Ord_P(V):=\varprojlim_{n} \Ord_P(V/\varpi_E^n V)
\end{equation}
which is a $\varpi_E$-adically continuous representation of $L_P(L)$ over $A$ (cf. \cite[Prop. 3.4.6]{EOrd1}). We have the canonical lifting map (cf. \cite[(3.4.7)]{EOrd1}):
\begin{equation}\label{equ: ord-can2}
\iota_{\can}: \Ord_P(V) \longrightarrow V^{N_0}
\end{equation}
which is $L_P^+$-equivariant. By \cite[Thm. 3.4.8]{EOrd1}, if $V$ is moreover admissible (\cite[Def.~2.4.7]{EOrd1}), $\Ord_P(V)$ is also admissible and $\iota_{\can}$ is a closed embedding (where the target and the source are equipped with the $\varpi_E$-adic topology).\\

\noindent
Let $V$ be a unitary Banach space representation of $G(L)$ over $E$ and $V^0$ an open bounded $G(L)$-invariant lattice of $V$ (i.e. a unit ball preserved by $G(L)$, which exists by definition as the representation is unitary). Then $V^0$ is a $\varpi_E$-adically continuous representation of $G$ over $\co_E$ and we put $\Ord_P(V):=\Ord_P(V^0)[1/p]$, which is easily checked to be independent of the choice of $V^0$. For any compact group $K$ we endow $\cC(K, \co_E)$ and $\cC(K, E)$ with the left action of $K$ by right translation on functions.

\begin{corollary}\label{coro: ord-inj}
Assume moreover that $V^0|_{I_0}$ is isomorphic to a direct factor of $\cC(I_0, \co_E)^{\oplus r}$ for some integer $r>0$. Then $\Ord_P(V^0)|_{L_0}$ (resp. $\Ord_P(V)|_{L_0}$) is isomorphic to a direct factor of $\cC(L_0, \co_E)^{\oplus r}$ (resp. $\cC(L_0, E)^{\oplus r}$) for some integer $s\geq 0$.
\end{corollary}
\begin{proof}
Let $n_1, n_2\in \Z_{>0}$ with $n_2>n_1$ and consider the exact sequence:
\begin{equation*}
0 \ra V^0/\varpi_E^{n_2-n_1} \xrightarrow{\varpi_E^{n_1}} V^0/\varpi_E^{n_2} \ra V^0/\varpi_E^{n_1} \ra 0.
\end{equation*}
Since $V^0|_{I_0} $ is a direct factor of $\cC(I_0,\co_E)^{\oplus r}$, arguing as in (\ref{equ: ord-N0inv}) we deduce an exact sequence (which is equivariant for the action of $Z_{L_P}^+$ and $L_0$):
\begin{equation*}
0 \ra (V^0/\varpi_E^{n_2-n_1})^{N_0} \xrightarrow{\varpi_E^{n_1}} (V^0/\varpi_E^{n_2})^{N_0} \ra (V^0/\varpi_E^{n_1})^{N_0} \ra 0.
\end{equation*}
Together with Theorem \ref{thm: Ord-dirc}, it follows that $\Ord_P(V^0/\varpi_E^{n_2})/\varpi_E^{n_1} \cong \Ord_P(V^0/\varpi_E^{n_1})$. Moreover, from Corollary \ref{coro: ordinj} we deduce that the dual $\Hom_{\co_E}(\Ord_P(V^0/\varpi_E^n), \co_E/\varpi_E^n)$ (= $\co_E$-linear maps) is a finitely generated projective $\co_E/\varpi_E^n \lbrack \lbrack L_0\rbrack \rbrack$-module. By a projective limit argument, it is then not difficult to deduce that $\Hom_{\co_E}(\Ord_P(V^0), \co_E)$ is also a finitely generated projective $\co_E[[L_0]]$-module. Dualizing back using \cite{ST} the corollary follows.
\end{proof}

\subsection{Ordinary parts of locally algebraic representations}\label{sec: ord-lalg}

\noindent
We review and generalize the ordinary part of a locally algebraic representation of $G(L)$ that admits an invariant lattice (see \cite[\S~5.6]{Em4}).\\

\noindent
We keep the notation of \S\S~\ref{sec: ord-1}~\&~\ref{sec: ord-2} and now assume that $G$ is split. We fix a split torus $T$ over $L$ and a Borel subgroup containing $T$ such that $B\subseteq P$ (where $P$ is the parabolic subgroup of {\it loc.cit.}). We let $V_{\infty}$ be a smooth \emph{admissible} representation of $G(L)$ over $E$, $L(\lambda)$ the irreducible $\Q_p$-algebraic representation of $G(L)$ over $E$ of highest weight $\lambda\in \Hom({\rm Res}_{L/\Q_p}T,{{\mathbb G}_{\rm m}}_{/\Q_p})$ where $\lambda$ is dominant with respect to ${\rm Res}_{L/\Q_p}B$ and we set:
$$V:=V_{\infty}\otimes_E L(\lambda).$$
We denote by $L_P(\lambda)$ the irreducible $\Q_p$-algebraic representation of $L_P$ over $E$ of highest weight $\lambda$ and by $\delta_{L_P, \lambda}$ the central character of $L_P(\lambda)$. Note that we have $L_P(\lambda)\cong L(\lambda)^{N_0}\cong L(\lambda)^{N_P(L)}$ and by \cite[Prop. 4.3.6]{Em11}:
\begin{equation}\label{jpanal}
 J_P(V)\cong J_P(V_{\infty}) \otimes_E L_P(\lambda)
\end{equation}
where $J_P(V)$ on the left is the Jacquet-Emerton functor of the locally algebraic representation $V$ relative to the parabolic subgroup $P(L)$ and $J_P(V_{\infty})$ is the usual Jacquet functor of the smooth representation $V_\infty$.\\

\noindent
For $i\geq 0$ consider:
\begin{equation*}
V_i:=V_{\infty}^{I_{i,i}}\otimes_E L_P(\lambda)\subseteq V^{N_0}\cong V_{\infty}^{N_0} \otimes_E L_P(\lambda)
\end{equation*}
which is finite dimensional over $E$ since $V_{\infty}$ is admissible. We equip $V_{\infty}^{N_0}$ and $V^{N_0}$ with the Hecke action of $L_P^+$ given by (\ref{equ: ord-Hec}). Note that we have $z\cdot (v\otimes u)=(z\cdot v)\otimes (zu)$ for $z\in L_P^+$, $v\in V_{\infty}^{N_0}$ and $u\in L_P(\lambda)$. In particular by Hypothesis \ref{hypo: ord}, $V_i\subseteq V^{N_0}$ is invariant under this $Z_{L_P}^+$-action. Denote by $B_i$ the $E$-subalgebra of $\End_E(V_i)$ generated by the operators in $Z_{L_P}^+$, then $B_i$ is an Artinian $E$-algebra. Similarly to what we did in the proof of Theorem \ref{thm: Ord-dirc}, a maximal ideal $\fm$ of $B_i$ is called \emph{of finite slope} if ${\rm Image}(Z_{L_P}^+)\cap \fm=\emptyset$ (inside $\End_E(V_i)$). Let $\fm$ be such a maximal ideal of finite slope and consider:
\begin{equation*}
Z_{L_P}^+ \lra B_i \twoheadrightarrow B_i/\fm\hooklongrightarrow \overline{\Q_p}.
\end{equation*}
Note that the image of $ Z_{L_P}^+$ lies in $ \overline{\Q_p}^\times$. We call $\fm$ \emph{of slope zero} if the above composition has image contained in the units $ \overline{\Z_p}^{\times}$ (this is independent of the choice of the last embedding). Denote by $(V_i)_*$ with $*\in\{\cfs,\ \nul,\ 0,\ >0\}$ the direct sum of the localizations of $V_i$ at maximal ideals which are respectively: of finite slope, not of finite slope, of slope zero, not of slope zero. We have thus:
\begin{equation}\label{equ: ord-alga}
 V_i\cong (V_i)_{\cfs}\oplus (V_i)_{\nul} \cong (V_i)_0 \oplus (V_i)_{>0}
\end{equation}
and note that $v\in (V_i)_{\nul}$ if and only if there exists $z\in Z_{L_P}^+$ such that $z\cdot v=0$. Moreover, as in the proof of Theorem \ref{thm: Ord-dirc}, for $j\geq i$ the natural injection $V_i \hookrightarrow V_j$ induces a $Z_{L_P}^+$-equivariant map for $*\in \{\cfs, \nul, 0, >0\}$:
\begin{equation*}
(V_i)_* \hooklongrightarrow (V_j)_*.
\end{equation*}
For $*\in \{\cfs, 0\}$, this action (uniquely) extends to $Z_{L_P}(L)$ since the action of $Z_{L_P}^+$ on $(V_i)_*$ is invertible. For $*\in \{\cfs, 0, \nul, >0\}$, we set:
\begin{equation}\label{equ: ord-algalim}
(V^{N_0})_{*}:=\varinjlim_i (V_i)_*
\end{equation}
which is an $E$-vector subspace of $V^{N_0}$ stable by $L_P^+$ (indeed, each $(V_i)_*$ is a generalized eigenspace of some sort for the action of $Z_{L_P}^+$ on $V_i$, and the action of $L_P^+$ on $V^{N_0}=\varinjlim_{i} V_i$ commutes with that of $Z_{L_P}^+$, so preserves generalized eigenspaces of $Z_{L_P}^+$ even though it may send a vector of $V_i$ to $V_j$ for some $j\gg i$). Moreover, for $*\in \{\cfs, 0\}$ this action of $L_P^+$ on $(V^{N_0})_{*}$ uniquely extends to $L_P(L)$ by \cite[Prop.~3.3.6]{Em11}. The decomposition (\ref{equ: ord-alga}) induces $L_P^+$-equivariant decompositions:
\begin{equation}\label{equ: ord-algb}
 V^{N_0}\cong (V^{N_0})_{\cfs} \oplus (V^{N_0})_{\nul} \cong (V^{N_0})_0 \oplus (V^{N_0})_{>0}.
\end{equation}
Moreover it follows from (\ref{jpanal}), $V^{N_0}\cong V_{\infty}^{N_0} \otimes_E L_P(\lambda)$ and the proof of \cite[Prop.~4.3.2]{Em11} (we leave here the details to the reader) that we have an isomorphism of locally algebraic representations of $L_P(L)$ (called the canonical lifting):
\begin{equation}\label{equ: lgOrd-relev}
J_P(V) \xlongrightarrow{\sim} (V^{N_0})_{\cfs}(\delta_P)
\end{equation}
where $(\delta_P)$ means the twist by the modulus character $\delta_P$.\\

\noindent
If $W$ is an $E$-vector space, recall an $\co_E$-lattice of $W$ is by definition an $\co_E$-submodule which generates $W$ over $E$ and doesn't contain any nonzero $E$-line. If $W$ is a $E[Z_{L_P}(L)]$-module such that the $Z_{L_P}(L)$-orbit of any element of $W$ is of finite dimension, by the very same construction as above we have a decomposition $W=W_0\oplus W_{>0}$ analogous to (\ref{equ: ord-algb}).

\begin{lemma}\label{lem: ord-alg7}
Let $W$ be an $E$-vector space equipped with a $Z_{L_P}^+$-action and let $f: W \ra V^{N_0}$ be an $E$-linear $Z_{L_P}^+$-equivariant map.\\
\noindent
(1) If $W$ is moreover an $E[Z_{L_P}(L)]$-module, then $f$ factors through a $Z_{L_P}(L)$-equivariant map $f: W \lra (V^{N_0})_{\cfs}$.\\
\noindent
(2) If $W$ is an $E[Z_{L_P}(L)]$-module such that the $Z_{L_P}(L)$-orbit of any element of $W$ is of finite dimension, then $f$ restricts to a $Z_{L_P}(L)$-equivariant map $W_0 \lra (V^{N_0})_0$. In particular, if $W$ admits a $Z_{L_P}(L)$-invariant $\co_E$-lattice, then $f$ factors through $W \lra (V^{N_0})_0$.
\end{lemma}
\begin{proof}
(1) For $v\in V^{N_0}$, we have $v\in (V^{N_0})_{\nul}$ if and only if there exists $z\in Z_{L_P}^+$ such that $z\cdot v=0$, which easily implies (1) using the first isomorphism in (\ref{equ: ord-algb}).\\
\noindent
(2) From the assumption on $W$ we can write $W=\varinjlim_\alpha (W_\alpha)$ where the $W_\alpha\subseteq W$ are finite dimensional and preserved by $Z_{L_P}(L)$, and by (1) it is enough to prove $f((W_\alpha)_0)\subseteq (V^{N_0})_0$, but this is clear from the definition. If $W^0$ is a $Z_{L_P}(L)$-invariant $\co_E$-lattice of $W$, then $W^0\cap W_\alpha$ is a $Z_{L_P}(L)$-invariant $\co_E$-lattice in $W_\alpha$ which easily implies $(W_\alpha)_0= W_\alpha$ and (2) follows.
\end{proof}

\begin{remark}\label{choicen0}
{\rm It easily follows from the first statement in Lemma \ref{lem: ord-alg7}(2) and the fact the $L_P(L)$-representations $(V^{N_0})_{\cfs}$ doesn't depend on the choice of $N_0$ up to isomorphism (see \cite[Prop.~3.4.11]{Em11}) that the ${L_P}(L)$-representation $(V^{N_0})_{0}$ also doesn't depend on the choice of $N_0$ up to isomorphism.}
\end{remark}

\noindent
Assume from now on that $V$ is a {\it unitary} $G(L)$-representation, i.e. admits an $\co_E$-lattice $V^0$ which is stable by $G(L)$, and set $V_i^0:=V_i \cap V^0$, which is thus an $\co_E$-lattice of $V_i$ stable by $Z_{L_P}^+$ (note that $(V^0)^{N_0}=\varinjlim_{i} V_i^0$). Denote by $A_i$ the $\co_E$-subalgebra of $\End_{\co_E}(V_i^0)$ generated by $Z_{L_P}^+$, then $A_i$ is an $\co_E$-algebra which is a free $\co_E$-module of finite type and we have $B_i\cong A_i\otimes_{\co_E} E$ and $A_i=\prod_{\fn}(A_i)_{\fn}$ where the product runs over the maximal ideals $\fn$ of $A_i$. As in the proof of Theorem \ref{thm: Ord-dirc}, a maximal ideal $\fn$ of $A_i$ is called ordinary if ${\rm Image}(Z_{L_P}^+) \cap \fn=\emptyset$ and we put:
\begin{equation*}
(V_i^0)_{\ord}:=\oplus_{\fn \text{ ordinary}} (V_i^0)_{\fn} \ \ \ \ (V_i^0)_{\nord}:=\oplus_{\fn \text{ nonordinary}} (V_i^0)_{\fn}.
\end{equation*}
We have $V_i^0\cong (V_i^0)_{\ord} \oplus (V_i^0)_{\nord}$ and we set $(V_i)_{\ord}:=(V_i^0)_{\ord} \otimes_{\co_E} E$.

\begin{lemma}\label{lem: ord-alg}
We have $(V_i^0)_{\ord}\cong V_i^0\cap (V_i)_0$, and hence $(V_i)_{\ord}\cong (V_i)_0$.
\end{lemma}
\begin{proof}
Let $\fm$ be a maximal ideal of $B_i$ and $\fn$ the unique maximal ideal of $A_i$ containing $\fm\cap A_i$ and $j: B_i/\fm \hookrightarrow \overline{\Q_p}$ an embedding as above. Then the restriction of $j$ to $A_i/(\fm\cap A_i)$ induces $j: A_i/(\fm\cap A_i) \hookrightarrow \overline{\Z_p}$ and we have $\fn/(\fm\cap A_i)=j^{-1}(\fm_{\overline{\Z_p}})$ (where $\fm_{\overline{\Z_p}}$ is the maximal ideal of $\overline{\Z_p}$). It is then easy to see that $\fn$ is ordinary if and only if $\fm$ is of slope zero. The inclusions $(V_i^0)_{\fn} \subseteq (V_i)_{\fn}\subseteq (V_i)_{\fm}$ thus imply $(V_i^0)_{\ord} \subseteq V_i^0 \cap (V_i)_0$. On the other hand, we have $V_i^0 \cap (V_i)_{\fm}\subseteq (V_i^0)_{\fn}$ and thus $V_i^0 \cap (V_i)_{0}\subseteq (V_i^0)_{\ord}$. The lemma follows.
\end{proof}

\noindent
The action of $Z_{L_P}^+$ on $(V_i^0)_{\ord}$ being invertible, it (uniquely) extends to an action of $Z_{L_P}(L)$ and the isomorphism $(V_i^0)_{\ord}\otimes_{\co_E} E \cong (V_i)_0$ of Lemma \ref{lem: ord-alg} is equivariant under the action of $Z_{L_P}(L)$. We set (using Lemma \ref{lem: ord-alg} for the second equality):
\begin{equation}\label{equ: ord-ordalg}
\Ord_P(V^0):=\varinjlim_i (V_i^0)_{\ord} =V^0\cap (V^{N_0})_0\hooklongrightarrow (V^0)^{N_0}=\varinjlim_i (V_i^0)
\end{equation}
and $\Ord_P(V):=\Ord_P(V^0)\otimes_{\co_E} E\hookrightarrow V^{N_0}$. The combined actions of $Z_{L_P}(L)$ and of $L_P^+$ on $\Ord_P(V^0)$ (the action of $L_P^+$ being induced by that on $(V^0)^{N_0}$) imply with \cite[Prop.~3.3.6]{Em11} that this $L_P^+$-action uniquely extends to $L_P(L)$. We deduce that $\Ord_P(V)$ is a unitary representation of $L_P(L)$ over $E$ and we call it the \emph{$P$-ordinary part} of $V$.

\begin{lemma}\label{lem: ord-alg6}
We have an isomorphism $\Ord_P(V)\cong (V^{N_0})_0$, in particular the $L_P(L)$-represen\-tation $\Ord_P(V)$ is independent of the choice of $V^0$ and $N_0$, and $(V^{N_0})_0$ is a unitary representation of $L_P(L)$ over $E$.
\end{lemma}
\begin{proof}
The isomorphism follows from the second equality in (\ref{equ: ord-ordalg}). The lemma follows since $(V^{N_0})_0$ doesn't depend on any lattice.
\end{proof}

\begin{remark}
{\rm If we drop the assumption that $V$ admits an invariant $\co_E$-lattice, then the $L_P(L)$-representation $(V^{N_0})_0$ might not be a unitary representation of $L_P(L)$ over $E$.}
\end{remark}

\begin{lemma}\label{ordpart}
Let $P'\supseteq P$ be another parabolic subgroup of $G$ and $L_{P'}$ the Levi subgroup of $P'$ (containing $L_P$). Then we have:
\begin{equation*}
  \Ord_P(V) \cong \Ord_{P\cap L_{P'}}\big(\Ord_{P'}(V)\big).
\end{equation*}
\end{lemma}
\begin{proof}
Let $N_0':=N_0\cap N_{P'}(L)$ and $N_0'':=N_0\cap N_{P\cap L_{P'}}(L)$. We have $N_0\cong N_0'\rtimes N_0''$ and thus an isomorphism $V^{N_0}\cong (V^{N_0'})^{N_0''}$. By Lemma \ref{lem: ord-alg6} and (the first statement in) Lemma \ref{lem: ord-alg7}(2), we see that the embedding $(((V^{N_0'})_{0})^{N_0''})_{0} \hookrightarrow V^{N_0}$ factors through $(V^{N_0})_{0}$. On the other hand, we have an embedding $(V^{N_0})_{0} \hookrightarrow (V^{N_0'})_{0}$ (using $Z_{L_{P'}}(L) \subseteq Z_{L_P}(L)$) which factors through an embedding $(V^{N_0})_{0} \hookrightarrow (((V^{N_0'})_{0})^{N_0''})_{0}$ using $L_{P\cap L_{P'}}\cong L_P$ and (again the first statement in) Lemma \ref{lem: ord-alg7}(2). We deduce an isomorphism $(V^{N_0})_{0}\cong (((V^{N_0'})_{0})^{N_0''})_{0}$ whence the result by Lemma \ref{lem: ord-alg6}.
\end{proof}

\begin{remark}\label{ordpart0}
{\rm If we drop the assumption that $V$ admits an invariant $\co_E$-lattice, the proof of Lemma \ref{ordpart} still gives $(V^{N_0})_{0}\cong (((V^{N_0'})_{0})^{N_0''})_{0}$ (with the notation in the proof of {\it loc.cit.}), and if we use Lemma \ref{lem: ord-alg7}(1) instead of Lemma \ref{lem: ord-alg7}(2), the same proof gives $(V^{N_0})_{\cfs}\cong (((V^{N_0'})_{\cfs})^{N_0''})_{\cfs}$ (which can also be deduced from (\ref{equ: lgOrd-relev})).}
\end{remark}

\noindent
Fix $n\in \Z_{>0}$ and consider $V^0/\varpi_E^n$ which is a smooth representation of $G(L)$ over $\co_E/\varpi_E^n$. We have $(V^0)^{N_0}/\varpi_E^n=\varinjlim_i (V_i^0/\varpi_E^n)$. For $i\in \Z_{\geq 0}$ the quotient $A_i/\varpi_E^n$ of $A_i$ is isomorphic to the $\co_E/\varpi_E^n$-subalgebra of $\End_{\co_{E}/\varpi_E^n}(V_i^0/\varpi_E^n)$ generated by $Z_{L_P}^+$. We have a natural bijection between the maximal ideals $\fm$ of $A_i$ and the maximal ideals $\overline{\fm}$ of $A_i/\varpi_E^n$ (since any maximal ideal of $A_i$ contains $\varpi_E$) and it is easy to see that $\fm\subset A_i$ is ordinary if and only if $\overline{\fm}$ is ordinary (see the proof of Theorem \ref{thm: Ord-dirc}). We deduce an isomorphism of $A_i/\varpi_E^n$-modules (see (\ref{equ: ord-dec})):
\begin{equation}\label{equ: ord-alg3}
(V_i^0)_{\ord}/\varpi_E^n \xlongrightarrow{\sim} (V_i^0/\varpi_E^n)_{\ord}.
\end{equation}

\begin{lemma}\label{lem: ord-alg2}
We have an $L_P(L)$-equivariant injection where $\Ord_P(V^0/\varpi_E^n)$ is defined as in (\ref{deford}):
\begin{equation}\label{equ: ord-alg}
\varinjlim_i (V_i^0/\varpi_E^n)_{\ord} \hooklongrightarrow \Ord_P(V^0/\varpi_E^n).
\end{equation}
Moreover, the composition of (\ref{equ: ord-alg}) with the canonical lifting (\ref{equ: ord-can}) gives the natural injection $\varinjlim_i (V_i^0/\varpi_E^n)_{\ord} \hooklongrightarrow (V^0/\varpi_E^n)^{N_0}$.
\end{lemma}
\begin{proof}
For any $i\geq 0$, by the last isomorphism in the proof of \cite[Lem. 3.1.5]{EOrd1} we have:
\begin{eqnarray*}
(V_i^0/\varpi_E^n)_{\ord} &\cong & \Hom_{\co_E/\varpi_E^n[Z_{L_P}^+]}\big(\co_E/\varpi_E^n[Z_{L_P}(L)], V_i^0/\varpi_E^n\big) \\
  &\cong & \Hom_{\co_E/\varpi_E^n[Z_{L_P}^+]}\big(\co_E/\varpi_E^n[Z_{L_P}(L)], V_i^0/\varpi_E^n\big)_{Z_{L_P}(L)-\fini} \\
  & \hooklongrightarrow & \Hom_{\co_E/\varpi_E^n[Z_{L_P}^+]}\big(\co_E/\varpi_E^n[Z_{L_P}(L)], (V^0/\varpi_E^n)^{N_0}\big)_{Z_{L_P}(L)-\fini}
\end{eqnarray*}
where the second isomorphism follows from the fact $V_i^0$ is of finite rank over $\co_E$.
The first part of the lemma follows. By unwinding the maps, the second part also easily follows.
\end{proof}

\begin{remark}\label{rem: ord-sur}
{\rm (1) The embedding $\varinjlim_i V_i^0/\varpi_E^n \cong (V^0)^{N_0}/\varpi_E^n \hooklongrightarrow (V^0/\varpi_E^n)^{N_0}$ is not surjective in general. Consequently (e.g. by the proof of Lemma \ref{lem: ord-alg2}), (\ref{equ: ord-alg}) might not be surjective in general.\\
\noindent
(2) If the inclusion $V_i^0/\varpi_E^n \hookrightarrow (V^0/\varpi_E^n)^{I_{i,i}}$ is an isomorphism for all $i$ (which in particular implies $(V^0)^{N_0}/\varpi_E^n \buildrel\sim\over\rightarrow (V^0/\varpi_E^n)^{N_0}$ and that the $G(L)$-representation $V^0/\varpi_E^n$ is admissible), it follows from (\ref{equ: ord-fini}) and the proof of Lemma \ref{lem: ord-alg2} that (\ref{equ: ord-alg}) is an isomorphism.}
\end{remark}

\begin{lemma}\label{lem: ord-alg3}
We have a natural $L_P(L)$-equivariant injection:
\begin{equation}\label{equ: ord-alg2}
\Ord_P(V^0) \hooklongrightarrow \varprojlim_n \Ord_P(V^0/\varpi_E^n) = \Ord_P(\widehat{V}^0)\ \ ({see}\ (\ref{ordproj}))
\end{equation}
where $\widehat{V}^0:=\varprojlim_n V^0/\varpi_E^n$. Moreover, the composition of (\ref{equ: ord-alg2}) with the (projective limit over $n$ of the) canonical lifting (\ref{equ: ord-can2}) coincides with the composition of the natural injections:
\begin{equation*}
\Ord_P(V^0) \hooklongrightarrow (V^0)^{N_0} \hooklongrightarrow (\widehat{V}^0)^{N_0}.
\end{equation*}
\end{lemma}
\begin{proof}
For any $n\in \Z_{>0}$, by (\ref{equ: ord-ordalg}) and (\ref{equ: ord-alg3}) we have:
\begin{equation*}
 \Ord_P(V^0)\cong \varinjlim_i (V_i^0)_{\ord} \twoheadlongrightarrow \varinjlim_i (V_i^0/\varpi_E^n)_{\ord}.
\end{equation*}
It is easy to see $(\varpi_E^n V_{i+1}^0) \cap V_i^0=(\varpi_E^n V^0) \cap V_i^0=\varpi_E^n V_i^0$. Hence the above surjection induces an isomorphism $\Ord_P(V^0)/\varpi_E^n \xrightarrow{\sim} \varinjlim_i (V_i^0/\varpi_E^n)_{\ord}$. We also have $\cap_n \varpi_E^n \Ord_P(V^0)=0$ since the same holds for $V^0$, and thus we obtain an injection:
\begin{equation*}\label{lastdense}
\Ord_P(V^0) \hooklongrightarrow \varprojlim_n \big(\Ord_P(V^0)/\varpi_E^n\big) \cong \varprojlim_n \big( \varinjlim_i (V_i^0/\varpi_E^n)_{\ord}\big).
\end{equation*}
By (\ref{equ: ord-alg}) and taking the projective limit over $n$, (\ref{equ: ord-alg2}) follows. The second part of the lemma follows from the second part of Lemma \ref{lem: ord-alg2}.
\end{proof}

\begin{remark}\label{rem: ord-dense}
{\rm By (\ref{lastdense}) and Remark \ref{rem: ord-sur}(2), if $V_i^0/\varpi_E^n \xrightarrow{\sim} (V^0/\varpi_E^n)^{I_{i,i}}$ for all $i$, then we see that (\ref{equ: ord-alg2}) has dense image where $\Ord_P(\widehat{V}^0)$ is endowed with the $\varpi_E$-adic topology.}
\end{remark}

\begin{lemma}\label{lem: ord-alg4}
Let $W$ be a unitary Banach representation of $G(L)$ over $E$, $W^0\subset W$ an open bounded $G(L)$-invariant lattice and $f: V^0 \ra W^0$ an $\co_E$-linear $G(L)$-equivariant morphism, which induces a $G(L)$-equivariant morphism $f: V\ra W$. Then $f$ induces an $L_P(L)$-equivariant morphism:
\begin{equation}\label{equ: ord-alg4}
\Ord_P(V^0) \lra \Ord_P(W^0) \ \ (\text{resp. } \Ord_P(V) \lra \Ord_P(W))
\end{equation}
such that the following diagram commutes (resp. with $V^0$, $W^0$ replaced by $V$, $W$):
\begin{equation}\label{equ: ord-comm}
\begin{CD}
  \Ord_P(V^0) @>>> \Ord_P(W^0) \\
  @VVV @VVV \\
  (V^0)^{N_0} @>>> (W^0)^{N_0}.
\end{CD}
\end{equation}
Moreover, if $f$ is injective and $V^0=W^0 \cap V$, then the morphisms in (\ref{equ: ord-alg4}) are injective.
\end{lemma}
\begin{proof}
Since $W^0$ is $\varpi_E$-adically complete, the morphism $f$ induces a morphism $\widehat{f}: \widehat{V}^0 \ra W^0$. By (\ref{equ: ord-alg2}) and the functoriality of $\Ord_P(\cdot)$ on the category of $\varpi_E$-adically continuous representations of $G(L)$, we deduce the morphisms in (\ref{equ: ord-alg4}). By the functoriality of the canonical lifting (\ref{equ: ord-can2}) and the second part of Lemma \ref{lem: ord-alg3}, the commutative diagram (\ref{equ: ord-comm}) follows. At last, if $f$ is injective and $V^0=W^0\cap V$, we have $\varpi_E^n V^0=(\varpi_E^n W^0)\cap V$ hence $V^0/\varpi_E^n\hookrightarrow W^0/\varpi_E^n$, and by (\ref{equ: ord-alg2}) and the left exactness of $\Ord_P(\cdot)$ (\cite[Prop.~3.2.4]{EOrd1}) the morphisms in (\ref{equ: ord-alg4}) are injective.
\end{proof}

\subsection{An adjunction property}

\noindent
We study some adjunction property of the functor $\Ord_P(\cdot)$ of \S~\ref{sec: ord-lalg} on locally algebraic representations.\\

\noindent
We keep the notation of \S\S~\ref{sec: ord-1},~\ref{sec: ord-2}~\&~\ref{sec: ord-lalg}. If $U$ is any $E$-vector space, denote by $\cC_c^{\infty}(N_P(L), U)$ the $E$-vector space of $U$-valued locally constant functions with compact support in $N_P(L)$ endowed with the left action of $N_P(L)$ by right translation on function. If $U_\infty$ is a smooth representation of $L_P(L)$ over $E$, recall that there is a natural $N_P(L)$-equivariant injection:
\begin{equation}\label{embednp}
\cC_c^{\infty}(N_P(L), U_{\infty}) \hooklongrightarrow \big(\Ind_{\overline{P}(L)}^{G(L)} U_{\infty}\big)^{\infty}
\end{equation}
sending $f\in \cC_c^{\infty}(N_P(L), U_{\infty})$ to $F\in (\Ind_{\overline{P}(L)}^{G(L)} U_{\infty})^{\infty}$ such that:
\begin{equation*}
 F(g)=\begin{cases}
   \overline{p}(f(n)) & {\rm for}\ g=\overline{p}n\in \overline{P}(L)N_P(L) \\
   0 & \text{otherwise}.
  \end{cases}
\end{equation*}

\begin{lemma}\label{lem: ord-alg5}
Let $U_{\infty}$ be a smooth admissible representation of $L_P(L)$ over $E$ and assume that $U:=U_{\infty} \otimes_E L_P(\lambda)$ is unitary as representation of $L_P(L)$ ($L_P(\lambda)$ as in the beginning of \S~\ref{sec: ord-lalg}). Then the locally algebraic representation $(\Ind_{\overline{P}(L)}^{G(L)} U_{\infty})^{\infty} \otimes_E L(\lambda)$ is unitary as representation of $G(L)$ and there is a natural $L_P(L)$-equivariant injection:
\begin{equation}\label{equ: ord-alg5}
U=U_{\infty}\otimes_E L_P(\lambda) \hooklongrightarrow \Ord_P\big((\Ind_{\overline{P}(L)}^{G(L)} U_{\infty})^{\infty}\otimes_E L(\lambda)\big)
\end{equation}
such that the composition of (\ref{equ: ord-alg5}) with the natural injection (see just after (\ref{equ: ord-ordalg})):
\begin{equation*}
\Ord_P\big((\Ind_{\overline{P}(L)}^{G(L)} U_{\infty})^{\infty} \otimes_E L(\lambda)\big) \hooklongrightarrow \big((\Ind_{\overline{P}(L)}^{G(L)} U_{\infty})^{\infty} \otimes_E L(\lambda)\big)^{N_0}
\end{equation*}
has image in $\cC_c^{\infty}(N_P(L), U)^{N_0}\cong(\cC_c^{\infty}(N_P(L), U^\infty)\otimes_EL_P(\lambda))^{N_0}$ via (\ref{embednp}) (tensored with $L(\lambda)$) and maps $u\in U$ to the unique function $f_u\in \cC_c^{\infty}(N_P(L), U)^{N_0}$ such that $f_u(n)=u$ for all $n\in N_0$ and $f_u(n)=0$ otherwise.
\end{lemma}
\begin{proof}
For simplicity, we write $V:=(\Ind_{\overline{P}(L)}^{G(L)} U_{\infty})^{\infty} \otimes_E L(\lambda)$. Let $U^0$ be an $L_P(L)$-invariant $\co_E$-lattice of $U$ and $\widehat{U^0}:=\varprojlim_n U^0/\varpi_E^n$. We have $G(L)$-equivariant embeddings:
\begin{equation}\label{equ: ord-alg6}
V \hooklongrightarrow \big(\Ind_{\overline{P}(L)}^{G(L)}U\big)^{\an}\hookrightarrow \big(\Ind_{\overline{P}(L)}^{G(L)}\widehat{U^0}\otimes_{\co_E}E\big)^{\cC^0}.
\end{equation}
Since the right hand side of (\ref{equ: ord-alg6}) has an obvious invariant lattice given by $(\Ind_{\overline{P}(L)}^{G(L)}\widehat{U^0})^{\cC^0}$, its intersection with the left hand side also gives an invariant lattice on $V$, hence $V$ is unitary. We have:
\begin{equation*}
U \xlongrightarrow{\sim} J_P\big(\cC_c^{\infty}(N_P(L), U)\big)(\delta_P^{-1}) \hooklongrightarrow J_P(V)(\delta_P^{-1})
\end{equation*}
where the first isomorphism follows from \cite[Lem. 3.5.2]{Em11} (the above action of $N_P(L)$ on $\cC_c^{\infty}(N_P(L), U)$ being extended to $P(L)$ as in \cite[\S~3.5]{Em11}) and the second injection follows from the left exactness of $J_P(\cdot)$. Since $U$ is unitary, by Lemma \ref{lem: ord-alg7}(2) and Lemma \ref{lem: ord-alg6} we deduce an injection:
\begin{equation*}
U \hooklongrightarrow \Ord_P(V)\big(\hooklongrightarrow J_P(V)(\delta_P^{-1}) \hooklongrightarrow V^{N_0}\big)
\end{equation*}
(recall the second embedding follows from (\ref{equ: lgOrd-relev}) and the third from (\ref{equ: ord-algb})). Moreover the composition is equal to the composition:
\begin{equation*}
U \xlongrightarrow{\sim} J_P\big(\cC_c^{\infty}(N_P(L), U)\big)(\delta_P^{-1}) \hooklongrightarrow \cC_c^{\infty}(N_P(L), U)^{N_0} \hooklongrightarrow V^{N_0}
\end{equation*}
sending $u\in U$ to $f_u\in \cC_c^{\infty}(N_P(L), U)^{N_0}$ as in the statement of the lemma (see \cite[\S~3.5]{Em11}, in particular the proof of \cite[Lem. 3.5.2]{Em11}, see also the beginning of \cite[\S~2.8]{Em2}).
\end{proof}

\begin{lemma}\label{lem: ord-ordind}
Keep the notation and assumptions of Lemma \ref{lem: ord-alg5} and let $U^0$ be an $L_P(L)$-invariant $\co_E$-lattice of $U$ and $\widehat{U}:=(\varprojlim_n U^0/\varpi_E^n)\otimes_{\co_E}E$. Assume that $\widehat{U}$ is an admissible Banach representation of $G(L)$ over $E$ (\cite[\S~3]{ST}). We have a natural commutative diagram:
\begin{equation*}
  \begin{CD}
   U @>>> \widehat{U} \\
   @V (\ref{equ: ord-alg5}) VV @V \wr VV \\
   \Ord_P\big((\Ind_{\overline{P}(L)}^{G(L)} U_{\infty})^{\infty}\otimes_E L(\lambda)\big) @>>> \Ord_P\big((\Ind_{\overline{P}(L)}^{G(L)} \widehat{U})^{\cC^0}\big)
  \end{CD}
\end{equation*}
where the bottom map is induced by (\ref{equ: ord-alg6}) and Lemma \ref{lem: ord-alg4}, and where the isomorphism on the right is \cite[Cor. 4.3.5]{EOrd1}.
\end{lemma}
\begin{proof}
By (\ref{equ: ord-comm}) and the fact (\ref{equ: ord-can2}) (with $V=(\Ind_{\overline{P}(L)}^{G(L)} \widehat{U})^{\cC^0}$) is an embedding (note that $V$ is admissible by assumption), it is sufficient to prove that the following diagram:
\begin{equation}\label{equ: ord-alg7}
  \begin{CD}
   U @>>> \widehat{U} \\
   @VVV @VVV\\
   \big((\Ind_{\overline{P}(L)}^{G(L)} U_{\infty})^{\infty}\otimes_E L(\lambda)\big)^{N_0} @>>> \big((\Ind_{\overline{P}(L)}^{G(L)}
   \widehat{U})^{\cC^0}\big)^{N_0}
  \end{CD}
\end{equation}
is commutative. By Lemma \ref{lem: ord-alg5} the left map sends $u\in U^0$ to $f_u\in \cC_c^{\infty}(N_P(L),U^0)^{N_0}$, and by \cite[\S~4]{EOrd1} the right map is induced by the maps (with obvious notation):
\begin{equation*}
  U^0/\varpi_E^n \lra \cC_c^{\infty}(N_P(L), U^0/\varpi_E^n)^{N_0}, \ \ \overline{u} \longmapsto \overline{f_u}
\end{equation*}
then taking the inverse limit over $n$ and inverting $p$. We see (\ref{equ: ord-alg7}) commutes.
\end{proof}

\begin{proposition}\label{prop: ord-adj}
Let $U_{\infty}$ be a smooth admissible representation of $L_P(L)$ over $E$, $U:=U_{\infty} \otimes_E L_P(\lambda)$ and $V$ a unitary admissible Banach representation of $G(L)$ over $E$. Let $f: U \hooklongrightarrow \Ord_P(V)$ be an $L_P(L)$-equivariant injection and denote by $\widehat{U}$ the closure of $U$ in the Banach space $\Ord_P(V)$. Then $f$ induces $G(L)$-equivariant morphisms:
\begin{equation}\label{equ: ord-adj}
 (\Ind_{\overline{P}(L)}^{G(L)} U_{\infty})^{\infty} \otimes_E L(\lambda) \hooklongrightarrow \big(\Ind_{\overline{P}(L)}^{G(L)}\widehat{U}\big)^{\cC^0} \lra V
 \end{equation}
from which $f$ can be recovered as the following composition:
\begin{equation}
U \xlongrightarrow{(\ref{equ: ord-alg5})} \Ord_P\big((\Ind_{\overline{P}(L)}^{G(L)} U_{\infty})^{\infty}\otimes_E L(\lambda)\big) \lra \Ord_P(V)
\end{equation}
where the last map is induced from the composition (\ref{equ: ord-adj}) and Lemma \ref{lem: ord-alg4}.
\end{proposition}
\begin{proof}
Note that $U$ is a unitary representation of $L_P(L)$ and that $\widehat{U}$ is a unitary admissible Banach representation of $L_P(L)$ over $E$ by \cite[Thm.~3.4.8]{EOrd1}. The second map in (\ref{equ: ord-adj}) is then obtained by applying \cite[Thm. 4.4.6]{EOrd1}, and the first map is obtained as in (\ref{equ: ord-alg6}) (with $U^0:=\Ord_P(V^0)\cap U$ where $V^0$ is an open bounded $G(L)$-invariant lattice in $V$). The second part of the proposition follows from \cite[Thm. 4.4.6]{EOrd1} together with Lemma \ref{lem: ord-ordind}.
\end{proof}

\section{$P$-ordinary Galois representations and local Langlands correspondence}\label{galoisclassicallanglands}

\noindent
In this section, for $P$ a parabolic subgroup we define $P$-ordinary Galois representations and prove some standard compatibility with classical local Langlands correspondence which will be used later. We denote by $L$ a finite extension of $\Q_p$.

\subsection{$P$-ordinary Galois deformations}\label{sec: ord-galoisdef}

\noindent
We define $P$-ordinary Galois deformations and recall some standard useful statements.\\

\noindent
We fix $P$ a parabolic subgroup of $\GL_n$ containing the Borel subgroup of upper triangular matrices and with a Levi subgroup $L_P$ given by (where $\sum_{i=1}^k n_i=n$):
\begin{equation}\label{equ: ord-LP}
 \begin{pmatrix}
  \GL_{n_1} & 0 & \cdots & 0 \\
  0 & \GL_{n_2} & \cdots & 0 \\
  \vdots & \vdots & \ddots & 0 \\
  0 & 0 &\cdots & \GL_{n_k}
 \end{pmatrix}.
\end{equation}

\begin{definition}\label{Porddef}
Let $A$ be a topological commutative ring and $(\rho_A,T_A)$ a continuous $A$-linear representation of $\Gal_L$ on a free $A$-module $T_A$ of rank $n$ (we often just write $\rho_A$ for simplicity). The representation $\rho_A$ is $P$-ordinary (over $A$) if there exists an increasing filtration of $T_A$ by invariant free $A$-submodules which are direct summands as $A$-modules such that the graded pieces are of rank $n_1$, $n_2$, $\cdots$, $n_k$ over $A$.
\end{definition}

\noindent
Choosing a basis of $T_A$ over $A$, we see that a $P$-ordinary representation gives rise to a continuous group homomorphism $\Gal_L\rightarrow P(A)$. Fix a $P$-ordinary representation $\overline{\rho}=(\overline{\rho},T_{k_E})$ of $\Gal_L$ over $k_E$ together with an invariant increasing filtration $0=T_{k_E,0}\subsetneq T_{k_E,1}\subsetneq \cdots \subsetneq T_{k_E,k}=T_{k_E}$ as in Definition \ref{Porddef} and denote by $(\overline{\rho}_i,\gr_iT_{k_E,\bullet}:=T_{k_E,i}/T_{k_E,i-1})$, $i\in \{1,\cdots ,k\}$, or simply $\overline{\rho}_i$, for the representations of $\Gal_L$ over $k_E$ given by the graded pieces (thus $\overline{\rho}_i$ is of dimension $n_i$). We assume the following hypothesis on $\overline{\rho}$ and the $\overline{\rho}_i$.

\begin{hypothesis}\label{hypo: ord-Pord1}
We have $\End_{\Gal_L}(\overline{\rho}) \cong k_E$, $\End_{\Gal_L}(\overline{\rho}_i)\cong k_E$ for $i=1, \cdots, k$ and $\Hom_{\Gal_L}(\overline{\rho}_i, \overline{\rho}_j)=0$ for all $i\neq j$.
\end{hypothesis}

\noindent
Let $\Art(\co_E)$ be the category of local artinian $\co_E$-algebras with residue field $k_E$ and $\Def_{\overline{\rho}}$ (resp. $\Def_{\overline{\rho}_i}$) the usual functor of deformations of $\overline{\rho}$ (resp. of $\overline{\rho}_i$), i.e. the functor from $\Art(\co_E)$ to sets which sends $A\in \Art(\co_E)$ to the set $\{((\rho_A,T_A),i_A)\}/\!\sim$ where $(\rho_A,T_A)$ is a representation of $\Gal_L$ over $A$ as above, $i_A$ is a $\Gal_L$-equivariant isomorphism $T_A\otimes_Ak_E\buildrel\sim\over\rightarrow T_{k_E}$ ($T_{k_E}$ being the underlying vector space of $\overline{\rho}$) and $\sim$ means modulo the $\Gal_L$-equivariant isomorphisms $T_A\buildrel\sim\over\rightarrow T'_A$ such that the following induced diagram commutes:
\begin{equation}\label{deformA}
\begin{CD}
 T_A\otimes_A k_E @> \sim >> T'_A\otimes_A k_E\\
 @V \iota_A V\!\wr V @V \iota'_A V\!\wr V \\
T_{k_E} @= T_{k_E}
\end{CD}
\end{equation}
(resp. with $\overline{\rho}_i$ instead of $\overline{\rho}$). If $A\rightarrow B$ in $\Art(\co_E)$ then $T_A$ is sent to $T_A\otimes_AB$ (and $i_A$ to itself via $T_A\otimes_A B\otimes_Bk_E\cong T_A\otimes_Ak_E$). By choosing basis, the functor $\Def_{\overline{\rho}}(A)$ can also be described as the set:
\begin{multline*}
\{\rho_A:\Gal_L\rightarrow \GL_n(A)\ {\rm such\ that\ the\ composition\ with}\ \GL_n(A)\twoheadrightarrow \GL_n(k_E)\ {\rm gives}\\ \overline{\rho}:\Gal_L\rightarrow \GL_n(k_E)\}/\!\sim
\end{multline*}
where $\sim$ means modulo conjugation by matrices in $\GL_n(A)$ which are congruent to $1$ modulo the maximal ideal $\fm_A$ of $A$. Since $\End_{\Gal_L}(\overline{\rho})  \cong k_E$ (resp. $\End_{\Gal_L}(\overline{\rho}_i) \cong k_E$), it is a standard result (first due to Mazur) that this functor is pro-representable, and we denote by $R_{\overline{\rho}}$ (resp. $R_{\overline{\rho}_i}$) the universal deformation ring of $\overline{\rho}$ (resp. of $\overline{\rho}_i$), which is a complete local noetherian $\co_E$-algebra of residue field $k_E$.\\

\noindent
We now switch to $P$-ordinary deformations. We define the functor $\Def_{\overline{\rho}, \{\overline{\rho}_i\}}^{P-\ord}: \Art(\co_E) \ra \{\text{Sets}\}$ by sending $A\in \Art(\co_E)$ to the set:
$$\left\{\big((\rho_A,T_A),T_{A,\bullet},i_A\big)\right\}/\!\sim$$
where $((\rho_A,T_A),i_A)$ is as above, $T_{A,\bullet}=(0=T_{A,0}\subsetneq T_{A,1}\subsetneq \cdots \subsetneq T_{A,k}=T_{A})$ is an increasing filtration of $T_A$ by invariant free $A$-submodules which are direct summands as $A$-modules such that $i_A$ induces a $\Gal_L$-equivariant isomorphism $T_{A,i}\otimes_Ak_E\buildrel\sim\over\rightarrow T_{k_E,i}$ for $i\in \{1,\cdots,k\}$, and where $\sim$ means modulo the $\Gal_L$-equivariant isomorphisms $T_A\buildrel\sim\over\rightarrow T'_A$ satisfying (\ref{deformA}) and which moreover respect the increasing filtration on both sides. Alternatively, by choosing adapted basis one can describe $\Def_{\overline{\rho}, \{\overline{\rho}_i\}}^{P-\ord}(A)$ as the set:
\begin{multline}\label{equ: ord-Pdef}
\{\rho:\Gal_L\rightarrow P(A)\ {\rm such\ that\ the\ composition\ with}\ P(A)\twoheadrightarrow P(k_E)\ {\rm gives}\\
\overline{\rho}:\Gal_L\rightarrow P(k_E)\}/\!\sim
\end{multline}
where $\sim$ means modulo conjugation by matrices in $P(A)$ which are congruent to $1$ modulo the maximal ideal $\fm_A$ of $A$. The following two propositions are standard, we provide short proofs for the convenience of the reader.

\begin{lemma}\label{lem: ord-Pdef}
The functor $\Def_{\overline{\rho}, \{\overline{\rho}_i\}}^{P-\ord}$ is a subfunctor of $\Def_{\overline{\rho}}$.
\end{lemma}
\begin{proof}
Let $A\in \Art(\co_E)$, starting from $((\rho_A,T_A),i_A)\in \Def_{\overline{\rho}}(A)$, it is enough to prove that there is at most one filtration $T_{A,\bullet}$ on $T_A$ such that $i_A$ induces isomorphisms $T_{A,i}\otimes_Ak_E\buildrel\sim\over\rightarrow T_{k_E,i}$ and that any isomorphism $T_A\buildrel\sim\over\rightarrow T'_A$ satisfying (\ref{deformA}) is automatically compatible with the filtrations (when they exist). For the first statement, by d\'evissage it is enough to prove $T^{(1)}_{A,1}=T^{(2)}_{A,1}$ (where $T^{(1)}_{A,\bullet}$, $T^{(2)}_{A,\bullet}$ are two filtrations). But the equivariant map $T^{(1)}_{A,1}\rightarrow T_A/T^{(2)}_{A,1}$ must be zero (and hence $T^{(1)}_{A,1}=T^{(2)}_{A,1}$) since the $\Gal_L$-representation $T_A/T^{(2)}_{A,1}$ is by definition a successive extension of $\overline{\rho}_i$, $i\ne 1$ and we have $\Hom_{\Gal_L}(T^{(1)}_{A,1}, \overline{\rho}_i)=0$ for $i\ne 1$ by Hypothesis \ref{hypo: ord-Pord1} (and an obvious d\'evissage). The same argument replacing $T_A/T^{(2)}_{A,1}$ by $T'_A/T'_{A,1}$ shows that any equivariant isomorphism $T_A\buildrel\sim\over\rightarrow T'_A$ must send $T_{A,i}$ to $T'_{A,i}$.
\end{proof}

\begin{proposition}\label{prop: ord-Pdef}
The functor $\Def_{\overline{\rho}, \{\overline{\rho}_i\}}^{P-\ord}$ is pro-representable by a complete local noetherian $\co_E$-algebra $R_{\overline{\rho}, \{\overline{\rho}_i\}}^{P-\ord}$ of residue field $k_E$.
\end{proposition}
\begin{proof}
By Schlessinger's criterion (\cite{Schl}), Lemma \ref{lem: ord-Pdef} and the fact that $\Def_{\overline{\rho}}$ is pro-representa\-ble, it is enough to check that, given morphisms $f_1: A\ra C$, $f_2: B\ra C$ in $\Art(\co_E)$ with $f_2$ surjective and small, the induced map:
\begin{equation*}
\Def_{\overline{\rho}, \{\overline{\rho}_i\}}^{P-\ord}(A\times_C B)\lra \Def_{\overline{\rho}, \{\overline{\rho}_i\}}^{P-\ord}(A)\times_{\Def_{\overline{\rho}, \{\overline{\rho}_i\}}^{P-\ord}(C)} \Def_{\overline{\rho}, \{\overline{\rho}_i\}}^{P-\ord}(B)
\end{equation*}
is surjective. But this is immediate from the description (\ref{equ: ord-Pdef}).
\end{proof}

\noindent
By Proposition \ref{prop: ord-Pdef}, Lemma \ref{lem: ord-Pdef} and the fact that $R_{\overline{\rho}}$ is a complete local noetherian $\co_E$-algebra, we see (e.g. by \cite[Lem. 2.1]{Fer}) that the natural morphism $R_{\overline{\rho}} \ra R_{\overline{\rho}, \{\overline{\rho}_i\}}^{P-\ord}$ is surjective. Moreover, for $i\in \{1,\cdots, k\}$, we have a natural transformation of functors $\Def^{P-\ord}_{\overline{\rho}, \{\overline{\rho}_i\}} \ra \Def_{\overline{\rho}_i}$ sending $((\rho_A,T_A),T_{A,\bullet},i_A)$ to $\gr_iT_{A,\bullet}$ with the induced $i_A$. It corresponds to a canonical morphism of $\co_E$-algebras $R_{\overline{\rho}_i} \ra R_{\overline{\rho}, \{\overline{\rho}_i\}}^{P-{\ord}}$, and we deduce a morphism of local complete $\co_E$-algebras (with obvious notation):
\begin{equation}\label{equ: ord-diag}
 \widehat{\bigotimes}_{i=1, \cdots, k} R_{\overline{\rho}_i} \lra R_{\overline{\rho}, \{\overline{\rho}_i\}}^{P-\ord}.
\end{equation}

\noindent
Let us now consider equal characteristic $0$ deformations. Fix a $P$-ordinary representation ${\rho}$ of $\Gal_L$ over $E$ together with an invariant increasing filtration $0=T_{E,0}\subsetneq T_{E,1}\subsetneq \cdots \subsetneq T_{E,k}=T_{E}$ as in Definition \ref{Porddef} and denote by $({\rho}_i,\gr_iT_{E,\bullet}:=T_{E,i}/T_{E,i-1})$, $i\in \{1,\cdots ,k\}$ the graded pieces. As previously we assume the following hypothesis on $\rho$ and the $\rho_i$.

\begin{hypothesis}\label{hypo: ord-Pord2}
We \ have \ $\End_{\Gal_L}(\rho)\cong E$, \ $\End_{\Gal_L}(\rho_i)\cong E$ \ for \ $i=1, \cdots, k$ \ and $\Hom_{\Gal_L}(\rho_i, \rho_j)=0$ for $i\neq j$.
\end{hypothesis}

\noindent
Let $\Art(E)$ be the category of local artinian $E$-algebras with residue field $E$ and define $\Def_{{\rho}}$ (resp. $\Def_{{\rho}_i}$) as $\Def_{\overline{\rho}}$ (resp. $\Def_{\overline{\rho}_i}$) but replacing $\Art(E)$ by $\Art(\co_E)$ and $\overline{\rho}$ (resp. $\overline{\rho}_i$) by $\rho$ (resp. $\rho_i$). Then from Hypothesis \ref{hypo: ord-Pord2} the functor $\Def_{\rho}$ (resp. $\Def_{\rho_i}$) is pro-representable by a complete local noetherian $E$-algebra of residue field $E$ denoted by $R_{\rho}$ (resp. $R_{\rho_i}$). Likewise we define the functor $\Def_{\rho, \{\rho_i\}}^{P-\ord}$ of $P$-ordinary deformations of $\rho$ on $\Art(E)$ in a similar way as (\ref{equ: ord-Pdef}) and before by replacing $\overline{\rho}$, $T_{k_E,i}$ and $\overline{\rho}_i$ by $\rho$, $T_{E,i}$ and $\rho_i$. By the same proof as for Lemma \ref{lem: ord-Pdef} and Proposition \ref{prop: ord-Pdef}, we obtain the following proposition.

\begin{proposition}
The functor $\Def_{\rho, \{\rho_i\}}^{P-\ord}$ is a subfunctor of $\Def_{\rho}$ and is pro-representable by a complete local noetherian $E$-algebra $R_{\rho, \{\rho_i\}}^{P-\ord}$ of residue field $E$.
\end{proposition}

\noindent
Let $(\overline{\rho}, \{\overline{\rho}_i\})$ as before satisfying Hypothesis \ref{hypo: ord-Pord1}. Let $\xi: R_{\overline{\rho}, \{\overline{\rho}_i\}}^{P-\ord}\twoheadrightarrow \co_E$ be a homomorphism of local $\co_E$-algebras and denote by $\rho_{\xi}^0$ (resp. $\rho_{\xi,i}^0$) the deformation of $\overline{\rho}$ (resp. of $\overline{\rho}_i$) over $\co_E$ associated to $\xi$ via $\Def^{P-\ord}_{\overline{\rho}, \{\overline{\rho}_i\}} \ra \Def_{\overline{\rho}}$ (resp. $\Def^{P-\ord}_{\overline{\rho}, \{\overline{\rho}_i\}} \ra \Def_{\overline{\rho}_i}$). In particular, $\rho_{\xi}^0$ is a representation of $\Gal_L$ over a free $\co_E$-module $T_{\co_E}$ endowed with an invariant filtration by direct summands $T_{\co_E,i}$ as $\co_E$-modules such that the graded pieces give the representations $\rho_{\xi,i}^0$, $i=1, \cdots, k$. Let $\rho_{\xi}:=\rho_{\xi}^0 \otimes_{\co_E} E$ and $\rho_{\xi,i}:=\rho_{\xi,i}^0\otimes_{\co_E} E$.

\begin{proposition}\label{prop: ord-Pdef2}
(1) We have that $(\rho_{\xi}, \{\rho_{\xi,i}\})$ satisfies Hypothesis \ref{hypo: ord-Pord2}.\\
\noindent
(2) The \ $E$-algebra \ $R_{\rho_{\xi}, \{\rho_{\xi,i}\}}^{P-\ord}$ \ is \ isomorphic \ to \ the \ $(\Ker(\xi)\otimes_{\co_E} E)$-adic \ completion \ of $R_{\overline{\rho}, \{\overline{\rho}_i\}}^{P-\ord}\otimes_{\co_E} E$.
\end{proposition}
\begin{proof}
(1) is straightforward from Hypothesis \ref{hypo: ord-Pord1} and a d\'evissage.\\
\noindent
(2) Denote by $\Def_{\overline{\rho}, \{\overline{\rho}_i\}, (\xi)}^{P-\ord}$ (resp. $\Def_{\overline{\rho}, (\xi)}$) the generic fiber of $\Def_{\overline{\rho}, \{\overline{\rho}_i\}}^{P-\ord}$ (resp. $\Def_{\overline{\rho}}$) at $\xi$ in the sense of \cite[\S~2.3]{Kis09}. By \cite[Lem. 2.3.3]{Kis09}, it is sufficient to prove $\Def_{\overline{\rho}, \{\overline{\rho}_i\}, (\xi)}^{P-\ord}\cong \Def_{\rho_{\xi}, \{\rho_{\xi,i}\}}^{P-\ord}$. By \cite[Prop. 2.3.5]{Kis09}, the generic fiber $\Def_{\overline{\rho}, (\xi)}$ is isomorphic to $\Def_{\rho_{\xi}}$. Moreover, by the argument in the proof of \emph{loc.cit.} (together with Lemma \ref{lem: ord-Pdef} and Proposition \ref{prop: ord-Pdef2}), the isomorphism $\Def_{\overline{\rho}, (\xi)} \xrightarrow{\sim} \Def_{{\rho}_{\xi}}$ induces an injection of functors:
\begin{equation}\label{equ: ord-Pdef2}
 \Def_{\overline{\rho}, \{\overline{\rho}_i\}, (\xi)}^{P-\ord} \hooklongrightarrow \Def_{\rho_{\xi}, \{\rho_{\xi,i}\}}^{P-\ord}.
\end{equation}
For $A\in \Art(E)$, let $A_0$ be an $\co_E$-subalgebra of $A$ such that $A_0$ is finitely generated as $\co_E$-module and $A_0[1/p]\cong A$. The canonical surjection of $E$-algebras $A\twoheadrightarrow E$ induces a surjection of $\co_E$-algebras $A_0\twoheadrightarrow \co_E$. Let $((\rho_A,T_A),T_{A,\bullet},i_A)\in \Def_{\rho_{\xi}, \{\rho_{\xi,i}\}}^{P-\ord}(A)$. As in the proof of \cite[Prop. 2.3.5]{Kis09}, the free $A$-module $T_A$ admits a $\Gal_L$-invariant $A_0$-lattice $T_{A_0}$ such that $T_{A_0}\otimes_{A_0}\co_E\cong \rho_{\xi}^0$. We define an invariant filtration on $T_{A_0}$ by $T_{A_0,i}:=T_{A,i}\cap T_{A_0}$ (inside $T_A$) and it is not difficult to check that $T_{A_0,i}$ is a direct summand of $T_{A_0}$ as $A_0$-module and that $T_{A_0,i}\otimes_{A_0}\co_E\cong T_{\co_E,i}$. Hence $((\rho_A,T_A),T_{A,\bullet},i_A)\in \Def_{\overline{\rho}, \{\overline{\rho}_i\}, (\xi)}^{P-\ord} (A)$ (see \cite[\S~2.3]{Kis09}) which implies (\ref{equ: ord-Pdef2}) is also surjective, and thus an isomorphism.
\end{proof}

\begin{definition}\label{def: ord-strict}
Let $\overline{\rho}$ (resp. $\rho$) be a $P$-ordinary representation of $\Gal_L$ over $k_E$ (resp. $E$) and fix an invariant increasing filtration of the underlying space $T_{k_E}$ (resp. $T_E$) as in Definition \ref{Porddef} leading to representations $\overline{\rho}_i$ (resp. ${\rho}_i$) for $i\in \{1,\cdots ,k\}$ on the graded pieces. The representation $\overline{\rho}$ (resp. $\rho$) is strictly $P$-ordinary if the following conditions are satisfied:
\begin{itemize}
\item $(\overline{\rho}, \{\overline{\rho}_i\})$ satisfies Hypothesis \ref{hypo: ord-Pord1} (resp. $(\rho, \{\rho_i\})$ satisfies Hypothesis \ref{hypo: ord-Pord2})
\item if $\overline{\rho}$ (resp. $\rho$) is isomorphic to a successive extension of $n_i$-dimensional representations $\overline{\rho}_i'$ (resp. $\rho_i'$) for $i=1,\cdots,k$, then $\overline{\rho}_i'\cong \overline{\rho}_i$ (resp. $\rho_i' \cong \rho_i$) for all $i=1, \cdots, k$.
\end{itemize}
\end{definition}

\noindent
In particular, if $\overline{\rho}$ (resp. $\rho$) is strictly $P$-ordinary, there is a {\it unique} invariant increasing filtration on its underlying space as in Definition \ref{Porddef}.

\begin{lemma}\label{lem: ord-strict}
Let $\overline{\rho}$ be a strictly $P$-ordinary representation of $\Gal_L$ over $k_E$, $\xi: R_{\overline{\rho}} \twoheadrightarrow \co_E$ a surjection of local $\co_E$-algebras and $\rho_{\xi}^0$ the deformation of $\overline{\rho}$ over $\co_E$ associated to $\xi$. Assume that $\rho_{\xi}^0$, and thus $\rho_{\xi}:=\rho_{\xi}^0\otimes_{\co_E} E$, are $P$-ordinary.\\
\noindent
(1) The morphism $\xi$ factors through the quotient $R_{\overline{\rho}, \{\overline{\rho}_i\}}^{P-\ord}$ of $R_{\overline{\rho}}$.\\
\noindent
(2) The representation $\rho_{\xi}$ is strictly $P$-ordinary.
\end{lemma}
\begin{proof}
Any choice of filtration as in Definition \ref{Porddef} on the underlying space of $\rho_{\xi}^0$ is such that its reduction modulo $\varpi_E$ gives the above unique filtration on the underlying space of $\overline{\rho}$, from which (1) follows easily. The proof of (2) is by the same argument as for Lemma \ref{lem: ord-Pdef}.
\end{proof}

\noindent
When $\overline{\rho}$ (resp. $\rho$) is strictly $P$-ordinary, by Definition \ref{def: ord-strict} the representations $\overline{\rho}_i$ (resp. ${\rho}_i$) are defined without ambiguity and we then write $R_{\overline{\rho}}^{P-\ord}:=R_{\overline{\rho}, \{\overline{\rho}_i\}}^{P-\ord}$ (resp. $R_{\rho}^{P-\ord}:=R_{\rho, \{\rho_i\}}^{P-\ord}$).

\subsection{Classical local Langlands correspondence}\label{classical}

\noindent
We give a sufficient condition in terms of the (usual) local Langlands correspondence for a $p$-adic Galois representation to be $P$-ordinary. The results of this section will be used in \S\S~\ref{sec: Pord-3}~\&~\ref{gl2ordinarylocalglobal}.\\

\noindent
Let $\rho: \Gal_L \ra \GL_n(E)$ be a potentially semi-stable representation of $\Gal_L$ over $E$ and $L'$ a finite Galois extension of $L$ such that $\rho|_{\Gal_{L'}}$ is semi-stable, following Fontaine we can associate to $\rho$ a Deligne-Fontaine module:
\begin{equation*}
\DF(\rho):=\big((B_{\st}\otimes_{\Q_p} \rho)^{\Gal_{L'}}, \varphi, N, \Gal(L'/L)\big),
\end{equation*}
where $D_{L'}:=(B_{\st}\otimes_{\Q_p} \rho)^{\Gal_{L'}}$ is a finite free $L'_0\otimes_{\Q_p} E$-module of rank $n$, $L_0'$ being the maximal unramified subextension of $L'$ (over $\Q_p$), where the $(\varphi, N)$-action on $D_{L'}$ is induced from the $(\varphi,N)$-action on $B_{\st}$, and where the $\Gal(L'/L)$-action on $D_{L'}$ is the residual action of $\Gal_L$. As in \cite[\S~4]{BS07}, we associate to $\DF(\rho)$ an $n$-dimensional Weil-Deligne representation $\WD(\rho)$ in the following way. By enlarging $E$, we assume $E$ contains all the embeddings of $L'$ (and hence $L_0'$) in $\overline{\Q_p}$. We have thus $L_0'\otimes_{\Q_p} E\cong \prod_{\sigma: L_0'\hookrightarrow E} E$ and therefore an isomorphism:
\begin{equation*}
 D_{L'}\xlongrightarrow{\sim} \prod_{\sigma: L_0'\hookrightarrow E} D_{L',\sigma}
\end{equation*}
where $D_{L',\sigma}:=D_{L'}\otimes_{L_0'\otimes_{\Q_p} E,\sigma\otimes 1}E$. Each $D_{L',\sigma}$ is stable by the $N$-action. Moreover, for $w\in \W_L$ (the Weil group of $L$), we have that $r(w):=\varphi^{-\alpha(w)} \circ \overline{w}$ acts $L_0'\otimes_{\Q_p} E$-linearly on $D_{L'}$ where $\alpha(w)\in [L_0:\Q_p]\Z$ is such that the image of $w$ in $\Gal_{\F_p}$ is equal to $\Frob^{\alpha(w)}$, $\Frob$ being the absolute arithmetic Frobenius, and where $\overline{w}$ denotes the image of $w$ in $\Gal(L'/L)$. We still denote by $r(w)$ the induced map $D_{L',\sigma} \ra D_{L',\sigma}$ for $\sigma:L'_0\hookrightarrow E$, then we denote by $\W(\rho)$ the representation $(D_{L',\sigma}, r)$ of $\W_L$ and by $\WD(\rho):=(\W(\rho), N)$ the Weil-Deligne representation obtained when taking $N$ into account. Both $\W(\rho)$ and $\WD(\rho)$ are independent of the choice of $\sigma$: if we replace $\sigma$ by $\sigma\circ \Frob^{-j}$ for $j\in \Z$ ($\Frob$ being the absolute arithmetic Frobenius on $L'_0$), then $\varphi^{j}: D_{L'} \ra D_{L'}$ induces an isomorphism of Weil-Deligne representations $D_{L',\sigma}\xrightarrow{\sim} D_{L',\sigma\circ \Frob^{-j}}$ (cf. \cite[Lem. 2.2.1.2]{BM02}). In fact, we only make use of $\W(\rho)$ in the sequel.\\

\noindent
Let $\pi$ be the smooth irreducible (hence admissible) representation of $\GL_n(L)$ over $E$ associated to $\W(\rho)^{\sss}$ via the classical local Langlands correspondence normalized so that $\rec(\pi)\cong \W(\rho)^{\sss}$ where $\rec(\pi)$ is as in \cite[Thm. 1.2(a)]{Scho13} and $\W(\rho)^{\sss}$ denotes the semi-simplification of $\W(\rho)$. We assume moreover that for all $\sigma\in \Sigma_L$ the $\sigma$-Hodge-Tate weights $\HT_{\sigma}(\rho)$ of $\rho$ are given by $\HT_{\sigma}(\rho):=\{1-n, 2-n,\cdots, 0\}$. Let $P\subseteq \GL_n$ as in (\ref{equ: ord-LP}) and choose $N_0$ a compact open subgroup of the unipotent radical $N_P(L)$ as in \S~\ref{sec: ord-1}. Recall that we defined a canonical representation $(\pi^{N_0})_{0}$ of $L_P(L)=\prod_{i=1}^k\GL_{n_i}(L)$ in (\ref{equ: ord-algalim}) (see Remark \ref{choicen0}). For $i\in \{1,\cdots, k\}$, we denote by $s_i:=\sum_{j=0}^{i-1} n_j$ where we set $n_0:=0$ (hence $s_1=0$). For a representation $W$ of $\W_L$ and an integer $s$, we set $W(s):=W\otimes |\cdot|^s=W\otimes_E \unr(q_L^{-s})$.

\begin{proposition}\label{prop: ord-cLLO}
For $i=1, \cdots, k$ let $\pi_i$ be a smooth irreducible representation of $\GL_{n_i}(L)$ over $E$. If there is an embedding $\otimes_{i=1}^k \pi_i\hookrightarrow (\pi^{N_0})_{0}$ of smooth representations of $L_P(L)=\prod_{i=1}^k\GL_{n_i}(L)$, then there exist $\rho_i:\Gal_L\ra \GL_{n_i}(E)$ for $i=1, \cdots, k$ such that:
\begin{itemize}
\item $\rho$ is isomorphic to a successive extension of the $\rho_i$ (thus $\rho_i$ is potentially semi-stable for all $i$)
\item $\rec(\pi_i)(s_i)\cong \W(\rho_i)^{\sss}$
\item $\HT_{\sigma}(\rho_i)=\{1-s_{i+1}, \cdots, -s_i\}$ for all $\sigma\in \Sigma_L$.
\end{itemize}
In particular, if $(\pi^{N_0})_{0}\neq 0$, then $\rho$ is $P$-ordinary over $E$ in the sense of Definition \ref{Porddef}.
\end{proposition}
\begin{proof}
The very last assertion easily follows from the others and the finite length of the $L_P(L)$-representation $(\pi^{N_0})_{0}$ (which follows from $(\pi^{N_0})_{0}\subseteq J_P(\pi)(\delta_P^{-1})$, see (\ref{equ: lgOrd-relev}), and the finite length of $J_P(\pi)$). The general idea of the proof below is the following: by classical local Langlands correspondence, we deduce first a ``$P$-filtration'' of the Weil representation $\W(\rho)^{\sss}$, then we show that this filtration actually comes from a filtration of Galois representations.\\

\noindent
(a) First we reduce to the case $k=2$ (i.e. $P$ maximal). Take $P'\supseteq P$ such that the Levi subgroup $L_{P'}$ of $P'$ satisfies $L_{P'}\cong \begin{pmatrix} \GL_{n-n_k} & 0 \\ 0 & \GL_{n_k} \end{pmatrix}$. By the proof of Lemma \ref{ordpart} (see Remark \ref{ordpart0}), we have with the notation as {\it loc.cit.} $(\pi^{N_0})_{0}\cong (((\pi^{N_0'})_{0})^{N_0''})_{0}$. Thus if $\otimes_{i=1}^k\pi_i\hookrightarrow (\pi^{N_0})_{0}$, there exists a smooth irreducible representation $\pi'\otimes \pi_k$ of $L_{P'}(L)$ over $E$ such that $\pi'\otimes\pi_k\hookrightarrow (\pi^{N_0'})_{0}$ and $\otimes_{i=1}^k \pi_i \hookrightarrow ((\pi'\otimes \pi_k)^{N_0''})_0$. Assume the statement holds for $k=2$, we then obtain $\rho'$, $\rho_k$ corresponding to $\pi'$, $\pi_k$ respectively as in the proposition. Applying the same argument with $\rho'$, $\pi'$ instead of $\rho$, $\pi$ and using an easy induction, we deduce the statement for arbitrary $k$.\\

\noindent
(b) Assume now $L_P\cong \begin{pmatrix}\GL_{n_1} & 0 \\ 0 & \GL_{n_2}\end{pmatrix}$. The composition $\pi_1\otimes \pi_2 \hooklongrightarrow (\pi^{N_0})_{0} \hooklongrightarrow J_P(\pi)(\delta_P)$ induces a nonzero, hence surjective, morphism (recall $\overline{P}$ is the opposite parabolic):
\begin{equation*}
\big(\Ind_{\overline{P}(L)}^{\GL_n(L)} \pi_1 \otimes \pi_2\big)^{\infty} \twoheadrightarrow \pi.
\end{equation*}
Let $\W_i:=\rec(\pi_i)$, $i\in \{1,2\}$ be the semi-simple representation of $\W_L$ associated to $\pi$, we have (see \cite[Thm. 1.2(b)]{Scho13}):
\begin{equation*}
\W(\rho)^{\sss}\cong \W_1\oplus \W_2(n_1)
\end{equation*}
with $\W_i\vert_{\W_{L'}}$ being unramified. For $i\in \{1,2\}$ let $\DF_i:=(\sD_i, \varphi, N=0,\Gal(L'/L))$ be the Deligne-Fontaine module associated to $(\W_i, N=0)$ (\cite[Prop.~4.1]{BS07}). Enlarging $E$ if needed, there exists a $\varphi$-submodule $D_1$ of $D_{L'}=(B_{\st}\otimes_{\Q_p} \rho)^{\Gal_{L'}}$ such that the $\varphi^{[L_0':\Q_p]}$-semi-simplifica\-tion of $D_1$ is isomorphic to $\sD_1$ as $\varphi$-modules over $L_0'\otimes_{\Q_p} E$. Indeed, for $\sigma: L_0'\hookrightarrow E$, there exists a $\varphi^{[L_0':\Q_p]}$-submodule $D_{1, \sigma}$ of $D_{L',\sigma}$ such that $D_{1,\sigma}^{\sss}\cong \sD_{1,\sigma}$ (since $\sD_{1,\sigma}$ is a $\varphi^{[L_0':\Q_p]}$-submodule of $D_{L',\sigma}^{\sss}$ and $E$ is sufficiently large). We can then take $D_1$ to be the $\varphi$-submodule of $D_{L'}$ generated by $D_{1,\sigma}$. We will show that $D_1$ is stable by $N$ and by $\Gal(L'/L)$ (hence is a Deligne-Fontaine submodule of $D$) and that the induced filtration on $D_1$ is admissible.\\

\noindent
(c) We first show that we have (where $t_N(\cdot):=\frac{1}{[L'_0:\Q_p]}\val_p(\det_{L'_0}(\varphi^{[L'_0:\Q_p]}))$):
\begin{equation}\label{equ: ord-tnd}
t_N(D_1)=\frac{n_1(n_1-1)}{2}[E:\Q_p].
\end{equation}
From \cite[Thm. 14.1(iv)~\&~Thm. 12.1]{Scho13} (note that $\rec(\pi_1)=\sigma(\pi_1)(\frac{1-n_1}{2})$ with the notation of \emph{loc.cit.}), one easily deduces that $\wedge_E^{n_1}(\W_1(\frac{1-n_1}{2}))$ coincides with the central character $\omega_{\pi_1}$ of $\pi_1$. However, $\omega_{\pi_1}$ is unitary (in the $p$-adic sense) since we are in $(\pi^{N_0})_0$ (see \S~\ref{sec: ord-lalg}) and (\ref{equ: ord-tnd}) easily follows. We equip $D_1\otimes_{L'_0}L'$ with the Hodge filtration $\Fil^\bullet(D_1\otimes_{L'_0}L')$ induced by $D_{L'}\otimes_{L'_0}L'$. Since $\HT_{\sigma}(\rho)=\{1-n, 2-n, \cdots, 0\}$ for all $\sigma\in \Sigma_L$, it is easy to deduce (where $t_H(\cdot\otimes_{L'_0}L'):=\sum_{i\in \Z}\dim_{L'}i\Fil^i(\cdot\otimes_{L'_0}L')/\Fil^{i+1}(\cdot\otimes_{L'_0}L')$):
\begin{equation*}
t_H(D_1\otimes_{L'_0}L')\geq \big(0+1+\cdots + (n_1-1)\big)[E:\Q_p]=\frac{n_1(n_1-1)}{2}[E:\Q_p].
\end{equation*}

\noindent
(d) We show that $D_1$ is stable by the monodromy operator $N$ of $D_{L'}$. Let $\sigma: L_0'\hookrightarrow E$, by the relation $N\varphi=p\varphi N$ and the fact that $\varphi^j$ induces an isomorphism $D_{1,\sigma} \ra D_{1, \sigma\circ \Frob^{-j}}$ for $j\in \Z_{\geq 0}$, it is sufficient to prove that $D_{1,\sigma}$ is stable by $N$. Let $f':=[L'_0:\Q_p]$ and denote by $D_{\sigma}'$ the $(\varphi^{f'}, N)$-submodule of $D_{L',\sigma}$ generated by $D_{1,\sigma}$. Let $D'$ be the $(\varphi, N)$-submodule of $D$ generated by $D_{\sigma}'$, i.e. $D'_{\sigma\circ \Frob^{-j}}=\varphi^j (D'_{\sigma})$ for $j\in \Z_{\geq 0}$.\\
\noindent
\textbf{Claim.} If $D_{\sigma}'\neq D_{1,\sigma}$ then there exists a $(\varphi^{f'}, N)$-submodule $D_{\sigma}''$ of $D_{\sigma}'$ such that:
\begin{equation*}
\dim_E D_{\sigma}''=\dim_E D_{1,\sigma}=n_1\ \ {\rm and}\ \ t_N(D'') < t_N(D_1)
\end{equation*}
where $D''$ is the $(\varphi, N)$-submodule of $D'$ generated by $D''_{\sigma}$.\\
\noindent
We first prove the claim in the case where there is $\alpha\in E^{\times}$ and $m\in \Z_{>0}$ such that the $\varphi^{f'}$-eigenvalues on $D_{\sigma}'$ lie in $\{\alpha, p^{-f'}\alpha, \cdots, p^{-f'm}\alpha\}$ (enlarging $E$ if necessary) and $\alpha$ is an eigenvalue of $\varphi^{f'}$. Since $D_{\sigma}'$ is generated by $D_{1,\sigma}$, we see from $N\varphi=p\varphi N$ that $\alpha$ is also a $\varphi^{f'}$-eigenvalue on $D_{1,\sigma}$. Since $N$ is nilpotent on $D'_{\sigma}$, there exists $s\in \Z_{\geq 0}$ such that $\dim_E \Ker(N^s)\geq n_1$ and $\dim_E\Ker(N^{s-1})<n_1$ as $(\varphi^{f'},N)$-submodule of $D_{\sigma}'$. Consider the short exact sequence:
\begin{equation*}
0 \lra \Ker(N^{s-1}) \lra \Ker(N^s) \xlongrightarrow{N^{s-1}} N^{s-1} (\Ker(N^s))\ra 0.
\end{equation*}
Let $M$ be a $\varphi^{f'}$-submodule of $N^{s-1} (\Ker(N^s))$ of dimension $n_1-\dim_E \Ker(N^{s-1})$ and let $D''_{\sigma}$ be the preimage of $M$ in $\Ker(N^s)$, which is thus a $(\varphi^{f'},N)$-submodule of $D_{\sigma}'$ of dimension $n_1$. Since $D_{\sigma}'\neq D_{1,\sigma}$, we have $\Ker(N^{s-1})\nsubseteq D_{1,\sigma}$ or $D_{1,\sigma}\nsubseteq \Ker(N^s)$ (indeed, otherwise we have $\Ker(N^{s-1})\subseteq D_{1,\sigma}\subseteq \Ker(N^s)$ which implies $N(D_{1,\sigma}) \subseteq \Ker(N^{s-1}) \subseteq D_{1,\sigma}$ hence $D_{1,\sigma}$ stable by $N$ and $D'_{\sigma}=D_{1,\sigma}$). In both cases, by comparing the $\varphi^{f'}$-eigenvalues, it is not difficult to see $t_N(D_1)>t_N(D'')$. The claim in this case follows. In general, we have a decomposition $D_{1,\sigma}\cong \oplus_{j\in J} D_{1,\sigma, j}$ where the $\varphi^{f'}$-eigenvalues on the $D_{1,\sigma, j}$ lie in disjoint finite sets of elements of $E^\times$ of the form $\{\alpha_j, p^{-f'}\alpha_j, \cdots, p^{-f'm'_j}\alpha_j\}$ with $\alpha_j$ an eigenvalue of $\varphi^{f'}$ on $D_{1,\sigma, j}$. Since $D_{\sigma}'$ is generated by $D_{1,\sigma}$, from $N\varphi=p\varphi N$ we have $D_{\sigma}'\cong \oplus_{j\in J} D_{\sigma, j}'$ where $D_{\sigma, j}'$ is generated by $D_{1,\sigma, j}$ and the $\varphi^{f'}$-eigenvalues on $D_{\sigma, j}'$ lie in $\{\alpha_j, p^{-f'}\alpha_j, \cdots, p^{-f'm_j}\alpha_j\}$ for $m_j\geq m'_j$. We put $D_{\sigma,j}'':=D_{1,\sigma,j}$ if $D_{1,\sigma,j}=D_{\sigma,j}'$ and define $D_{\sigma,j}''\subseteq D_{\sigma, j}'$ as above when $D_{1,\sigma,j}\neq D_{\sigma, j}'$. The claim then follows with $D_{\sigma}'':=\oplus_{j\in J} D_{\sigma,j}''$.\\
\noindent
Assume now $D'_{\sigma}\neq D_{1,\sigma}$ and let $D''$ be as in the claim. The same argument as in (c) with the induced Hodge filtration gives then $t_H(D''\otimes_{L'_0}L')\geq \frac{n_1(n_1-1)}{2}[E:\Q_p]>t_N(D'')$, which contradicts the fact that $D_{L'}$ is admissible. So we have $D'_{\sigma}=D_{1,\sigma}$, $D_1=D'$ and these spaces are stable by $N$. By (c) and the fact that $D_{L'}$ is admissible, we deduce:
$$t_H(D_1\otimes_{L'_0}L')=\frac{n_1(n_1-1)}{2} [E:\Q_p]$$
and hence together with (\ref{equ: ord-tnd}) that $D_1$ is a weakly admissible $(\varphi, N)$-submodule of $D_{L'}$.\\

\noindent
(e) For a $\varphi^{f'}$-module $W$ over $E$ and $a\in \bR$, denote by $W_{<a}$ (resp. $W_{\leq a}$) the $E$-vector subspace of $W$ generated by the generalized $\varphi^{f'}$-eigenvectors of eigenvalues $\beta$ satisfying $\val_p(\beta)<a$ (resp. $\val_p(\beta)\leq a$). If $W$ is moreover a $(\varphi, N)$-module over $L_0'\otimes_{\Q_p} E$, it is easy to see that $W_{<a}$ and $W_{\leq a}$ are still $(\varphi, N)$-submodules (over $L_0'\otimes_{\Q_p} E$) of $W$. We now show $(D_{L'})_{<0}=0$ and $D_1=(D_{L'})_{\leq (n_1-1)f'}$. Since $t_H(W)>0$ for any nonzero $E$-vector subspace $W$ of $D_{L'}$ (with the induced filtration) and since $D_{L'}$ is admissible, it follows that $(D_{L'})_{<0}=0$. We show $(D_1)_{\leq (n_1-1)f'}=D_1$ (and hence $D_1=(D_{L'})_{\leq (n_1-1)f'}$). Assume not and let $n_1'<n_1$ such that $\dim_E (D_1)_{\leq (n_1-1)f'}=n_1'f'$ (note that $\dim_E D_1=n_1f'$ and that $ (D_1)_{\leq (n_1-1)f'}$ is free over $L_0'\otimes_{\Q_p} E$ (as is easily checked)). Then we deduce:
\begin{equation*}
t_N\big((D_1)_{\leq (n_1-1)f'}\big) < t_N(D_1)-(n_1-1)(n_1-n_1')[E:\Q_p]=\frac{(n_1-1)(2n'_1-n_1)}{2}[E:\Q_p]
\end{equation*}
where the equality follows from (\ref{equ: ord-tnd}). But we also have (with the induced Hodge filtration and by the same argument as in (c)):
\begin{equation*}
 t_H\big((D_1)_{\leq (n_1-1)f'}\big) \geq (0+1+\cdots + (n_1'-1))[E:\Q_p]=\frac{n_1'(n'_1-1)}{2}[E:\Q_p]
\end{equation*}
which contradicts the fact $D_1$ is weakly admissible (see the end of (d)) since we easily check that $(n_1-1)(2n'_1-n_1)\leq n_1'(n'_1-1)$ for $0\leq n'_1<n_1$.\\

\noindent
(f) Since $\Gal(L'/L)$ commutes with $\varphi$, we see that $D_1=(D_{L'})_{\leq (n_1-1)f'}$ is stable by $\Gal(L'/L)$. Let $\rho_1$ be the continuous representation of $\Gal_L$ over $E$ associated to $D_1$ by the Colmez-Fontaine theorem and $\rho_2:=\rho/\rho_1$. Thus $\W(\rho)^{\sss}\cong \W(\rho_1)^{\sss} \oplus \W(\rho_2)^{\sss}$ and the first and third properties in the statement are then clear. To finish the proof, we only need to show that the $\W_L$-representations $\W(\rho_1)^{\sss}$ and $\W_1$ (see (b)) are isomorphic. Let $\DF_1':=(D_1^{\sss}, \varphi, N=0, \Gal(L'/L))$ be the Deligne-Fontaine module associated to $(\W(\rho_1)^{\sss}, N=0)$ (\cite[Prop.~4.1]{BS07}) where $D_1^{\sss}$ denotes the semi-simplification of $D_1$ for the $\varphi^{f'}$-action, we are reduced to show that $\DF_1'$ and $\DF_1=(\sD_1, \varphi, N=0,\Gal(L'/L))$ (see (b)) are isomorphic (that is, one has to take care of the $\Gal(L'/L)$-action). The natural inclusion $\W_1\hookrightarrow \W(\rho)^{\sss}$ induces an embedding of Deligne-Fontaine modules:
\begin{equation*}
 \DF_1 \hooklongrightarrow \DF:=(D_{L'}^{\sss}, \varphi, N=0, \Gal(L'/L))
\end{equation*} where the latter is isomorphic to the Deligne-Fontaine module associated to $(\W(\rho)^{\sss}, N=0)$ and where $D_{L'}^{\sss}$ denotes the semi-simplification of $D_{L'}$ for the $\varphi^{f'}$-action. Similarly, the inclusion $\W(\rho_1)^{\sss}\hookrightarrow \W(\rho)^{\sss}$ induces an injection $\DF_1'\hooklongrightarrow \DF$. By construction, we also know $\sD_1\cong D_1^{\sss}$ as $\varphi$-module. However, by (e) we have $D_1^{\sss}=(D_{L'}^{\sss})_{\leq (n_1-1)f'}$, thus we also have $\sD_1=(D_{L'}^{\sss})_{\leq (n_1-1)f'}$ since $(D_{L'}^{\sss})_{\leq (n_1-1)f'}$ is only defined in terms of the $\varphi$-action. So both $\DF_1$ and $\DF_1'$ are isomorphic to the Deligne-Fontaine submodule $((D_{L'}^{\sss})_{\leq (n_1-1)f'}, \varphi, N=0, \Gal(L'/L))$ of $\DF$. This concludes the proof.
\end{proof}

\section{Automorphic and $P$-ordinary automorphic representations}\label{ordinaryautomorphic}\label{automor}

\noindent
In this section we start the global theory: we give the global setup, state our local-global compatibility conjecture for $\GL_3(\Q_p)$, and prove several useful results on the $P$-ordinary part of (localized) Banach spaces of $p$-adic automorphic forms on definite unitary groups.

\subsection{Global setup and main conjecture}\label{prelprel}

\noindent
We introduce the global setup and state our main local-global compatibility conjecture for $\GL_3(\Q_p)$.\\

\noindent
We fix field embeddings $\iota_{\infty}: \overline{\Q} \hookrightarrow \bC$, $\iota_p: \overline{\Q} \hookrightarrow \overline{\Q_p}$. We also fix $F^+$ a totally real number field, $F$ a quadratic totally imaginary extension of $F^+$ and $G/F^+$ a unitary group attached to the quadratic extension $F/F^+$ as in \cite[\S~6.2.2]{Bch} such that $G\times_{F^+} F\cong \GL_n$ ($n\geq 2$) and $G(F^+\otimes_{\Q} \bR)$ is compact. For a finite place $v$ of $F^+$ which is totally split in $F$ and $\tilde{v}$ a place of $F$ dividing $v$, we have thus isomorphisms $ i_{G, \widetilde{v}}: G(F^+_v)\xlongrightarrow{\sim} G(F_{\tilde{v}})\xlongrightarrow{\sim} \GL_n(F_{\tilde{v}})$. We let $\Sigma_p$ denote the set of places of $F^+$ dividing $p$ and we assume that each place in $\Sigma_p$ is totally split in $F$. We fix an open compact subgroup $U^p=\prod_{v\nmid p} U_v$ of $G(\bA_{F^+}^{p,\infty})$ and set:
\begin{equation*}
\widehat{S}(U^p,E):=\Big\{f: G(F^+) \setminus G(\bA_{F^+}^{\infty})/U^p\lra E,\ f \text{ is continuous}\Big\}.
\end{equation*}
Since $G(F^+\otimes_{\Q} \bR)$ is compact, $G(F^+)\setminus G(\bA_{F^+}^{\infty})/U^p$ is a profinite set, and we see that $\widehat{S}(U^p,E)$ is a Banach space over $E$ with the norm defined by the (complete) $\co_E$-lattice:
\begin{equation*}
\widehat{S}(U^p,\co_E):=\Big\{f: G(F^+)\setminus G(\bA_{F^+}^{\infty})/U^p \lra \co_E,\ f \text{ is continuous}\Big\}.\end{equation*}
Moreover, $\widehat{S}(U^p,E)$ is equipped with a continuous action of $G(F^+\otimes_{\Q}\Q_p)$ given by $(g'f)(g)=f(gg')$ for $f\in \widehat{S}(U^p,E)$, $g'\in G(F^+\otimes_{\Q} \Q_p)$, $g\in G(\bA_{F^+}^{\infty})$. The lattice $\widehat{S}(U^p,\co_E)$ is obviously stable by this action, so the Banach representation $\widehat{S}(U^p,E)$ of $G(F^+\otimes_{\Q} \Q_p)$ is unitary. Moreover, we know (see e.g. the proof of Lemma \ref{projH}) that $\widehat{S}(U^p,E)$ is admissible. Let $D(U^p)$ be the set of primes $v$ of $F^+$ satisfying:
\begin{itemize}
 \item $v\nmid p$ and $v$ is totally split in $F$
 \item $U_v$ is a maximal compact open subgroup of $G(F^+_v)$.
\end{itemize}Let $\bT(U^p):=\co_E[T_{\tilde{v}}^{(j)}]$ be the commutative polynomial $\co_E$-algebra generated by the formal variables $T_{\tilde{v}}^{(j)}$ where $j\in \{1,\cdots,n\}$ and $\tilde{v}$ is a finite place of $F$ above a finite place $v$ in $D(U^p)$. The $\co_E$-algebra $\bT(U^p)$ acts on $\widehat{S}(U^p,E)$ and $\widehat{S}(U^p, \co_E)$ by making $T_{\tilde{v}}^{(j)}$ act by the double coset operator:
\begin{equation}\label{equ: ord-hecke1}
 T_{\tilde{v}}^{(j)}:=\Big[U_{v} g_v i_{G,\tilde{v}}^{-1}\begin{pmatrix}
  \text{\textbf{1}}_{n-j} & 0 \\ 0 & \text{$\varpi_{\tilde{v}}$ \textbf{1}}_{j}
 \end{pmatrix} g_v^{-1}U_{v}\Big]
\end{equation}
where $\varpi_{\tilde{v}}$ is a uniformizer of $F_{\tilde{v}}$, and where $g_v\in G(F_v^+)$ is such that $i_{G,\tilde{v}}(g_v^{-1} U_v g_v)=\GL_n(\co_{F_{\tilde{v}}})$. This action commutes with that of $G(F^+ \otimes_{\Q} \Q_p)$.\\

\noindent
Recall that the automorphic representations of $G(\bA_{F^+})$ are the irreducible constituents of the $\bC$-vector space of functions $f: G(F^+)\backslash G(\bA_{F^+})\lra \bC$ which are:
\begin{itemize}
 \item $\cC^{\infty}$ when restricted to $G(F^+\otimes_{\Q} \bR)$
 \item locally constant when restricted to $G(\bA_{F^+}^{\infty})$
 \item $G(F^+\otimes_{\Q} \bR)$-finite,
\end{itemize}
where $G(\bA_{F^+})$ acts on this space via right translation. An automorphic representation $\pi$ is isomorphic to $\pi_{\infty}\otimes_{\bC} \pi^{\infty}$ where $\pi_{\infty}=W_{\infty}$ is an irreducible algebraic representation of $(\Res_{F^+/\Q}G)(\bR)=G(F^+\otimes_{\Q} \bR)$ over $\bC$ and $\pi^{\infty}\cong \Hom_{G(F^+\otimes_{\Q}\bR)}(W_{\infty}, \pi)\cong \otimes'_{v}\pi_v$ is an irreducible smooth representation of $G(\bA_{F^+}^{\infty})$. The algebraic representation $W_{\infty}\vert_{(\Res_{F^+/\Q}G)(\Q)}$ is defined over $\overline{\Q}$ via $\iota_{\infty}$ and we denote by $W_p$ its base change to $\overline{\Q_p}$ via $\iota_p$, which is thus an irreducible algebraic representation of $(\Res_{F^+/\Q}G)(\Q_p)=G(F^+\otimes_{\Q} \Q_p)$ over $\overline{\Q_p}$. Via the decomposition $G(F^+\otimes_{\Q} \Q_p)\xrightarrow{\sim} \prod_{v\in \Sigma_p} G(F^+_v)$, one has $W_p\cong \otimes_{v\in \Sigma_p} W_v$ where $W_v$ is an irreducible $\Q_p$-algebraic representation of $G(F^+_v)$ over $\overline{\Q_p}$. One can also prove $\pi^{\infty}$ is defined over a number field via $\iota_{\infty}$ (e.g. see \cite[\S~6.2.3]{Bch}). Denote by $\pi^{\infty,p}:=\otimes'_{v\nmid p} \pi_v$, so that we have $\pi^\infty \cong \pi^{\infty,p}\otimes_{\overline{\Q}} \pi_p$ (seen over $\overline{\Q}$ via $\iota_\infty$), and by $m(\pi)\in \Z_{\geq 1}$ the multiplicity of $\pi$ in the above space of functions $f: G(F^+)\backslash G(\bA_{F^+})\ra \bC$. Denote by $\widehat{S}(U^p,E)^{\lalg}$ the subspace of $\widehat{S}(U^p,E)$ of locally algebraic vectors for the $(\Res_{F^+/\Q}G)(\Q_p)=G(F^+\otimes_{\Q} \Q_p)$-action, which is stable by $\bT(U^p)$. We have an isomorphism which is equivariant under the action of $G(F^+\otimes_{\Q} \Q_p)\times \bT(U^p)$ (see e.g. \cite[Prop. 5.1]{Br13II} and the references in \cite[\S~5]{Br13II}):
\begin{equation}\label{equ: lalgaut}
 \widehat{S}(U^p,E)^{\lalg} \otimes_E \overline{\Q_p} \cong \bigoplus_{\pi}\Big((\pi^{\infty,p})^{U^p} \otimes_{\overline{\Q}}(\pi_p \otimes_{\overline{\Q}} W_p)\Big)^{\oplus m(\pi)}
\end{equation}
where $\pi\cong \pi_{\infty}\otimes_{\overline{\Q}} \pi^{\infty}$ runs through the automorphic representations of $G(\bA_{F^+})$ and $W_p$ is associated to $\pi_{\infty}=W_\infty$ as above, and where $T_{\tilde{v}}^{(j)}\in \bT(U^p)$ acts on $(\pi^{\infty,p})^{U^p}$ by the double coset operator (\ref{equ: ord-hecke1}).\\

\noindent
We now fix a place $\wp$ of $F^+$ above $p$, a place $\widetilde{\wp}$ of $F$ dividing $\wp$ and we set $L:=F^+_\wp\cong F_{\widetilde{\wp}}$. We have thus an isomorphism $i_{G,\widetilde{\wp}}: G(F^+_\wp)\xrightarrow{\sim} \GL_n(L)$. We also fix an irreducible $\Q_p$-algebraic representation $W^\wp$ of $\prod_{v|p, v\neq \wp} G(F^+_v)$ over $E$, that we see as a representation of $G(F^+\otimes_{\Q} \Q_p)$ via $G(F^+\otimes_{\Q} \Q_p)\twoheadrightarrow \prod_{v|p, v\neq \wp} G(F^+_v)$, and a compact open subgroup $U_p^\wp=\prod_{v|p,v\neq \wp} U_v$ of $\prod_{v|p, v\neq \wp}G(F^+_v)$. We put $U^\wp:=U^pU_p^\wp$ and:
$$\widehat{S}(U^\wp,W^\wp):=\big(\widehat{S}(U^p,E)\otimes_E W^\wp\big)^{U_p^\wp}$$
which is an admissible unitary Banach representation of $G(F^+_\wp)$ over $E$ (see the discussion below) which is equipped with a natural action of $\bT(U^p)$ commuting with the action of $G(F^+_\wp)\cong \GL_n(L)$. We have an isomorphism equivariant under the action of $G(F^+\otimes_{\Q} \Q_p)\times \bT(U^p)$:
\begin{equation*}
\widehat{S}(U^p,E)\otimes_E W^{\wp}=\Big\{f: G(F^+) \setminus G(\bA_{F^+}^{\infty})/U^p\lra W^{\wp},\ f \text{ is continuous}\Big\}
\end{equation*}
where $G(F^+\otimes_{\Q} \Q_p)$ acts on the right hand side by $(g_pf)(g)=g_p (f(gg_p))$ for $g_p\in G(F^+\otimes_{\Q} \Q_p)$ and $g\in G(\bA_{F^+})$. We deduce:
\begin{multline}\label{equ: ord-aut}
 \widehat{S}(U^{\wp},W^{\wp})=\Big\{f: G(F^+) \setminus G(\bA_{F^+}^{\infty})/U^p \lra W^{\wp},\ \ f \text{ is continuous and} \\ f(gg_p^{\wp})=(g_p^{\wp})^{-1} (f(g)) \ \text{for all}\ g\in G(\bA_{F^+}^{\infty}) \ \text{and all} \ g_p^{\wp}\in U_p^{\wp}\Big\}.
\end{multline}
Let $\bW^{\wp}$ be an $\co_E$-lattice of $W^{\wp}$ stable by $U_p^{\wp}$, we define $\widehat{S}(U^{\wp}, *)$ by replacing $W^{\wp}$ by $*$ in (\ref{equ: ord-aut}) for $*\in \{\bW^{\wp}, \bW^{\wp}/\varpi_E^s\}$ (where $s\geq 1$). We also define for $*\in \{W^{\wp}, \bW^{\wp}, \bW^{\wp}/\varpi_E^s\}$:
\begin{multline*}
 S(U^{\wp},*):=\Big\{f: G(F^+) \setminus G(\bA_{F^+}^{\infty})/U^p \lra *,\ \ f \text{ is locally constant and} \\ f(gg_p^{\wp})=(g_p^{\wp})^{-1}(f(g))\text{ for all }g\in G(\bA_{F^+}^{\infty})\text{ and all }g_p^{\wp}\in U_p^{\wp}\Big\}.
\end{multline*}
All these spaces are equipped with the action of $\GL_n(L)\times \bT(U^p)$ by right translation on functions for $\GL_n(L)$ and by the double coset operators (\ref{equ: ord-hecke1}) for $\bT(U^p)$. We have moreover $\GL_n(L)\times \bT(U^p)$-equivariant isomorphisms:
\begin{eqnarray}
S(U^{\wp}, \bW^{\wp}/\varpi_E^s) &\cong &\widehat{S}(U^{\wp}, \bW^{\wp}/\varpi_E^s) \ \ \ \ \cong \ \ S(U^{\wp}, \bW^{\wp})/\varpi_E^s \label{equ: ord-modps}\\
\widehat{S}(U^{\wp}, \bW^{\wp}) & \cong& \varprojlim_s S(U^{\wp}, \bW^{\wp}/\varpi_E^s) \label{equ: ord-compl}\\
\widehat{S}(U^{\wp}, W^{\wp}) & \cong &\widehat{S}(U^{\wp}, \ \bW^{\wp}) \otimes_{\co_E} E \nonumber
\\
S(U^{\wp}, W^{\wp}) &\cong & S(U^{\wp}, \bW^{\wp})\otimes_{\co_E} E\ \cong \ \widehat{S}(U^{\wp}, W^{\wp})^{\sm}.\label{equ: ord-sm}
\end{eqnarray}
Finally, for a compact open subgroup $U_{\wp}$ of $\GL_n(L)$, we define for $*\in \{W^{\wp}, \bW^{\wp}, \bW^{\wp}/\varpi_E^s\}$:
\begin{multline*}
S(U^{\wp}U_{\wp},*):=\Big\{f: G(F^+) \setminus G(\bA_{F^+}^{\infty})/U^p U_{\wp}\lra *,\\ f(gg_p^{\wp})=(g_p^{\wp})^{-1}(f(g))\text{ for all }g\in G(\bA_{F^+}^{\infty})\text{ and all }g_p^{\wp}\in U_p^{\wp}\Big\}
\end{multline*}
and we thus have $\varinjlim_{U_{\wp}} S(U^{\wp}U_{\wp},*) = S(U^{\wp}, *)$. We can then easily deduce from (\ref{equ: lalgaut}) a $\GL_n(L)\times \bT(U^p)$-equivariant isomorphism:
\begin{equation}\label{equ: lalgaut2}
\widehat{S}(U^{\wp},W^{\wp})^{\lalg}\otimes_{E} \overline{\Q_p} \cong \bigoplus_{\pi} \big( (\pi^{\infty,p})^{U^p} \otimes_{\overline{\Q}} (\otimes_{v|p, v\neq {\wp}} \pi_v^{U_v})\otimes_{\overline{\Q}} (\pi_{\wp} \otimes_{\overline{\Q}} W_{\wp})\big)^{m(\pi)}
\end{equation}
where $\pi\cong \pi_{\infty} \otimes_{\bC} \pi^{\infty}$ runs through the automorphic representations of $G(\bA_{F_+})$ such that the algebraic representation $W_p$ in (\ref{equ: lalgaut}) associated to $\pi_\infty$ is of the form $W_p\cong W_{\wp}\otimes_{E}(W^{\wp})^{\vee}$ where $(W^{\wp})^{\vee}$ is the dual of $W^{\wp}$ and $W_{\wp}$ is a $\Q_p$-algebraic representation of $ \GL_n(L)$ over $\overline{\Q_p}$.\\

\noindent
Following \cite[\S~3.3]{CHT}, we say that $U^p$ is {\it sufficiently small} if there is a place $v\nmid p$ such that $1$ is the only element of finite order in $U_v$. The following (standard) lemma will be useful.

\begin{lemma}\label{projH}
Assume $U^p$ sufficiently small, then for any $U_p^\wp$, $\bW^{\wp}$ as above and any compact open subgroup $U_{\wp}$ of $G(F^+_{\wp})$ there is an integer $r\geq 1$ such that $\widehat{S}(U^\wp,\bW^\wp)|_{U_{\wp}}$ is isomorphic to $\cC(U_{\wp},\co_E)^{\oplus r}$.
\end{lemma}
\begin{proof}
(a) We first show that for any compact open subgroup $U_p$ of $G(F^+\otimes_{\Q} \Q_p)$ there exist an integer $r'$ such that $\widehat S(U^p, \co_E)\vert_{U_p} \cong C(U_p, \co_E)^{r'}$. Since $U^p$ is sufficiently small, we have $U^pU_p\cap g G(F^+) g^{-1}=\{1\}$ for all $g\in G(\bA_{F^+}^{\infty})$ (the left hand side is a finite group as $G(F^+)$ is discrete in $G(\bA_{F^+}^{\infty})$, then $U^p$ being sufficiently small implies it has to be $\{1\}$). From which we deduce a $U_p$-invariant isomorphism:
\begin{equation}\label{cups}
\amalg_{s} sU_p \buildrel\sim\over\lra G(F^+)\backslash G(\bA_{F^+}^{\infty})/U^p, \ \ sh \longmapsto sh
\end{equation}
where $h\in U_p$ and $s$ runs through a representative set of $G(F^+)\backslash G(\bA_{F^+}^{\infty})/U^pU_p$. Indeed, first (\ref{cups}) is clearly surjective. If $s_1h_1=s_2h_2$ in $G(F^+)\backslash G(\bA_{F^+}^{\infty})/U^p$ (for $h_1,h_2\in U_p$), we have $s_1=s_2=s$, and there exist $g \in G(F^+)$, $u \in U^p$ such that $sh_1=gsh_2u$ in $G(\bA_{F^+}^{\infty})$, which implies $g$ lies in $s^{-1} (U^pU_p) s \cap G(F^+)=\{1\}$, and injectivity follows. From (\ref{cups}), we deduce (a).\\
\noindent
(b) We have $\widehat S(U^{\wp}, \bW^{\wp}) \cong (\widehat S(U^p, \co_E) \otimes_{\co_E} \bW^{\wp})^{U^{\wp}_p}$. Then from (a) we deduce using $U_p=U_{\wp}U^{\wp}_p$:
$$\widehat S(U^{\wp}, \bW^{\wp})|_{U_{\wp}} \cong \cC(U_{\wp}, \co_E)\widehat{\otimes}_{\co_E} [\cC(U_p^{\wp}, \co_E)^{r'} \otimes_{\co_E} \bW^{\wp}]^{U^{\wp}_p}.$$
Since $[\cC(U_p^{\wp}, \co_E)^{r'} \otimes_{\co_E} \bW^{\wp}]^{U^{\wp}_p}$ is easily checked to be a finite free $\co_E$-module, the lemma follows with $r$ the rank of this $\co_E$-module.
\end{proof}

\noindent
Let $\Gal_{F}:={\rm Gal}(\overline F/F)$, ${\rho}: \Gal_{F}\ra \GL_n(E)$ a continuous representation and assume ${\rho}$ is unramified for $v\in D(U^p)$. We associate to ${\rho}$ the unique maximal ideal $\fm_{{\rho}}$ of residue field $E$ of $\bT(U^p)[1/p]$ such that for any $v\in D(U^p)$ and $\tilde{v}$ a place of $F$ above $v$, the characteristic polynomial of ${\rho}(\Frob_{\tilde{v}})$, where $\Frob_{\tilde{v}}$ is a {\it geometric} Frobenius at $\tilde v$, is given by (compare \cite[\S~4.2]{BH}):
\begin{equation}\label{equ: ord-ideal}
X^n+\cdots + (-1)^j (N\tilde{v})^{\frac{j(j-1)}{2}} \theta_{{\rho}}(T_{\tilde{v}}^{(j)}) X^{n-j}+\cdots +(-1)^n(N\tilde{v})^{\frac{n(n-1)}{2}} \theta_{{\rho}}(T_{\tilde{v}}^{(n)})
\end{equation}
where $N\tilde{v}$ is the cardinality of the residue field at $\tilde v$ and $\theta_{{\rho}}: \bT(U^p)[1/p]/\fm_{\rho}\buildrel\sim\over\lra E$. Recall that (by the result of many people) if $\widehat{S}(U^{\wp}, W^{\wp})[\fm_{\rho}]^{\lalg}\ne 0$ then $\rho_{\widetilde{\wp}}$ is in particular de Rham with distinct Hodge-Tate weights. We end this section by our main local-global compatibility conjecture when $n=3$ and $L=\Q_p$. If $\rho_p:\Gal_{\Q_p}\lra \GL_3(E)$ is a semi-stable representation such that $N^2\ne 0$ on $D_{\st}(\rho_p)$, there exists a unique triangulation $\cR_E(\delta_1)-\cR_E(\delta_2)-\cR_E(\delta_3)$ on the $(\varphi,\Gamma)$-module $D_{\rig}(\rho_p)$ (with $\cR_E(\delta_1)$ as unique subobject and $\cR_E(\delta_3)$ as unique quotient). If $(D_{\rig}(\rho_p),(\delta_1,\delta_2,\delta_3))$ is (special) noncritical and if the $(\varphi,\Gamma)$-modules $\cR_E(\delta_1)-\cR_E(\delta_2)$ and $\cR_E(\delta_2)-\cR_E(\delta_3)$ satisfy Hypothesis \ref{hypo: hL-pLL0}, we say that $D_{\rig}(\rho_p)$ is {\it sufficiently generic}. We have then associated to such a $\rho_p$ a finite length locally analytic representation $\Pi(\rho_p)$ at the end of \S~\ref{linvariantgl3} which determines and only depends on $\rho_p$.

\begin{conjecture}\label{THEconjecture}
Assume $n=3$ and $F^+_\wp\cong F_{\widetilde{\wp}}=\Q_p$. Let $\rho: \Gal_F\ra \GL_3(E)$ be a continuous absolutely irreducible representation which is unramified at the places of $D(U^p)$ and such that:
\begin{itemize}
\item $\widehat{S}(U^{\wp}, W^{\wp})[\fm_{\rho}]^{\lalg}\neq 0$
\item $\rho_{\widetilde{\wp}}:={\rho}|_{\Gal_{F_{\widetilde{\wp}}}}$ is semi-stable with $N^2\ne 0$ on $D_{\st}(\rho_{\widetilde{\wp}})$
\item $D_{\rig}(\rho_{\widetilde{\wp}})$ is sufficiently generic.
\end{itemize}
Let $\Pi(\rho_{\widetilde{\wp}})$ be the locally analytic representation of $\GL_3(\Q_p)$ at the very end of \S~\ref{linvariantgl3}, then the following restriction morphism is bijective (recall we have $\Pi(\rho_{\widetilde{\wp}})^{\lalg}=\soc_{\GL_3(\Q_p)}\Pi(\rho_{\widetilde{\wp}})$):
\begin{equation*}
 \Hom_{\GL_3(\Q_p)}\big(\Pi(\rho_{\widetilde{\wp}}),\widehat{S}(U^{\wp}, W^{\wp})[\fm_{\rho}]\big)
 \xlongrightarrow{\sim} \Hom_{\GL_3(\Q_p)}\big(\Pi(\rho_{\widetilde{\wp}})^{\lalg}, \widehat{S}(U^{\wp}, W^{\wp})[\fm_{\rho}]\big).
\end{equation*}
\end{conjecture}

\subsection{Hecke operators}\label{sec: ord-hecke}

\noindent
We give (or recall) the definition of some useful pro-$p$-Hecke algebras and of their localisations.\\

\noindent
We keep the notation of \S~\ref{prelprel}. For $s\in \Z_{>0}$ and a compact open subgroup $U_{\wp}$ of $\GL_n(L)$, we let $\bT(U^{\wp}U_{\wp}, \bW^{\wp}/\varpi_E^s)$ (resp. $\bT(U^{\wp}U_{\wp}, \bW^{\wp})$) be the $\co_E/\varpi_E^s$-subalgebra (resp. $\co_E$-subalgebra) of the endomorphism ring of $S(U^{\wp}U_{\wp}, \bW^{\wp}/\varpi_E^s)$ (resp. $S(U^{\wp}U_{\wp}, \bW^{\wp})$) generated by the operators in $\bT(U^p)$. Since $S(U^{\wp}U_{\wp}, \bW^{\wp}/\varpi_E^s)$ (resp. $S(U^{\wp}U_{\wp}, \bW^{\wp})$) is a finite free $\co_E/\varpi_E^s$-module (resp. $\co_E$-module), $\bT(U^{\wp}U_{\wp}, \bW^{\wp}/\varpi_E^s)$ is a finite $\co_E/\varpi_E^s$ algebra \big(resp. $\bT(U^{\wp}U_{\wp}, \bW^{\wp})$ an $\co_E$-algebra which is finitely free as $\co_E$-module\big). For $s'\leq s$, since we have:
\begin{equation}\label{equ: ord-red}
S(U^{\wp}U_{\wp}, \bW^{\wp}/\varpi_E^s)\otimes_{\co_E/\varpi_E^s} \co_E/\varpi_E^{s'}\cong S(U^{\wp}U_{\wp}, \bW^{\wp}/\varpi_E^{s'}),
\end{equation}
it is easy to deduce $\bT(U^{\wp}U_{\wp}, \bW^{\wp}/\varpi_E^s)\otimes_{\co_E/\varpi_E^s} \co_E/\varpi_{s'}\xlongrightarrow{\sim} \bT(U^{\wp}U_{\wp}, \bW^{\wp}/\varpi_E^{s'})$. From (\ref{equ: ord-modps}), it is also easy to see:
\begin{equation}\label{equ: ord-redT}
\bT(U^{\wp}U_{\wp}, \bW^{\wp})\xlongrightarrow{\sim} \varprojlim_s \bT(U^{\wp}U_{\wp}, \bW^{\wp}/\varpi_E^s).
\end{equation}
For $U_{\wp,2}\subseteq U_{\wp,1}$ an inclusion of compact open subgroups of $\GL_n(L)$, the natural injections:
\begin{equation*}
S(U^{\wp}U_{\wp,1}, \bW^{\wp}/\varpi_E^s) \hooklongrightarrow S(U^{\wp}U_{\wp,2}, \bW^{\wp}/\varpi_E^s)
\text{ and } S(U^{\wp}U_{\wp,1}, \bW^{\wp}) \hooklongrightarrow S(U^{\wp}U_{\wp,2}, \bW^{\wp})
\end{equation*}
induce natural surjections:
\begin{equation*}
\bT(U^{\wp}U_{\wp,2}, \bW^{\wp}/\varpi_E^s) \twoheadlongrightarrow \bT(U^{\wp}U_{\wp,1}, \bW^{\wp}/\varpi_E^s)
\text{ and } \bT(U^{\wp}U_{\wp,2}, \bW^{\wp}) \twoheadlongrightarrow \bT(U^{\wp}U_{\wp,1}, \bW^{\wp}).
\end{equation*}
giving rise to projective systems when $U_{\wp}$ gets smaller. From (\ref{equ: ord-redT}) we deduce isomorphisms:
\begin{equation}\small\label{tildeproj}
\widetilde{\bT}(U^{\wp}):=\varprojlim_s \varprojlim_{U_{\wp}} \bT(U^{\wp}U_{\wp}, \bW^{\wp}/\varpi_E^s)
\cong \varprojlim_{U_{\wp}}\varprojlim_s \bT(U^{\wp}U_{\wp}, \bW^{\wp}/\varpi_E^s) \cong \varprojlim_{U_{\wp}} \bT(U^{\wp}U_{\wp}, \bW^{\wp}).
\end{equation}

\begin{lemma}\label{lem: ord-hecke}
The $\co_E$-algebra $\widetilde{\bT}(U^{\wp})$ is reduced and acts faithfully on $\widehat{S}(U^{\wp}, W^{\wp})$.
\end{lemma}
\begin{proof}
By construction, the algebra $\widetilde{\bT}(U^{\wp})$ acts $\co_E$-linearly and faithfully on $S(U^{\wp}, \bW^{\wp})\cong \varinjlim_{U_{\wp}} S(U^{\wp}U_{\wp}, \bW^{\wp})$. By (\ref{equ: ord-compl}) and (\ref{equ: ord-modps}), this action extends naturally to an $\co_E$-linear faithful action of $\widetilde{\bT}(U^{\wp})$ on $\widehat{S}(U^{\wp}, \bW^{\wp})$ and hence to an $E$-linear faithful action on $\widehat{S}(U^{\wp}, W^{\wp})$. Since the operators in $\bT(U^p)$ acting on $S(U^{\wp}, W^{\wp})$ are semi-simple (which easily follows from (\ref{equ: ord-sm}) and (\ref{equ: lalgaut2})), we deduce $\widetilde{\bT}(U^{\wp})$ is reduced.
\end{proof}

\noindent
To a continuous representation $\overline{\rho}: \Gal_{F}\ra \GL_n(k_E)$ which is unramified for $v\in D(U^p)$, we associate a maximal ideal $\fm_{\overline{\rho}}$ of residue field $k_E$ of $\bT(U^p)$ by the same formula as (\ref{equ: ord-ideal}) replacing $\theta_\rho$ by $\theta_{\overline{\rho}}: \bT(U^p)/\fm_{\overline{\rho}} \buildrel\sim\over\lra k_E$.

\begin{definition}
A maximal ideal $\fm$ of $\bT(U^p)$ is called $(U^{\wp}, \bW^{\wp})$-automorphic if there exist $s$, $U_{\wp}$ as above such that the image of $\fm$ in $\bT(U^{\wp}U_{\wp}, \bW^{\wp}/\varpi_E^s)$ is still a maximal ideal, or equivalently such that the localisation $S(U^{\wp}U_{\wp}, \bW^{\wp}/\varpi_E^s)_{\fm}$ is nonzero. A continuous representation $\overline{\rho}: \Gal_{F}\ra \GL_n(k_E)$ is called $(U^{\wp}, \bW^{\wp})$-automor\-phic if $\fm_{\overline{\rho}}$ is $(U^{\wp}, \bW^{\wp})$-automorphic.
\end{definition}

\begin{lemma}\label{rhodevientm}
There are finitely many $(U^{\wp}, \bW^{\wp})$-automorphic maximal ideals of $\bT(U^p)$.
\end{lemma}
\begin{proof}
By (\ref{equ: ord-red}), $\fm$ is $(U^{\wp}, \bW^{\wp})$-automorphic if and only if the image of $\fm$ is a maximal ideal of $\bT(U^{\wp}U_{\wp}, \bW^{\wp}/\varpi_E)$ for some $U_{\wp}$. Fix $U_{\wp,1}$ a pro-$p$ compact open subgroup of $\GL_n(L)$ which is sufficiently small. Let $(U_p^{\wp})'\subseteq U_p^{\wp}$ such that $(U_p^{\wp})'$ acts trivially on $\bW^{\wp}/\varpi_E$, and let $(U^{\wp})':=U^p(U_p^{\wp})'\subseteq U^{\wp}$. If $\fm$ is $(U^{\wp}, \bW^{\wp})$-automorphic, there exists a compact open subgroup $U_{\wp,2}\subseteq U_{\wp,1}$ (depending on $\fm$), which we can choose to be normal, such that $ S(U^{\wp}U_{\wp,2}, \bW^{\wp}/\varpi_E)_{\fm}\neq 0$, and hence $S((U^{\wp})'U_{\wp,2}, k_E)_{\fm}\neq 0$. We claim that this implies $S((U^{\wp})'U_{\wp,1}, k_E)_{\fm}\neq 0$. We first have a decomposition:
\begin{equation*}
S((U^{\wp})'U_{\wp,2}, k_E) \cong \oplus_{\fm'} S((U^{\wp})'U_{\wp,2}, k_E)_{\fm'}
\end{equation*}
where the sum is over the maximal ideals $\fm'$ of the Artinian ring $\bT((U^{\wp})'U_{\wp,2},k_E)$. Moreover this decomposition is equivariant under the $U_{\wp,1}/U_{\wp,2}$-action. By the proof of \cite[Lem. 3.3.1]{CHT}, $S((U^{\wp})'U_{\wp,2}, k_E) $ is a finite free $k_E[U_{\wp,1}/U_{\wp,2}]$-module, and hence so is $S((U^{\wp})'U_{\wp,2}, k_E)_{\fm'}$ for any $\fm'$ by the above decomposition. Using the trace map $\tr_{U_{\wp,1}/U_{\wp,2}}$ (which is $\bT(U^p)$-equivariant) and the proof of \cite[Lem. 3.3.1]{CHT}, we then deduce that if $S((U^{\wp})'U_{\wp,2}, k_E)_{\fm'}\neq 0$ then $S((U^{\wp})'U_{\wp,1}, k_E)_{\fm'}\neq 0$ (where we identify $\fm'\subseteq \bT((U^{\wp})'U_{\wp,2},k_E)$ with its image in $\bT((U^{\wp})'U_{\wp,1},k_E)$). In particular the image of $\fm\subset \bT(U^p)$ in $\bT((U^{\wp})'U_{\wp,1}, k_E)$ is still a maximal ideal in $\bT((U^{\wp})'U_{\wp,1}, k_E)$. Since $\bT((U^{\wp})'U_{\wp,1},k_E)$ is Artinian, it only has a finite number of maximal ideals, and the lemma follows.
\end{proof}

\noindent
If $\fm$ is $(U^{\wp}, \bW^{\wp})$-automorphic, by (\ref{tildeproj}) we can associate to $\fm$ a maximal ideal (still denoted) $\fm$ of $\widetilde{\bT}(U^{\wp})$ of residue field $k_E$. It then easily follows from Lemma \ref{rhodevientm} and its proof, (\ref{tildeproj}) and from $\bT(U^{\wp}U_{\wp}, \bW^{\wp}/\varpi_E^s)\cong \prod_{\fm}\bT(U^{\wp}U_{\wp}, \bW^{\wp}/\varpi_E^s)_{\fm}$ (where the product is over the $(U^{\wp}, \bW^{\wp})$-automorphic maximal ideals $\fm$ of $\bT(U^p)$), that we have a decomposition $\widetilde{\bT}(U^{\wp})\xlongrightarrow{\sim} \prod_{\fm} \widetilde{\bT}(U^{\wp})_{\fm}$, isomorphisms:
\begin{equation}\label{projtrho}
\widetilde{\bT}(U^{\wp})_{\fm}\cong \varprojlim_s \varprojlim_{U_{\wp}} \bT(U^{\wp}U_{\wp}, \bW^{\wp}/\varpi_E^s)_{\fm}
\cong \varprojlim_{U_{\wp}} \bT(U^{\wp}U_{\wp}, \bW^{\wp})_{\fm}
\end{equation}
and that $\bT(U^{\wp}U_{\wp}, \bW^{\wp}/\varpi_E^s)_{\fm}$ (resp. $\bT(U^{\wp}U_{\wp}, \bW^{\wp})_{\fm}$) is isomorphic to the $\co_E$-subalgebra of the \ endomorphism \ ring \ of \ $S(U^{\wp}U_{\wp}, \bW^{\wp}/\varpi_E^s)_{\fm}$ \ (resp. \ of \ the \ endomorphism \ ring \ of $S(U^{\wp}U_{\wp}, \bW^{\wp})_{\fm}\cong \varprojlim_s S(U^{\wp}U_{\wp}, \bW^{\wp}/\varpi_E^s)_{\fm}$) generated by the operators in $\bT(U^p)$. It is also easy to see that:
\begin{equation}
\widehat{S}(U^{\wp}, \bW^{\wp})_{\fm}\cong \varprojlim_s \varinjlim_{U_{\wp}} S(U^{\wp}U_{\wp}, \bW^{\wp}/\varpi_E^s)_{\fm}
\end{equation}
is a direct summand of $\widehat{S}(U^{\wp}, \bW^{\wp})$ (where the localisation $\widehat{S}(U^{\wp}, \bW^{\wp})_{\fm}$ is with respect to the $\widetilde{\bT}(U^{\wp})$-module structure on $\widehat{S}(U^{\wp}, \bW^{\wp})$, which might be different from the localisation at $\fm$ with respect to the $\bT(U^p)$-module structure). When $\fm=\fm_{\overline{\rho}}$ comes from a continuous $\overline{\rho}: \Gal_{F}\ra \GL_n(k_E)$ as at the beginning of \S~\ref{sec: ord-hecke}, we simply denote by $M_{\overline{\rho}}$ the localisation of a $\widetilde{\bT}(U^{\wp})$-module (resp. of a $\bT(U^p)$-module) $M$ at $\fm_{\overline{\rho}}$. We easily check $\widehat{S}(U^{\wp}, W^{\wp})_{\overline{\rho}}\cong \widehat{S}(U^{\wp}, \bW^{\wp})_{\overline{\rho}}\otimes_{\co_E} E$. The following result is then a consequence of Lemma \ref{lem: ord-hecke} and its proof.

\begin{lemma}
Let $\overline{\rho}$ be $(U^{\wp}, \bW^{\wp})$-automorphic, then the local $\co_E$-algebra $\widetilde{\bT}(U^{\wp})_{\overline{\rho}}$ is reduced and acts faithfully on $\widehat{S}(U^{\wp}, W^{\wp})_{\overline{\rho}}$.
\end{lemma}

\subsection{$P$-ordinary automorphic representations}\label{sec: Pord-3}

\noindent
We relate the space $\Ord_P(S(U^{\wp}, W^{\wp})_{\overline{\rho}})$ to $P$-ordinary Galois representations (\S~\ref{sec: ord-galoisdef}).\\

\noindent
We keep the previous notation. We let $\overline{\rho}: \Gal_{F}\ra \GL_n(k_E)$ be $(U^{\wp}, \bW^{\wp})$-automorphic and \emph{absolutely irreducible}. We fix $P$ a parabolic subgroup of $\GL_n$ as in \S~\ref{sec: ord-galoisdef}. Recall we have from (\ref{equ: ord-ordalg}):
\begin{equation*}
\Ord_P(S(U^{\wp}, \bW^{\wp})_{\overline{\rho}})\cong \varinjlim_{i} \big(S(U^{\wp}, \bW^{\wp})_{\overline{\rho}}^{I_{i,i}}\big)_{\ord}
\end{equation*}
where $(I_{i,i})_i$ is as in \S~\ref{sec: ord-1} with $(I_i)_i$ as in Example \ref{ex: ord-gln}. For any $i\geq 0$, $(S(U^{\wp}, \bW^{\wp})_{\overline{\rho}}^{I_{i,i}})_{\ord}=(S(U^{\wp}I_{i,i}, \bW^{\wp})_{\overline{\rho}})_{\ord}$ is stable by $\bT(U^{\wp})$ (since the action of $\bT(U^{\wp})$ on $S(U^{\wp}, \bW^{\wp})^{I_{i,i}}_{\overline{\rho}}$ commutes with that of $L_P^+$), and we denote by $\bT(U^{\wp}I_{i,i}, \bW^{\wp})_{\overline{\rho}}^{P-\ord}$ the $\co_E$-subalgebra of the endomorphism ring of $(S(U^{\wp}, \bW^{\wp})_{\overline{\rho}}^{I_{i,i}})_{\ord}$ generated by the operators in $\bT(U^p)$. Since:
\begin{equation*}
\big(S(U^{\wp}, \bW^{\wp})_{\overline{\rho}}^{I_{i,i}}\big)_{\ord} \hooklongrightarrow S(U^{\wp}, \bW^{\wp})_{\overline{\rho}}^{I_{i,i}}\cong S(U^{\wp}I_{i,i}, \bW^{\wp})_{\overline{\rho}},
\end{equation*}
we have a natural surjection of local $\co_E$-algebras (finite free over $\co_E$):
\begin{equation*}
  \bT(U^{\wp}I_{i,i}, \bW^{\wp})_{\overline{\rho}} \twoheadlongrightarrow \bT(U^{\wp}I_{i,i}, \bW^{\wp})_{\overline{\rho}}^{P-\ord}.
\end{equation*}
We set:
$$\widetilde{\bT}(U^{\wp})^{P-\ord}_{\overline{\rho}}:=\varprojlim_i \bT(U^{\wp}I_{i,i}, \bW^{\wp})_{\overline{\rho}}^{P-\ord}$$
which is thus easily checked to be a quotient of $\widetilde{\bT}(U^{\wp})_{\overline{\rho}}$ and is also a complete local $\co_E$-algebra of residue field $k_E$. Moreover, as in the proof of Lemma \ref{lem: ord-hecke}, the operators in $\bT(U^p)$ acting on $(S(U^{\wp}, \bW^{\wp})_{\overline{\rho}}^{I_{i,i}})_{\ord}\otimes_{\co_E}E$ are semi-simple (since they are so on $S(U^{\wp}, W^{\wp}$)), and we have as in {\it loc.cit.} the following consequence.

\begin{lemma}\label{lem: Pord-reduce}
The $\co_E$-algebra $\widetilde{\bT}(U^{\wp})_{\overline{\rho}}^{P-\ord}$ is reduced and the natural action of $\widetilde{\bT}(U^{\wp})_{\overline{\rho}}^{P-\ord}$ on $\Ord_p(S(U^{\wp}, \bW^{\wp})_{\overline{\rho}})$ and $\Ord_P(S(U^{\wp}, W^{\wp})_{\overline{\rho}})$ is faithful
\end{lemma}

\noindent
From now on we assume that the compact open subgroup $U^p$ is sufficiently small (see the end of \S~\ref{prelprel}).

\begin{lemma}\label{lem: Pord-comp}
(1) The $\co_E$-module $\Ord_P(S(U^{\wp}, \bW^{\wp})_{\overline{\rho}})$ is dense for the $p$-adic topology in $\Ord_P(\widehat{S}(U^{\wp}, \bW^{\wp})_{\overline{\rho}})$ (see (\ref{ordproj}) for the latter). Consequently, the action of $\widetilde{\bT}(U^{\wp})_{\overline{\rho}}^{P-\ord}$ on $\Ord_P(S(U^{\wp}, \bW^{\wp})_{\overline{\rho}})$ extends to a faithful action on $\Ord_P(\widehat{S}(U^{\wp}, \bW^{\wp})_{\overline{\rho}})$.\\
\noindent
(2) The \ representation \ $\Ord_P(\widehat S(U^{\wp}, \bW^{\wp})_{\overline{\rho}})|_{L_P(\Z_p)}$ \ is \ isomorphic \ to \ a \ direct \ summand \ of $\cC(L_P(\Z_p), \co_E)^{\oplus r}$ for some $r\geq 1$.
\end{lemma}
\begin{proof}
(1) From Lemma \ref{projH} we deduce that there exist $r\geq 1$ and a $\GL_n(\co_L)$-representation $Q$ such that:
\begin{equation}\label{equ: ord-dir}
\widehat{S}(U^{\wp}, \bW^{\wp})_{\overline{\rho}}\vert_{\GL_n(\co_L)} \oplus Q\cong \cC(\GL_n(\co_L), \co_E)^{\oplus r}
\end{equation}
which implies using (\ref{equ: ord-sm}) that $S(U^{\wp}\!, \!\bW^{\wp})_{\overline{\rho}}\vert_{\GL_n(\co_L)}$ is a direct summand of $\cC^{\infty}\!(\GL_n(\co_L), \!\co_E)^{\!\oplus r}$. It \ is \ easy \ to \ see \ that \ the \ condition \ in \ Remark \ \ref{rem: ord-sur}(2) \ is \ satisfied \ with \ $V^0=\cC^{\infty}(\GL_n(\co_L),\co_E)$, which then implies it is also satisfied with $V^0=S(U^{\wp}, \bW^{\wp})_{\overline{\rho}}$. Thus the natural injection from (\ref{equ: ord-alg}):
\begin{equation}\small\label{equ: ord-dense}
\Ord_P (S(U^{\wp}, \bW^{\wp})_{\overline{\rho}})/\varpi_E^s \cong \big(\varinjlim_{i} \big(S(U^{\wp}, \bW^{\wp})_{\overline{\rho}}^{I_{i,i}}\big)_{\ord}\big)/\varpi_E^s \hooklongrightarrow \Ord_P\big(\widehat{S}(U^{\wp}, \bW^{\wp})_{\overline{\rho}}/\varpi_E^s\big)
\end{equation}
is actually an isomorphism for all $s\geq 1$ by the proof of Lemma \ref{lem: ord-alg3}. Then (1) follows (see also Remark \ref{rem: ord-dense}).\\
\noindent
(2) The statement follows from (\ref{equ: ord-dir}) and Corollary \ref{coro: ord-inj}.
\end{proof}

\noindent
We now make the following further hypothesis on $G$ and $F$ {\it till the end of the paper}.

\begin{hypothesis}\label{hypoglobal}
We have either ($p>2$, $n\leq 3$) or ($p>2$, $F/F^+$ is unramified and $G$ is quasi-split at all finite places of $F^+$).
\end{hypothesis}

\noindent
When $n\leq 3$, Rogawski's well-known results (\cite{Ro}) imply that strong base change holds from $G/F^+$ to $\GL_n/F$. When $F/F^+$ is moreover unramified, it also holds by well-known results of Labesse (\cite{Lab}).

\begin{remark}
{\rm It is possible that for $n>3$ the recent results (\cite{Mo}, \cite{KMSW}) now allow to relax (for this paper) some of the assumptions in Hypothesis \ref{hypoglobal}. Note that the main result of the paper will be anyway for $n=3$.}
\end{remark}

\noindent
We now also assume that $U_v$ is maximal in $\GL_n(L)=\GL_n(F_{\tilde{v}})$ for all $v|p$, $v\neq \wp$. Let $S(U^p)$ be the union of $\Sigma_p$ and of the places $v\notin \Sigma_p$ such that $U_v$ is not hyperspecial. Since $\overline{\rho}$ is $(U^{\wp}, \bW^{\wp})$-automorphic, recall we have in particular that $\overline{\rho}$ is unramified outside $S(U^p)$ and $\overline{\rho}^{\vee} \circ c \cong \overline{\rho}\otimes \overline{\varepsilon}^{n-1}$ where $\overline{\rho}^{\vee}$ is the dual of $\overline{\rho}$ and $c$ is the nontrivial element in $\Gal(F/F^+)$. The functor $A\mapsto \rho_A$ of (isomorphism classes of) deformations of $\overline{\rho}$ on the category of local artinian $\co_E$-algebras $A$ of residue field $k_E$ satisfying that $\rho_A$ is unramified outside $S(U^p)$ and that $\rho_A^{\vee}\circ c \cong \rho_A\otimes \varepsilon^{n-1}$ is pro-representable by a complete local noetherian algebra of residue field $k_E$ denoted by $R_{\overline{\rho}, S(U^p)}$. By \cite[Prop. 6.7]{Thor} (which holds under Hypothesis \ref{hypoglobal}, this is the place where $p>2$ is required), for any compact open subgroup $U_{\wp}$ of $\GL_n(L)$, we have a natural surjection of local $\co_E$-algebras $R_{\overline{\rho}, S(U^p)} \twoheadrightarrow \bT(U^{\wp}U_{\wp}, \bW^{\wp})_{\overline{\rho}}$, from which we easily deduce using (\ref{projtrho}) a surjection of local complete $\co_E$-algebras:
\begin{equation}\label{equ: Pord-RT2}
R_{\overline{\rho}, S(U^p)}\twoheadlongrightarrow \widetilde{\bT}(U^{\wp})_{\overline{\rho}}.
\end{equation}
In particular, $\widetilde{\bT}(U^{\wp})_{\overline{\rho}}$ and $\widetilde{\bT}(U^{\wp})_{\overline{\rho}}^{P-\ord}$ are noetherian (local complete) $\co_E$-algebras.

\begin{lemma}
The representation $\Ord_P(\widehat{S}(U^{\wp}, \bW^{\wp})_{\overline{\rho}})$ is a $\varpi_E$-adically admissible representation of $L_P(L)$ over $\widetilde{\bT}(U^{\wp})_{\overline{\rho}}^{P-\ord}$ in the sense of \cite[Def.~3.1.1]{Em4}.
\end{lemma}
\begin{proof}
The lemma follows by the same argument as in the proof of \cite[Lem. 5.3.5]{Em4} with (5.3.3) of \emph{loc.cit.} replaced by the isomorphism (\ref{equ: ord-dense}).
\end{proof}

\noindent
Assume now that $\overline{\rho}_{\widetilde{\wp}}:=\overline{\rho}|_{\Gal_{F_{\widetilde{\wp}}}}$ is strictly $P$-ordinary (cf. Definition \ref{def: ord-strict}) and is isomorphic to a successive extension of $\overline{\rho}_i$ for $i=1,\cdots,k$ with $\overline{\rho}_i: \Gal_{L} \ra \GL_{n_i}(k_E)$ (recall $L\cong F_{\widetilde{\wp}}$). The restriction to $\Gal_{F_{\widetilde{\wp}}}$ gives a natural morphism:
\begin{equation}\label{mapdeform}
R_{\overline{\rho}_{\widetilde{\wp}}} \lra R_{\overline{\rho}, S(U^p)}.
\end{equation}
We fix $\rho: \Gal_F\ra \GL_n(E)$ a continuous representation such that $\rho$ is unramified outside $S(U^p)$ and $\rho^{\vee}\circ c\cong \rho\otimes \varepsilon^{1-n}$. We set $\fp_{\rho}:=\fm_{\rho}\cap \bT(U^p)$, which is a prime ideal of $\bT(U^p)$ (see (\ref{equ: ord-ideal}) for $\fm_{\rho}$), and ${\rho}_{\widetilde{\wp}}:={\rho}|_{\Gal_{F_{\widetilde{\wp}}}}$. We assume $\widehat{S}(U^{\wp}, \bW^{\wp})_{\overline{\rho}}[\fp_{\rho}]\neq 0$, then $\fp_{\rho}$ can also be seen as a prime ideal of $\widetilde{\bT}(U^{\wp})_{\overline{\rho}}$ (using (\ref{tildeproj})). Note that this implies that the mod $p$ semi-simplification of $\rho$ is isomorphic to $\overline{\rho}$ (and is thus irreducible).

\begin{theorem}\label{thm: ord-ordlg}
(1) The action of $R_{\overline{\rho}_{\widetilde{\wp}}}$ on $\Ord_P(\widehat{S}(U^{\wp}, \bW^{\wp})_{\overline{\rho}})$ via (\ref{mapdeform}) and (\ref{equ: Pord-RT2}) factors through $R_{\overline{\rho}_{\widetilde{\wp}}}^{P-\ord}$ (see the very end of \S~\ref{sec: ord-galoisdef}).\\
\noindent
(2) If $\Ord_P(\widehat{S}(U^{\wp}, \bW^{\wp})_{\overline{\rho}}[\fp_{\rho}])\neq 0$ then $\rho_{\widetilde{\wp}}$ is $P$-ordinary.
\end{theorem}
\begin{proof}
(1) Assume first $S(U^{\wp}, \bW^{\wp})_{\overline{\rho}}[\fp_{\rho}]\neq 0$. By (\ref{equ: ord-sm}) and (\ref{equ: lalgaut2}), there is an automorphic representation $\pi$ of $G(\bA_{F_+})$ (with $W_{\wp}$ trivial in (\ref{equ: lalgaut2})) which contributes to:
$$S(U^{\wp}, W^{\wp})_{\overline{\rho}}[\fm_{\rho}]\cong S(U^{\wp}, \bW^{\wp})_{\overline{\rho}}[\fp_{\rho}] \otimes_{\co_E} E.$$
By the local-global compatibility for classical local Langlands correspondence (see e.g. \cite[Thm.~6.5(v)]{Thor} and \cite{Car12} taking into account our various normalisations and note that this uses Hypothesis \ref{hypoglobal} via strong base change), $\rho_{\widetilde{\wp}}$ is potentially semi-stable with $\HT_{\sigma}(\rho_{\widetilde{\wp}})=\{1-n, \cdots, 0\}$ for all $\sigma: L\hookrightarrow E$ and $\rec(\pi_{\wp}) \cong \W(\rho_{\widetilde{\wp}})^{\sss}$ where $\pi_{\wp}$ is the $\wp$-th component of $\pi$ and is viewed as a representation of $\GL_n(L)$ via $i_{G,\widetilde{\wp}}$ (see \S~\ref{classical} for the notation). If $\Ord_P(S(U^{\wp}, W^{\wp})_{\overline{\rho}}[\fm_{\rho}])\neq 0$, then there exists $\pi$ as above such that moreover $\Ord_P(\pi_{\wp})\neq 0$ (since we actually have $S(U^{\wp}, W^{\wp})_{\overline{\rho}}[\fm_{\rho}]\cong \pi_{\wp}^{\oplus r}$ as $\GL_n(L)$-representations for some $r\geq 1$). It follows from Lemma \ref{lem: ord-alg6} and Proposition \ref{prop: ord-cLLO} that $\rho_{\widetilde{\wp}}$ is $P$-ordinary. Denote by $I^{P-\ord}$ the kernel of the natural surjection $R_{\overline{\rho}_{\widetilde{\wp}}}\twoheadrightarrow R_{\overline{\rho}_{\widetilde{\wp}}}^{P-\ord}$, which we also view as an ideal of $\widetilde{\bT}(U^{\wp})_{\overline{\rho}}$ via:
\begin{equation}\label{equ: ord-compo}
R_{\overline{\rho}_{\widetilde{\wp}}} \buildrel(\ref{mapdeform})\over\lra R_{\overline{\rho}, S(U^p)} \buildrel(\ref{equ: Pord-RT2})\over\lra \widetilde{\bT}(U^{\wp})_{\overline{\rho}}.
\end{equation}
Then Lemma \ref{lem: ord-strict} easily implies $I^{P-\ord}\subseteq \fp_{\rho}$, in particular $S(U^{\wp}, \bW^{\wp})_{\overline{\rho}}[\fp_{\rho}]$ is killed by $I^{P-\ord}$, and with (\ref{equ: ord-sm}) and (\ref{equ: lalgaut2}) we deduce that $\Ord_P(S(U^{\wp}, \bW^{\wp})_{\overline{\rho}})$ is also killed by $I^{P-\ord}$. By Lemma \ref{lem: Pord-comp}(1), $\Ord_P(S(U^{\wp}, \bW^{\wp})_{\overline{\rho}})$ is dense in $\Ord_P(\widehat{S}(U^{\wp}, \bW^{\wp})_{\overline{\rho}})$ for the $\varpi_E$-adic topology. We deduce then:
$$I^{P-\ord}\Ord_P\big(\widehat{S}(U^{\wp}, \bW^{\wp})_{\overline{\rho}}\big)\subset \cap_{i\in \Z_{\geq 0}} \varpi_E^i \Ord_P\big(\widehat{S}(U^{\wp}, \bW^{\wp})_{\overline{\rho}}\big)=0$$
and (1) follows.\\
\noindent
(2) Let $\fp_{\rho_{\widetilde{\wp}}}$ be the prime ideal of $R_{\overline{\rho}_{\widetilde{\wp}}}$ attached to $\rho_{\widetilde{\wp}}$, which is just the preimage of $\fp_{\rho}$ via (\ref{equ: ord-compo}), and $\fm_{\rho_{\widetilde{\wp}}}:=\fp_{\rho_{\widetilde{\wp}}}[1/p]$, which is a maximal ideal of $R_{\overline{\rho}_{\widetilde{\wp}}}[1/p]$. If $\Ord_P(\widehat{S}(U^{\wp}, \bW^{\wp})_{\overline{\rho}}[\fp_{\rho}])\neq 0$ then we have $I^{P-\ord}[1/p]\subseteq \fm_{\rho_{\widetilde{\wp}}}$, since otherwise $1\in \fm_{\rho_{\widetilde{\wp}}}+I^{P-\ord}[1/p]$ annihilates $\Ord_P(\widehat{S}(U^{\wp}, W^{\wp})_{\overline{\rho}}[\fm_{\rho}])=\Ord_P(\widehat{S}(U^{\wp}, \bW^{\wp})_{\overline{\rho}}[\fp_{\rho}])\otimes_{\co_E} E$ by the first part. From the discussion above Proposition \ref{prop: ord-Pdef2}, we obtain that $\rho_{\widetilde{\wp}}$ is $P$-ordinary.
\end{proof}

\noindent
By Theorem \ref{thm: ord-ordlg}(1) and the last part in Lemma \ref{lem: Pord-comp}(1), the surjection $\widetilde{\bT}(U^{\wp})_{\overline{\rho}}\!\twoheadrightarrow \!\widetilde{\bT}(U^{\wp})_{\overline{\rho}}^{P-\ord}$ factors through:
\begin{equation*}
\widetilde{\bT}(U^{\wp})_{\overline{\rho}} \otimes_{R_{\overline{\rho}_{\widetilde{\wp}}}} R_{\overline{\rho}_{\widetilde{\wp}}}^{P-\ord} \twoheadlongrightarrow \widetilde{\bT}(U^{\wp})_{\overline{\rho}}^{P-\ord}.
\end{equation*}
In particular, we have natural morphisms of local complete noetherian $\co_E$-algebras of residue field $k_E$:
\begin{equation}\label{equ: ord-RT}
\omega :\ \widehat{\bigotimes}_{i=1, \cdots, k} R_{\overline{\rho}_i} \xlongrightarrow{(\ref{equ: ord-diag})} R_{\overline{\rho}_{\widetilde{\wp}}}^{P-\ord} \lra \widetilde{\bT}(U^{\wp})_{\overline{\rho}}^{P-\ord}.
\end{equation}
We end this section by the following proposition.

\begin{proposition}\label{prop: lg1-semisimp}
Let $\pi_{\infty}$ be a smooth admissible representation of $L_P(L)$, $\lambda$ be a dominant weight as in the beginning of \S~\ref{sec: ord-lalg} (for $G=\GL_n$ and $P$ as above), $x$ be a closed point of $\Spec (\widetilde{\bT}(U^{\wp})_{\overline{\rho}}^{P-\ord}[1/p])$, and $\fm_x$ be the associated maximal ideal. Then any $L_P(L)$-equivariant morphism:
\begin{equation*}
\pi_{\infty}\otimes_E L_P(\lambda) \lra \Ord_P\big(\widehat{S}(U^{\wp}, W^{\wp})\big)\{\fm_x\}
\end{equation*}
has image in $\Ord_P\big(\widehat{S}(U^{\wp}, W^{\wp})^{\lalg}\big)[\fm_x]$.
\end{proposition}
\begin{proof}
Replacing $\pi_{\infty}\otimes_E L_P(\lambda)$ by its image, we can assume the morphism is injective. From Proposition \ref{prop: ord-adj} we deduce that the image is in $\Ord_P(\widehat{S}(U^{\wp}, W^{\wp})^{\lalg})$, hence also in $\Ord_P(\widehat{S}(U^{\wp}, W^{\wp})^{\lalg})\{\fm_x\}$. From (\ref{equ: lalgaut2}) it is easy to check that $\Ord_P(\widehat{S}(U^{\wp}, W^{\wp})^{\lalg})[\fm_x]=\Ord_P(\widehat{S}(U^{\wp}, W^{\wp})^{\lalg})\{\fm_x\}$, whence the result.
\end{proof}

\section{$\cL$-invariants, $\GL_2(\Q_p)$-ordinary families and local-global compatibility}\label{maintheroemlocalglobal}

\noindent
We now assume that the field $L=F^+_\wp\cong F_{\widetilde{\wp}}$ in \S~\ref{prelprel} is $\Q_p$ and study $\Ord_P(\widehat{S}(U^{\wp}, W^{\wp})_{\overline{\rho}})$ when the factors in the Levi $L_P$ of the parabolic subgroup $P$ are either $\GL_1$ or $\GL_2$. We derive several local-global compatibility results in this case. In particular we prove Conjecture \ref{THEconjecture} when $\HT(\rho_{\widetilde{\wp}})=\{k_1, k_1-1, k_1-2\}$ for some integer $k_1$ (under mild genericity assumptions).

\subsection{$\GL_2(\Q_p)$-ordinary families and local-global compatibility}\label{gl2ordinarylocalglobal}

\noindent
When the factors of the $L_P$ are either $\GL_1$ or $\GL_2$ we prove local-global compatibility results for the $L_P(\Q_p)$-representation $\Ord_P(\widehat{S}(U^{\wp}, \bW^{\wp})_{\overline{\rho}})$ by generalizing Emerton's method (\cite{Em4}).\\

\subsubsection{Dominant algebraic vectors}\label{domialgebraic}

\noindent
In this section, which is purely local, we prove density results of subspaces of algebraic functions.\\

\noindent
We fix $H$ a connected reductive algebraic group over $\Z_p$ and denote by $A$ the finitely generated $\Z_p$-algebra which represents $H$. For any $f\in A$, the natural map $\Hom_{\Z_p-\alg}(A, \Z_p)=H(\Z_p) \ra \Z_p, \ z \mapsto z(f)$ lies in $\cC(H(\Z_p), \Z_p)$ and induces an $E$-linear morphism $A\otimes_{\Z_p} E\rightarrow \cC(H(\Z_p), E)$. We denote by $\cC^{\alg}(H(\Z_p), E)$ its image, which is called the vector space of algebraic functions on the compact group $H(\Z_p)$. By \cite[Lem. A.1]{Pas11}, $\cC^{\alg}(H(\Z_p), E)$ is dense in the Banach space $\cC(H(\Z_p), E)$. For $f\in \cC(H(\Z_p), E)$, we set $\us(f):=\inf_{z\in H(\Z_p)} \val_p(f(z))$ and note that the associated norm gives the Banach topology on $\cC(H(\Z_p), E)$. Now we let $H=\GL_r$, $r\geq 1$. By \cite[Prop. A.3]{Pas11} we have an isomorphism:
\begin{equation}\label{isoalgebraic}
\cC^{\alg}(\GL_r(\Z_p), E)\cong \bigoplus_{\sigma} \Hom_{\GL_r(\Z_p)}(\sigma, \cC(\GL_r(\Z_p), E)) \otimes_E \sigma
\end{equation}
where $\sigma$ runs through the irreducible algebraic representations of $\GL_r$ over $E$ and where $\Hom_{\GL_r(\Z_p)}(\sigma, \cC(\GL_r(\Z_p), E))$ denotes the $E$-linear $\GL_r(\Z_p)$-equivariant morphisms with $\GL_r(\Z_p)$ acting on $\cC(\GL_r(\Z_p), E)$ by the usual right translation on functions. Recall there exists a one-to-one correspondence between the integral dominant weights $\lambda=(\lambda_1, \cdots, \lambda_r)$ for $\GL_r$ with respect to the Borel subgroup of upper triangular matrices, i.e. such that $\lambda_1\geq \lambda_2 \geq \cdots \geq \lambda_r$, and the irreducible algebraic representations $L(\lambda)$ of $\GL_r$. For $a\in \Z$, we put:
\begin{equation}\label{smallera}
\cC^{\alg}_{\leq a}(\GL_r(\Z_p),E):=\bigoplus_{\substack{\lambda=(\lambda_1,\cdots, \lambda_n)\\ \lambda_1\leq a}} \Hom_{\GL_r(\Z_p)}\big(L(\lambda), \cC(\GL_r(\Z_p), E)\big) \otimes_E L(\lambda).
\end{equation}

\begin{lemma}\label{lem: lg1-dense}
For any $a\in \Z$, the vector space $\cC^{\alg}_{\leq a}(\GL_r(\Z_p), E)$ is dense in $\cC(\GL_r(\Z_p), E)$.
\end{lemma}
\begin{proof}
We first prove the lemma for $r=1$, in which case we have by (\ref{isoalgebraic}) (with obvious notation) $\cC^{\alg}(\Z_p^\times, E)\cong \oplus_{j\in \Z}Ex^j$. Let $W$ be the closure of $\oplus_{j\leq a} E x^j$, we have to prove $x^j\in W$ for any $j\in \Z$. It is enough to prove that, for any $j\in \Z$ and $M>0$, there exists $j'\leq a$ such that $\us(x^{j'}-x^{j})\geq M$. If we consider $j':=j-(p-1)p^{M'}$ with $M'>M$ sufficiently large so that $j'\leq a$, then we indeed have $\val_p(x^{j'}-x^j)=\val_p(x^{(p-1)p^{M'}}-1)>M$ for any $z\in \Z_p^{\times}$. The case $r=1$ follows.\\
\noindent
For general $r$, denote by $\iota_{11}: \Z_p^{\times}\hookrightarrow \GL_r(\Z_p)$, $u\mapsto \diag(u, 1, \cdots, 1)$ and consider the induced map $\SL_r(\Z_p)\times \Z_p^{\times} \ra \GL_r(\Z_p)$, $(u, v)\mapsto u\iota_{11}(v)$. This map is a homeomorphism and thus induces an isomorphism:
\begin{equation*}
h: \cC(\GL_r(\Z_p), E) \xlongrightarrow{\sim} \cC(\SL_r(\Z_p)\times \Z_p^{\times}, E) \cong \cC(\SL_r(\Z_p),E)\widehat{\otimes}_E \cC(\Z_p^{\times}, E).
\end{equation*}
For a dominant weight $\lambda=(\lambda_1, \cdots, \lambda_r)$ as above, let $L(\lambda)_0:=L(\lambda)|_{\SL_r(\Z_p)}$. We claim that $h\vert_{\cC^{\alg}(\GL_r(\Z_p),E)}$ induces an isomorphism via (\ref{isoalgebraic}):
\begin{multline}\label{equ:lg1-dec}
\Hom_{\GL_r(\Z_p)}\big(L(\lambda), \cC(\GL_r(\Z_p), E)\big)\otimes_E L(\lambda)\\
\xlongrightarrow{\sim} \big( \Hom_{\SL_r(\Z_p)}\big(L(\lambda)_0, \cC(\SL_r(\Z_p),E)\big)\otimes_E L(\lambda)_0\big) \otimes_E Ex^{\lambda_1}.
\end{multline}
Indeed, we have a natural commutative diagram (induced by the restriction map):
\begin{equation*}
 \begin{CD}
  \Hom_{\GL_r(\Z_p)}\big(L(\lambda), \cC(\GL_r(\Z_p), E)\big)\otimes_E L(\lambda) @>>> \cC(\GL_r(\Z_p),E) \\
  @VVV @VVV \\
  \Hom_{\SL_r(\Z_p)}\big(L(\lambda)_0, \cC(\SL_r(\Z_p), E)\big)\otimes_E L(\lambda)_0 @>>> \cC(\SL_r(\Z_p),E)
 \end{CD}
\end{equation*}
where the horizontal maps are the evaluation maps and are injective by (\ref{isoalgebraic}). The morphism $\cC(\GL_r(\Z_p),E)\ra \cC(\Z_p^{\times}, E)$ induced by $\iota_{11}$ is easily checked to send (via (\ref{isoalgebraic})) $\Hom_{\GL_r(\Z_p)}(L(\lambda), \cC(\GL_r(\Z_p), E))\otimes_E L(\lambda)$ (on)to $E x^{\lambda_1}$. We thus obtain the morphism in (\ref{equ:lg1-dec}), which is moreover injective since $h$ is. Since we have from the proof of \cite[Prop. A.3]{Pas11}:
\begin{multline}\label{dimfinie}
\dim_E \Hom_{\GL_r(\Z_p)}\big(L(\lambda), \cC(\GL_r(\Z_p), E)\big)\\
=\dim_E \Hom_{\SL_r(\Z_p)}\big(L(\lambda)_0, \cC(\SL_r(\Z_p), E)\big)=\dim_E L(\lambda),
\end{multline}
we deduce that (\ref{equ:lg1-dec}) is an isomorphism. The isomorphism $h$ then induces a bijection:
\begin{equation*}
\cC^{\alg}_{\leq a}(\GL_r(\Z_p), E) \xlongrightarrow{\sim} \cC^{\alg}(\SL_r(\Z_p), E) \otimes_E \cC^{\alg}_{\leq a}(\Z_p^{\times}, E).
\end{equation*}
Since $\cC^{\alg}_{\leq a}(\Z_p^{\times}, E)$ is dense in $\cC(\Z_p^{\times}, E)$ and $\cC^{\alg}(\SL_r(\Z_p), E)$ is dense in $\cC(\SL_r(\Z_p),E)$, we deduce that $\cC^{\alg}(\SL_r(\Z_p), E) \otimes_E \cC^{\alg}_{\leq a}(\Z_p^{\times}, E)$ is dense in $\cC(\SL_r(\Z_p),E)\widehat{\otimes}_E \cC(\Z_p^{\times}, E)$, that is $\cC^{\alg}_{\leq a}(\GL_r(\Z_p), E)$ is dense in $\cC(\GL_r(\Z_p), E)$.
\end{proof}

\noindent
We fix $P$ a parabolic subgroup of $\GL_n$ as in \S~\ref{sec: ord-galoisdef} (or \S~\ref{sec: Pord-3}) with $L_P$ as in (\ref{equ: ord-LP}). We have in particular:
\begin{equation*}
\cC(L_P(\Z_p), E)\cong \widehat{\bigotimes}_{i=1,\cdots, k}\cC(\GL_{n_i}(\Z_p), E)\ {\rm and}\ \cC^{\alg}(L_P(\Z_p), E) \cong \bigotimes_{i=1, \cdots, k} \cC^{\alg}(\GL_{n_i}(\Z_p), E).
\end{equation*}
For $i\in \{1,\cdots,k\}$ we define $s_i:=\sum_{j=0}^{i-1} n_j$ (with $n_0:=0$) as in \S~\ref{classical} and set:
\begin{equation}\label{equ: lg1-Lp+}
 \cC^{\alg}_+(L_P(\Z_p), E):=\bigoplus_{\substack{(\lambda_1, \cdots, \lambda_n)\in \Z^{n}\\ \lambda_1\geq \lambda_2\geq \cdots \geq \lambda_n}} \Big(\bigotimes_{i=1,\cdots, k} \cC^{\ul{\lambda}_i-\alg}(\GL_{n_i}(\Z_p),E)\Big)
\end{equation}
where $\ul{\lambda}_i:=(\lambda_{s_{i}+1}, \cdots, \lambda_{s_{i+1}})$ and:
\begin{equation*}
\cC^{\ul{\lambda}_i-\alg}(\GL_{n_i}(\Z_p), E):= \Hom_{\GL_{n_i}(\Z_p)}\big(L(\ul{\lambda}_i), \cC(\GL_{n_i}(\Z_p), E)\big)\otimes_EL(\ul{\lambda}_i).
\end{equation*}
We define the subspace $\cC^{\alg}_{++}(L_P(\Z_p), E)$ of $\cC^{\alg}_+(L_P(\Z_p),E)$ in the same way but taking in (\ref{equ: lg1-Lp+}) the direct sum only over those (dominant) $\lambda$ such that $\lambda_{s_i}>\lambda_{s_i+1}$ for $i=2,\cdots, k$. We call vectors in $ \cC^{\alg}_+(L_P(\Z_p), E)$ dominant $L_P(\Z_p)$-algebraic vectors.

\begin{proposition}\label{prop: lg1-dense}
The \ vector \ spaces \ $\cC^{\alg}_{++}(L_P(\Z_p), E)$ \ and \ $\cC^{\alg}_+(L_P(\Z_p),E)$ \ are \ dense \ in $\cC(L_P(\Z_p), E)$.
\end{proposition}
\begin{proof}
It is enough to prove the result for the first one. Using an easy induction argument, we can reduce to the case where $k=2$. In this case, we have (see (\ref{smallera})):
\begin{equation*}
\cC^{\alg}_{++}(L_P(\Z_p), E) \cong \bigoplus_{\substack{\ul{\lambda}_1=(\lambda_1, \cdots, \lambda_{n_1})\\ \lambda_1\geq\cdots \geq \lambda_{n_1}}}\Big( \cC^{\ul{\lambda}_1-\alg}(\GL_{n_1}(\Z_p), E)\otimes_E \cC^{\alg}_{\leq \lambda_{n_1}-1}(\GL_{n_2}(\Z_p),E)\Big).
\end{equation*}
From (\ref{dimfinie}) we have that $\dim_E\cC^{\ul{\lambda}_1-\alg}(\GL_{n_1}(\Z_p), E)<+\infty$, which implies that $F_{\ul{\lambda}_1}:=\cC^{\ul{\lambda}_1-\alg}(\GL_{n_1}(\Z_p), E)\otimes_E \cC(\GL_{n_2}(\Z_p),E)$ is a Banach space. From Lemma \ \ref{lem: lg1-dense} we have that $\cC^{\ul{\lambda}_1-\alg}(\GL_{n_1}(\Z_p), E)\otimes_E \cC^{\alg}_{\leq \lambda_{n_1}-1}(\GL_{n_2}(\Z_p),E)$ is dense in $F_{\ul{\lambda}_1}$. We deduce that the closure of $\cC^{\alg}_+(L_P(\Z_p), E)$ in $\cC(L_P(\Z_p), E)$ contains $\oplus_{\ul{\lambda}_1}F_{\ul{\lambda}_1}\cong \cC^{\alg}(\GL_{n_1}(\Z_p), E)\otimes_E\cC(\GL_{n_2}(\Z_p), E)$. But $\cC^{\alg}(\GL_{n_1}(\Z_p), E)$ is dense in $\cC(\GL_{n_1}(\Z_p), E)$, hence $\sum_{\ul{\lambda}_1}F_{\ul{\lambda}_1}$ is dense in $\cC(L_P(\Z_p), E)$ and the lemma follows.
\end{proof}

\noindent
Let $V$ be an admissible continuous Banach representation of $L_P(\Q_p)$ over $E$ and put:
\begin{equation}\label{equ: lg1-alg}
V^{L_P(\Z_p)-\alg}:=\bigoplus_{\sigma} \Hom_{L_P(\Z_p)}(\sigma, V)\otimes_E \sigma\\
\cong \bigoplus_{\sigma} (V\otimes_E \sigma^\vee)^{L_P(\Z_p)} \otimes_E \sigma
\end{equation}
where $\sigma$ runs through the irreducible algebraic representations of $L_P$ and $\sigma^\vee$ is the dual of $\sigma$. By \cite[Prop. 4.2.4]{Em04}, the evaluation map induces a natural injection $V^{L_P(\Z_p)-\alg}\hookrightarrow V$. We denote by $V^{L_P(\Z_p)-\alg}_+$ (resp. $V^{L_P(\Z_p)-\alg}_{++}$) the subspace of $V^{L_P(\Z_p)-\alg}$ defined as in (\ref{equ: lg1-alg}) but taking the direct sum over those irreducible algebraic representations of $L_P$ of highest weight $(\lambda_1, \cdots, \lambda_n)$ such that $\lambda_1\geq \cdots\geq \lambda_n$ (resp. such that $\lambda_1\geq \cdots \geq \lambda_n$ and $\lambda_{s_i}>\lambda_{s_i+1}$ for $i=2,\cdots, k$). If $W$ is a closed subrepresentation of $V$, one easily checks that $W^{L_P(\Z_p)-\alg}_*\cong W\cap V^{L_P(\Z_p)-\alg}_*$ with $*\in \{\emptyset, +, ++\}$.

\begin{corollary}\label{coro: lg1-dens}
Assume that $V|_{L_P(\Z_p)}$ is isomorphic to a direct summand of $\cC(L_P(\Z_p), E)^{\oplus r}$ for some $r\geq 1$. Then $V^{L_P(\Z_p)-\alg}_*$ is dense in $V$ for $*\in \{\emptyset, +, ++\}$.
\end{corollary}
\begin{proof}
If $V_1$, $V_2$ are two locally convex $E$-vector spaces and $X_i\subseteq V_i$, $i=1,2$ two $E$-vector subspaces, then $X_1\oplus X_2$ is dense in $V_1\oplus V_2$ (with the direct sum topology) if and only if $X_i$ is dense in $V_i$ for $i=1,2$. The result follows then from Proposition \ref{prop: lg1-dense} together with $(V_1\oplus V_2)^{L_P(\Z_p)-\alg}_*$ = $(V_1)^{L_P(\Z_p)-\alg}_*\oplus (V_2)^{L_P(\Z_p)-\alg}_*$ for $*\in \{\emptyset, +, ++\}$.
\end{proof}

\subsubsection{Benign points}\label{sec: lg1-bpts}

\noindent
We define benign points of $\Spec \widetilde{\bT}(U^{\wp})_{\overline{\rho}}^{P-\ord}[1/p]$ and prove several results on them.\\

\noindent
We keep the previous notation. We also keep all the notation and assumption of \S~\ref{sec: Pord-3} with $L=\Q_p$ (in particular $U^p$ is sufficiently small, $U_v$ is maximal for $v|p$, $v\neq \wp$, and we assume Hypothesis \ref{hypoglobal}). We denote by $B$ the subgroup of upper triangular matrices in $\GL_n$ and by $T$ the torus of diagonal matrices. We assume moreover $n_i\leq 2$ for all $i=1,\cdots, k$ (though many results in this section hold more generally). For $x$ a closed point of $\Spec \widetilde{\bT}(U^{\wp})_{\overline{\rho}}^{P-\ord}[1/p]$, we denote by $\fm_x$ the associated maximal ideal, $k(x)$ the residue field (a finite extension of $E$) and by $\fp_x:=\fm_x\cap \widetilde{\bT}(U^{\wp})_{\overline{\rho}}^{P-\ord}$ (a prime ideal). We also denote by $\fm_x$ (resp. $\fp_x$) the corresponding maximal ideal of $\widetilde{\bT}(U^{\wp})_{\overline{\rho}}[1/p]$ (resp. the corresponding prime ideal of $\widetilde{\bT}(U^{\wp})_{\overline{\rho}}$). We easily deduce from the left exactness of $\Ord_P$ (\cite[Prop.~3.2.4]{EOrd1}) an $L_P(\Q_p)$-equivariant isomorphism:
\begin{equation*}
\Ord_P\big(\widehat{S}(U^{\wp}, \bW^{\wp})_{\overline{\rho}}[\fp_x]\big) \cong \Ord_P\big(\widehat{S}(U^{\wp}, \bW^{\wp})_{\overline{\rho}}\big)[\fp_x].
\end{equation*}
and we recall that $\Ord_P(\widehat{S}(U^{\wp}, \bW^{\wp})_{\overline{\rho}}[\fp_x])$ is an invariant lattice in the admissible unitary $L_p(\Q_p)$-representation $\Ord_P(\widehat{S}(U^{\wp}, W^{\wp})_{\overline{\rho}}[\fm_x])$. We denote by:
$$\rho_x: \Gal_F \lra \GL_n(R_{\overline{\rho}, S(U^p)})\lra \GL_n(\widetilde{\bT}(U^{\wp})_{\overline{\rho}}^{P-\ord})\lra \GL_n(k(x))$$
the continuous representation attached to $x$ and set $\rho_{x, \widetilde{\wp}}:=\rho_{x}|_{\Gal_{F_{\widetilde{\wp}}}}$. We also denote by $x_i$ for $i\in \{1,\cdots,k\}$ the associated point of $\Spec R_{\overline{\rho}_i}[1/p]$ via (\ref{equ: ord-RT}) and $\rho_{x_i}: \Gal_{\Q_p} \ra \GL_{n_i}(k(x))$ the attached representation. Thus $\rho_{x, \widetilde{\wp}}$ is a successive extension of the $\rho_{x_i}$ for $i=1,\cdots,k$ and is strictly $P$-ordinary by Lemma \ref{lem: ord-strict} (applied with $E=k(x)$). In particular each $\rho_{x_i}$ is indecomposable by Hypothesis \ref{hypo: ord-Pord2}.

\begin{definition}
A closed point $x\in \Spec \widetilde{\bT}(U^{\wp})_{\overline{\rho}}^{P-\ord}[1/p]$ is benign if:
\begin{equation*}
\Ord_P\big(\widehat{S}(U^{\wp}, W^{\wp}\big)_{\overline{\rho}}[\fm_x]\big)^{L_P(\Z_p)-\alg}_+\neq 0.
\end{equation*}
\end{definition}

\noindent
We recall that a closed point $x\in \Spec \widetilde{\bT}(U^{\wp})_{\overline{\rho}}[1/p]$ is {\it classical} if $\widehat{S}(U^{\wp}, W^{\wp})_{\overline{\rho}}[\fm_x]^{\lalg} \neq 0$. If $x$ classical, then it follows \cite[Prop. 4.2.4]{Em04} that there is an integral dominant $\lambda=(\lambda_1, \cdots, \lambda_n)$ as in \S~\ref{domialgebraic} such that~:
$$\big(\widehat{S}(U^{\wp}, W^{\wp})_{\overline{\rho}})[\fm_x]\otimes_E L_P(\lambda)^\vee\big)^{\sm}\otimes_EL(\lambda)\hookrightarrow \widehat{S}(U^{\wp}, W^{\wp})_{\overline{\rho}}[\fm_x]^{\lalg}\hookrightarrow (\widehat{S}(U^{\wp}, W^{\wp})_{\overline{\rho}})[\fm_x].$$
One then easily deduces from (\ref{equ: lalgaut2}) and, e.g. \cite[Thm.~6.5(v)]{Thor} (taking into account the normalisations) and \cite[Rem.~4.2.4]{BH}, that $\HT(\rho_{x,\widetilde{\wp}})=\{\lambda_1,\lambda_2-1,\cdots,\lambda_n-(n-1)\}$. In particular, $\lambda$ is uniquely determined by $x$.

\begin{proposition}\label{thm: lg1-bp}
(1) A benign point is classical.\\
\noindent
(2) The set of benign points is Zariski-dense in $\Spec \widetilde{\bT}(U^{\wp})_{\overline{\rho}}^{P-\ord}[1/p]$.
\end{proposition}
\begin{proof}
(1) Let $x$ be a benign point. The admissibility of the $L_P(\Q_p)$-continuous representation $\Ord_P(\widehat{S}(U^{\wp}, W^{\wp})_{\overline{\rho}}[\fm_x]$ together with \cite[Thm.~7.1]{ST03} and \cite[Prop.~6.3.6]{Em04} imply that there exist a smooth admissible representation $\pi_x^{\infty}$ of $L_P(\Q_p)$ over $k(x)$ with $(\pi_x^{\infty})^{L_P(\Z_p)}\neq 0$ and $\lambda$ integral dominant such that:
\begin{equation}\label{equ: lg1-bp1}
\pi_x:=\pi_x^{\infty}\otimes_E L_P(\lambda) \hooklongrightarrow \Ord_P\big(\widehat{S}(U^{\wp}, W^{\wp})_{\overline{\rho}}[\fm_x]\big).
\end{equation}
Denote by $\widehat{\pi}_x$ the closure of $\pi_x$ in $\Ord_P\big(\widehat{S}(U^{\wp}, W^{\wp})_{\overline{\rho}}[\fm_x]\big)$, by Proposition \ref{prop: ord-adj} we have continuous $L_P(\Q_p)$-equivariant morphisms:
\begin{multline*}
(\Ind_{\overline{P}(\Q_p)}^{\GL_n(\Q_p)} \pi_x^{\infty})^{\infty}\otimes_E L(\lambda) \hooklongrightarrow (\Ind_{\overline{P}(\Q_p)}^{\GL_n(\Q_p)} \pi_x)^{\an} \hooklongrightarrow (\Ind_{\overline{P}(\Q_p)}^{\GL_n(\Q_p)} \widehat{\pi}_x)^{\cC^0}\\ \lra \widehat{S}(U^{\wp}, W^{\wp})_{\overline{\rho}}[\fm_x],
\end{multline*}
the composition of which is nonzero. (1) follows (and $\lambda$ is the unique dominant weight as discussed just before Proposition \ref{thm: lg1-bp}).\\
\noindent
(2) Let $\cI:=\bigcap_{x \in Z_0} \fm_x$ where $Z_0$ is the set of benign points of $\Spec \widetilde{\bT}(U^{\wp})_{\overline{\rho}}^{P-\ord}[1/p]$, we have to prove $\cI=0$. By Lemma \ref{lem: Pord-comp}(2) and Corollary \ref{coro: lg1-dens}, $\Ord_P(\widehat{S}(U^{\wp}, W^{\wp})_{\overline{\rho}})^{L_P(\Z_p)-\alg}_{+}$ is dense in $\Ord_P(\widehat{S}(U^{\wp}, W^{\wp})_{\overline{\rho}})$. Since by Lemma \ref{lem: Pord-comp}(1) the action of $\widetilde{\bT}(U^{\wp})_{\overline{\rho}}^{P-\ord}$ on $\Ord_P(\widehat{S}(U^{\wp}, W^{\wp})_{\overline{\rho}})$ is faithful, it is thus sufficient to prove that $\Ord_P(\widehat{S}(U^{\wp}, W^{\wp})_{\overline{\rho}})^{L_P(\Z_p)-\alg}_{+}$ is annihilated by $\cI$. Let $\lambda$ be an integral dominant weight, by (\ref{equ: lg1-alg}) we are reduced to prove that any $v\in (\Ord_P(\widehat{S}(U^{\wp}, W^{\wp})_{\overline{\rho}})\otimes_E L_P(\lambda)^\vee)^{L_P(\Z_p)} \otimes_E L_P(\lambda)$ is annihilated by $\cI$. It is enough to consider the case $v=v_{\infty}\otimes u$ with $v_{\infty}\in (\Ord_P(\widehat{S}(U^{\wp}, W^{\wp})_{\overline{\rho}})\otimes_E L_P(\lambda)^\vee)^{L_P(\Z_p)}$ and $u \in L_P(\lambda)$. Let $V_{\infty}$ be the smooth $L_P(\Q_p)$-subrepresentation of $(\Ord_P(\widehat{S}(U^{\wp}, W^{\wp})_{\overline{\rho}})\otimes_E L_P(\lambda)^\vee)^{\sm}$ generated by $v_{\infty}$ and consider the $L_P(\Q_p)$-equivariant injection (see \cite[Prop. 4.2.4]{Em04}):
\begin{equation}\label{equ: lg1-inj1}
V_{\infty}\otimes_E L_P(\lambda) \hooklongrightarrow  \Ord_P(\widehat{S}(U^{\wp}, W^{\wp})_{\overline{\rho}}).
\end{equation}
By Proposition \ref{prop: ord-adj} again, this injection induces:
\begin{multline}\label{equ: lg1-adj1}
(\Ind_{\overline{P}(\Q_p)}^{\GL_n(\Q_p)}V_{\infty})^{\infty} \otimes_E L(\lambda) \lra \widehat{S}(U^{\wp}, W^{\wp})_{\overline{\rho}}^{\lalg} \\
\cong \bigoplus_{x \text{ classical}} \widehat{S}(U^{\wp}, W^{\wp})_{\overline{\rho}}[\fm_x]^{\lalg} \hooklongrightarrow \widehat{S}(U^{\wp}, W^{\wp})_{\overline{\rho}}
\end{multline}
where the middle isomorphism follows from (\ref{equ: lalgaut2}). Since we can recover the injection (\ref{equ: lg1-inj1}) from (\ref{equ: lg1-adj1}) by applying the functor $\Ord_P(\cdot)$ (cf. Proposition \ref{prop: ord-adj}), we see (\ref{equ: lg1-inj1}) factors through:
\begin{multline*}
\Ord_P\Big(\bigoplus_{x \text{ classical}} \widehat{S}(U^{\wp}, W^{\wp})_{\overline{\rho}}[\fm_x]\Big)\cong \bigoplus_{x \text{ classical}} \Ord_P\big(\widehat{S}(U^{\wp}, W^{\wp})_{\overline{\rho}}[\fm_x]\big)\hookrightarrow \\
\Ord_P\big(\widehat{S}(U^{\wp}, W^{\wp})_{\overline{\rho}}[\fm_x]\big).
\end{multline*}
Since $V_{\infty}$ is generated by $v_\infty$ and each $\Ord_P(\widehat{S}(U^{\wp}, W^{\wp})_{\overline{\rho}}[\fm_x])$ is preserved by $L_P(\Q_p)$, there is a finite set $C$ of classical points such that (\ref{equ: lg1-inj1}) has image in $\oplus_{x\in C} \Ord_P(\widehat{S}(U^{\wp}, W^{\wp})_{\overline{\rho}}[\fm_x])$. In particular $v\in \Ord_P(\widehat{S}(U^{\wp}, W^{\wp})_{\overline{\rho}})^{L_P(\Z_p)-\alg}_{+}$ is contained in:
\begin{equation*}
\bigoplus_{x\in C} \Ord_P\big(\widehat{S}(U^{\wp}, W^{\wp})_{\overline{\rho}}[\fm_x]\big)^{L_P(\Z_p)-\alg}_{+}=\bigoplus_{x\in C\cap Z_0} \Ord_P\big(\widehat{S}(U^{\wp}, W^{\wp})_{\overline{\rho}}[\fm_x]\big)^{L_P(\Z_p)-\alg}_{+}
\end{equation*}
and hence is annihilated by $\cI$. (2) follows.
\end{proof}

\noindent
Let $x$ be a closed point of $\Spec \widetilde{\bT}(U^{\wp})_{\overline{\rho}}^{P-\ord}[1/p]$. For $i=1,\cdots, k$ we denote by $\widehat{\pi}(\rho_{x_i})$ the continuous finite length representation of $\GL_{n_i}(\Q_p)$ over $k(x)$ associated to $\rho_{x_i}$ via the $p$-adic local Langlands correspondence for $\GL_2(\Q_p)$ (\cite{Colm10a}) normalized as in \cite[\S~3.1]{Br11b} when $n_i=2$, via local class field theory for $\GL_1(\Q_p)=\Q_p^{\times}$ normalized as in \S~\ref{intronota} when $n_i=1$. Recall that $\overline{B}_2$ denotes the lower triangular matrices of $\GL_2$.

\begin{proposition}\label{prop: lg1-bpp}
(1) If $x$ is a benign point then $\rho_{x, \widetilde{\wp}}$ is semi-stable.\\
\noindent
(2) If $x$ is benign (hence classical by Proposition \ref{thm: lg1-bp}(1)) and $\lambda=(\lambda_1,\lambda_2,\cdots,\lambda_n)$ is the unique integral dominant weight associated to $x$ before Proposition \ref{thm: lg1-bp}, then for $i=1,\cdots, k$, $\rho_{x_i}$ is semi-stable with $\HT(\rho_{x_i})=\{\lambda_{s_i+1}-s_{i}, \lambda_{s_{i}+n_i}-(s_{i}+n_i-1)\}$ (note these two integers are the same when $n_i=1$ and recall $s_i=\sum_{j=0}^{i-1} n_j$).
\end{proposition}
\begin{proof}
We fix a benign point $x$ and use the notation of the proof of Proposition \ref{thm: lg1-bp}(1).\\
\noindent
(1) Let $0\neq v\in (\pi_x^{\infty})^{L_P(\Z_p)}$ be an eigenvector for the spherical Hecke algebra of $L_P(\Q_p)$ with respect to $L_P(\Z_p)$ and let $\pi^{\infty}$ be the $L_P(\Q_p)$-subrepresentation of $\pi_x^{\infty}$ generated by $v$. Then it is easy to check that we have:
\begin{equation*}
\pi^{\infty}\cong \bigotimes_{i=1,\cdots, k} \pi_i^{\infty}
\end{equation*}
where if $n_i=1$, $\psi_{s_i+1}:=\pi_i^{\infty}$ is an unramified character of $\Q_p^{\times}$ and if $n_i=2$, either there exist unramified characters $\psi_{s_i+1}$, $\psi_{s_i+2}$ of $\Q_p^{\times}$ such that $\pi_i^{\infty}\cong (\Ind_{\overline{B}_2(\Q_p)}^{\GL_2(\Q_p)} \psi_{s_i+1} \otimes \psi_{s_i+2})^{\infty}$ with $\psi_{s_i+1} \ne \psi_{s_i+2}$ or $\pi_i^{\infty}$ is isomorphic to the composition of an unramified character of $\Q_p^{\times}$ with the determinant character (note that we can assume $\psi_{s_i+1} \ne \psi_{s_i+2}$ in the first case since otherwise we would in fact be in the second). As in (\ref{equ: lg1-bp1}), we have an $L_P(\Q_p)$-equivariant embedding:
\begin{equation}\label{equ: lg1-bp2}
\Big(\bigotimes_{i=1,\cdots, k} \pi_i^{\infty}\Big)\otimes_{E} L_P(\lambda) \hooklongrightarrow \Ord\big(\widehat{S}(U^{\wp},W^{\wp})_{\overline{\rho}}[\fm_x]\big)
\end{equation}
which, by Proposition \ref{prop: ord-adj}, induces a nonzero morphism:
\begin{equation}\label{equ: lg1-adj2}
 \big(\Ind_{\overline{P}(\Q_p)}^{\GL_n(\Q_p)}\otimes_{i=1,\cdots, k} \pi_i^{\infty}\big)^{\infty}\otimes_{E} L(\lambda) \lra \widehat{S}(U^{\wp},W^{\wp})_{\overline{\rho}}[\fm_x].
\end{equation}
By (\ref{equ: lalgaut2}) and the local-global compatibility at $\ell=p$ in the classical local Langlands correspondence (cf. \cite{Car12}), there exists an automorphic representation $\pi$ of $G$ associated to $\rho_x$ such that the factor of $\pi$ at the place $\wp$ is of the form $\pi_{\wp}\otimes_{k(x)}\overline{\Q_p}$ where $\pi_{\wp}$ is an irreducible constituent of $(\Ind_{\overline{P}(\Q_p)}^{\GL_n(\Q_p)}\otimes_{i=1,\cdots, k} \pi_i^{\infty})^{\infty}$ (note that the action of $\GL_n(\Q_p)$ on $\pi_{\wp}$ actually also depends on $\tilde\wp$). Since the representation $\pi_i^{\infty}$ is unramified for all $i$, it easily follows from \cite{Car12} and properties of the local Langlands correspondence that the potentially semi-stable $\rho_{x, \widetilde{\wp}}$ must be semi-stable. This proves (1). Moreover, since $\overline{\rho}$ is irreducible so is $\rho_x$, and thus $\pi_{\wp}$ is a generic representation of $\GL_n(\Q_p)$ as follows by base change to $\GL_n$ (\cite{Ro}, \cite{Lab}) and genericity of local components of cuspidal automorphic representations of $\GL_n$. This implies that $\pi_i^{\infty}$ is infinite dimensional when $n_i=2$ since otherwise it is easy to check that $(\Ind_{\overline{P}(\Q_p)}^{\GL_n(\Q_p)}\otimes_{i=1,\cdots, k} \pi_i^{\infty})^{\infty}$ has no generic irreducible constituent.\\
\noindent
(2) The fact that $\rho_{x_i}$ is semi-stable follows from (1). We prove the statement on $\HT(\rho_{x_i})$. If $n_i=1$, set $\alpha_{s_i+1}:=p^{s_i} \psi_{s_i+1}(p)$. If $n_i=2$, set $\alpha_{s_i+1}:=\psi_{s_i+1}(p)p^{s_i}$, $\alpha_{s_i+2}:=\psi_{s_i+2}(p)p^{s_{i}+1}$ (so $\alpha_{s_i+1}\alpha_{s_i+2}^{-1}\neq p^{-1}$). It follows from \cite{Car12} and \cite[Thm. 1.2(b)]{Scho13} that we have:
\begin{equation}\label{equ: lg1-wd}
\rec(\pi_{\widetilde{\wp}})\cong  W(\rho_{x, \widetilde{\wp}})^{\sss}\cong \oplus_{j=1}^n \unr(\alpha_j)
\end{equation}
(rec := semi-simplified local Langlands correspondence, see \S~\ref{classical}). Denote by $L_i(\ul{\lambda}_i)$ the algebraic representation of $\GL_{n_i}(\Q_p)$ over $E$ of highest weight $\ul{\lambda}_i=(\lambda_{s_i+1}, \lambda_{s_i+n_i})$. Since $\pi_i^{\infty}\otimes_{E} L_i(\ul{\lambda}_i)$ is unitary by (\ref{equ: lg1-bp2}), we have if $n_i=1$:
\begin{equation}\label{equ: lg1-val}
  \val_p(\alpha_{s_i+1})=-\lambda_{s_i+1} +s_i
\end{equation}
and if $n_i=2$:
\begin{equation}\label{equ: lg1-val2}
 \val_p(\alpha_{s_i+1})+\val_p(\alpha_{s_i+2})=-\lambda_{s_i+1} -\lambda_{s_i+2}+s_i+(s_i+1).
\end{equation}
For $j\in \{1,\cdots,n\}$ set $\mu_j:=-\lambda_j+(j-1)$: the $\mu_j$ are the opposite of the Hodge-Tate weights of $\rho_{x,\widetilde{\wp}}$ and are strictly increasing with $j$. Let $D^{k-1}$ be a $\varphi$-submodule of $D_{\st}(\rho_{x, \widetilde{\wp}})$ such that the $\varphi$-semi-simplification $(D^{k-1})^{\sss}$ is isomorphic to $\oplus_{j=1}^{n-n_k} \unr(\alpha_j)$, we thus have $t_H(D^{k-1})\geq [k(x):\Q_p](\sum_{j=1}^{n-n_k} \mu_j)$ and:
\begin{equation*}
 t_N(D^{k-1})=[k(x):\Q_p]\Big(\sum_{j=1}^{n-n_k} \mu_j\Big), \ \ \ t_N(D/D^{k-1})=[k(x):\Q_p]\Big(\sum_{j=n-n_k+1}^n \mu_j\Big)
\end{equation*}
($t_N$ and $t_H$ as in the proof of Proposition \ref{prop: ord-cLLO}). As in (d) of the proof of Proposition \ref{prop: ord-cLLO}, we can prove $D^{k-1}$ is stable by $N$ and (hence) that $t_H(D^{k-1})=[k(x):\Q_p](\sum_{j=1}^{n-n_k} \mu_j)$. Thus $D^{k-1}$ corresponds to a subrepresentation $\rho^{k-1}$ of $\rho_{x, \widetilde{\wp}}$. Let $\rho_{x_k}':=\rho_{x, \widetilde{\wp}}/\rho^{k-1}$, we have $D_{\st}(\rho_{x_k}')\cong D/D^{k-1}$ and in particular (note $s_k+n_k=n$):
\begin{itemize}
\item $\{-\mu_{s_k+1}, -\mu_{s_k+n_k}\}=\HT(\rho'_{x_k})$
\item $\{\alpha_{s_k+1}, \alpha_{s_k+n_k}\}$= eigenvalue(s) of $\varphi$ on $D_{\st}(\rho_{x_k}')$.
 \end{itemize}
We can then continue the same argument with $D_{\st}(\rho_{x, \widetilde{\wp}})$ replaced by $D^{k-1}$ etc., and obtain representations $\rho_{x_i}'$ for $i=1,\cdots,k$ which are semi-stable with $\HT(\rho'_{x_i})=\{\lambda_{s_i+1}-s_{i}, \lambda_{s_{i}+n_i}-(s_{i}+n_i-1)\}$ and $\varphi$-eigenvalues $\{\alpha_{s_i+1}, \alpha_{s_i+n_i}\}$. But since $\rho_{x,\wp}$ is strictly $P$-ordinary, we have $\rho_{x_i}'\cong \rho_{x_i}$ for all $i$, which finishes the proof of (2).
\end{proof}

\begin{remark}\label{rem: lg1-distin}
{\rm If $n_i=2$, we have by weak admissibility:
\begin{equation}\label{equ: lg1-evF}
\mu_{s_i+1} \leq \val_p(\alpha_{s_i+l}) \leq \mu_{s_i+2}\ \  \forall \  l=1,2.
\end{equation}
Together with (\ref{equ: lg1-val}), we see $\alpha_j\neq \alpha_{j'}$ if $j$, $j'$ do not lie in $\{s_i+1, s_i+{n_i}\}$ for any $i\in \{1,\cdots,k\}$. If $\lambda$ is moreover strictly dominant, i.e. $\lambda_j>\lambda_{j+1}$ for all $j$, we deduce $\alpha_j \alpha_{j'}^{-1}\notin \{1,p,p^{-1}\}$ if $j$, $j'$ do not lie in $\{s_i+1, s_i+{n_i}\}$ for any $i\in \{1,\cdots,k\}$.}
\end{remark}

\noindent
With the notation in the proof of Proposition \ref{prop: lg1-bpp}, there exists $m(x)\in \Z_{\geq 1}$ such that:
\begin{equation}\label{equ: lg1-mult0}
\widehat{S}(U^{\wp}, W^{\wp})_{\overline{\rho}}[\fm_x]^{\lalg} \cong (\pi_{\wp} \otimes_{E} L(\lambda))^{\oplus m(x)}.\end{equation}
In fact, we have:
\begin{equation}\label{equ: lg2-mult1}
 m(x)= \sum_{\pi} m(\pi)\dim_{\overline{\Q_p}} (\pi^{\infty, \wp})^{U^{\wp}}= \sum_{\pi} m(\pi)\dim_{\overline{\Q_p}} (\pi^{\infty, p})^{U^{p}}
\end{equation}
where $\pi$ runs through the automorphic representations of $G(\bA_{F_+})$ whose factor at the place $\wp$ is isomorphic to $\pi_{\wp}\otimes_{k(x)} \overline{E}$ and where the second equality follows from the fact that $U_{v}$ is maximal for $v|p$, $v\neq \wp$. By (\ref{equ: lg1-bp2}) and the fact that each $\pi_i^{\infty}$ for $i=1,\cdots,k$, and thus $\otimes_{i=1,\cdots, k} \pi_i^{\infty}$, has an irreducible socle (see the proof of Proposition \ref{prop: lg1-bpp}(1)), we deduce an $L_P(\Q_p)$-equivariant injection:
\begin{equation}\label{equ: lg1-ordlg}
\big(\bigotimes_{i=1,\cdots, k} \pi_i^{\infty}\big)\otimes_{E} L_P(\lambda) \hooklongrightarrow \Ord_P\big(\pi_{\wp}\otimes_{E} L(\lambda)\big).
\end{equation}

\begin{lemma}\label{lem: lg1-ordalg}
The injection (\ref{equ: lg1-ordlg}) is bijective.
\end{lemma}
\begin{proof}
Denote by $I_P:=\{i=1,\cdots, k,\ n_i=2\}$, we have:
\begin{equation}\label{equ: lg1-Jacord}
J_{B\cap L_P}\big((\otimes_{i=1,\cdots, k} \pi_i^{\infty})\otimes_{E} L_P(\lambda)\big)^{\sss} \cong \delta_{\lambda} \otimes\big(\oplus_{w=(w_i)\in S_2^{|I_P|}} (\otimes_{j=1}^n \unr(\beta_{w,j}))\big)
\end{equation}
where $\delta_{\lambda}$ is the algebraic character of $T(\Q_p)$ of weight $\lambda$, $\sss$ denotes the semi-simplification as $T(\Q_p)$-representations, and where, if $n_i=1$, $\beta_{w,s_i+1}:=p^{-s_i}\alpha_{s_i+1}$, and if $n_i=2$, $\beta_{w,s_i+1}:=p^{-s_i-1}\alpha_{s_i+w_i(1)}$, $\beta_{w,s_i+2}:=p^{-s_i}\alpha_{s_i+w_i(2)}$ with $S_2$ the Weyl group of $\GL_2$ identified with the permutations on the set $\{1,2\}$. On the other hand, we deduce from Remark \ref{ordpart0}:
\begin{equation*}
J_{B\cap L_P}\big( \Ord_P(\pi_{\wp}\otimes_{E}L(\lambda))\big)\hooklongrightarrow J_B(\pi_{\wp}\otimes_{E}L(\lambda))(\delta_P^{-1}).
\end{equation*}
Comparing \cite[Thm. 5.4]{PR} with (\ref{equ: lg1-Jacord}) (recall $\pi_{\wp}$ is a constituent of $(\Ind_{\overline{P}(\Q_p)}^{\GL_n(\Q_p)}\otimes_{i=1,\cdots, k} \pi_i^{\infty})^{\infty}$) and using Remark \ref{rem: lg1-distin}, one can check that any character:
\begin{equation}\label{equ: lg1-chi'}
\chi'\hooklongrightarrow  J_B(\pi_{\wp}\otimes_{E}L(\lambda))[\delta_P^{-1}]^{\sss}/ J_{B\cap L_P}\big((\otimes_{i=1,\cdots, k} \pi_i^{\infty})\otimes_{E} L_P(\lambda)\big)^{\sss}
\end{equation}
does not appear on the right hand-side of (\ref{equ: lg1-Jacord}). Let $\pi_P^{\infty}$ be the smooth admissible representation of $L_P(\Q_p)$ over $k(x)$ such that $\Ord_P(\pi_{\wp}\otimes_{E}L(\lambda))\cong\pi_P^{\infty}\otimes_{E}L_P(\lambda)$. Let $\chi'$ be as in (\ref{equ: lg1-chi'}). If $\chi'$ injects into $J_{B\cap L_P}\big(\pi_P^{\infty}\otimes_{E}L_P(\lambda)\big)$ (which is equivalent to $\chi'\delta_{\lambda}^{-1}\hookrightarrow  J_{B\cap L_P}(\pi_P^{\infty})$), by \cite[(0.1)]{Em2} we deduce a nonzero morphism~:
$$\big(\Ind_{\overline{B}(\Q_p)\cap L_P(\Q_p)}^{L_P(\Q_p)} \chi'\delta_{\lambda}\delta_{B\cap L_P}\big)^{\infty}\lra \pi_P^{\infty}$$
and hence a nonzero morphism:
\begin{equation}\label{equ: lg1-adj10}
\big(\Ind_{\overline{B}(\Q_p)\cap L_P(\Q_p)}^{L_P(\Q_p)} \chi'\delta_{\lambda}\delta_{B\cap L_P}\big)^{\infty}\otimes_{E} L_P(\lambda) \lra \Ord_P(\pi_{\wp}\otimes_{E}L(\lambda)).
\end{equation}
However, $\Ord_P(\pi_{\wp}\otimes_{E}L(\lambda))$ is unitary, while, by (\ref{equ: lg1-val}), (\ref{equ: lg1-val2}) and (\ref{equ: lg1-evF}), one can check that the left hand-side of (\ref{equ: lg1-adj10}) does not have any unitary subquotient (e.g. by considering the central characters, the key point being that, for $w$ in the Weyl group of $\GL_n$ which does not lie in the Weyl group of $L_P$, if we replace the $\alpha_j$ by the $\alpha_j':=\alpha_{w^{-1}(j)}$ for $j=1,\cdots,n$, then at least one of (\ref{equ: lg1-val}), (\ref{equ: lg1-val2}) or (\ref{equ: lg1-evF}) cannot hold). Consequently, any $\chi'$ as in (\ref{equ: lg1-chi'}) cannot inject into $J_{B\cap L_P}( \Ord_P(\pi_{\wp}\otimes_{E}L(\lambda)))$, and hence cannot appear in the semi-simplification of the latter (using that there does not exist nontrivial extension between different characters of $T(\Q_p)$). It follows that the natural injection induced by (\ref{equ: lg1-ordlg}):
\begin{equation*}
J_{B\cap L_P}\big( (\otimes_{i=1,\cdots, k} \pi_i^{\infty})\otimes_{E} L_P(\lambda)\big)\hooklongrightarrow J_{B\cap L_P}\big(\Ord_P(\pi_{\wp} \otimes_{E} L(\lambda))\big)
\end{equation*}
is bijective. Since $J_P(\pi_{\wp})$ does not have cuspidal constituents and $J_{B\cap L_P}$ is an exact functor, we deduce that the injection (\ref{equ: lg1-ordlg}) must be bijective.
\end{proof}

\begin{proposition}\label{prop: lg1-oralgJ}
With the notation of Proposition \ref{prop: lg1-bpp} and its proof, we have an $L_P(\Q_p)$-equivariant isomorphism:
\begin{equation}\label{equ: lg1-ordalg2}
\Ord_P\big(\widehat{S}(U^{\wp}, W^{\wp})_{\overline{\rho}}[\fm_x]^{\lalg}\big) \cong \Big(\big(\bigotimes_{i=1,\cdots, k} \pi_i^{\infty}\big)\otimes_{E} L_P(\lambda)\Big)^{\oplus m(x)}.
\end{equation}
\end{proposition}
\begin{proof}
This is an immediate consequence of (\ref{equ: lg1-mult0}) and Lemma \ref{lem: lg1-ordalg}.
\end{proof}

\begin{corollary}\label{rajoutcristalline}
(1) If $x$ is benign, the representations $\rho_{x_i}$ are crystalline for $i=1,\cdots, k$.\\
\noindent
(2) If $x$ is benign, there exists an $L_P(\Q_p)$-equivariant injection:
 \begin{equation}\label{equ: lg1-lalginj}
\bigotimes_{i=1,\cdots, k} \big(\widehat{\pi}(\rho_{x_i})^{\lalg}\otimes \varepsilon^{s_i}\circ \dett\big) \hooklongrightarrow \Ord_P\big(\widehat{S}(U^{\wp}, W^{\wp})_{\overline{\rho}}[\fm_x]\big).
 \end{equation}
\end{corollary}
\begin{proof}
(1) We use the notation of Proposition \ref{prop: lg1-bpp} and its proof. The first statement is clear when $n_i=1$ by Proposition \ref{prop: lg1-bpp}(2). By {\it loc.cit.} and its proof, it is enough to prove that for $n_i=2$ we have $\alpha_{s_i+1}\alpha_{s_i+2}^{-1}\neq p^{\pm 1}$. From the proof of Proposition \ref{prop: lg1-bpp}, we have already seen $\alpha_{s_i+1}\alpha_{s_i+2}^{-1}\neq p^{-1}$. Assume there exists $i$ such that $n_i=2$ and $\alpha_{s_i+1}\alpha_{s_i+2}^{-1}=p$, then $\pi_i^{\infty}$ is reducible and has a $1$-dimensional quotient. Let $\pi_j'$ be the (unique) irreducible quotient of $\pi_j^{\infty}$ for $j=1,\cdots, k$, we have $\otimes_{j=1}^k \pi_j^{\infty}\twoheadrightarrow \otimes_{j=1}^k \pi_j'$ where $\pi_i'$ is $1$-dimensional. By Lemma \ref{lem: lg1-ordalg} and the fact that $\Ord_P(\pi_{\wp}\otimes_{E} L(\lambda))$ is a direct summand of $J_P\big(\pi_{\wp}\otimes_{E} L(\lambda)\big)(\delta_P^{-1})$ (which follows from (\ref{equ: ord-algb})), we deduce an $L_P(\Q_p)$-equivariant surjection $J_P(\pi_{\wp}) \twoheadlongrightarrow (\otimes_{j=1}^k \pi_j')(\delta_P)$. By \cite[Thm.~5.3(3)]{PR} this induces a nonzero morphism $\pi_{\wp} \ra (\Ind_{\overline{P}(\Q_p)}^{\GL_n(\Q_p)}(\otimes_{j=1}^k \pi_j')(\delta_P))^{\infty}$, which is an injection since $\pi_{\wp}$ is irreducible. However $(\Ind_{\overline{P}(\Q_p)}^{\GL_n(\Q_p)}(\otimes_{j=1}^k \pi_j')(\delta_P))^{\infty}$ does not have any generic irreducible constituent since $\dim_{k(x)}\pi_i'=1$. This gives a contradiction and finishes the proof of (1). Note that we also obtain that $\pi_i^{\infty}$ is irreducible for $i=1,\cdots, k$.\\
\noindent
(2) By well-known properties of the $p$-adic local Langlands correspondence for $\GL_2(\Q_p)$ we have:
\begin{equation*}
\widehat{\pi}(\rho_{x_i})^{\lalg}\cong
\begin{cases}
\unr(\alpha_{s_i+1}) x^{\lambda_{s_i+1}-s_i} & n_i=1\\
\big(\Ind_{\overline{B}_2(\Q_p)}^{\GL_n(\Q_p)} \unr(\alpha_{s_i+1})\otimes \unr(\alpha_{s_i+2}p^{-1})\big)^{\infty}\otimes_{E} L_i(\ul{\lambda}_i-s_i) & n_i=2
\end{cases}
\end{equation*}
where $\ul{\lambda}_i-s_i$ is by definition the weight $(\lambda_{s_i+1}-s_i, \lambda_{s_i+2}-s_i)$. Using $\varepsilon=z\unr(p^{-1})$, we easily deduce:
\begin{equation}\label{equ: lg1-algbis}
\bigotimes_{i=1,\cdots, k} \big(\widehat{\pi}(\rho_{x_i})^{\lalg}\otimes \varepsilon^{s_i}\circ \dett\big) \cong \Big(\bigotimes_{i=1,\cdots, k} \pi_i^{\infty}\Big)\otimes_{E} L_P(\lambda),
\end{equation}
whence (2) by (\ref{equ: lg1-bp2}).
\end{proof}

\subsubsection{$P$-ordinary eigenvarieties}\label{pordinary}

\noindent
We define and study $P$-ordinary Hecke eigenvarieties and use them to prove geometric properties of $\Spec \widetilde{\bT}(U^{\wp})_{\overline{\rho}}^{P-\ord}[1/p]$.\\

\noindent
We keep the notation and assumptions of the previous sections. We now consider the locally analytic representation of $T(\Q_p)$~:
$$J_{B\cap L_P}\big(\Ord_P(\widehat{S}(U^{\wp}, W^{\wp})_{\overline{\rho}})^{\an}\big)$$
where $J_{B\cap L_P}$ is Emerton's locally analytic Jacquet functor (\cite[\S~3.4]{Em11}). This is an essentially admissible representation of $T(\Q_p)$ over $E$ (\cite[Def.~6.4.9]{Em04}) which is equipped with an action of $\widetilde{\bT}(U^{\wp})^{P-\ord}_{\overline{\rho}}$ commuting with $T(\Q_p)$. Let $\cT$ be the rigid analytic space over $E$ parametrizing the locally analytic characters of $T(\Q_p)$ and $(\Spf \widetilde{\bT}(U^{\wp})_{\overline{\rho}}^{P-\ord})^{\rig}$ the generic rigid fiber (\`a la Raynaud-Berthelot) of the formal scheme $\Spf \widetilde{\bT}(U^{\wp})_{\overline{\rho}}^{P-\ord}$ associated to the complete noetherian local ring $\widetilde{\bT}(U^{\wp})_{\overline{\rho}}^{P-\ord}$ (in particular the points of $(\Spf \widetilde{\bT}(U^{\wp})_{\overline{\rho}}^{P-\ord})^{\rig}$ are the closed points of $\Spec \widetilde{\bT}(U^{\wp})_{\overline{\rho}}^{P-\ord}[1/p]$). Then, following \cite[\S~ 2.3]{Em1} the continuous dual $J_{B\cap L_P}(\Ord_P(\widehat{S}(U^{\wp}, W^{\wp})_{\overline{\rho}})^{\an})^\vee$ is the global sections of a coherent sheaf on the rigid analytic space $(\Spf \widetilde{\bT}(U^{\wp})_{\overline{\rho}}^{P-\ord})^{\rig}\times_E \cT$, the schematic support of which defines a Zariski-closed immersion of rigid spaces~:
$$\cE^{P-\ord}\hookrightarrow (\Spf \widetilde{\bT}(U^{\wp})_{\overline{\rho}}^{P-\ord})^{\rig}\times_E \cT.$$
In particular $y=(x, \chi)\in \cE^{P-\ord}$ if and only if there is a $T(\Q_p)$-equivariant embedding:
\begin{equation*}
\chi \hooklongrightarrow J_{B\cap L_P}\big(\Ord_P(\widehat{S}(U^{\wp}, W^{\wp})_{\overline{\rho}})^{\an}\big)[\fm_x].
\end{equation*}
By the same proof (in fact simpler) as for \cite[Cor. 3.12]{BHS1} using Lemma \ref{lem: Pord-comp}(2) to ensure that the analogous results of the ones in \cite[\S\S~3.3~\&~5.2]{BHS1} hold in our setting, we have the following proposition.

\begin{proposition}
The rigid analytic space $\cE^{P-\ord}$ is equidimensional of dimension $n$.
\end{proposition}

\begin{definition}\label{pordclas}
A point $y=(x, \chi)\in \cE^{P-\ord}$ is $P$-ordinary classical if:
\begin{itemize}
\item $\chi$ is of the form $\chi^\infty\delta_\lambda$ where $\chi^\infty$ is smooth and $\lambda=(\lambda_1, \cdots, \lambda_n)$ is integral dominant
\item $J_{B\cap L_P}\big(\Ord_P(\widehat{S}(U^{\wp}, W^{\wp})_{\overline{\rho}})^{\lalg}\big)[\fm_x, T(\Q_p)=\chi]\neq 0$.
\end{itemize}
\end{definition}

\begin{lemma}\label{equ: lg1-pclass}
Let $y=(x, \chi)\in \cE^{P-\ord}$ be $P$-ordinary classical, then the point $x\in $ is classical.
\end{lemma}
\begin{proof}
This follows by the same argument as in the proof of Proposition \ref{thm: lg1-bp}(1) (except that we don't necessarily have $(\pi_x^{\infty})^{L_P(\Z_p)}\neq 0$ anymore), using the adjunction property of the functor $J_{B\cap L_P}(\cdot)$ on locally algebraic representations and then applying Proposition \ref{prop: ord-adj}.
\end{proof}

\begin{lemma}\label{lem: lg2-classici}
Let $\lambda=(\lambda_1, \cdots, \lambda_n)$ be an integral dominant weight, $\chi^{\infty}=\chi^{\infty}_1\otimes \cdots \otimes \chi^{\infty}_n$ be an unramified character of $T(\Q_p)$, and $y=(x, \chi)\in \cE^{P-\ord}$ with $\chi=\delta_{\lambda}\chi^{\infty}=\chi_1\otimes \cdots \otimes \chi_n$. If we have for all $i=1, \cdots, k$ such that $n_i=2$:
\begin{equation}\label{equ: lg1-slopenc}
\val_p(\chi_{s_i+1}(p))<\lambda_{s_i+1}-\lambda_{s_i+2} \ \ \ (\text{equivalently }\val_p(\chi_{s_i+1}^{\infty}(p))<-\lambda_{s_i+2}),
\end{equation}
then $y$ is $P$-ordinary classical.
\end{lemma}
\begin{proof}
As in \S~\ref{sec: hL-GL30} we use without comment in this proof the theory of \cite{OS} (see \cite[\S~2]{Br13I} for a summary). For $i\in \{1,\cdots,k\}$ let $\pi_{s_i+1}:= x^{\lambda_{s_i+1}} \chi^{\infty}_{s_i+1}$ if $n_i=1$ and
$$\pi_i:= \cF_{\overline{B}_2}^{\GL_2}\big(\overline{M}_i(-\ul{\lambda}_i), |\cdot|^{-1}\chi_{s_i+1}^{\infty}\otimes |\cdot|\chi_{s_i+2}^{\infty}\big)$$
if $n_i=2$ where $-\ul{\lambda}_i$ is the algebraic weight $(-\lambda_{s_i+1}, -\lambda_{s_i+2})$ and $\overline{M}_i(-\ul{\lambda}_i):=\text{U}(\gl_2)\otimes_{\text{U}(\overline\ub_2)} (-\ul{\lambda}_i)$ ($\overline \ub_2$ being the Lie algebra of $\overline B_2$). It follows from \cite[Thm. 4.3]{Br13II} that the injection $\chi\hookrightarrow J_{B\cap L_P}(\Ord_P(\widehat{S}(U^{\wp}, W^{\wp})_{\overline{\rho}})^{\an})[\fm_x]$ \ induces \ a \ nonzero \ continuous \ $L_P(\Q_p)$-equivariant \ morphism:
\begin{equation}\label{equ: lg1-adjJac}
\widehat{\bigotimes}_{i=1, \dots, k} \pi_i \lra \Ord_P(\widehat{S}(U^{\wp}, W^{\wp})_{\overline{\rho}})^{\an}[\fm_x]
\end{equation}
where the completed tensor product on the left hand side is with respect to the projective limit topology, or equivalently by \cite[Prop.~1.1.31]{Em04}, the inductive limit topology, on $\pi_1\otimes_{k(x)}\cdots\otimes_{k(x)}\pi_k$. If $\val_p(p\chi^{\infty}_{s_i+1}(p))<1-\lambda_{s_i+2}$, by \cite[Cor. 3.6]{Br13I} the representation:
\begin{multline*}
\cF_{\overline{B}_2}^{\GL_2}\big(\overline{L}_i(-s\cdot \ul{\lambda}_i), |\cdot|^{-1}\chi_{s_i+1}^{\infty}\otimes |\cdot|\chi_{s_i+2}^{\infty}\big)\\ \cong \big(\Ind_{\overline{B}_2(\Q_p)}^{\GL_2(\Q_p)} |\cdot|^{-1}\chi^{\infty}_{s_i+1}x^{\lambda_{s_i+2}-1}\otimes |\cdot| \chi^{\infty}_{s_i+2}x^{\lambda_{s_i+1}+1}\big)^{\an}
 \end{multline*}
does not have a $\GL_2(\Q_p)$-invariant lattice, where $\overline{L}_i(-s\cdot \ul{\lambda}_i)$ is the unique simple subobject of $\overline{M}_i(-\ul{\lambda}_i)$. We then easily deduce that the map in (\ref{equ: lg1-adjJac}) factors through a (nonzero) morphism:
\begin{equation}\label{equ: lg1-adjclass}
\Big(\bigotimes_{i=1, \cdots, k} \pi_i^{\infty}\Big)\otimes_E L_P(\lambda)\lra \Ord_P(\widehat{S}(U^{\wp}, W^{\wp})_{\overline{\rho}})^{\an}[\fm_x]
\end{equation}
where $\pi_i^{\infty}:= \unr(\beta_i)$ if $n_i=1$ and $\pi_i^\infty:= (\Ind_{\overline{B}_2(\Q_p)}^{\GL_2(\Q_p)} |\cdot|^{-1}\chi_{s_i+1}^{\infty}\otimes|\cdot| \chi_{s_i+2}^{\infty})^{\infty}$ if $n_i=2$. The lemma follows (using the adjunction property of $J_{B\cap L_P}(\cdot)$ on locally algebraic representations).
\end{proof}

\noindent
The proof of the following lemma is standard and we omit it (see e.g. the proof of \cite[Thm. 3.19]{BHS1}).

\begin{lemma}\label{zdensebis}
The set of points satisfying the conditions in Lemma \ref{lem: lg2-classici} is Zariski-dense in $\cE^{P-\ord}$.
\end{lemma}

\begin{proposition}\label{prop: lg1-Pdense}
The set of $P$-ordinary classical points is Zariski-dense in $\cE^{P-\ord}$.
\end{proposition}
\begin{proof}
This follows from Lemma \ref{lem: lg2-classici} and Lemma \ref{zdensebis}.
\end{proof}

\noindent
Replacing \ the \ locally \ analytic \ $T(\Q_p)$-representation \ $J_{B\cap L_P}(\Ord_P(\widehat{S}(U^{\wp}, W^{\wp})_{\overline{\rho}})^{\an})$ \ by $J_B(\widehat{S}(U^{\wp}, W^{\wp})_{\overline{\rho}}^{\an})$, we obtain in the same way a rigid analytic variety $\cE$ over $E$ together with a Zariski-closed immersion:
$$\cE\hookrightarrow (\Spf \widetilde{\bT}(U^{\wp})_{\overline{\rho}})^{\rig}\times_E \cT$$
such \ that \ $(x, \chi)\in \cE$ \ if \ and \ only \ if \ there \ is \ a \ $T(\Q_p)$-equivariant \ embedding $\chi\hookrightarrow J_B(\widehat{S}(U^{\wp}, W^{\wp})_{\overline{\rho}}^{\an})[\fm_x]$. Moreover $\cE$ is also equidimensional of dimension $n$. Consider now the following closed immersion:
\begin{equation*}
\iota^{P-\ord}:(\Spf \widetilde{\bT}(U^{\wp})^{P-\ord}_{\overline{\rho}})^{\rig}\times_E \cT \hooklongrightarrow (\Spf \widetilde{\bT}(U^{\wp})_{\overline{\rho}})^{\rig}\times_E \cT, \ \ \ (x, \chi)\longmapsto (x, \chi\delta_P^{-1}).
\end{equation*}
Let $y\in \cE^{P-\ord}$ be $P$-ordinary classical. By Lemma \ref{equ: lg1-pclass} and $J_{B\cap L_P}\circ \Ord_P\hookrightarrow J_B(\delta_P^{-1})$ (see \S~\ref{sec: ord-lalg}), we see $\iota^{P-\ord}(y)$ is a classical point in $\cE$. Together with Proposition \ref{prop: lg1-Pdense}, we deduce that $\iota^{P-\ord}$ induces a closed immersion of reduced rigid analytic spaces:
\begin{equation}\label{equiforP}
\iota^{P-\ord}: \cE^{P-\ord}_{\red} \hooklongrightarrow \cE_{\red}
\end{equation}
where ``$\red$" means the reduced closed rigid subspace.

\begin{corollary}\label{cor: lg1-PordEv}
The rigid space $\cE^{P-\ord}_{\red}$ is isomorphic to a union of irreducible components of $\cE_{\red}$.
\end{corollary}
\begin{proof}
This follows from (\ref{equiforP}) and the fact both $\cE^{P-\ord}_{\red}$ and $\cE_{\red}$ are equidimensional of dimension $n$.
\end{proof}

\noindent
Recall that, for any $(x,\chi)\in \cE$, the associated $\Gal_{\Q_p}$-representation $\rho_{x,\widetilde{\wp}}$ is trianguline (see \cite{KPX} and also \cite{Liu}) and that $(x,\chi)$ is called \emph{noncritical} if $\chi \delta_B^{-1} (1\otimes \varepsilon^{-1}\otimes \cdots \otimes \varepsilon^{1-n})$ gives a parameter of the trianguline $(\varphi,\Gamma)$-module $D_{\rig}(\rho_{x,\widetilde{\wp}})$ associated to $\rho_{x,\widetilde{\wp}}$ (with the usual identification of the $T(\Q_p)$-character $\delta_1\otimes\cdots\otimes\delta_n$ and the parameter $(\delta_1,\cdots,\delta_n)$). We call a point $y=(x, \chi)$ of $\cE^{P-\ord}$ \emph{noncritical} if $\iota^{P-\ord}(y)$ is noncritical, or equivalently if $\chi\delta_{B\cap L_P}^{-1}(1\otimes \varepsilon^{-1}\otimes \cdots \otimes \varepsilon^{1-n})$ is a trianguline parameter of $D_{\rig}(\rho_{x,\widetilde{\wp}})$.

\begin{lemma}\label{lem: lg1-noncrit1}
Let $y=(x, \chi)$ be a $P$-ordinary classical point with $x$ benign, then $y$ is noncritical.
\end{lemma}
\begin{proof}
We use the notation of Definition \ref{pordclas} and of Lemma \ref{lem: lg1-ordalg} and its proof. By Lemma \ref{lem: lg1-ordalg} and (\ref{equ: lg1-Jacord}), there exists $w\in S_2^{|I_P|}$ such that $\chi^{\infty}=\unr(\beta_{w,1})\otimes \cdots \otimes \unr(\beta_{w,n})$. It follows from Proposition \ref{prop: lg1-bpp}(2) and its proof together with Corollary \ref{rajoutcristalline}(1) that:
$$\big(\chi^{\infty}_{s_i+1}|\cdot|^{-1}x^{\lambda_{s_i+1}}\varepsilon^{-s_i}, \chi^{\infty}_{s_i+2}|\cdot|x^{\lambda_{s_i+2}}\varepsilon^{-s_i-1}\big)$$
is a trianguline parameter of $D_{\rig}(\rho_{x_i})$ if $n_i=2$ and $\chi_{s_i+1}^{\infty}x^{\lambda_{s_i+1}} \varepsilon^{-s_i}\cong \rho_{x_i}$ if $n_i=1$ (where $i\in \{1,\cdots,k\}$). Together with the fact $\rho_{x, \widetilde{\wp}}$ is isomorphic to a successive extension of the $\rho_{x_i}$, we deduce that $\chi \delta_{B\cap L_P}^{-1}(1\otimes \varepsilon^{-1}\otimes \cdots \otimes \varepsilon^{1-n})$ is a trianguline parameter of $\rho_{x,\widetilde{\wp}}$.
\end{proof}

\noindent
We say that an $r$-dimensional crystalline representation $V$ of $\Gal_{\Q_p}$ is {\it generic} if the eigenvalues $(\varphi_i)_{i=1,\cdots,r}$ of $\varphi$ on $D_{\cris}(V)$ are such that $\varphi_i\varphi_j^{-1}\notin \{1,p,p^{-1}\}$ for $i\ne j$.

\begin{lemma}\label{lem: lg1-classici2}
Let $y=(x, \chi)\in \cE^{P-\ord}$ as in Lemma \ref{lem: lg2-classici} and assume moreover for all $i=1,\cdots, k$ such that $n_i=2$ (with the notation of {\it loc.cit.}):
\begin{equation}\label{equ: lg1-slopenc2}
\val_p(\chi_{s_i+1}(p))<\frac{\lambda_{s_i+1}-\lambda_{s_i+2}}{2}-1.
\end{equation}
Then $x$ is a benign point and $\rho_{x_i}$ is crystalline generic for $i=1, \cdots, k$.
\end{lemma}
\begin{proof}
We use the notation of Lemma \ref{lem: lg2-classici}. Since $\chi_{s_i+1}(p)+\chi_{s_i+2}(p)=0$ (as follows from (\ref{equ: lg1-adjclass})), we easily deduce from (\ref{equ: lg1-slopenc2}) that:
\begin{equation}\label{equ: lg1-igeneric}
\chi^{\infty}_{s_i+1}(p)\chi^{\infty}_{s_i+2}(p)^{-1}\notin \{p^{-2}, p^{-1}, 1\}.
\end{equation}
which implies that $\pi_i^\infty$ in the proof of Lemma \ref{lem: lg2-classici} is irreducible. It then follows from (\ref{equ: lg1-adjclass}) that $x$ is benign. Hence the $\rho_{x_i}$ are crystalline by Corollary \ref{rajoutcristalline}(1). Moreover, by the proof of Proposition \ref{prop: lg1-bpp}(2), the crystalline eigenvalues of $\varphi$ on $D_{\cris}(\rho_{x_i})$ are given by $\{p^{s_i+1}\chi_{s_i+1}^{\infty}(p), p^{s_i}\chi_{s_i+2}^{\infty}(p)\}$ if $n_i=2$ and $p^{s_i}\chi^{\infty}_{s_i+1}$ if $n_i=1$. We deduce then from (\ref{equ: lg1-igeneric}) that $\rho_{x_i}$ is generic.
\end{proof}

\noindent
Denote by $\omega^1$ the following composition:
\begin{equation*}
\omega^1: \cE^{P-\ord} \hooklongrightarrow (\Spf \widetilde{\bT}(U^{\wp})^{P-\ord}_{\overline{\rho}})^{\rig}\times_E \cT \xlongrightarrow{\pr_1} (\Spf \widetilde{\bT}(U^{\wp})^{P-\ord}_{\overline{\rho}})^{\rig}.
\end{equation*}
Denote by $Z_1'$ the set of $P$-ordinary classical points $y=(x,\chi)\in \cE^{P-\ord}$ such that:
\begin{itemize}
\item $y$ satisfies all the conditions in Lemma \ref{lem: lg2-classici} and Lemma \ref{lem: lg1-classici2} (in particular $x$ is benign)
\item $\chi=\chi^\infty\delta_\lambda$ is such that $\lambda=(\lambda_1, \cdots, \lambda_n)$ is {\it strictly} dominant, i.e. $\lambda_j>\lambda_{j+1}$ for all $j$.
\end{itemize}
We let $Z_1:=\omega^1(Z_1')\subseteq (\Spf \widetilde{\bT}(U^{\wp})^{P-\ord}_{\overline{\rho}})^{\rig}$, which we can also view as a subset of (closed) points of the scheme $\Spec \widetilde{\bT}(U^{\wp})^{P-\ord}_{\overline{\rho}}[1/p]$.

\begin{proposition}\label{prop: lg1-dense2}
(1) The set $Z_1'$ is Zariski-dense in $\cE^{P-\ord}$ and accumulates (see \cite[D\'ef.~2.2]{BHS1}) at any point $(x,\chi)$ with $\chi$ locally algebraic such that $\chi^\infty$ is unramified.\\
\noindent
(2) The set $Z_1$ is Zariski-dense in the scheme $\Spec \widetilde{\bT}(U^{\wp})^{P-\ord}_{\overline{\rho}}[1/p]$.
\end{proposition}
\begin{proof}
(1) The proof is standard and we omit it.\\
\noindent
(2) Let $X_0$ be the Zariski closure of $Z_1$ in the scheme $\Spec \widetilde{\bT}(U^{\wp})^{P-\ord}_{\overline{\rho}}[1/p]$ and $X$ be the associated closed subspace of $(\Spf \widetilde{\bT}(U^{\wp})^{P-\ord}_{\overline{\rho}})^{\rig}$. Note that $X$ contains the Zariski closure of $Z_1$ in the rigid space $(\Spf \widetilde{\bT}(U^{\wp})^{P-\ord}_{\overline{\rho}})^{\rig}$. By Proposition \ref{thm: lg1-bp}(2) it is enough to show any benign point of $\Spec \widetilde{\bT}(U^{\wp})^{P-\ord}_{\overline{\rho}}[1/p]$ belongs to $X_0$, or equivalently to $X$ when seen in $(\Spf \widetilde{\bT}(U^{\wp})^{P-\ord}_{\overline{\rho}})^{\rig}$. Let $x$ be a benign point and $y=(x, \chi)$ a $P$-ordinary classical point of $\cE^{P-\ord}$ lying above $x$. The existence of $y$ follows easily from Corollary \ref{rajoutcristalline}(2) and its proof. By (1), $Z_1'$ accumulates at $y$, in particular $y$ lies in the Zariski closure of $(\omega^1)^{-1}(Z_1)$ in $\cE^{P-\ord}$, from which we easily deduce that $\omega^1(y)=x$ lies in the Zariski closure of $Z_1$ in the rigid space $(\Spf \widetilde{\bT}(U^{\wp})^{P-\ord}_{\overline{\rho}})^{\rig}$. As the latter is contained in $X$, (2) follows.
\end{proof}

\begin{remark}
{\rm We do not know if $Z_1$ is also Zariski-dense in the rigid analytic space $(\Spf \widetilde{\bT}(U^{\wp})^{P-\ord}_{\overline{\rho}})^{\rig}$.}
\end{remark}

\begin{lemma}\label{lem: lg1-glotri1}
Let $y=(x,\chi)\in \cE^{P-\ord}$ such that $Z_1'$ accumulates at $y$. Let $i\in \{1,\cdots, k\}$ and assume $\chi_{s_i+1}\chi_{s_i+2}^{-1}\neq x^m|\cdot|^2$ for any $m\in \Z$ if $n_i=2$. Then $\rho_{x_i}$ is trianguline and there exists an injection of $(\varphi,\Gamma)$-modules over $\cR_{k(x)}$:
\begin{equation}\label{equ: lg1-glotri1}
\begin{cases} \cR_{k(x)}(\chi_{s_i+1}\varepsilon^{-s_i}) \xlongrightarrow{\sim} D_{\rig}(\rho_{x_i}) & \ n_i=1 \\
\cR_{k(x)}(\chi_{s_i+1}\varepsilon^{-s_i}|\cdot|^{-1}) \hooklongrightarrow D_{\rig}(\rho_{x_i}) & \ n_i=2.
\end{cases}
\end{equation}
\end{lemma}
\begin{proof}
Let $i\in \{1,\cdots, k\}$, by Lemma \ref{lem: lg1-noncrit1} we have (\ref{equ: lg1-glotri1}) for any point in $Z_1'$. The result then follows from the global triangulation theory (\cite{KPX}, \cite{Liu}), we leave the (standard) details to the reader.
\end{proof}

\begin{proposition}\label{equ: lg1-adjan1}
Let $y=(x,\chi)\in \cE^{P-\ord}$ be a $P$-ordinary classical point with $x$ benign. Then any injection as in (\ref{equ: lg1-lalginj}) extends to an injection of locally analytic representations of $L_P(\Q_p)$ over $k(x)$:
\begin{equation}
\big(\Ind_{\overline{B}\cap L_P(\Q_p)}^{L_P(\Q_p)} \chi \delta_{B\cap L_P}^{-1}\big)^{\an} \hooklongrightarrow \Ord_P\big(\widehat{S}(U^{\wp}, W^{\wp})_{\overline{\rho}}[\fm_x]\big)^{\an}.
\end{equation}
\end{proposition}
\begin{proof}
We use the notation in the proofs of Proposition \ref{prop: lg1-bpp} and Lemma \ref{lem: lg2-classici}. Let $V$ be an irreducible constituent of $(\Ind_{\overline{B}\cap L_P(\Q_p)}^{L_P(\Q_p)} \chi \delta_{B\cap L_P}^{-1})^{\an}$. It follows from \cite[Cor.~4.25]{OS2} that $V\cong \widehat{\otimes}_{i=1,\cdots,k} V_i$ where $V_i\cong \pi_i^{\infty}\otimes_{E} L_i(\ul{\lambda}_i)$ if $n_i=1$ and $V_i\cong \pi_i^{\infty}\otimes_{E}L_i(\ul{\lambda}_i)$ or $\cF_{\overline{B}_2}^{\GL_2}(\overline{L}_i(-s\cdot \ul{\lambda}_i), |\cdot|^{-1}\chi_{s_i+1}^{\infty}\otimes |\cdot|\chi_{s_i+2}^{\infty})$ if $n_i=2$. Assume that we have an injection:
\begin{equation}\label{equ: lg1-comp2}
V\hooklongrightarrow \Ord_P\big(\widehat{S}(U^{\wp}, W^{\wp})_{\overline{\rho}}[\fm_x]\big)^{\an}
\end{equation}
for a constituent $V$ such that there exists $i\in \{1,\cdots,k\}$ with $V_i$ {\it not} locally algebraic (so $n_i=2$ and $V_i=\cF_{\overline{B}_2}^{\GL_2}(\overline{L}_i(-s\cdot \ul{\lambda}_i), |\cdot|^{-1}\chi_{s_i+1}^{\infty}\otimes |\cdot|\chi_{s_i+2}^{\infty})$). Applying the functor $J_{B\cap L_P}(\cdot)$ to (\ref{equ: lg1-comp2}) gives a point $y'=(x, \chi')\in \cE^{P-\ord}$ with $(\chi')^{\infty}=\chi^{\infty}$ (which is unramified) and $\chi'_{s_i+1}=\chi_{s_i+1} x^{\lambda_{s_i+2}-\lambda_{s_i+1}-1}$, $\chi'_{s_i+2}=\chi_{s_i+2} x^{\lambda_{s_i+1}-\lambda_{s_i+2}+1}$. If $\chi_{s_i+1}\chi_{s_i+2}^{-1}\neq x^{\lambda_{s_i+1}-\lambda_{s_i+2}}|\cdot|^2$ (thus $\chi_{s_i+1}\chi_{s_i+2}^{-1}\neq x^m |\cdot|^2$ for any $m\in \Z$, and hence also $\chi'_{s_i+1}(\chi'_{s_i+2})^{-1}\neq x^m |\cdot|^2$ for any $m\in \Z$), applying Lemma \ref{lem: lg1-glotri1} to the point $y'$ (via Proposition \ref{prop: lg1-dense2}(1)), we easily deduce a contradiction with the fact the $2$-dimensional crystalline $\Gal_{\Q_p}$-representation $\rho_{x_i}$ is nonsplit. Hence such a point $y'$ doesn't exist on $\cE^{P-\ord}$ (and we can't have (\ref{equ: lg1-comp2})). If $\chi_{s_i+1}\chi_{s_i+2}^{-1}= x^{\lambda_{s_i+1}-\lambda_{s_i+2}}|\cdot|^2$, we have $\val_p(\chi_{s_i+1}(p))=\frac{\lambda_{s_i+1}-\lambda_{s_i+2}}{2}-1<\lambda_{s_i+1}-\lambda_{s_i+2}$, and as in the proof of Lemma \ref{lem: lg2-classici}, we then see by \cite[Cor. 3.6]{Br13I} that $V_i$ does not admit a $\GL_2(\Q_p)$-invariant lattice, a contradiction with (\ref{equ: lg1-comp2}). Using \cite[Cor.~4.5]{Br13II} we deduce that $y'$ again doesn't exist on $\cE^{P-\ord}$. The proposition then follows by the same arguments as in \cite[\S~6.4\ Cas\ $i=1$]{Br16} (or as in \cite[\S~5.6]{BE} when $k=1$) using Lemma \ref{lem: Pord-comp}(2) as a replacement for \cite[Lem.~6.3.1]{Br16} and the above discussion as a replacement for \cite[Prop.~6.3.4]{Br16}.
\end{proof}

\begin{corollary}\label{coro: lg1-bpBana}
Let $x$ be a benign point, then any injection as in (\ref{equ: lg1-lalginj}) extends to a closed injection of Banach representations of $L_P(\Q_p)$ over $k(x)$:
\begin{equation}
\widehat{\bigotimes}_{i=1,\cdots, k} \big(\widehat{\pi}(\rho_{x_i})\otimes \varepsilon^{s_i}\circ \dett\big) \hooklongrightarrow \Ord_P\big(\widehat{S}(U^{\wp}, W^{\wp})_{\overline{\rho}}[\fm_x]\big).
\end{equation}
\end{corollary}
\begin{proof}
(a) We use the notation in the proofs of Proposition \ref{prop: lg1-bpp} or Corollary \ref{rajoutcristalline}. When $n_i=2$, by exchanging $\alpha_{s_i+1}$ and $\alpha_{s_i+2}$ if necessary, we can assume $\val_p(\alpha_{s_i+1})\geq \val_p(\alpha_{s_i+2})$. Let $\chi:=\delta_{\lambda}\chi^{\infty}$ with $\chi^{\infty}_{s_i+1}:=\unr(p^{-s_i}\alpha_{s_i+1})$ if $n_i=1$ and $\chi_{s_i+1}^{\infty}:=\unr(p^{-s_i-1}\alpha_{s_i+1})$, $\chi_{s_i+2}^{\infty}:=\unr(p^{-s_i}\alpha_{s_i+2})$ if $n_i=2$. We have $(x,\chi)\in \cE^{P-\ord}$. From Proposition \ref{equ: lg1-adjan1}, we deduce a continuous $L_P(\Q_p)$-equivariant injection:
\begin{equation}\label{equ: lg1-bana1}
\widehat{\bigotimes}_{i=1,\cdots, k} \pi_i^{\an} \cong
\big(\Ind_{\overline{B}\cap L_P(\Q_p)}^{L_P(\Q_p)} \chi \delta_{B\cap L_P}^{-1}\big)^{\an} \hooklongrightarrow \Ord_P\big(\widehat{S}(U^{\wp}, W^{\wp})_{\overline{\rho}}[\fm_x]\big)
\end{equation}
where $\pi_i^{\an}:= \chi_{s_i+1}$ if $n_i=1$ and $\pi_i^{\an} := (\Ind_{\overline{B}_2(\Q_p)}^{\GL_2(\Q_p)} \chi_{s_i+1} |\cdot|^{-1} \otimes \chi_{s_i+2} |\cdot|)^{\an}$ if $n_i=2$. By the above condition on $\alpha_{s_i+1}$, $\alpha_{s_i+2}$, we know that $\widehat{\pi}(\rho_{x_i})\otimes \varepsilon^{s_i}\circ \dett$ is isomorphic to the universal unitary completion of $\pi_i^{\an}$ (see \cite{BeBr} for the case where $\alpha_{s_i+1}\neq \alpha_{s_i+2}$ and \cite{Pas09} for the case where $\alpha_{s_i+1}=\alpha_{s_i+2}$). It then follows from \cite[Lem.~3.4]{BH2} that the universal unitary completion of $\widehat{\otimes}_{i=1,\cdots, k} \pi_i^{\an}$ is isomorphic to $\widehat{\otimes}_{i=1,\cdots, k} (\widehat{\pi}(\rho_{x_i})\otimes \varepsilon^{s_i}\circ \dett)$. We deduce that (\ref{equ: lg1-bana1}) induces a continuous $L_P(\Q_p)$-equivariant morphism:
\begin{equation}\label{equ: lg1-comp}
\widehat{\bigotimes}_{i=1,\cdots, k} \big(\widehat{\pi}(\rho_{x_i})\otimes_{k(x)} \varepsilon^{s_i}\circ \dett\big) \longrightarrow \Ord_P\big(\widehat{S}(U^{\wp} ,W^{\wp})_{\overline{\rho}}[\fm_x]\big)
\end{equation}
which restricts to (\ref{equ: lg1-bana1}) on the left hand side. Since (\ref{equ: lg1-bana1}) is injective and the two Banach representations in (\ref{equ: lg1-comp}) are admissible (for the left hand side, this follows by induction e.g. from \cite[Lem.~2.14]{BH2}), it follows from \cite[\S~7]{ST03} that (\ref{equ: lg1-comp}) is also injective, and from \cite[\S~3]{ST} that it is automatically closed.
\end{proof}

\noindent
We now give a lower bound on the Krull dimension of $\Spec \widetilde{\bT}(U^{\wp})^{P-\ord}_{\overline{\rho}}[1/p]$. We denote by $\cW$ the rigid analytic space over $E$ parametrizing the locally analytic characters of $T(\Z_p)$, by $\omega^2$ the composition:
\begin{equation*}
\omega^2: \cE^{P-\ord}_{\red} \hooklongrightarrow (\Spf \widetilde{\bT}(U^{\wp})^{P-\ord}_{\overline{\rho}})^{\rig}\times_E \cT \xlongrightarrow{\pr_2} \cT
\end{equation*}
and by $\omega^2_0$ the composition of $\omega^2$ with the natural surjection $\cT\twoheadrightarrow \cW$.\\

\noindent
We fix $x\in Z_1$ and use the notation in the proof of Lemma \ref{lem: lg1-ordalg}. For $J\subseteq I_P$, we let $w_J:=(w_{J,i})_{i\in I_P}\in S_2^{|I_P|}$ with $w_{J,i}\neq 1$ if and only if $i \in J$, and put $\chi_J:=\delta_{\lambda}(\otimes_{j=1}^n \unr(\beta_{w_J,j}))$ with the notation of (\ref{equ: lg1-Jacord}). By definition, $\lambda$ is strictly dominant. By (\ref{equ: lg1-Jacord}) and the proof of Lemma \ref{lem: lg1-noncrit1}, we have that $y_J:=(x, \chi_J)\in \cE^{P-\ord}$ and $y_J$ is noncritical. By Proposition \ref{prop: lg1-bpp}, the second part of Remark \ref{rem: lg1-distin} and Lemma \ref{lem: lg1-classici2}, we easily deduce that $\rho_{x,\widetilde{\wp}}$ is crystalline generic. Recall we have assumed Hypothesis \ref{hypoglobal}. We now assume one more condition {\it till the end of the paper}.

\begin{hypothesis}\label{hypo: lg2-globhypo}
If $n>3$, we have $U_v$ maximal hyperspecial at all inert places $v$.
\end{hypothesis}

\noindent
It then follows from \cite[Thm. 4.8~\&~4.10]{Che11} and the smoothness of $\cW$ that the rigid variety $\cE_{\red}$ is smooth at the point $\iota^{P-\ord}(y_J)$ (see (\ref{equiforP})), which therefore belongs to only one irreducible component of $\cE_{\red}$. Combining \cite[Thm. 4.8~\&~4.10]{Che11} with Corollary \ref{cor: lg1-PordEv}, we deduce the following result.

\begin{proposition}\label{prop: lg1-etale}
The morphism $\omega^2_0$ is \'etale at the point $y_J$.
\end{proposition}

\noindent
For $i\in \{1, \cdots, k\}$, we denote by $\omega_i: R_{\overline{\rho}_i}\rightarrow \widetilde{\bT}(U^{\wp})^{P-\ord}_{\overline{\rho}}$ the $i$-th factor of $\omega$ in (\ref{equ: ord-RT}) and we still denote by $\omega_i$ the induced morphism on the respective $(\Spf \cdot)^{\rig}$. We fix $i\in \{1, \cdots, k\}$ and denote by $\omega^1_i$ the following composition:
\begin{equation}\label{omegai}
\omega^1_i:  \cE_{\red}^{P-\ord} \xlongrightarrow{\omega^1} (\Spf \widetilde{\bT}(U^{\wp})^{P-\ord}_{\overline{\rho}})^{\rig} \xlongrightarrow{\omega_i} (\Spf R_{\overline{\rho}_i})^{\rig}.
\end{equation}

\noindent
Recall we have $\widehat {\mathcal O}_{(\Spf R_{\overline{\rho}_i})^{\rig}}\cong R_{\rho_{x_i}}$ (\cite[\S~2.3]{Kis09} and see \S~\ref{sec: ord-galoisdef} for $R_{\rho_{x_i}}$), hence the tangent space~:
$$V_{(\Spf R_{\overline{\rho}_i})^{\rig},x_i}=\Hom_{k(x_i)-\alg}(\widehat {\mathcal O}_{(\Spf R_{\overline{\rho}_i})^{\rig},x_i},k(x_i)[\epsilon]/\epsilon^2)$$
of the rigid analytic variety $(\Spf R_{\overline{\rho}_i})^{\rig}$ at $x_i$ is naturally isomorphic to $\Ext^1_{\Gal_{\Q_p}}(\rho_{x_i}, \rho_{x_i})$. Extending scalars if necessary, we can see everything over the finite extension $k(x)$ of $k(x_i)$. Assume first $n_i=1$, then we have $\dim_{k(x)} \Ext^1_{\Gal_{\Q_p}}(\rho_{x_i}, \rho_{x_i})=2$ and we denote by $\Ext^1_g(\rho_{x_i}, \rho_{x_i})$ the $1$-dimensional ${k(x)}$-vector subspace of de Rham (or equivalently crystalline) deformations. Assume $n_i=2$, since $\rho_{x_i}$ is crystalline, generic and nonsplit, we have $\dim_{k(x)} \Ext^1_{\Gal_{\Q_p}}(\rho_{x_i}, \rho_{x_i})=5$ (e.g. by similar arguments as in Lemma \ref{lem: hL-ext2}). For each refinement $(\alpha_{s_i+w_i(1)}, \alpha_{s_i+w_i(2)})$ on the Frobenius eigenvalues $\{\alpha_{s_i+1}, \alpha_{s_i+2}\}$ of $D_{\cris}(\rho_{x_i})$ with $w_i\in S_2$, one can proceed as in (\ref{equ: l3-exa2}) and Lemma \ref{lem: hL-tri3} and define a ${k(x)}$-vector subspace $\Ext^1_{w_i}(\rho_{x_i}, \rho_{x_i})$ of $\Ext^1_{\Gal_{\Q_p}}(\rho_{x_i}, \rho_{x_i})\cong \Ext^1_{(\varphi,\Gamma)}(D_{\rig}(\rho_{x_i}),D_{\rig}(\rho_{x_i}))$, analogous to the subspace $\Ext^1_{\tri}(D,D)$ of $\Ext^1_{(\varphi,\Gamma)}(D,D)$ in Lemma \ref{lem: hL-tri3}, consisting of trianguline deformations of $\rho_{x_i}$ over ${k(x)}[\epsilon]/\epsilon^2$ with respect to the triangulation on $D_{\rig}(\rho_{x_i})$ associated to the refinement $(\alpha_{s_i+w_i(1)}, \alpha_{s_i+w_i(2)})$. We denote by $\Ext^1_g( \rho_{x_i}, \rho_{x_i})\subseteq \Ext^1_{\Gal_{\Q_p}}(\rho_{x_i}, \rho_{x_i})$ the ${k(x)}$-vector subspace of de Rham deformations, or equivalently of crystalline deformations (since $\rho_{x_i}$ is crystalline generic).

\begin{lemma}\label{lem: lg1-trid}
Let $i\in \{1, \cdots, k\}$ such that $n_i=2$.\\
\noindent
(1) For any $w_i\in S_2$, we have $\dim_{k(x)} \Ext^1_{w_i}(\rho_{x_i}, \rho_{x_i})=4$, $\dim_{k(x)} \Ext^1_g(\rho_{x_i}, \rho_{x_i})=2$ and $\Ext^1_g(\rho_{x_i}, \rho_{x_i})\subseteq \Ext^1_{w_i}(\rho_{x_i}, \rho_{x_i})$.\\
\noindent
(2) We have $\sum_{w_i\in S_2} \Ext^1_{w_i}(\rho_{x_i}, \rho_{x_i})=\Ext^1_{\Gal_{\Q_p}}(\rho_{x_i}, \rho_{x_i})$.
\end{lemma}
\begin{proof}
(1) follows by arguments similar to the ones in the proofs of Lemma \ref{lem: hL-tri3} and Lemma \ref{lem: hL-dRg}. (2) easily follows from $\dim_{k(x)} \Ext^1_{w_i}(\rho_{x_i}, \rho_{x_i})=4$ and $\dim_{k(x)} \Ext^1_{\Gal_{\Q_p}}(\rho_{x_i}, \rho_{x_i})=5$.
\end{proof}

\noindent
For a morphism $f:X\rightarrow Y$ of rigid analytic varieties and a point $x\in X$, we denote by $df_x:V_{X,x}\rightarrow V_{Y,f(x)}$ the $k(x)$-linear map induced by $f$ on the respective tangent spaces of $X$ and $Y$ at $x$ and $f(x)$.\\

\noindent
We fix $J\subseteq I_P$ and denote by $V_J=V_{\cE_{\red}^{P-\ord} ,y_J}$ the tangent space of $\cE_{\red}^{P-\ord}$ at the point $y_J$. We let $\overline{d}\omega^1_{i, y_J}$ be the composition:
\begin{equation*}
\overline{d}\omega^1_{i, y_J}: V_J\xlongrightarrow{d \omega^1_{i,y_J}} \Ext^1_{\Gal_{\Q_p}}(\rho_{x_i}, \rho_{x_i}) \twoheadlongrightarrow \Ext^1_{\Gal_{\Q_p}}(\rho_{x_i}, \rho_{x_i}) /\Ext^1_g(\rho_{x_i}, \rho_{x_i})
\end{equation*}
where we recall that $V_{(\Spf R_{\overline{\rho}_i})^{\rig},x_i}\cong \Ext^1_{\Gal_{\Q_p}}(\rho_{x_i}, \rho_{x_i})$. We set:
$$\overline{d}\omega^1_{y_J}:=(\overline{d}\omega^1_{i, y_J})_{i=1,\cdots, k}:V_J \lra \bigoplus_{i=1, \cdots, k} \Ext^1_{\Gal_{\Q_p}}(\rho_{x_i}, \rho_{x_i}) /\Ext^1_g(\rho_{x_i}, \rho_{x_i}).$$

\begin{proposition}\label{prop: lg1-1fern}
(1) Let $i\in \{1,\cdots, k\}$, we have $\Ima(d\omega^1_{i,y_J})\subseteq \Ext^1_{w_{J,i}}(\rho_{x_i}, \rho_{x_i})$ where $\Ext^1_{w_{J,i}}(\rho_{x_i}, \rho_{x_i}):=\Ext^1_{\Gal_{\Q_p}}(\rho_{x_i}, \rho_{x_i})$ if $n_i=1$.\\
\noindent
(2) The morphism $\overline{d}\omega^1_{y_J}$ induces a bijection (using $\Ext^1_g(\rho_{x_i}, \rho_{x_i})\subseteq \Ext^1_{w_{J,i}}(\rho_{x_i}, \rho_{x_i})$):
\begin{equation*}
\overline{d}\omega^1_{y_J}: V_J \xlongrightarrow{\sim} \bigoplus_{i=1, \cdots, k} \Ext^1_{w_{J,i}}(\rho_{x_i}, \rho_{x_i}) /\Ext^1_g(\rho_{x_i}, \rho_{x_i}).
\end{equation*}
\end{proposition}
\begin{proof}
(1) Let $v\in V_J$, set $\widetilde{\rho}_{x_i}:=d\omega^1_{i,y_J}(v)$, which we view as a deformation of $\rho_{x_i}$ over $k(x)[\epsilon]/\epsilon^2$, and let $\widetilde{\chi}_J:= d\omega^2_{y_J}(v)$, which we view as a deformation of $\chi_J$ over $k(x)[\epsilon]/\epsilon^2$. From the global triangulation theory (see for instance \cite[Prop. 5.13]{Liu}) and Lemma \ref{lem: lg1-glotri1}, we derive:
\begin{equation}\label{equ: lg-glotri2}
\begin{cases} \cR_{k(x)[\epsilon]/\epsilon^2}(\widetilde{\chi}_{J,s_i+1}\varepsilon^{-s_i})  \xlongrightarrow{\sim} D_{\rig}(\widetilde{\rho}_{x_i}) & n_i=1 \\
\cR_{k(x)[\epsilon]/\epsilon^2}(\widetilde{\chi}_{J,s_i+1}\varepsilon^{-s_i}|\cdot|^{-1}) \hooklongrightarrow D_{\rig}(\widetilde{\rho}_{x_i}) & n_i=2.
\end{cases}
\end{equation}
Then (1) follows by definition of $\Ext^1_{w_{J,i}}(\rho_{x_i}, \rho_{x_i})$.\\
\noindent
(2) By Proposition \ref{prop: lg1-etale}, we have $\dim_{k(x)} V_J=n$. By Lemma \ref{lem: lg1-trid}(1) and the discussion before it we have:
\begin{equation}\label{dimi}
\begin{cases} \dim_{k(x)}\Ext^1_{w_{J,i}}(\rho_{x_i}, \rho_{x_i}) /\Ext^1_g(\rho_{x_i}, \rho_{x_i})=1 & n_i=1 \\
\dim_{k(x)}\Ext^1_{w_{J,i}}(\rho_{x_i}, \rho_{x_i}) /\Ext^1_g(\rho_{x_i}, \rho_{x_i})=2 & n_i=2.
\end{cases}
\end{equation}
Hence it is enough to prove that $\overline{d}\omega^1_{y_J}$ is injective. If $0\neq v\in V_J$ then we have $d\omega^2_{0, y_J}(v)\neq 0$ by Proposition \ref{prop: lg1-etale}, and hence there exists $j\in \{1,\cdots,n\}$ such that the character $\widetilde{\chi}_{J,j}$ is {\it not} locally algebraic (i.e. doesn't come from an extension of ${\chi}_{J,j}$ by ${\chi}_{J,j}$ given by $E\val_p$). It then follows from (\ref{equ: lg-glotri2}) and (\ref{ginfty}) that $\widetilde{\rho}_{x_i}\notin\Ext^1_g(\rho_{x_i}, \rho_{x_i})$ if $j\in \{s_i+1, s_i+n_i\}$, whence $\overline{d}\omega^1_{y_J}(v)\ne 0$.
\end{proof}

\noindent
We denote by $V_x$ the tangent space of the rigid variety $(\Spf \widetilde{\bT}(U^{\wp})^{P-\ord}_{\overline{\rho}})^{\rig}$ at the point $x$.

\begin{corollary}\label{coro: lg1-ferns}
We have $\dim_{k(x)} V_x \geq n+(n-k)$.
\end{corollary}
\begin{proof}
For any $J\subseteq I_P$ the morphism $\overline{d}\omega^1_{y_J}$ factors as:
$$\overline{d}\omega^1_{y_J}:V_J\buildrel {d}\omega^1_{y_J} \over \longrightarrow V_x\buildrel \oplus_i d\omega_{i,x}\over\longrightarrow \bigoplus_{i=1, \cdots, k}\Ext^1_{\Gal_{\Q_p}}(\rho_{x_i}, \rho_{x_i}) \twoheadlongrightarrow \bigoplus_{i=1, \cdots, k}\Ext^1_{\Gal_{\Q_p}}(\rho_{x_i}, \rho_{x_i}) /\Ext^1_g(\rho_{x_i}, \rho_{x_i})$$
which implies an inclusion of $k(x)$-vector spaces:
\begin{equation}\label{inclrajout}
\sum_{J\subseteq I_P}\Ima(\overline{d}\omega^1_{y_J})\subseteq \Ima\Big(V_x\longrightarrow \bigoplus_{i=1, \cdots, k}\Ext^1_{\Gal_{\Q_p}}(\rho_{x_i}, \rho_{x_i}) /\Ext^1_g(\rho_{x_i}, \rho_{x_i})\Big).
\end{equation}
But, by Proposition \ref{prop: lg1-1fern} we have:
\begin{eqnarray*}
\bigoplus_{J\subseteq I_P} \Ima(\overline{d}\omega^1_{y_J}) &\cong &\bigoplus_{J\subseteq I_P} \Big(\bigoplus_{i=1, \cdots, k} \Ext^1_{w_{J,i}}(\rho_{x_i}, \rho_{x_i}) /\Ext^1_g(\rho_{x_i}, \rho_{x_i})\Big)\\
&\cong &\bigoplus_{i=1, \cdots, k} \Big(\bigoplus_{J\subseteq I_P} \big(\Ext^1_{w_{J,i}}(\rho_{x_i}, \rho_{x_i}) /\Ext^1_g(\rho_{x_i}, \rho_{x_i})\big)\Big)\\
&\twoheadlongrightarrow& \bigoplus_{i=1,\cdots, k} \Ext^1_{\Gal_{\Q_p}}(\rho_{x_i}, \rho_{x_i})/\Ext^1_g(\rho_{x_i}, \rho_{x_i})
\end{eqnarray*}
where the last morphism is surjective by Lemma \ref{lem: lg1-trid}(2). Together with (\ref{inclrajout}) it follows that the morphism $V_x\rightarrow \oplus_{i=1, \cdots, k}\Ext^1_{\Gal_{\Q_p}}(\rho_{x_i}, \rho_{x_i}) /\Ext^1_g(\rho_{x_i}, \rho_{x_i})$ is in fact surjective. Since the right hand side has dimension $n+\vert I_P\vert=n+(n-k)$ by Lemma \ref{lem: lg1-trid}(1) and the discussion before it, the corollary follows.
\end{proof}

\begin{proposition}\label{prop: lg1-dim}
Each irreducible component of $\Spec \widetilde{\bT}(U^{\wp})_{\overline{\rho}}^{P-\ord}[1/p]$ has (Krull) dimension $\geq n+(n-k)$.
\end{proposition}
\begin{proof}
By Lemma \ref{lem: Pord-reduce} $\Spec\widetilde{\bT}(U^{\wp})_{\overline{\rho}}^{P-\ord}[1/p]$ is a reduced scheme and by Proposition \ref{prop: lg1-dense2}(2) the set of closed points $Z_1$ is Zariski-dense in $\Spec \widetilde{\bT}(U^{\wp})_{\overline{\rho}}^{P-\ord}[1/p]$. Thus for each irreducible component $X$ of $\Spec \widetilde{\bT}(U^{\wp})_{\overline{\rho}}^{P-\ord}[1/p]$, there exists a closed point of $X$ which is in $Z_1$ and such that $X$ is smooth at $x$. Since the completed local rings of the scheme $\Spec \widetilde{\bT}(U^{\wp})_{\overline{\rho}}^{P-\ord}[1/p]$ and of the rigid space $(\Spf \widetilde{\bT}(U^{\wp})^{P-\ord}_{\overline{\rho}})^{\rig}$ at $x$ are isomorphic (see e.g. \cite[Lem.~7.1.9]{DJ}), the tangent space of $\Spec \widetilde{\bT}(U^{\wp})_{\overline{\rho}}^{P-\ord}[1/p]$ at the point $x$ is isomorphic to $V_x$. The result then follows from Corollary \ref{coro: lg1-ferns}.
\end{proof}

\subsubsection{Local-global compatibility}\label{sec: lg1-lg1}

\noindent
We prove local-global compatibility results for the $L_P(\Q_p)$-representation $\Ord_P(\widehat{S}(U^{\wp}, \bW^{\wp})_{\overline{\rho}})$ by generalizing Emerton's method (\cite{Em4}).\\

\noindent
We keep the notation and assumptions of \S\S~\ref{sec: Pord-3}, \ref{domialgebraic}, \ref{sec: lg1-bpts}, \ref{pordinary} (in particular we assume Hypothesis \ref{hypoglobal}~\&~\ref{hypo: lg2-globhypo}) and we assume moreover that the $\GL_2(\Q_p)$-representations $\overline{\rho}_i$ satisfy the assumption (\ref{hypo: hL-modp}) in the appendix when $n_i=2$. We denote by $\overline{\pi}_i$ the representation of $\GL_{n_i}(\Q_p)$ over $k_E$ associated to $\overline{\rho}_i$ by the modulo $p$ Langlands correspondence for $\GL_2(\Q_p)$ normalized as in \cite[\S~3.1]{Br11b} when $n_i=2$ and by local class field theory for $\GL_1(\Q_p)$ normalized as in \S~\ref{intronota} when $n_i=1$. We denote by $\pi_i^{\univ}$ the universal deformation of $\overline{\pi}_i$ over $R_{\overline{\rho}_i}$ in the sense of \cite[Def. 3.3.7]{Em4} (see also \S~\ref{sec: app-def}). We set:
\begin{equation}\label{equ: lg1-pii}
\pi_i(U^{\wp}):=\pi_i^{\univ} \widetilde{\otimes}_{R_{\overline{\rho}_i}} \widetilde{\bT}(U^{\wp})_{\overline{\rho}}^{P-\ord}
\end{equation}
where $\widetilde{\otimes}$ means the $\fm_{\overline{\rho}}$-adic completion of the tensor product (still denoting by $\fm_{\overline{\rho}}$ the maximal ideal of $\widetilde{\bT}(U^{\wp})_{\overline{\rho}}^{P-\ord}$). One can check that this is an orthonormalizable admissible representation of $\GL_{n_i}(\Q_p)$ over $\widetilde{\bT}(U^{\wp})_{\overline{\rho}}^{P-\ord}$ in the sense of \cite[Def. 3.1.11]{Em4}. We set:
\begin{equation}\label{equ: lg1-piup}
\pi^{\otimes}_P(U^{\wp}):=\widetilde{\bigotimes}_{i=1, \cdots, k} \big(\pi_i(U^{\wp}) \otimes \varepsilon^{s_i} \circ \dett\big)
\end{equation}
(the $\fm_{\overline{\rho}}$-completed tensor product being over $\widetilde{\bT}(U^{\wp})_{\overline{\rho}}^{P-\ord}$) which is an orthonormalizable admissible representation of $L_P(\Q_p)$ over $\widetilde{\bT}(U^{\wp})_{\overline{\rho}}^{P-\ord}$. We have:
\begin{eqnarray*}
\pi^{\otimes}_P(U^{\wp})\otimes_{\widetilde{\bT}(U^{\wp})_{\overline{\rho}}^{P-\ord}}\big( \widetilde{\bT}(U^{\wp})_{\overline{\rho}}^{P-\ord}/\fm_{\overline{\rho}}\big) &\cong &
\bigotimes_{i=1,\cdots, k} \big((\pi_i^{\univ} \otimes_{R_{\overline{\rho}_i}} k_E)\otimes \overline{\varepsilon}^{s_i} \circ \dett\big)\\
&\cong & \bigotimes_{i=1, \cdots, k} (\overline{\pi}_i \otimes \overline{\varepsilon}^{s_i} \circ \dett).
\end{eqnarray*}
As in \cite[Def. 6.3.4]{Em4}, we define the $\co_E$-module:
\begin{equation}\label{xpup}
 X_P(U^{\wp}):=
 \Hom_{\widetilde{\bT}(U^{\wp})_{\overline{\rho}}^{P-\ord}[L_P(\Q_p)]}^{\cts}\big(\pi^{\otimes}_P(U^{\wp}), \Ord_P(\widehat{S}(U^{\wp}, \bW^{\wp})_{\overline{\rho}})\big)
\end{equation}
where ``$\cts$" denotes the continuous maps for the $\fm_{\overline{\rho}}$-adic topology on the source and the $\varpi_E$-adic topology on the target. Note that $X_P(U^{\wp})$ is equipped with a natural action of $\widetilde{\bT}(U^{\wp})_{\overline{\rho}}^{P-\ord}$.\\

\noindent
We fix a point $x$ of $(\Spf \widetilde{\bT}(U^{\wp})_{\overline{\rho}}^{P-\ord})^{\rig}$ and let $x_i$ for $i\in \{1,\cdots,k\}$ the associated closed points of $\Spec R_{\overline{\rho}_i}[1/p]$ as in \S~\ref{sec: lg1-bpts}. For $i\in \{1,\cdots,k\}$ we let $\widehat{\pi}(\rho_{x_i})^0$ be the open bounded $\GL_{n_i}(\Q_p)$-invariant $\co_{k(x)}$-lattice of $\widehat{\pi}(\rho_{x_i})$ given by $\widehat{\pi}(\rho_{x_i})^0:= \pi_i^{\univ}\otimes_{R_{\overline{\rho}_i}} \co_{k(x)}$ where the morphism $R_{\overline{\rho}_i}\rightarrow \co_{k(x)}$ is given by $x_i$. We can deduce then (note that the $\fm_{\overline{\rho}}$-adic topology on $\widetilde{\bT}(U^{\wp})^{P-\ord}_{\overline{\rho}}/\fp_x$ coincides with the $p$-adic topology):
\begin{equation}\label{piPmodpx}
\pi^{\otimes}_P(U^{\wp}) \otimes_{\widetilde{\bT}(U^{\wp})^{P-\ord}_{\overline{\rho}}} \widetilde{\bT}(U^{\wp})^{P-\ord}_{\overline{\rho}}/\fp_x \cong \widehat{\bigotimes}_{i=1, \cdots, k} \big(\widehat{\pi}(\rho_{x_i})^0\otimes \varepsilon^{s_i}\circ \dett\big),
\end{equation}
from which we easily get:
\begin{multline}\label{equ: lg1-Xx}
X_P(U^{\wp})[\fp_x]\\
\cong \Hom_{\co_{k(x)}[L_P(\Q_p)]}\Big(\widehat{\bigotimes}_{i=1, \cdots, k} \big(\widehat{\pi}(\rho_{x_i})^0\otimes \varepsilon^{s_i}\circ \dett\big), \Ord_P(\widehat{S}(U^{\wp}, \bW^{\wp})_{\overline{\rho}})[\fp_x]\Big)
\end{multline}
(where $\widehat\otimes$ means the $p$-adic completion of the tensor product). We refer to \cite[Def.~C.1]{Em4} for the definition of a cofinitely generated $\widetilde{\bT}(U^{\wp})_{\overline{\rho}}^{P-\ord}$-module.

\begin{lemma}
The $\widetilde{\bT}(U^{\wp})_{\overline{\rho}}^{P-\ord}$-module $X_P(U^{\wp})$ is cofinitely generated.
\end{lemma}
\begin{proof}
We verify the conditions in \cite[Def. C.1]{Em4}. The first three conditions are easy to check from the definition (\ref{xpup}). We have an injection of $k_E$-vector spaces:
$$X_P(U^{\wp})/\varpi_E\hooklongrightarrow \Hom_{\widetilde{\bT}(U^{\wp})_{\overline{\rho}}^{P-\ord}[L_P(\Q_p)]}\Big(\bigotimes_{i=1,\cdots, k}(\overline{\pi}_i\otimes \overline{\varepsilon}^{s_i}\circ \dett), \Ord_P(\widehat{S}(U^{\wp}, \bW^{\wp})_{\overline{\rho}})/\varpi_E\Big)$$
from which we deduce an injection of $k_E$-vector spaces:
\begin{multline}\label{equ: lg1-xupm}
(X_P(U^{\wp})/\varpi_E) [\fm_{\overline{\rho}}]\\ \hooklongrightarrow \Hom_{k_E[L_P(\Q_p)]}\Big(\bigotimes_{i=1,\cdots, k}\big(\overline{\pi}_i\otimes \overline{\varepsilon}^{s_i}\circ \dett\big), \big(\Ord_P(\widehat{S}(U^{\wp}, \bW^{\wp})_{\overline{\rho}})/\varpi_E\big)[\fm_{\overline{\rho}}]\Big).
\end{multline}
By Lemma \ref{lem: Pord-comp}(1) and its proof (see the isomorphism (\ref{equ: ord-dense})), we have an isomorphism:
\begin{equation}\label{equ: isomodp}
\big(\Ord_P(\widehat{S}(U^{\wp}, \bW^{\wp})_{\overline{\rho}})/\varpi_E\big)[\fm_{\overline{\rho}}] \cong \Ord_P(S(U^{\wp}, \bW^{\wp}/\varpi_E)_{\overline{\rho}})[\fm_{\overline{\rho}}]
\end{equation}
which is a smooth admissible representation of $L_P(\Q_p)$ over $k_E$. Together with the fact that $\otimes_{i=1,\cdots, k}(\overline{\pi}_i\otimes \overline{\varepsilon}^{s_i}\circ \dett)$ can be generated over $L_P(\Q_p)$ by a finite dimensional $k_E$-vector subspace, we easily deduce that the right hand side of (\ref{equ: lg1-xupm}) is finite dimensional over $k_E$. The lemma follows.
\end{proof}

\begin{theorem}\label{thm: lg1-pt}
(1) The $\widetilde{\bT}(U^{\wp})_{\overline{\rho}}^{P-\ord}$-module $X_P(U^{\wp})$ is faithful.\\
\noindent
(2) For any point $x\in (\Spf \widetilde{\bT}(U^{\wp})_{\overline{\rho}}^{P-\ord})^{\rig}$, we have $X_P(U^{\wp})[\fp_x]\neq 0$,
equivalently by (\ref{equ: lg1-Xx}) there exists a nonzero morphism of admissible Banach representations of $L_P(\Q_p)$ over $k(x)$:
\begin{equation}\label{equ: lg1-lgOrd}
 \widehat{\bigotimes}_{i=1,\cdots, k} \big(\widehat{\pi}(\rho_{x_i}) \otimes_{k(x)} \varepsilon^{s_i}\circ \dett\big) \lra \Ord_P\big(\widehat{S}(U^{\wp}, W^{\wp})_{\overline{\rho}}\big)[\fm_x].
\end{equation}
\end{theorem}
\begin{proof}
By \cite[Prop. C.36]{Em4}, (1) and (2) are equivalent, hence it is enough to prove (1). By Corollary \ref{coro: lg1-bpBana}, if $x$ is a benign point we have $X_P(U^{\wp})[\fp_x]\neq 0$. By Proposition \ref{thm: lg1-bp}(2) the benign points are Zariski-dense in $\Spec \widetilde{\bT}(U^{\wp})_{\overline{\rho}}^{P-\ord}[1/p]$. The theorem then follows by the same argument as in the proof of \cite[Prop. C.36]{Em4} (see also \cite[Prop. 4.7]{Br11b}).
\end{proof}

\begin{corollary}\label{coro: lg1-pt}
Let $x\in (\Spf \widetilde{\bT}(U^{\wp})_{\overline{\rho}}^{P-\ord})^{\rig}$, there exists a nonzero morphism of admissible Banach representations of $\GL_n(\Q_p)$ over $k(x)$:
\begin{equation}\label{equ: lg1-pt}
\Big(\Ind_{\overline{P}(\Q_p)}^{\GL_n(\Q_p)}\widehat{\bigotimes}_{i=1,\cdots, k} \big(\widehat{\pi}(\rho_{x_i}) \otimes \varepsilon^{s_i}\circ \dett\big)\Big)^{\cC^0} \lra \widehat{S}(U^{\wp}, W^{\wp})_{\overline{\rho}}[\fm_x].
\end{equation}
\end{corollary}
\begin{proof}
This follows from (\ref{equ: lg1-lgOrd}) and \cite[Thm. 4.4.6]{EOrd1}.
\end{proof}

\begin{corollary}\label{coro: lg1-clas}
Let $x\in (\Spf \widetilde{\bT}(U^{\wp})_{\overline{\rho}}^{P-\ord})^{\rig}$ and assume:
\begin{itemize}
\item for any $i\in \{1,\cdots, k\}$, the $\Gal_{\Q_p}$-representation $\rho_{x_i}$ is irreducible de Rham with distinct Hodge-Tate weights $\{-\mu_{s_i+1}, -\mu_{s_i+n_i}\}$
\item $-\mu_1> -\mu_2>\cdots > -\mu_n$.
\end{itemize}
Then the point $x$ is classical.
\end{corollary}
\begin{proof}
Let $\lambda_j:=-\mu_j+(j-1)$, thus $\lambda:=(\lambda_1, \cdots, \lambda_n)$ is a dominant weight. It follows from \cite[Thm. VI.5.7~\&~VI.6.18]{Colm10a} that there exists a nonzero smooth representation $\pi^{\infty}_i$ of $\GL_{n_i}(\Q_p)$ over $k(x)$ such that $\widehat{\pi}(\rho_{x_i})^{\lalg}\cong \pi^{\infty}_i\otimes_{k(x)} L_i(\ul{\lambda}_i)$ where $\ul{\lambda}_i:=(\lambda_{s_i+1}, \lambda_{s_i+n_i})$. Moreover, since $\rho_{x_i}$ is irreducible, we know that $\widehat{\pi}(\rho_{x_i})$ is also irreducible as a continuous representation of $\GL_{n_i}(\Q_p)$. It then follows that $\widehat{\otimes}_{i=1,\cdots, k} (\widehat{\pi}(\rho_{x_i}) \otimes \varepsilon^{s_i}\circ \det)$ is also irreducible as a continuous representation of $L_P(\Q_p)$ (for lack of a direct reference, one can use locally analytic vectors as follows: one easily checks that it is enough to prove that any nonzero closed invariant subspace of $\widehat{\otimes}_{i=1,\cdots, k} (\widehat{\pi}(\rho_{x_i}) \otimes \varepsilon^{s_i}\circ \det)$ contains a nonzero vector of the form $v_1\otimes\cdots\otimes v_k$ with $v_i\in \widehat{\pi}(\rho_{x_i}) \otimes \varepsilon^{s_i}\circ \det$, but this follows from successively using \cite[Thm.~7.1]{ST03}, \cite[Lem.~2.14]{BH2}, \cite[Cor.~4.25]{OS2} and \cite[Lem.~2.10]{BH2}). It follows that the morphism (\ref{equ: lg1-lgOrd}) is injective and thus restricts to an injective $L_P(\Q_p)$-equivariant morphism:
\begin{multline}\label{equ: lg1-lgalg}
\Big(\bigotimes_{i=1,\cdots, k} \pi_i^{\infty}\Big) \otimes_{k(x)} L_P(\lambda)\\ \cong \bigotimes_{i=1,\cdots, k} \big(\widehat{\pi}(\rho_{x_i})^{\lalg} \otimes \varepsilon^{s_i}\circ \dett\big) \hooklongrightarrow \Ord_P\big(\widehat{S}(U^{\wp}, W^{\wp})_{\overline{\rho}}[\fm_x]\big).
\end{multline}
By Proposition \ref{prop: ord-adj}, the morphism (\ref{equ: lg1-lgalg}) induces a nonzero $\GL_n(\Q_p)$-equivariant morphism:
\begin{equation*}
\Big(\Ind_{\overline{P}(\Q_p)}^{\GL_n(\Q_p)} \bigotimes_{i=1,\cdots, k} \pi_i^{\infty}\Big)^{\infty}\otimes_{k(x)} L(\lambda) \lra \widehat{S}(U^{\wp}, W^{\wp})_{\overline{\rho}}[\fm_x]
\end{equation*}
which implies $\widehat{S}(U^{\wp}, W^{\wp})_{\overline{\rho}}[\fm_x]^{\lalg} \neq 0$, whence the result.
\end{proof}

\begin{remark}
{\rm For $x\in (\Spf \widetilde{\bT}(U^{\wp})_{\overline{\rho}}^{P-\ord})^{\rig}$, by passing to a smaller parabolic subgroup, it should be possible to prove that Corollary \ref{coro: lg1-clas} still holds when $\rho_{x_i}$ is reducible for some $i$.}
\end{remark}

\noindent
We set (where $\Hom_{\co_E}=\co_E$-linear homomorphisms):
\begin{equation}\label{mu}
M_P(U^{\wp}):=\Hom_{\co_E}(X_P(U^{\wp}), \co_E)
\end{equation}
which, by \cite[Prop. C.5]{Em4}, is a finitely generated $\widetilde{\bT}(U^{\wp})_{\overline{\rho}}^{P-\ord}$-module which is $\co_E$-torsion free. Moreover by \cite[Lem. C.14]{Em4}, for any $x\in (\Spf \widetilde{\bT}(U^{\wp})_{\overline{\rho}}^{P-\ord})^{\rig}$, the $\co_{k(x)}$-modules $M_P(U^{\wp}) /\fp_x$ and $X_P(U^{\wp})[\fp_x]$ are finitely generated free of the same rank, that we denote by $m_P(x)$.

\begin{lemma}\label{lem: lg1-mult1}
Let $x$ be a benign point, then $m_P(x)=m(x)$ (see (\ref{equ: lg2-mult1}) for $m(x)$).
\end{lemma}
\begin{proof}
Consider the following composition:
\begin{align*}
\MoveEqLeft[10] {\Hom_{L_P(\Q_p)}\Big(\widehat{\bigotimes}_{i=1, \cdots, k} \big(\widehat{\pi}(\rho_{x_i})\otimes \varepsilon^{s_i}\circ \dett\big), \Ord_P(\widehat{S}(U^{\wp}, W^{\wp})_{\overline{\rho}})[\fm_x]\Big)} & \\
\lra{}& \Hom_{L_P(\Q_p)}\Big(\widehat{\bigotimes}_{i=1, \cdots, k} (\pi_i^{\an}\otimes \varepsilon^{s_i}\circ \dett), \Ord_P(\widehat{S}(U^{\wp}, W^{\wp})_{\overline{\rho}})[\fm_x]\Big)\\
\lra{}& \Hom_{L_P(\Q_p)}\Big(\bigotimes_{i=1, \cdots, k} \big(\widehat{\pi}(\rho_{x_i})^{\lalg}\otimes \varepsilon^{s_i}\circ \dett\big), \Ord_P(\widehat{S}(U^{\wp}, W^{\wp})_{\overline{\rho}})[\fm_x]\Big)
\end{align*}
where $\pi_i^{\an}$ is as in the proof of Corollary \ref{coro: lg1-bpBana}. The first map is injective since $\widehat{\otimes}_{i=1, \cdots, k} \pi_i^{\an}$ is dense in $\widehat{\otimes}_{i=1, \cdots, k} \widehat{\pi}(\rho_{x_i})$ (see the proof of Corollary \ref{coro: lg1-bpBana}). By Corollary \ref{coro: lg1-bpBana}, the composition is surjective. By the proof of Proposition \ref{equ: lg1-adjan1}, the second map is injective. We deduce then that all these maps are bijective. From Proposition \ref{prop: ord-adj}, we deduce an isomorphism:
\begin{multline*}
\Hom_{L_P(\Q_p)}\Big(\bigotimes_{i=1, \cdots, k} \big(\widehat{\pi}(\rho_{x_i})^{\lalg}\otimes \varepsilon^{s_i}\circ \dett\big), \Ord_P\big(\widehat{S}(U^{\wp}, W^{\wp})_{\overline{\rho}}[\fm_x]^{\lalg}\big)\Big)\\
\xlongrightarrow{\sim} \Hom_{L_P(\Q_p)}\Big(\bigotimes_{i=1, \cdots, k} \big(\widehat{\pi}(\rho_{x_i})^{\lalg}\otimes\varepsilon^{s_i}\circ \dett\big), \Ord_P(\widehat{S}(U^{\wp}, W^{\wp})_{\overline{\rho}})[\fm_x]\Big).
\end{multline*}
The lemma follows from these isomorphisms together with (\ref{equ: lg1-Xx}), (\ref{equ: lg1-algbis}), Proposition \ref{prop: lg1-oralgJ}.
\end{proof}

\noindent
Let $S_{\glo}^{P-\ord}(\overline{\rho}_{\widetilde{\wp}})$ be the set of (isomorphism classes of) irreducible $L_P(\Z_p)$-representations $\sigma=\otimes_{i=1}^k\sigma_i$ over $k_E$ such that:
$$\Hom_{L_P(\Z_p)}\big(\sigma, \Ord_P\big({S}(U^{\wp}, \bW^{\wp}/\varpi_E)_{\overline{\rho}}[\fm_{\overline{\rho}}]\big)\big)\neq 0.$$
For $i\in \{1,\cdots,k\}$ let $S_{\glo}(\overline{\rho}_i)$ be the set of (isomorphism classes of) irreducible $\GL_{n_i}(\Z_p)$-representations $\sigma_i$ such that there exist irreducible $\GL_{n_j}(\Z_p)$-representations $\sigma_j$ over $k_E$ for $j\neq i$ such that $\otimes_{j=1}^k (\sigma_j\otimes\overline{\varepsilon}^{s_j}\circ \dett)\in S_{\glo}^{P-\ord}(\overline{\rho}_{\widetilde{\wp}})$. Finally let $S(\overline{\rho}_i)$ be the set of Serre weights attached to $\overline{\rho}_i$, that is the set of irreducible summands in $\soc(\overline{\pi}_i\vert_{\GL_{n_i}(\Z_p)})$, and let $S^{P-\ord}$ be the set of (isomorphism classes of) irreducible $L_P(\Z_p)$-representations $\sigma\cong \otimes_{i=1}^k\ (\sigma_i\otimes \overline{\varepsilon}^{s_i}\circ \dett)$ with $\sigma_i\in S(\overline{\rho}_i)$.

\begin{proposition}\label{prop: lg2-wt}
We have $S_{\glo}^{P-\ord}(\overline{\rho}_{\widetilde{\wp}})\subseteq S^{P-\ord}(\overline{\rho}_{\widetilde{\wp}})$, hence $S_{\glo}(\overline{\rho}_i)\subseteq S(\overline{\rho}_i)$ for any $i\in \{1,\cdots,k\}$.
\end{proposition}
\begin{proof}
The proposition follows by similar arguments as in the proof of \cite[Thm. 5.7.7(1)]{Em4}. For $\sigma=\otimes_{i=1}^k\ (\sigma_i\otimes \overline{\varepsilon}^{s_i}\circ \dett)\in S_{\glo}^{P-\ord}(\overline{\rho}_{\widetilde{\wp}})$, we lift $\sigma$ to an algebraic representation $\Theta\cong \otimes_{i=1,\cdots, k} \Theta_i$ of $L_P(\Z_p)$ over $\co_E$ of (dominant) weight $\lambda$ such that $\lambda_{s_i}\geq\lambda_{s_i+1}$ and $0\leq \lambda_{s_i+n_i}-\lambda_{s_i+1}\leq p-1$ for $i=1,\cdots,k$. Since $\Ord_P(\widehat{S}(U^{\wp}, \bW^{\wp})_{\overline{\rho}})$ is isomorphic to a direct factor of $\cC(L_P(\Z_p),\co_E)^{\oplus r}$ (cf. Lemma \ref{lem: Pord-comp}(2)), we have an isomorphism (e.g. by \cite[Lem. 2.14]{Pas15}):
\begin{equation*}
  \Hom_{L_P(\Z_p)}\big(\Theta, \Ord_P(\widehat S(U^{\wp}, \bW^{\wp})_{\overline{\rho}})\big)/\varpi_E
  \xlongrightarrow{\sim}
  \Hom_{L_P(\Z_p)}\big(\sigma, \Ord_P(S(U^{\wp}, \bW^{\wp}/\varpi_E)_{\overline{\rho}})\big).
\end{equation*}
We deduce, using that $\lambda$ is dominant:
\begin{equation}\small\label{projmodp}
0\ne \Hom_{L_P(\Z_p)}\big(\Theta, \Ord_P(\widehat S(U^{\wp}, \bW^{\wp})_{\overline{\rho}}\big)
\cong \Hom_{L_P(\Z_p)}\big(\Theta, \Ord_P(\widehat S(U^{\wp}, \bW^{\wp})_{\overline{\rho}})^{L_P(\Z_p)-\alg}_{+}\big).
\end{equation}
By (\ref{equ: lalgaut2}) and the same argument as in the proof of Proposition \ref{thm: lg1-bp}(2), it follows that there exists a nonempty finite set $C$ of benign points such that the $\co_E$-module (\ref{projmodp}) is isomorphic to:
\begin{equation*}
\bigoplus_{x\in C} \Hom_{L_P(\Z_p)}\big(\Theta, \Ord_P(\widehat S(U^{\wp}, \bW^{\wp})_{\overline{\rho}})[\fp_x]\big)
\end{equation*}
with each factor in the direct sum being nonzero. Let $x\in C$ and consider (recall $\Theta$ is an $\co_E$-lattice in $L_P(\lambda)$ stable by $L_P(\Z_p)$):
\begin{multline}
\pi_x^{\infty}:=\big(\Ord_P\big(\widehat S(U^{\wp}, W^{\wp})_{\overline{\rho}}[\fm_x]\big) \otimes_{\co_E} L_P(\lambda)^{\vee}\big)^{\sm}\\ \supseteq\Big(\Ord_P\big(\widehat S(U^{\wp}, \bW^{\wp})_{\overline{\rho}}[\fp_x]^{\lalg}\big) \otimes_{\co_E} \Theta^{\vee}\big)^{L_P(\Z_p)}.
\end{multline}
By assumption we have $(\pi_x^{\infty})^{L_P(\Z_p)}\neq 0$, so that (picking up $0\neq v\in (\pi_x^{\infty})^{L_P(\Z_p)}$) we can define smooth irreducible $\GL_{n_i}(\Q_p)$-representations $\pi_i^{\infty}$ as in the proofs of Proposition \ref{prop: lg1-bpp} and Corollary \ref{rajoutcristalline}(1). In particular we have $\otimes_{i=1, \cdots, k} \pi_i^{\infty}\hookrightarrow \pi_x^{\infty}$ and $(\pi_i^{\infty})^{\GL_{n_i}(\Z_p)}\neq 0$, and from (\ref{equ: lg1-algbis}) we also have $\pi_i^{\infty}\otimes L_i(\ul{\lambda}_i)\cong \widehat{\pi}(\rho_{x_i})^{\lalg} \otimes \varepsilon^{s_i} \circ \dett$ where $\ul{\lambda}_i=(\lambda_{s_i+1}, \lambda_{s_i+n_i})$. But the latter isomorphism together with $(\pi_i^{\infty})^{\GL_{n_i}(\Z_p)}\neq 0$ easily imply, using that $\Theta_i$ is up to scaling the only $\co_E$-lattice in $L_i(\ul{\lambda}_i)$ which is stable by $\GL_{n_i}(\Z_p)$:
\begin{equation*}
\sigma_i=\overline{\Theta_i} \otimes \overline{\varepsilon}^{-s_i}\circ \dett \in S(\overline{\rho}_i).
\end{equation*}
The proposition follows.
\end{proof}

\begin{theorem}\label{thm: lg2-lg}
If there exists $i$ such that $n_i=2$ and $\overline{\rho}_i$ is peu ramifi\'e (up to twist), assume that any subrepresentation $\pi=\otimes_{i=1,\cdots, k} \pi_i$ of $\Ord_P(S(U^{\wp}, \bW^{\wp}/\varpi_E)_{\overline{\rho}})$ is such that $\pi_i$ is infinite dimensional. Then the evaluation map:
\begin{equation}
 \label{equ: lg2-ev}
 \ev: X_P(U^{\wp}) \widehat{\otimes}_{\widetilde{\bT}(U^{\wp})_{\overline{\rho}}^{P-\ord}} \pi^{\otimes}_P(U^{\wp})\lra \Ord_P(\widehat{S}(U^{\wp}, \bW^{\wp})_{\overline{\rho}})
\end{equation}
is an isomorphism where $\widehat{\otimes}$ denotes the $\varpi_E$-adic completion of the usual tensor product.
\end{theorem}
\begin{proof}
(a) By \cite[Lem. C.46]{Em4}, the map $\ev$ is injective with saturated image (see \cite[Def.~C.6]{Em4}) if and only if the induced morphism:
\begin{equation}\label{equ: lg2-evmodp}
(X_P(U^{\wp})/\varpi_E)[\fm_{\overline{\rho}}]\otimes_{k_E} \Big(\bigotimes_{i=1,\cdots, k} (\overline{\pi}_i \otimes \overline{\varepsilon}^{s_i}\circ \dett) \Big) \lra \Ord_P(S(U^{\wp}, \bW^{\wp}/\varpi_E)_{\overline{\rho}})[\fm_{\overline{\rho}}]
\end{equation}
is injective. By (\ref{equ: lg1-xupm}), it is enough to prove that the evaluation map:
\begin{multline}
\Hom_{k_E[L_P(\Q_p)]}\Big(\bigotimes_{i=1,\cdots, k}(\overline{\pi}_i\otimes_{k_E} \overline{\varepsilon}^{s_i}\circ \dett), \big(\Ord_P(\widehat{S}(U^{\wp}, \bW^{\wp})_{\overline{\rho}})/\varpi_E\big)[\fm_{\overline{\rho}}]\Big) \otimes_{k_E} \\ \Big(\bigotimes_{i=1,\cdots, k} (\overline{\pi}_i \otimes_{k_E} \overline{\varepsilon}^{s_i}\circ \dett) \Big) \lra \Ord_P(S(U^{\wp}, \bW^{\wp}/\varpi_E)_{\overline{\rho}})[\fm_{\overline{\rho}}]
\end{multline}
is injective.
By \cite[Lem. 6.4.15]{Em4}, it is enough to show that any nonzero homomorphism in:
\begin{equation*}
\Hom_{k_E[L_P(\Q_p)]}\Big(\bigotimes_{i=1,\cdots, k}\big(\overline{\pi}_i\otimes_{k_E} \overline{\varepsilon}^{s_i}\circ \dett\big), \big(\Ord_P(\widehat{S}(U^{\wp}, \bW^{\wp})_{\overline{\rho}})/\varpi_E\big)[\fm_{\overline{\rho}}]\Big)
\end{equation*}
is injective. But this follows from the same argument as in the proof of \cite[Thm. 6.4.16]{Em4} (using Proposition \ref{prop: lg2-wt} and the assumption to deal with those $\overline{\pi}_i$ which are reducible).\\
\noindent
(b) We show that the map $\ev$ is surjective. Since its image is saturated, it is enough to prove the surjection after inverting $p$. By \cite[Lem. 3.1.16]{Em4} and the proof of \cite[Prop. 3.1.3]{Em4}, $\Ima(\ev\otimes E)$ is a closed $L_P(\Q_p)$-subrepresentation of $\Ord_P(\widehat{S}(U^{\wp}, W^{\wp})_{\overline{\rho}})$ which is preserved by $\widetilde{\bT}(U^{\wp})_{\overline{\rho}}^{P-\ord}$. By Lemma \ref{lem: Pord-comp}, Corollary \ref{coro: lg1-dens} and the same argument as in the proof of Proposition \ref{thm: lg1-bp}(2), it is enough to prove that for any benign point $x$, we have:
\begin{equation}\label{equ: lg2-bp}
\Ord_P\big(\widehat{S}(U^{\wp}, W^{\wp})_{\overline{\rho}}[\fm_x]\big)^{L_P(\Z_p)-\alg}_+\subset\Ima(\ev\otimes E).
\end{equation}
Using the adjunction formula of Proposition \ref{prop: ord-adj}, we can deduce (see the proof of Proposition \ref{thm: lg1-bp}(2)):
\begin{equation}\label{equ: lg2-1}
\Ord_P\big(\widehat{S}(U^{\wp}, W^{\wp})_{\overline{\rho}}[\fm_x]\big)^{L_P(\Z_p)-\alg}_+ \subseteq \Ord_P\big(\widehat{S}(U^{\wp}, W^{\wp})_{\overline{\rho}}[\fm_x]^{\lalg}\big).
\end{equation}
Then (\ref{equ: lg2-bp}) easily follows from Proposition \ref{prop: lg1-oralgJ}, (\ref{equ: lg1-algbis}), Corollary \ref{coro: lg1-bpBana} and (\ref{equ: lg1-Xx}).
\end{proof}

\begin{remark}
{\rm The assumption in Theorem \ref{thm: lg2-lg} when $\overline{\rho}_i$ is peu ramifi\'e is in the style of ``Ihara's lemma'' (see e.g. the proof of \cite[Thm. 5.7.7(3)]{Em4}) and one can conjecture that it is always satisfied.}
\end{remark}

\begin{corollary}
Keep the assumption of Theorem \ref{thm: lg2-lg}. There exists $s\geq 1$ such that we have an isomorphism of smooth admissible $L_P(\Q_p)$-representations over $k_E$:
\begin{equation*}
\Big(\bigotimes_{i=1,\cdots, k} \big(\overline{\pi}_i \otimes_{k_E} \overline{\varepsilon}^{s_i}\circ \dett\big) \Big)^{\oplus s} \xlongrightarrow{\sim} \Ord_P\big({S}(U^{\wp}, \bW^{\wp}/\varpi_E)_{\overline{\rho}}[\fm_{\overline{\rho}}]\big).
\end{equation*}
Consequently, $S^{P-\ord}_{\glo}(\overline{\rho}_{\widetilde{\wp}})=S^{P-\ord}(\overline{\rho}_{\widetilde{\wp}})$.
\end{corollary}
\begin{proof}
By Theorem \ref{thm: lg2-lg}, (\ref{equ: isomodp}) and (\ref{piPmodpx}), (\ref{equ: lg2-evmodp}) is actually an isomorphism. The corollary follows since $(X_P(U^{\wp})/\varpi_E)[\fm_{\overline{\rho}}]$ is a finite dimensional $k_E$-vector space.
\end{proof}

\begin{corollary}\label{coro: lg2-lg}
Keep the assumption of Theorem \ref{thm: lg2-lg}. Let $x\in (\Spf \widetilde{\bT}(U^{\wp})_{\overline{\rho}}^{P-\ord})^{\rig}$, then:
\begin{equation*}
\Ord_P\big(\widehat{S}(U^{\wp}, W^{\wp})_{\overline{\rho}}\big)[\fm_x]\cong \big(\widehat{\otimes}_{i=1, \cdots, r} (\widehat{\pi}(\rho_{x_i})\otimes_{k(x)} \varepsilon^{s_i}\circ \dett)\big)^{\oplus m_P(x)}.
\end{equation*}
\end{corollary}
\begin{proof}
The corollary follows from Theorem \ref{thm: lg2-lg}, \cite[Lem. 3.1.17]{Em4} (applied with $A=\widetilde{\bT}(U^{\wp})_{\overline{\rho}}^{P-\ord}$ and $M=\widetilde{\bT}(U^{\wp})_{\overline{\rho}}^{P-\ord}/\fp_x$), (\ref{piPmodpx}) and the definition of $m_P(x)$.
\end{proof}

\subsection{$\cL$-invariants}\label{sec: lg2}

\noindent
We prove Conjecture \ref{THEconjecture} when $\rho_{\widetilde{\wp}}$ has consecutive Hodge-Tate weights assuming weak genericity conditions.

\subsubsection{Preliminaries}\label{fewprel}

\noindent
We start with easy preliminaries.\\

\noindent
Throughout \S~\ref{sec: lg2} we keep the notation and assumptions of \S~\ref{sec: Pord-3} and of all the subsections of \S~\ref{gl2ordinarylocalglobal}, in particular we assume Hypothesis \ref{hypoglobal} and that the open compact subgroup $U^{\wp}$ is such that $U^p$ is sufficiently small, $U_v$ is maximal for $v|p$, $v\neq \wp$, and $U_v$ is maximal hyperspecial at all inert places $v$ if $n>3$ (Hypothesis \ref{hypo: lg2-globhypo}). We assume moreover that $\overline{\rho}$ is such that $\overline{\rho}_{\widetilde{\wp}}$ is a successive extension of characters $\overline{\chi}_i$ for $i=1,\cdots, n$ with $\overline{\chi}_i\overline{\chi}_{i+1}^{-1}=\overline{\varepsilon}$ (so in particular all the $n_i$ are $1$ and $k=n$) and that $\overline{\rho}_{\widetilde{\wp}}$ is strictly $B$-ordinary (Definition \ref{def: ord-strict}). This implies $\overline \chi_i=\overline \varepsilon^{1-i}\overline \chi_1$ for $i=1,\cdots,n$ and $p>n$. We fix $\rho: \Gal_F\ra \GL_n(E)$ a continuous representation such that $\rho$ is unramified outside $S(U^p)$ and such that:
\begin{itemize}
\item $\widehat{S}(U^{\wp}, W^{\wp})_{\overline{\rho}}[\fm_{\rho}]^{\lalg}\neq 0$
\item $\rho_{\widetilde{\wp}}$ is semi-stable noncrystalline and is isomorphic to a successive extension of characters $\chi_i:\Gal_{\Q_p}\rightarrow E^\times$ such that $\chi_i\chi_{i+1}^{-1}=\varepsilon$.
\end{itemize}
\noindent
The first assumption implies that $\rho$ is absolutely irreducible (since $\overline{\rho}$ is), is automorphic (by (\ref{equ: lalgaut2})) and satisfies $\rho^{\vee}\circ c\cong \rho\otimes \varepsilon^{1-n}$, and then the second implies that the monodromy operator on $D_{\st}(\rho_{\widetilde{\wp}})$ satisfies $N^{n-1}\neq 0$ (use \cite{Car12} together with the fact that the automorphic representation associated to $\rho$ has a generic local component at $\wp$ by base change to $\GL_n$ and the irreducibility of $\rho$, see the proof of Proposition \ref{prop: lg1-bpp}(1)). In particular $(\chi_1,\cdots,\chi_n)$ is the unique parameter of the $(\varphi,\Gamma)$-module $D:=D_{\rig}(\rho_{\widetilde{\wp}})$ and it is moreover special (see Definition \ref{def: hL-spe} and use \cite[Thm.~A]{Berger2}) and such that $\chi_i=\varepsilon^{1-i}\chi_1$ for $i=1,\cdots,n$. We also easily deduce that $\rho_{\widetilde{\wp}}$ is strictly $P$-ordinary for any parabolic subgroup $P$ of $\GL_n$ containing $B$.\\

\noindent
Using \cite{Car12}, we see that there exists $m(\rho)$ such that:
\begin{equation}\label{equ: lg2-GL3lalg}
\widehat{S}(U^{\wp}, W^{\wp})_{\overline{\rho}}[\fm_{\rho}]^{\lalg}\cong (\St_n^{\infty} \otimes \chi_1\circ \dett)^{\oplus m(\rho)}
\end{equation}
where $\St_n^{\infty}$ denotes the standard smooth Steinberg representation of $\GL_n(\Q_p)$ over $E$. As in (\ref{equ: lg2-mult1}), we have by (\ref{equ: lalgaut2}) and our assumptions on $U_v$ for $v|p$, $v\neq \wp$:
\begin{equation}\label{equ: lg2-mult2}
 m(\rho)= \sum_{\pi} m(\pi)\dim_{\overline{\Q_p}} (\pi^{\infty, \wp})^{U^{\wp}}= \sum_{\pi} m(\pi)\dim_{\overline{\Q_p}} (\pi^{\infty, p})^{U^{p}}
\end{equation}
where $\pi$ runs through the automorphic representations of $G(\bA_{F_+})$ which contribute to the locally algebraic representation $\widehat{S}(U^{\wp}, W^{\wp})_{\overline{\rho}}[\fm_{\rho}]^{\lalg}$. We easily check that:
\begin{equation}\label{equ: lg2-ordB}
 \Ord_B(\St_n^{\infty} \otimes \chi_1\circ \dett)\cong J_B(\St_n^{\infty}\otimes \chi_1\circ \dett)(\delta_B^{-1})\cong
\chi_1\circ \dett.
\end{equation}

\begin{lemma}\label{lem: lg2-jacncpt}
We have an isomorphism of $T(\Q_p)$-representations:
\begin{equation*}
\soc_{T(\Q_p)}J_B\big(\widehat{S}(U^{\wp}, W^{\wp})_{\overline{\rho}}^{\an}[\fm_\rho]\big) \cong ((\chi_1 \circ \dett)\otimes \delta_B)^{\oplus m(\rho)}.
\end{equation*}
\end{lemma}
\begin{proof}
From the global triangulation theory (\cite{KPX}, \cite{Liu}) applied to the eigenvariety $\cE$ (see \S~\ref{pordinary}), exactly the same proof as the one of \cite[Prop.~6.3.4]{Br16} gives:
\begin{equation}\label{equ: ord-ncomp}
 \Hom_{T(\Q_p)}\big(\delta, J_B\big(\widehat{S}(U^{\wp}, W^{\wp})_{\overline{\rho}}[\fm_{\rho}]\big) \big)\neq 0 \Rightarrow \delta\delta_B^{-1}=\chi_1 \circ \dett.
\end{equation}
There exists thus an integer $m'\geq 1$ such that the isomorphism in the statement holds with $m(\rho)$ replaced by $m'$. By (\ref{equ: lg2-GL3lalg}) and (\ref{equ: lg2-ordB}), we have $m'\geq m(\rho)$. Using \cite[Thm. 4.3]{Br13II} together with (\ref{equ: ord-ncomp}), we see that an ``extra'' copy of $(\chi_1 \circ \dett)\otimes \delta_B$ in the socle would yield an extra copy of $\St_n^{\infty} \otimes \chi_1\circ \dett$ in $\widehat{S}(U^{\wp}, W^{\wp})_{\overline{\rho}}[\fm_{\rho}]^{\lalg}$, hence $m'=m(\rho)$.
\end{proof}

\noindent
We denote by $x$ the point of $(\Spf \widetilde{\bT}(U^{\wp})_{\overline{\rho}})^{\rig}$ associated to $\rho$ (thus $\fm_x=\fm_{\rho}$). By (\ref{equ: lg2-ordB}) and Lemma \ref{ordpart}, we obtain that $x\in (\Spf \widetilde{\bT}(U^{\wp})^{P-\ord}_{\overline{\rho}})^{\rig}$ for all $P\supseteq B$.\\

\noindent
For $1\leq i <i'\leq n$, we denote by $\rho_i^{i'}$ the (unique) subquotient of $\rho_{\widetilde{\wp}}$ which is isomorphic to a successive extension of the characters $\chi_j$ for $i\leq j\leq i'$. We have $D_{\rig}(\rho_i^{i'})=D_i^{i'}=$ the (unique) subquotient of $D$ isomorphic to a successive extension of the $\cR_E(\chi_j)$ for $i\leq j\leq i'$ (see the beginning of \S~\ref{sec: hL-FM} for this notation).

\subsubsection{Simple $\cL$-invariants}\label{simpleinv}

\noindent
For $L_P$ with only one factor being $\GL_2$, we show that one can recover the corresponding simple $\cL$-invariant in $\Ord_{P}(\widehat{S}(U^{\wp}, W^{\wp})_{\overline{\rho}})[\fm_\rho]$ (Corollary \ref{coro: lg2-simpL}). We work in arbitrary dimension.

\noindent
We keep the notation and assumptions of \S~\ref{fewprel}. By Theorem \ref{thm: lg2-lg}, we have an isomorphism (note that the assumption in {\it loc.cit.} is here automatic since $n_i=1$ for all $i$):
\begin{equation}\label{equ: lg2-ordBlg}
X_B(U^{\wp})\widehat{\otimes}_{\widetilde{\bT}(U^{\wp})_{\overline{\rho}}^{B-\ord}} \pi^{\otimes}_B(U^{\wp}) \xlongrightarrow{\sim} \Ord_B(\widehat{S}(U^{\wp}, \bW^{\wp})_{\overline{\rho}}).
\end{equation}
Recall we defined the integer $m_B(x)$ just before Lemma \ref{lem: lg1-mult1}.

\begin{lemma}\label{lem: lg2-mult2}
We have $m_B(x)=m(\rho)$.
\end{lemma}
\begin{proof}
By Corollary \ref{coro: lg2-lg} combined with (\ref{equ: ord-alg2}), (\ref{equ: lg2-GL3lalg}) and (\ref{equ: lg2-ordB}), we have $m_B(x)\geq m(\rho)$. By \cite[Thm. 4.4.6]{EOrd1}, we have:
\begin{multline}\label{equ: lg2-adj1}
 \Hom_{\GL_n(\Q_p)}\big((\Ind_{\overline{B}(\Q_p)}^{\GL_n(\Q_p)} 1)^{\cC^0}\otimes \chi_1\circ \dett, \widehat{S}(U^{\wp}, W^{\wp})_{\overline{\rho}}[\fm_{\rho}]\big)\\ \xlongrightarrow{\sim} \Hom_{T(\Q_p)}\big(\chi_1\circ \dett, \Ord_B( \widehat{S}(U^{\wp}, W^{\wp})_{\overline{\rho}}[\fm_{\rho}])\big).
\end{multline}
We have an obvious injection:
\begin{multline}\label{equ: lg2-anv1}
 \Hom_{\GL_n(\Q_p)}\big((\Ind_{\overline{B}(\Q_p)}^{\GL_n(\Q_p)} 1)^{\cC^0}\otimes \chi_1\circ \dett, \widehat{S}(U^{\wp}, W^{\wp})_{\overline{\rho}}[\fm_{\rho}]\big)\\ \hooklongrightarrow \Hom_{\GL_n(\Q_p)}\big((\Ind_{\overline{B}(\Q_p)}^{\GL_n(\Q_p)} 1)^{\an}\otimes \chi_1\circ \dett, \widehat{S}(U^{\wp}, W^{\wp})_{\overline{\rho}}[\fm_{\rho}]^{\an}\big).
\end{multline}
From the description of irreducible constituents of $(\Ind_{\overline{P}(\Q_p)}^{\GL_n(\Q_p)} 1)^{\an}$ for $B\subseteq P$ (see \cite[\S~6]{OS}), Lemma \ref{lem: lg2-jacncpt} and \cite[Cor. 3.4]{Br13I}, we obtain if $P\supsetneq B$:
$$\Hom_{\GL_n(\Q_p)}\big((\Ind_{\overline{P}(\Q_p)}^{\GL_n(\Q_p)} 1)^{\an}\otimes \chi_1\circ \dett, \widehat{S}(U^{\wp}, W^{\wp})_{\overline{\rho}}[\fm_{\rho}]^{\an}\big)=0,$$
from which we deduce:
\begin{align*}
\MoveEqLeft[19]  {\Hom_{\GL_n(\Q_p)}\big((\Ind_{\overline{B}(\Q_p)}^{\GL_n(\Q_p)} 1)^{\an}\otimes \chi_1\circ \dett, \widehat{S}(U^{\wp}, W^{\wp})_{\overline{\rho}}[\fm_{\rho}]^{\an}\big)} &\\
\xlongleftarrow{\sim}{}& \Hom_{\GL_n(\Q_p)}\big(\St^{\an}_n \otimes \chi_1\circ \dett, \widehat{S}(U^{\wp}, W^{\wp})_{\overline{\rho}}[\fm_{\rho}]^{\an}\big)\\
\hooklongrightarrow{}& \Hom_{T(\Q_p)}\big(\St^{\infty}_n \otimes \chi_1\circ \dett, \widehat{S}(U^{\wp}, W^{\wp})_{\overline{\rho}}[\fm_{\rho}]^{\an}\big)
\end{align*}
where $\St^{\an}_n :=(\Ind_{\overline{B}(\Q_p)}^{\GL_n(\Q_p)} 1)^{\an}/\sum_{P\supsetneq B} (\Ind_{\overline{P}(\Q_p)}^{\GL_n(\Q_p)} 1)^{\an}$. Together with (\ref{equ: lg2-GL3lalg}), (\ref{equ: lg2-anv1}), (\ref{equ: lg2-adj1}) and Corollary \ref{coro: lg2-lg}, we deduce then $m_B(x)\leq m(\rho)$. The lemma follows.
\end{proof}

\begin{lemma}\label{lem: lg2-lfree1}
The $\widetilde{\bT}(U^{\wp})_{\overline{\rho}}^{B-\ord}[1/p]$-module $M_B(U^{\wp})[1/p]$ is locally free at the point $x$.
\end{lemma}
\begin{proof}
(a) Let $X:= \Spec A := \Spec (\widetilde{\bT}(U^{\wp})_{\overline{\rho}}^{B-\ord}[1/p]/\fp)$ be any irreducible component containing the closed point $x$, we show that the $A$-module $M_B(U^{\wp})[1/p]/\fp$ (see (\ref{mu})) is locally free at $x$, from which the corollary follows by \cite[Ex. II.5.8(c)]{Harts} (recall that $\widetilde{\bT}(U^{\wp})_{\overline{\rho}}^{B-\ord}$, hence $\widetilde{\bT}(U^{\wp})_{\overline{\rho}}^{B-\ord}[1/p]$ and $X$, are reduced by Lemma \ref{lem: Pord-reduce}). We define:
$$Z:=\{{\rm benign\ points\ in\ }X\}\cup \{x\}.$$
By Proposition \ref{thm: lg1-bp}(2), we know that $Z$ is Zariski-dense in $X$. By Lemma \ref{lem: lg1-mult1} (resp. by Lemma \ref{lem: lg2-mult2}), we know $m_B(x')=m(x')$ for $x'\in Z\setminus \{x\}$ (resp. $m_B(x)=m(\rho)$).\\
\noindent
(b) For any finite place $\fl\nmid p$ of $F$, we deduce from (\ref{equ: Pord-RT2}) a continuous representation $\rho_{A,\fl}: \Gal_{F_{\fl}} \ra \GL_n(A)$. By \cite[Prop. 4.1.6]{EH}, we can associate to $\rho_{A,\fl}$ a Weil-Deligne representation over $A$. Then the statement of \cite[Prop. 7.8.19]{Bch} (with ``open affinoid'' replaced by ``open affine'') still holds where the rigid space $X$ of {\it loc.cit.} is replaced by the scheme $X$ in (a) and the Weil-Deligne representation in \cite[Prop. 7.8.14]{Bch} is replaced by the one above (the argument of the proof of \cite[Prop. 7.8.19]{Bch} is then analogous, and even easier since we are in the setting of affine schemes). An examination of their proofs then shows that \cite[Lem. 4.5]{Che11} (for any $n$) and \cite[Lem. 4.6]{Che11} (for $n\leq 3$) both hold {\it verbatim} with $(\rho, \co(X))$ of \emph{loc.cit.} replaced by $(\rho_{A,\fl}|_{W_{F_{\fl}}}, A)$. From (\ref{equ: lg2-mult1}), (\ref{equ: lg2-mult2}) together with $m(\pi)=1$ (which follows from \cite{Ro} and \cite{Lab}), we then deduce $m(x')=m(\rho)$ for all $x'\in Z$, and hence $m_B(x')=m(\rho)$ by (a) for all $x'\in Z$.\\
\noindent
(c) Denote by $\cM$ the coherent sheaf on $X$ attached to the $A$-module $M_B(U^{\wp})[1/p]/\fp$. For any prime ideal $\fp'$ of $A$, set:
$$m_B(\fp'):=\dim_{\Frac(A/\fp')} \big((M_B(U^{\wp})[1/p]/\fp')\otimes_{A/\fp'} \Frac(A/\fp')\big)$$
which is upper semi-continuous on $\Spec A$ by \cite[Ex. II.5.8(a)]{Harts}. In particular, the sets:
\begin{equation*}
U_{m}:=\{\fp'\in \Spec A,\ m_B(\fp')\leq m\}=\{\fp'\in \Spec A,\ m_B(\fp')< m+1\}
\end{equation*}
are Zariski-open for $m\in \Z_{\geq 0}$. It follows from (b) that we have $Z\subseteq U_{m(\rho)}$ and $Z \cap U_{m(\rho)-1}=\emptyset$. Since $Z$ is Zariski-dense in $X$, this implies $U_{m(\rho)-1}=\emptyset$, and thus the function $\fp'\mapsto m(\fp')$ is constant of value $m(\rho)$ on the open set $U_{m(\rho)}$ which contains the point $x$. By \cite[Ex. II.5.8(c)]{Harts}, we deduce that $\cM$ is locally free on $U_{m(\rho)}$, which finishes the proof.
\end{proof}

\noindent
Denote by $V_x$ the tangent space of $(\Spf \widetilde{\bT}(U^{\wp})^{B-\ord}_{\overline{\rho}})^{\rig}$ at $x$. Recall that we have a natural morphism (see (\ref{omegai})):
$$\omega=(\omega_i)_{i=1, \cdots, n}: (\Spf \widetilde{\bT}(U^{\wp})^{B-\ord}_{\overline{\rho}})^{\rig} \lra \prod_{i=1}^n (\Spf R_{\overline{\rho}_i})^{\rig}$$
where $\overline{\rho}_i= \overline{\chi}_i$, and that we uniformly (in $i=1,\cdots,n$) identify the tangent space of $(\Spf R_{\overline{\rho}_i})^{\rig}$ at $\omega_i(x)$ with $\Hom(\Q_p^{\times}, E)$ via:
$$\Ext^1_{\Gal_{\Q_p}}(\rho_{x_i}, \rho_{x_i})=\Ext^1_{\Gal_{\Q_p}}(\chi_i, \chi_i)\cong \Ext^1_{(\varphi,\Gamma)}(\cR_E(\chi_i), \cR_E(\chi_i))\buildrel(\ref{equ: hL-cad})\over\cong  \Hom(\Q_p^{\times}, E).$$

\begin{lemma}\label{lem: lg2-tang1}
The morphism $d\omega_x:V_x \lra \oplus_{i=1,\cdots, n} \Hom(\Q_p^{\times},E)$ induced by $\omega$ on the tangent spaces is injective. Moreover, the induced morphism:
\begin{equation}\label{tgtspace}
\overline{d}\omega_x: V_x \lra \bigoplus_{i=1,\cdots, n} \big(\Hom(\Q_p^{\times},E)/\Hom_{\infty}(\Q_p^{\times}, E)\big)
\end{equation}
is bijective.
\end{lemma}
\begin{proof}
By Proposition \ref{prop: lg1-dim}, we have $\dim_E V_x\geq n$. Since this is also the dimension of the right hand side of (\ref{tgtspace}), it is enough to prove that $\overline{d}\omega_x$ is injective. Let $0\ne v\in V_x$ such that $\overline{d}\omega_x(v)=0$ and denote by $\cI_v:=\Ker(\widetilde{\bT}(U^{\wp})^{B-\ord}_{\overline{\rho}}\rightarrow E[\epsilon]/\epsilon^2)$ the ideal attached to $v$ (so $\widetilde{\bT}(U^{\wp})^{B-\ord}_{\overline{\rho}}/\cI_v\cong \co_E[\epsilon]/\epsilon^2$). From \cite[Prop. C.11]{Em4} applied with $M=\widetilde{\bT}(U^{\wp})^{B-\ord}_{\overline{\rho}}/\cI_v$ we obtain:
\begin{eqnarray}\label{alglin}
(M_B(U^{\wp})/\cI_v)[1/p]&\cong &\Hom_{\co_E}\big(\Hom_{\widetilde{\bT}(U^{\wp})_{\overline{\rho}}^{B-\ord}}(\widetilde{\bT}(U^{\wp})_{\overline{\rho}}^{B-\ord}/\cI_v,X_B(U^{\wp})),\co_E\big)[1/p] \nonumber\\
&\cong & \Hom_{\co_E}\big(X_B(U^{\wp})[\cI_v],\co_E\big)[1/p].\label{dualdual}
\end{eqnarray}
From (\ref{dualdual}) and Lemma \ref{lem: lg2-lfree1} it easily follows that $(X_B(U^{\wp})[\cI_v])[1/p]$ is free of rank $m_B(x)$ over $(\widetilde{\bT}(U^{\wp})^{B-\ord}_{\overline{\rho}}/\cI_v)[1/p]\cong E[\epsilon]/\epsilon^2$. For $i=1, \cdots, n$ denote by $\widetilde{\chi}_i$ the extension of $\chi_i$ by $\chi_i$ associated to $d\omega_{i,x}(v)\in \Hom(\Q_p^{\times}, E)$. From (\ref{equ: lg1-pii}) we get $(\pi_i(U^{\wp})/\cI_v)[1/p]\cong \widetilde{\chi}_i$ (since $n_i=1$). Let ${\chi}:=\otimes_{i=1}^n ({\chi}_i\otimes \varepsilon^{s_i})$ and $\widetilde{\chi}:=\otimes_{i=1}^n (\widetilde{\chi}_i\otimes \varepsilon^{s_i})$ where the tensor product $\otimes_{i=1}^n$ on the latter is over $E[\epsilon]/\epsilon^2$. By (\ref{equ: lg2-ordBlg}) together with \cite[Lem. 3.1.17]{Em4} applied with $M=\widetilde{\bT}(U^{\wp})_{\overline{\rho}}^{B-\ord}/\cI_v$, we obtain a commutative diagram:
\begin{equation}\label{equ: lg2-linvT}
\begin{CD}
\chi^{\oplus m_B(x)}@> \sim >> \Ord_B(\widehat{S}(U^{\wp}, W^{\wp})_{\overline{\rho}})[\fm_\rho] \\
@VVV @VVV\\
\widetilde{\chi}^{\oplus m_B(x)} @> \sim>> \Ord_B(\widehat{S}(U^{\wp}, W^{\wp})_{\overline{\rho}})[\cI_v].
\end{CD}
\end{equation}
Since $\overline{d}\omega_x(v)=0$, the character $\widetilde{\chi}$ is locally algebraic by (\ref{ginfty}). It then follows from $\fm_\rho^2\subseteq \cI_v[1/p]$ and Proposition \ref{prop: lg1-semisimp} that the bottom horizontal map in (\ref{equ: lg2-linvT}) factors through $\Ord_B(\widehat{S}(U^{\wp}, W^{\wp})_{\overline{\rho}})[\fm_\rho]$, which contradicts (\ref{equ: lg2-linvT}). The lemma follows.
\end{proof}

\noindent
Recall that for $i=1,\cdots, n-1$ the $(\varphi,\Gamma)$-module $D_i^{i+1}$ was defined at the end of \S~\ref{fewprel}, and that $\cL_{\FM}(D_i^{i+1}: \cR_E(\chi_i))$ is the line in $\Ext^1_{(\varphi,\Gamma_L)}(\cR_E(\chi_i), \cR_E(\chi_i))\cong \Hom(\Q_p^\times,E)$ defined as the orthogonal of $ED_i^{i+1}\subseteq \Ext^1_{(\varphi,\Gamma_L)}(\cR_E(\chi_{i+1}), \cR_E(\chi_i))$ via the pairing as in (\ref{equ: hL-cup}), see \S~\ref{sec: hL-FM}.

\begin{proposition}\label{prop: lg2-sim}
For $i=1,\cdots, n-1$, the morphism $d\omega_{i,x}-d\omega_{i+1,x}$ factors through a surjection:
\begin{equation*}
 d\omega_{i,x}-d\omega_{i+1,x}: V_x \twoheadlongrightarrow \cL_{\FM}(D_i^{i+1}: \cR_E(\chi_i))\subsetneq \Hom(\Q_p^\times,E).
\end{equation*}
\end{proposition}
\begin{proof}
Recall we have a morphism of rigid spaces (see (\ref{equ: ord-RT})):
\begin{equation*}
\omega': (\Spf \widetilde{\bT}(U^{\wp})_{\overline{\rho}}^{B-\ord})^{\rig}\lra(\Spf R_{\overline{\rho}_{\widetilde{\wp}}}^{B-\ord})^{\rig}.
\end{equation*}
For any nonzero $v$ in $V_x$, let $\widetilde{\rho}$ (resp. $\widetilde{\chi}_i$) be the $\Gal_{\Q_p}$-representation over $E[\epsilon]/\epsilon^2$ attached to $d\omega'_{x}(v)$ (resp. $d\omega_{i,x}(v)$). We know that $\widetilde{\rho}$ (resp. $\widetilde{\chi}_i$) is a deformation of $\rho_{\widetilde{\wp}}$ (resp. $\chi_i$) over $E[\epsilon]/\epsilon^2$. It follows from Proposition \ref{prop: ord-Pdef2}(2) that $v$ can be seen as an $E[\epsilon]/\epsilon^2$-valued point of $\Spec R_{{\rho}_{\widetilde{\wp}},\{\chi_i\}}^{B-\ord}$, hence that $\widetilde{\rho}$ is isomorphic to a successive extension of the $\widetilde{\chi}_i$ as $\Gal_{\Q_p}$-representation over $E[\epsilon]/\epsilon^2$. Then from Theorem \ref{thm: hL-hL} we easily deduce $(d\omega_{i,x}-d\omega_{i+1,x})(v)\in \cL_{\FM}(D_i^{i+1}: \cR_E(\chi_i))$ for all $i=1,\cdots, n-1$. If $d\omega_{i,x}-d\omega_{i+1,x}=0$ for some $i$ (equivalently $d\omega_{i,x}-d\omega_{i+1,x}$ is not surjective), then the morphism in Lemma \ref{lem: lg2-tang1} cannot be surjective, a contradiction.
\end{proof}

\noindent
For $r\in \{1, \cdots, n-1\}$, we denote by $P_r$ the parabolic subgroup as in (\ref{equ: ord-LP}) with $k=n-1$, $n_i=1$ for $i\in \{1,\cdots, n-1\}\backslash \{r\}$ and $n_i=2$ for $i=r$ (note that this implies $n\geq 3$). We have isomorphisms of smooth representations of $L_{P_r}(\Q_p)$ over $E$:
\begin{equation}\small\label{equ: lg1-ordPr}
\Ord_{P_r}(\St_n^{\infty} \otimes \chi_1 \circ \dett)\cong J_{P_r}(\St_n^{\infty}\otimes \chi_1\circ \dett)(\delta_{P_r}^{-1})\cong \Big(\big(\bigotimes_{\substack{i=1, \cdots, n-1\\ i\neq r}} 1\big) \otimes \St_2^{\infty}\Big)\otimes (\chi_1\circ \dett)
\end{equation}
where the first isomorphism follows from the second (see \S~\ref{sec: ord-lalg} for $\Ord_{P_r}$) and where the second easily follows from $J_{B\cap L_{P_r}}(J_{P_r}(\St_n^{\infty}))\cong J_B(\St_n^{\infty})\cong \delta_B$ and the usual adjunction for $J_{B\cap L_{P_r}}(\cdot)$.

\begin{corollary}\label{coro: lg2-simpL}
For $r=1,\cdots, n-1$, the restriction morphism:
\begin{multline}\label{equ: lg2-simpL}
\Hom_{L_{P_r}(\Q_p)}\Big(\big(\bigotimes_{\substack{i=1, \dots, n-1 \\ i \neq r}} \chi_1\big) \otimes ( \widehat{\pi}(\rho_r^{r+1})\otimes \varepsilon^{r-1} \circ \dett), \Ord_{P_r}(\widehat{S}(U^{\wp}, W^{\wp})_{\overline{\rho}})[\fm_\rho]\Big)\\
\lra \Hom_{L_{P_r}(\Q_p)}\Big( \Big(\big(\bigotimes_{\substack{i=1, \cdots, n-1\\ i\neq r}} 1\big) \otimes \St_2^{\infty}\Big)\otimes(
\chi_1\circ \dett), \Ord_{P_r}(\widehat{S}(U^{\wp}, W^{\wp})_{\overline{\rho}})[\fm_\rho]\Big)
\end{multline}
is an isomorphism. In particular, we have (see (\ref{equ: lg2-GL3lalg}) for $m(\rho)$):
\begin{equation*}
\dim_E \Hom_{L_{P_r}(\Q_p)}\Big(\big(\bigotimes_{\substack{i=1, \dots, n-1 \\ i \neq r}} \chi_1\big) \otimes ( \widehat{\pi}(\rho_r^{r+1})\otimes \varepsilon^{r-1} \circ \dett), \Ord_{P_r}(\widehat{S}(U^{\wp}, W^{\wp})_{\overline{\rho}})[\fm_\rho]\Big)=m(\rho).
\end{equation*}
\end{corollary}
\begin{proof}
Note first that $\chi_i\otimes \varepsilon^{s_i}=\chi_1$ for $i=1,\cdots,n$. Let $0\neq \psi \in \cL_{\FM}(D_r^{r+1}:\cR_E(\chi_r))$, then we have the following restriction maps:
\begin{align*}
\MoveEqLeft[5]  {\Hom_{L_{P_r}(\Q_p)}\Big(\big(\bigotimes_{\substack{i=1, \dots, n-1 \\ i \neq r}} \chi_1\big) \otimes ( \widehat{\pi}(\rho_r^{r+1})\otimes \varepsilon^{r-1} \circ \dett), \Ord_{P_r}(\widehat{S}(U^{\wp}, W^{\wp}))[\fm_x]\Big)}\\
\xlongrightarrow{\sim}{}&\Hom_{L_{P_r}(\Q_p)}\Big(\big(\bigotimes_{\substack{i=1, \dots, n-1 \\ i \neq r}} \chi_1\big) \otimes ( \widehat{\pi}(\rho_r^{r+1})^{\an}\otimes \varepsilon^{r-1} \circ \dett), \Ord_{P_r}(\widehat{S}(U^{\wp}, W^{\wp}))[\fm_x]^{\an}\Big)\\
\lra{}& \Hom_{L_{P_r}(\Q_p)}\Big(\big(\bigotimes_{\substack{i=1, \dots, n-1 \\ i \neq r}} \chi_1\big) \otimes ( \pi(0,\psi)^-\otimes (\chi_1 \circ \dett)), \Ord_{P_r}(\widehat{S}(U^{\wp}, W^{\wp}))[\fm_x]^{\an}\Big)\\
\lra{}& \Hom_{L_{P_r}(\Q_p)}\Big(\Big(\big(\bigotimes_{\substack{i=1, \cdots, n-1\\ i\neq r}} 1\big) \otimes \St_2^{\infty}\Big)\otimes(
\chi_1\circ \dett), \Ord_{P_r}(\widehat{S}(U^{\wp}, W^{\wp}))[\fm_x]^{\an}\Big)
\end{align*}
where the first isomorphism follows from the fact that the universal completion of $\widehat{\pi}(\rho_r^{r+1})^{\an}\cong \pi(0,\psi)\otimes \chi_1 \circ \dett$ is $\widehat{\pi}(\rho_r^{r+1})$ (\cite{CD} and see \S~\ref{sec: hL-ext1} for $\pi(0,\psi)$ and $\pi(0, \psi)^-$). Using \cite[Cor. 3.4]{Br13I}, (\ref{equiforP}) and Lemma \ref{lem: lg2-jacncpt}, we deduce that the second and third morphisms are injective by the same type of argument as in the proof of Proposition \ref{equ: lg1-adjan1}. By the same arguments as in \cite[\S~6.4\ Cas\ $i=1$]{Br16} using Lemma \ref{lem: Pord-comp}(2) and Lemma \ref{lem: lg2-jacncpt}, one can prove that the second morphism is moreover surjective (see also the end of the proof of Proposition \ref{equ: lg1-adjan1} for analogous considerations). By Proposition \ref{prop: lg2-sim} and an easier variation of step (c) in the proof of Theorem \ref{thm: lg2-GL3} below, it follows that the third morphism is also surjective (see also the proof of \cite[Prop. 12]{Ding3} for similar arguments). The last assertion follows from (\ref{equ: lg2-GL3lalg}), (\ref{equ: lg1-ordPr}) and Proposition \ref{prop: ord-adj}.
\end{proof}

\begin{remark}
{\rm Corollary \ref{coro: lg2-simpL} would actually be an easy consequence of Theorem \ref{thm: lg2-lg}, but we prove it here {\it without} the assumption in Theorem \ref{thm: lg2-lg}. This is important as it is used in the proof of the main result.}
\end{remark}

\subsubsection{Higher $\cL$-invariants}\label{sec: lg2-1}

\noindent
The main result of this section is Proposition \ref{prop: lg2-L-inv}, which can be seen as a version of Proposition \ref{prop: lg2-sim} for higher $\cL$-invariants. We still work in arbitrary dimension.\\

\noindent
We keep the notation and assumptions of \S\S~\ref{fewprel}~\&~\ref{simpleinv}. We fix $r\in \{1, \cdots, n-1\}$ and set $P:=P_r$ (so $p>n\geq 3$). Since $\overline{\rho}_{\widetilde{\wp}}$ is strictly $B$-ordinary, one can check that $\overline{\rho}_{\widetilde{\wp}}$ is strictly $P$-ordinary. With the notation of \S~\ref{sec: ord-galoisdef} we have $k=n-1$ and:
\begin{equation*}
\overline{\rho}_i = \begin{cases}
  \overline{\chi}_i & i\in \{1,\cdots,r-1\} \\
  {\rm nonsplit\ extension\ of}\ \overline{\chi}_{r+1}\ {\rm by}\ \overline{\chi}_r & i=r \\
  \overline{\chi}_{i+1} & i\in \{r+1,\cdots,n-1\}
 \end{cases}
\end{equation*}
with $\overline{\rho}_r$ satisfying (\ref{hypo: hL-modp}).

\begin{lemma}\label{lem: lg2-lfree2}
The $\widetilde{\bT}(U^{\wp})^{P-\ord}_{\overline{\rho}}[1/p]$-module $M_P(U^{\wp})[1/p]$ is locally free at $x$.
\end{lemma}
\begin{proof}
By (\ref{equ: lg1-Xx}) and the last statement in Corollary \ref{coro: lg2-simpL} we have $m_P(x)=m(\rho)$. Together with Lemma \ref{lem: lg1-mult1}, the lemma then follows by the same argument as in the proof of Lemma \ref{lem: lg2-lfree1}.
\end{proof}

\noindent
We denote by $V_x$ the tangent space of $(\Spf \widetilde{\bT}(U^{\wp})^{P-\ord}_{\overline{\rho}})^{\rig}$ at $x$ and by $\overline{d}\omega_x$ the induced morphism:
\begin{equation}\label{domegabar}
\overline{d}\omega_x: V_x \lra \bigoplus_{i=1, \cdots, n-1} \big(\Ext^1_{\Gal_{\Q_p}}(\rho_{x_{i}}, \rho_{x_{i}})/\Ext^1_g(\rho_{x_i}, \rho_{x_i})\big)
\end{equation}
where $\omega=(\omega_i)_{i=1,\cdots, n-1}: (\Spf \widetilde{\bT}(U^{\wp})^{P-\ord}_{\overline{\rho}})^{\rig} \lra \prod_{i=1,\cdots, n-1} (\Spf R_{\overline{\rho}_i})^{\rig}$ (there will be no confusion with the tangent space $V_x$ and the map $\omega$ in \S~\ref{simpleinv} and note that $\rho_{x_i}=\chi_i$ if $i<r$, $\rho_{x_i}=\chi_{i+1}$ if $i>r$ and $\rho_{x_r}=\rho_r^{r+1}$). The following lemma is analogous to Lemma \ref{lem: lg2-tang1}.

\begin{lemma}\label{lem: lg2-tang2}
The morphism $\overline{d}\omega_x$ is bijective.
\end{lemma}
\begin{proof}
Since $\dim_E V_x\geq n+1$ by Proposition \ref{prop: lg1-dim} and the right hand side of (\ref{domegabar}) has dimension $(n-2)+(5-2)=n+1$ by Lemma \ref{lem: hL-ext2} and Lemma \ref{lem: hL-dRg}(3), it is enough to prove that $\overline{d}\omega_x$ is injective.\\
\noindent
(a) Let $v\in V_x$, $\cI_v:=\Ker(\widetilde{\bT}(U^{\wp})^{P-\ord}_{\overline{\rho}}\rightarrow E[\epsilon]/\epsilon^2)$ the ideal attached to $v$ (so $\widetilde{\bT}(U^{\wp})^{P-\ord}_{\overline{\rho}}/\cI_v\!\cong \co_E[\epsilon]/\epsilon^2$) and $\widetilde{\rho}_i$ for $i=1, \cdots, n-1$ the extension of $\rho_{x_i}$ by $\rho_{x_i}$ associated to $d\omega_{i,x}(v)$. Denote by $\widetilde{\pi}_i:=(\pi_i(U^{\wp})\otimes_{\widetilde{\bT}(U^{\wp})^{P-\ord}_{\overline{\rho}}} \widetilde{\bT}(U^{\wp})^{P-\ord}_{\overline{\rho}}/\cI_v)[1/p]$ (cf. (\ref{equ: lg1-pii})), which is isomorphic to the unitary Banach representation of $\GL_{n_i}(\Q_p)$ over $E[\epsilon]/\epsilon^2$ attached to $\widetilde{\rho}_i$ via (\ref{equ: hL-pLL0}), Proposition \ref{prop: hL-gl2pLL} and Remark \ref{rem: hL-pLL2}(2). Note that for $i\ne r$ we have $\widetilde{\pi}_i\cong \widetilde{\rho}_i$ as characters of $\Q_p^{\times}$ over $E[\epsilon]/\epsilon^2$. We set (cf. (\ref{equ: lg1-piup})):
\begin{equation*}
\pi:=\big(\pi^{\otimes}_P(U^{\wp}) \otimes_{\widetilde{\bT}(U^{\wp})^{P-\ord}_{\overline{\rho}}} \widetilde{\bT}(U^{\wp})^{P-\ord}_{\overline{\rho}}/\fp_x\big)[1/p]\cong \Big(\bigotimes_{\substack{i=1,\cdots, n-1\\ i\neq r}}\!\!\! \rho_{x_i}\otimes \varepsilon^{s_i}\Big)\otimes_E (\widehat{\pi}(\rho_{x_r})\otimes \varepsilon^{r-1}\circ \dett)
\end{equation*}
\begin{equation*}
\widetilde{\pi}:=\big(\pi^{\otimes}_P(U^{\wp}) \otimes_{\widetilde{\bT}(U^{\wp})^{P-\ord}_{\overline{\rho}}} \widetilde{\bT}(U^{\wp})^{P-\ord}_{\overline{\rho}}/\cI_v\big)[1/p]\cong \Big(\bigotimes_{\substack{i=1,\cdots, n-1\\ i\neq r}} \widetilde{\rho}_i\otimes \varepsilon^{s_i}\Big)\otimes_{E[\epsilon]/\epsilon^2} (\widetilde{\pi}_r \otimes \varepsilon^{r-1}\circ \dett)
\end{equation*}
where the tensor product of the $\widetilde{\rho}_i$ in the last term is over $E[\epsilon]/\epsilon^2$. Since $\widetilde{\pi}_i$ is free of rank one over $E[\epsilon]/\epsilon^2$ for $i\neq r$, we see that $\widetilde{\pi}$ is isomorphic to an extension of $\pi$ by $\pi$.
Since $M_P(U^{\wp})[1/p]$ is locally free at $x$ by Lemma \ref{lem: lg2-lfree2}, by (\ref{alglin}) and the discussion that follows we see that the evaluation map (\ref{equ: lg2-ev}) induces a commutative diagram:
\begin{equation}\label{equ: lg2-linvP}
\begin{CD}
\pi^{m_P(x)}@>>> \Ord_P(\widehat{S}(U^{\wp}, W^{\wp})_{\overline{\rho}})[\fm_\rho] \\
@V \iota VV @VVV\\
\widetilde{\pi}^{m_P(x)} @>>> \Ord_P(\widehat{S}(U^{\wp}, W^{\wp})_{\overline{\rho}})[\cI_v]
\end{CD}
\end{equation}
where the vertical maps are the natural injections (coming from $\pi\subseteq \widetilde{\pi}[\epsilon]$ for the first) and where the top horizontal map is also injective by Corollary \ref{coro: lg2-simpL} (and its proof).\\
\noindent
(b) We prove that the injection $\pi\xrightarrow{\iota} \widetilde{\pi}$ has image exactly $\epsilon \widetilde{\pi}$. It is enough to prove that $\iota$ induces $\pi\buildrel\sim\over\rightarrow \widetilde{\pi}[\epsilon]$ (since then we have a short exact sequence $0\rightarrow \pi \buildrel\iota\over\rightarrow \widetilde{\pi} \rightarrow \epsilon \widetilde{\pi}\rightarrow0$ and we use that $\widetilde{\pi}$ is an extension of $\pi$ by $\pi$). From \cite[Lem.~3.1.17]{Em04} we deduce isomorphisms:
\begin{eqnarray*}
\Big(\big(X_P(U^{\wp})\widehat\otimes_{\widetilde{\bT}(U^{\wp})_{\overline{\rho}}^{P-\ord}}\pi^{\otimes}_P(U^{\wp})\big)[\cI_v]\Big)[1/p]&\cong &(X_P(U^{\wp})[\cI_v])[1/p]\otimes_{E[\epsilon]/\epsilon^2}\widetilde{\pi}\cong \widetilde{\pi}^{m_P(x)}\\
\Big(\big(X_P(U^{\wp})\widehat\otimes_{\widetilde{\bT}(U^{\wp})_{\overline{\rho}}^{P-\ord}}\pi^{\otimes}_P(U^{\wp})\big)[\fp_x]\Big)[1/p]&\cong &(X_P(U^{\wp})[\fp_x])[1/p]\otimes_{E[\epsilon]/\epsilon^2}{\pi}\cong {\pi}^{m_P(x)}
\end{eqnarray*}
using that $(X_P(U^{\wp})[\cI_v])[1/p]$ is free of rank $m_P(x)$ over $E[\epsilon]/\epsilon^2$ by the same argument as in the proof of Lemma \ref{lem: lg2-tang1}. The result follows using $(\cdot)[\cI_v][\epsilon]=(\cdot)[\fp_x]$.\\
\noindent
(c) Suppose now that we have $\overline{d}\omega_x(v)=0$. Then it follows that $\widetilde{\pi}_i^{\lalg}=\widetilde{\pi}_i=\widetilde{\rho}_i$ is locally algebraic when $i\ne r$ (use (\ref{ginfty})) and that $\widetilde{\pi}_r^{\lalg}$ is an extension of $\widehat{\pi}(\rho_{x_r})^{\lalg}$ by $\widehat{\pi}(\rho_{x_r})^{\lalg}$ when $i=r$ (use (\ref{equ: hL-pLLdR})). In particular we have a commutative diagram:
\begin{equation}\label{equ: lg2-split1}\begin{CD}
 0 @>>> \pi^{\lalg} @>\iota >> \widetilde{\pi}^{\lalg} @>>> \pi^{\lalg} @>>> 0 \\
 @. @VVV @VVV @VVV @. \\
 0 @>>> \pi @>\iota >> \widetilde{\pi} @>>> \pi @>>> 0
 \end{CD}
\end{equation}
where the vertical maps are the natural inclusions. By (b), the multiplication by $\epsilon$ on $\widetilde{\pi}$ factors as $\widetilde{\pi}\twoheadrightarrow {\pi}\buildrel\sim\over\rightarrow \epsilon\widetilde{\pi}\hookrightarrow \widetilde{\pi}$. It follows from (\ref{equ: lg2-split1}) that the multiplication by $\epsilon$ on $\widetilde{\pi}^{\lalg}$ also factors as $\widetilde{\pi}^{\lalg}\twoheadrightarrow {\pi}^{\lalg}\buildrel\sim\over\rightarrow \epsilon\widetilde{\pi}^{\lalg}\hookrightarrow \widetilde{\pi}^{\lalg}$, in particular we have $\iota({\pi}^{\lalg})=\epsilon \widetilde{\pi}^{\lalg}$ inside $\widetilde{\pi}^{\lalg}$. From $\fm_\rho^2\subseteq \cI_v[1/p]$ and Proposition \ref{prop: lg1-semisimp}, any morphism $\widetilde{\pi}^{\lalg} \ra \Ord_P(\widehat{S}(U^{\wp}, W^{\wp})_{\overline{\rho}})[\cI_v]$ factors through $\widetilde{\pi}^{\lalg} \ra \Ord_P(\widehat{S}(U^{\wp}, W^{\wp})_{\overline{\rho}}^{\lalg})[\fm_{\rho}]$. It then follows that the bottom horizontal morphism in (\ref{equ: lg2-linvP}), which is $E[\epsilon]/\epsilon^2$-linear, sends $(\epsilon \widetilde{\pi}^{\lalg})^{m_P(x)}=\iota(\pi^{\lalg})^{m_P(x)}$ to $0$, which contradicts the injections in (\ref{equ: lg2-linvP}). The lemma follows.
\end{proof}

\noindent
We consider the $E$-linear injection $\xi: \Hom(\Q_p^{\times}, E) \hookrightarrow \Ext^1_{\Gal_{\Q_p}}(\rho_r^{r+1}, \rho_r^{r+1}), \ \psi \mapsto \rho_r^{r+1}\otimes_E (1+\psi\epsilon)$ and set $d\omega_{r,x}^+:=d\omega_{r,x} -\xi \circ d \omega_{r+1,x}$ (if $r<n-1$) and $d\omega_{r,x}^-:=d\omega_{r,x} -\xi \circ d \omega_{r-1,x}$ (if $r>1$). The following result is somewhat analogous to Proposition \ref{prop: lg2-sim} (see \S~\ref{sec: hL-FM} for $\cL_{\FM}(\cdot)$ and $\ell_{\FM}(\cdot)$).

\begin{proposition}\label{prop: lg2-L-inv}
(1) Inside $\Ext^1_{(\varphi,\Gamma)}(D_{r}^{r+1}\!, D_{r}^{r+1})\!\cong \!\Ext^1_{\Gal_{\Q_p}}\!(\rho_r^{r+1}\!, \rho_r^{r+1})$ we have $\Ima (d\omega_{r,x}^+)\!\subseteq \cL_{\FM}(D_{r}^{r+2}:D_{r}^{r+1})$ (if $r<n-1$) and $\Ima (d\omega_{r,x}^-) \subseteq \cL_{\FM}(D_{r-1}^{r+1}: D_{r}^{r+1})$ (if $r>1$).\\
\noindent
(2) If $r<n-1$ (resp. if $r>1$) the composition:
\begin{equation*}
\Ima (d\omega_{r,x}^+) \hooklongrightarrow \Ext^1_{(\varphi,\Gamma)}(D_{r}^{r+1}, D_{r}^{r+1}) \twoheadlongrightarrow \Ext^1_{(\varphi,\Gamma)}(D_{r}^{r+1}, \cR_E(\chi_{r+1}))
\end{equation*}
\begin{equation*}
\text{(resp. }\Ima (d\omega_{r,x}^-) \hooklongrightarrow \Ext^1_{(\varphi,\Gamma)}(D_{r}^{r+1}, D_{r}^{r+1}) \twoheadlongrightarrow \Ext^1_{(\varphi,\Gamma)}(\cR_E(\chi_{r}), D_{r}^{r+1})\text{)}
\end{equation*}
induces a surjective map $\Ima (d\omega_{r,x}^+) \twoheadrightarrow \ell_{\FM}(D_{r}^{r+2}: D_{r}^{r+1})\subseteq \Ext^1_{(\varphi,\Gamma)}(D_{r}^{r+1}, \cR_E(\chi_{r+1}))$ (resp. $\Ima (d\omega_{r,x}^-) \twoheadrightarrow \ell_{\FM}(D_{r-1}^{r+1}: D_{r}^{r+1})\subseteq \Ext^1_{(\varphi,\Gamma)}(\cR_E(\chi_{r}),D_{r}^{r+1})$).
\end{proposition}
\begin{proof}
(1) From (\ref{equ: ord-RT}) we have $\omega': (\Spf R_{\overline{\rho}_{\widetilde{\wp}}}^{P-\ord})^{\rig} \ra (\Spf \widetilde{\bT}(U^{\wp})_{\overline{\rho}}^{P-\ord})^{\rig}$. Let $0\neq v\in V_x$ and $\widetilde{\rho}$ (resp. $\widetilde{\rho}_{i}$) the $\Gal_{\Q_p}$-representation over $E[\epsilon]/\epsilon^2$ attached to $d\omega'_{x}(v)$ (resp. $d\omega_{i,x}(v)$). We know that $\widetilde{\rho}$ (resp. $\widetilde{\rho}$) is a deformation of $\rho_{\widetilde{\wp}}$ (resp. $\rho_{x_i}$) over $E[\epsilon]/\epsilon^2$, and using Proposition \ref{prop: ord-Pdef2}(2) we see as in the proof of Proposition \ref{prop: lg2-sim} that $\widetilde{\rho}$ is isomorphic to a successive extension of $\widetilde{\rho}_i$ as representations over $E[\epsilon]/\epsilon^2$. (1) follows then from Theorem \ref{thm: hL-hL}.\\
\noindent
(2) We prove the statement for $\Ima (d\omega_{r,x}^+)$ (and $r<n-1$), the other case being similar. By Corollary \ref{codimension}(2) and Lemma \ref{lem: hL-ext2} (and the assumptions on $\rho_{\widetilde{\wp}}$ in \S~\ref{fewprel}), we have $\dim_E \ell_{\FM}(D_r^{r+2}: D_r^{r+1})=2$. Recall that we have by Lemma \ref{lem: hL-ext2} and Lemma \ref{lem: hL-dRg}(3):
$$\dim_E \Ext^1_{(\varphi,\Gamma)}(D_{r}^{r+1}, D_{r}^{r+1})/\Ext^1_g(D_{r}^{r+1}, D_{r}^{r+1})=5-2=3.$$
By Lemma \ref{lem: lg2-tang2}, it is then not difficult to deduce $\dim_E \Ima (d\omega_{r,x}^+)\geq 3$. By (1) and (\ref{sesker}), we have an exact sequence (see (\ref{equ: l3-diag2}) for the morphism $\kappa$):
\begin{equation*}
0 \lra \Ima (d\omega_{r,x}^+)\cap \Ker (\kappa) \lra \Ima (d\omega_{r,x}^+) \lra \ell_{\FM}(D_{r}^{r+2}: D_{r}^{r+1}).
\end{equation*}
We have $\dim_E \Ker(\kappa)=2$ by (\ref{sesker}), Corollary \ref{codimension}(2) and $\dim_E \ell_{\FM}(D_r^{r+2}: D_r^{r+1})=2$. If $\dim_E \Ima (d\omega_{r,x}^+)\geq 4$, the result is thus clear. Assume $\dim_E \Ima (d\omega_{r,x}^+)=3$, it is enough to prove $\dim_E \Ima (d\omega_{r,x}^+)\cap \Ker(\kappa)\leq 1$. From Lemma \ref{lem: hL-kk} and Lemma \ref{lem: hL-dRg}(1)\&(2), we deduce $\dim_E \Ker (\kappa) \cap \Ext^1_g(D_{r}^{r+1}, D_{r}^{r+1})=1$. It easily follows from Lemma \ref{lem: lg2-tang2} that the morphism $\overline{d}\omega_{r,x}^+:=\overline{d}\omega_{r,x}-\xi\circ \overline{d}\omega_{r+1,x}$ is surjective (note that it is well-defined since $\xi$ sends $\Hom_{\infty}(\Q_p^{\times}, E)$ to $\Ext^1_g(\rho_r^{r+1}, \rho_r^{r+1})$), which implies that the composition:
\begin{equation}\label{equ: lg2-compSur}
\Ima (d\omega_{r,x}^+) \hooklongrightarrow \Ext^1_{(\varphi,\Gamma)}(D_{r}^{r+1}, D_{r}^{r+1}) \twoheadrightarrow \Ext^1_{(\varphi,\Gamma)}(D_{r}^{r+1}, D_{r}^{r+1})/\Ext^1_g(D_{r}^{r+1}, D_{r}^{r+1})
\end{equation}
is also surjective, hence bijective as source and target have dimension $3$. If $\dim_E \Ima (d\omega_{r,x}^+)\cap \Ker (\kappa)=2$, we have $\Ker (\kappa)\subseteq \Ima (d\omega_{r,x}^+)$ since $\dim_E \Ker(\kappa)=2$, and thus $\Ima (d\omega_{r,x}^+)\cap \Ext^1_g(D_{r}^{r+1}, D_{r}^{r+1}) \neq 0$ as $\Ker (\kappa) \cap \Ext^1_g(D_{r}^{r+1}, D_{r}^{r+1})\ne 0$, which contradicts the fact (\ref{equ: lg2-compSur}) is bijective. This concludes the proof.
\end{proof}

\subsubsection{Local-global compatibility for $\GL_3(\Q_p)$}

\noindent
In dimension $3$, we finally use most of the previous material to prove our main local-global compatibility result (Corollary \ref{main}).\\

\noindent
We keep all the notation of \S\S~\ref{fewprel}, \ref{simpleinv}, \ref{sec: lg2-1} and now assume $n=3$ (and thus $p>3$). For $r=1, 2$, we let $\cL_r\in E$ such that:
\begin{equation*}
\psi_{\cL_r}:=\log_p -\cL_r \val_p \in \cL_{\FM}(D_r^{r+1}: \cR_E(\chi_r))\subsetneq \Hom(\Q_p^\times,E).
\end{equation*}
We set $\lambda:=(\wt(\chi_1), \wt(\chi_1),\wt(\chi_1))\in \Z^3$ and let $\alpha\in E^\times$ such that $\chi_1=\unr(\alpha)x^{\wt(\chi_1)}$. We define $v_{\overline{P}_r}^{\infty}(\lambda)$ for $r=1,2$ as in \S~\ref{sec: hL-GL30}, $\widetilde\Pi^1(\lambda, \psi_{\cL_1})$ as in (\ref{sigmatilde1}) and set $v_{\overline{P}_r}^{\infty}(\alpha,\lambda):=v_{\overline{P}_r}^{\infty}(\lambda)\otimes \unr(\alpha)\circ\dett$, $\widetilde\Pi^1(\alpha,\lambda, \psi_{\cL_1}):=\widetilde\Pi^1(\lambda, \psi_{\cL_1})\otimes \unr(\alpha)\circ\dett$ and $\cL_{\aut}(D:D_1^2)\subseteq \Ext^1_{{\GL_3}(\Q_p)}(v_{\overline{P}_2}^{\infty}(\alpha,\lambda),\widetilde\Pi^1(\alpha,\lambda, \psi))$ as in (\ref{callaut}) (tensoring by $\unr(\alpha)\circ\dett$). The assumptions on $\rho_{\widetilde{\wp}}$ imply in particular that $D$ is sufficiently generic in the sense of (the end of) \S~\ref{prelprel}, and we can then define $\sE(\widetilde\Pi^1(\alpha,\lambda, \psi_{\cL_1}), v_{\overline{P}_2}^{\infty}(\alpha,\lambda)^{\oplus 2}, \cL_{\aut}(D:D_1^2))$ as in Notation \ref{not: hL-ext} and set (see (\ref{equ: L3-tildPi}) when $\alpha=1$):
$$\widetilde{\Pi}^1(D)^-:=\sE\big(\widetilde\Pi^1(\alpha,\lambda, \psi_{\cL_1}), v_{\overline{P}_2}^{\infty}(\alpha,\lambda)^{\oplus 2}, \cL_{\aut}(D:D_1^2)\big).$$
Likewise we define $\widetilde\Pi^2(\alpha,\lambda, \psi_{\cL_2}):=\widetilde\Pi^2(\lambda, \psi_{\cL_2})\otimes \unr(\alpha)\circ\dett$ (see before \S~\ref{linvariantgl3}), $\cL_{\aut}(D:D_2^3)\subseteq \Ext^1_{{\GL_3}(\Q_p)}(v_{\overline{P}_1}^{\infty}(\alpha,\lambda),\widetilde\Pi^2(\alpha,\lambda, \psi))$ (see (\ref{aut2})) and we set (see (\ref{pitilde2}) when $\alpha=1$):
$$\widetilde{\Pi}^2(D)^-:=\sE\big(\widetilde\Pi^2(\alpha,\lambda, \psi_{\cL_2}), v_{\overline{P}_1}^{\infty}(\alpha,\lambda)^{\oplus 2}, \cL_{\aut}(D:D_2^3)\big).$$

\begin{theorem}\label{thm: lg2-GL3}
For $r\in \{1, 2\}$, the following restriction morphism is bijective:
\begin{equation}\small\label{equ: lg2-linv}
\Hom_{\GL_3(\Q_p)}\!\big(\widetilde{\Pi}^r(D)^-\!, \widehat{S}(U^{\wp}, W^{\wp})_{\overline{\rho}}[\fm_{\rho}]\big)
\!\xlongrightarrow{\sim} \!\Hom_{\GL_3(\Q_p)}\!\big(\!\St_3^{\infty} \otimes_E \chi_1\circ \dett, \widehat{S}(U^{\wp}, W^{\wp})_{\overline{\rho}}[\fm_{\rho}]\big).
\end{equation}
\end{theorem}
\begin{proof}
We only prove the case $r=1$, the case $r=2$ being symmetric.\\
\noindent
(a) It follows from (\ref{equ: ord-ncomp}) that (\ref{equ: lg2-linv}) is injective (by the usual argument: if (\ref{equ: lg2-linv}) is not injective, there exists an irreducible constituent $V$ of $\widetilde{\Pi}^1(D)^-$ such that $V\hookrightarrow \widehat{S}(U^{\wp}, W^{\wp})_{\overline{\rho}}[\fm_{\rho}]$, hence $J_B(V)\hookrightarrow J_B(\widehat{S}(U^{\wp}, W^{\wp})_{\overline{\rho}}[\fm_{\rho}])$, which contradicts (\ref{equ: ord-ncomp}) using \cite[Cor. 3.4]{Br13I}).\\
\noindent
(b) We have natural morphisms:
\begin{align}\label{equ: lg2-isomos}
\MoveEqLeft[10]  {\Hom_{\GL_3(\Q_p)}\big(\St_3^\infty\otimes (\chi_1 \circ \dett), \widehat{S}(U^{\wp}, W^{\wp})_{\overline{\rho}}[\fm_{\rho}]\big)} & \nonumber \\
 \xlongrightarrow{\sim}{}& \Hom_{L_{P_1}(\Q_p)}\big((\St_2^{\infty} \otimes 1) \otimes (\chi_1 \circ \dett), \Ord_{P_1}(\widehat{S}(U^{\wp}, W^{\wp})_{\overline{\rho}}[\fm_{\rho}])\big) \nonumber\\
\xlongrightarrow{\sim}{}& \Hom_{L_{P_1}(\Q_p)}\big(\widehat{\pi}(\rho_1^2)\otimes \chi_1, \Ord_{P_1}(\widehat{S}(U^{\wp}, W^{\wp})_{\overline{\rho}}[\fm_{\rho}])\big)\nonumber\\
\xlongrightarrow{\sim}{}& \Hom_{\GL_3(\Q_p)}\big((\Ind_{\overline{P}_1(\Q_p)}^{\GL_3(\Q_p)}\widehat{\pi}(\rho_1^2)\otimes \chi_1)^{\cC^0}, \widehat{S}(U^{\wp}, W^{\wp})_{\overline{\rho}}[\fm_{\rho}]\big)\nonumber\\
\hooklongrightarrow{}& \Hom_{\GL_3(\Q_p)}\big((\Ind_{\overline{P}_1(\Q_p)}^{\GL_3(\Q_p)}\widehat{\pi}(\rho_1^2)^{\an}\otimes \chi_1)^{\an}, \widehat{S}(U^{\wp}, W^{\wp})_{\overline{\rho}}[\fm_{\rho}]^{\an}\big)\nonumber\\
\xlongrightarrow{\sim}{}& \Hom_{\GL_3(\Q_p)}\big(\widetilde\Pi^1(\alpha,\lambda, \psi_{\cL_1}), \widehat{S}(U^{\wp}, W^{\wp})_{\overline{\rho}}[\fm_{\rho}]^{\an}\big)
\end{align}
where the first map is given by Lemma \ref{lem: ord-alg4} together with (\ref{equ: lg1-ordPr}) and is bijective by Proposition \ref{prop: ord-adj} together with (\ref{equ: lg2-GL3lalg}), the second isomorphism follows from Corollary \ref{coro: lg2-simpL}, the third from \cite[Thm. 4.4.6]{EOrd1}, the fourth map is injective since the locally analytic vectors are dense in the corresponding Banach representation, and where the last bijection follows from the fact that any irreducible constituent of the kernel $W$ of the surjection $(\Ind_{\overline{P}_1(\Q_p)}^{\GL_3(\Q_p)}\widehat{\pi}(\rho_1^2)^{\an}\otimes \chi_1)^{\an} \twoheadrightarrow \widetilde\Pi^1(\alpha,\lambda, \psi_{\cL_1})$ does not occur in $\soc_{\GL_3(\Q_p)} \widehat{S}(U^{\wp}, W^{\wp})_{\overline{\rho}}[\fm_{\rho}]^{\an}$ (see the discussion below (\ref{equ: hL-L3c}) and argue as in (a)). One can check by using the functor $J_B(\cdot)$ that the composition in (\ref{equ: lg2-isomos}) gives a section of the restriction morphism:
\begin{multline}\label{equ: lg2-GL3simp}
\Hom_{\GL_3(\Q_p)}\big(\widetilde\Pi^1(\alpha,\lambda, \psi_{\cL_1}), \widehat{S}(U^{\wp}, W^{\wp})_{\overline{\rho}}[\fm_{\rho}]\big)\\
\lra \Hom_{\GL_3(\Q_p)}\big(\St_3^\infty\otimes (\chi_1 \circ \dett), \widehat{S}(U^{\wp}, W^{\wp})_{\overline{\rho}}[\fm_{\rho}]\big),
\end{multline}
which is therefore surjective. Since (\ref{equ: lg2-GL3simp}) is injective by (again) the same argument as in (a), it follows that (\ref{equ: lg2-GL3simp}) is bijective. Consequently, the fourth injection in (\ref{equ: lg2-isomos}) is also bijective.\\
\noindent
(c) By (a) and (b), it is enough to prove that, for any line $Ew\subseteq \cL_{\aut}(D: D_1^2)$, setting $\Pi:=\sE(\widetilde\Pi^1(\alpha,\lambda, \psi_{\cL_1}), v_{\overline{P}_2}^{\infty}(\alpha,\lambda), Ew)$ (see Notation \ref{not: hL-ext}, in fact this is just here the representation associated to the extension $w$), the following restriction morphism is surjective:
\begin{equation}\small\label{equ: lg2-GL3}
\Hom_{\GL_3(\Q_p)}\big(\Pi, \widehat{S}(U^{\wp}, W^{\wp})_{\overline{\rho}}[\fm_{\rho}]\big)
\lra \Hom_{\GL_3(\Q_p)}\big(\St_3^\infty\otimes (\chi_1 \circ \dett), \widehat{S}(U^{\wp}, W^{\wp})_{\overline{\rho}}[\fm_{\rho}]\big).
\end{equation}
As in (\ref{equ: lg2-isomos}), we have:
\begin{multline}\label{equ: lg2-GL3b}
\Hom_{\GL_3(\Q_p)}\big(\St_3^\infty\otimes (\chi_1 \circ \dett), \widehat{S}(U^{\wp}, W^{\wp})_{\overline{\rho}}[\fm_{\rho}]\big)\\
\xlongrightarrow{\sim} \Hom_{L_{P_1}(\Q_p)}\big(\widehat{\pi}(\rho_1^2)\otimes \chi_1, \Ord_{P_1}(\widehat{S}(U^{\wp}, W^{\wp})_{\overline{\rho}}[\fm_{\rho}])\big)\cong X_P(U^{\wp})[\fp_x] \otimes_{\co_E} E
\end{multline}
where we use the notation in \S~\ref{sec: lg2-1}. Let $0\neq w\in \cL_{\aut}(D: D_1^2)\cong \ell_{\FM}(D:D_1^2)$ (cf. (\ref{rem: hL-LgalLaut})), by Proposition \ref{prop: lg2-L-inv}(2) there exists $v\in V_x$ such that $d \omega_{1, x}^+(v)\mapsto w \in \Ext^1_{(\varphi,\Gamma)}(D_1^2, \cR_E(\chi_2))$. Denote by $\cI_v$ the ideal of $\tilde{\bT}(U^{\wp})^{P_1-\ord}_{\overline{\rho}}$ attached to $v$ (see e.g. the beginning of the proof of Lemma \ref{lem: lg2-tang2}), recall we have (see e.g. (b) in the proof of Lemma \ref{lem: lg2-tang2}):
\begin{equation}\label{equ: lg2-tgfree}
\text{$(X(U^{\wp})[\cI_v])[1/p]$ is free over $E[\epsilon]/\epsilon^2$.}
\end{equation}
Let $f_0$ be a nonzero element in the right hand side of (\ref{equ: lg2-GL3}), and let $f: \widehat{\pi}(\rho_1^2)\otimes \chi_1 \hookrightarrow \Ord_{P_1}(\widehat{S}(U^{\wp}, W^{\wp})_{\overline{\rho}}[\fm_{\rho}])$ and $e\in (X_P(U^{\wp})[\fp_x])[1/p]$ be the corresponding elements via (\ref{equ: lg2-GL3b}). By (\ref{equ: lg2-tgfree}) and $X_P(U^{\wp})[\fp_x]=X_P(U^{\wp})[\cI_v][\epsilon]$, there exists $\tilde{e}\in (X_P(U^{\wp})[\cI_v])[1/p]$ such that $\epsilon \tilde{e}=e$. As in (\ref{equ: lg2-linvP}), letting $\widetilde{\pi}_1$ (resp. $\widetilde \chi_1$) be the deformation of $\widehat{\pi}(\rho_1^2)$ (resp. $\chi_1$) over $E[\epsilon]/\epsilon^2$ attached to $d\omega_{1,x}(v)$ (resp. $d\omega_{2,x}(v)$), we have a commutative diagram:
\begin{equation}\label{equ: lg2-GL3ordx}
\begin{CD}
\pi: = \widehat{\pi}(\rho_1^2)\boxtimes_E \chi_1 @> f>> \Ord_P(\widehat{S}(U^{\wp}, W^{\wp})_{\overline{\rho}})[\fm_{\rho}] \\
@V\iota VV @VVV\\
\widetilde{\pi}:=\widetilde{\pi}_1\boxtimes_{E[\epsilon]/\epsilon^2} \widetilde\chi_1 @>\tilde{f}>> \Ord_P(\widehat{S}(U^{\wp}, W^{\wp})_{\overline{\rho}})[\cI_v]
\end{CD}
\end{equation}
where $\tilde{f}$ is the morphism corresponding to $\tilde{e}$ and where we write $\boxtimes$ instead of $\otimes$ to emphasize that it is an {\it exterior} tensor product of representations ($\GL_2(\Q_p)$ acting on the left and $\Q_p^\times$ on the right). Using Proposition \ref{prop: lg2-L-inv}(1), let $w_0:=d \omega_{1, x}^+(v)\in \cL_{\FM}(D: D_1^2)$ and $\widetilde{\pi}_{w_0}$ the associated deformation of $\widehat{\pi}(\rho_1^2)$ over $E[\epsilon]/\epsilon^2$ via (\ref{equ: hL-pLL0}). From the definition of $d \omega_{1, x}^+$ we have:
\begin{equation}\label{equ: lg2-GL3det}
\widetilde{\pi} \cong ((\chi_1^{-1}\widetilde \chi_1)\circ \dett_{\GL_2} \otimes_{E[\epsilon]/\epsilon^2} \widetilde{\pi}_{w_0}) \boxtimes_{E[\epsilon]/\epsilon^2} \widetilde \chi_1\cong (\chi_1^{-1}\widetilde \chi_1)\circ \dett_{L_{P_1}} \otimes_{E[\epsilon]/\epsilon^2} (\widetilde{\pi}_{w_0}\boxtimes_E\chi_1).
\end{equation}
By \cite[Thm. 4.4.6]{EOrd1}, taking $(\Ind_{\overline{P}_1(\Q_p)}^{\GL_3(\Q_p)}\cdot )^{\cC^0}$ and then locally analytic vectors, the maps $\iota$ and $\widetilde f$ in (\ref{equ: lg2-GL3ordx}) induce morphisms of locally analytic representations of $\GL_3(\Q_p)$ over $E$:
\begin{equation}\label{equ: lg2-GL3ad}
 \big(\Ind_{\overline{P}_1(\Q_p)}^{\GL_3(\Q_p)} \pi^{\an}\big)^{\an}\hooklongrightarrow \big(\Ind_{\overline{P}_1(\Q_p)}^{\GL_3(\Q_p)} \widetilde{\pi}^{\an}\big)^{\an} \lra \widehat{S}(U^{\wp}, W^{\wp})_{\overline{\rho}}[\cI_v]^{\an}.
\end{equation}
Let $\widetilde{\pi}_0:=\widetilde{\pi}_{w_0} \boxtimes_E \chi_1$, from (\ref{equ: lg2-GL3det}) we deduce:
\begin{equation}\label{equ: lg2-GL3det2}
\big(\Ind_{\overline{P}_1(\Q_p)}^{\GL_3(\Q_p)} \widetilde{\pi}^{\an}\big)^{\an} \cong (\chi_1^{-1}\widetilde \chi_1)\circ \dett_{\GL_3} \otimes_{E[\epsilon]/\epsilon^2}\big(\Ind_{\overline{P}_1(\Q_p)}^{\GL_3(\Q_p)} \widetilde{\pi}_0^{\an}\big)^{\an}.
\end{equation}
As in (b), (\ref{equ: lg2-GL3ad}) factors as (see (a) for $W$):
\begin{equation}\label{equ: lg2-GL3comp}
\widetilde\Pi^1(\alpha,\lambda, \psi_{\cL_1}) \hooklongrightarrow \big(\Ind_{\overline{P}_1(\Q_p)}^{\GL_3(\Q_p)} \widetilde{\pi}^{\an}\big)^{\an}/W \lra \widehat{S}(U^{\wp}, W^{\wp})_{\overline{\rho}}[\cI_v]^{\an}.
\end{equation}
Since $w_0\in \cL_{\FM}(D: D_1^2)\mapsto w\in \ell_{\FM}(D:D_1^2)$, it follows from (\ref{equ: hL-L3p}) and Remark \ref{rk3.47} that $\Pi=\sE(\widetilde\Pi^1(\lambda, \psi_{\cL_1}), v_{\overline{P}_2}^{\infty}(\lambda), Ew)\otimes \unr(\alpha)\circ\dett$ is a subrepresentation of $(\Ind_{\overline{P}_1(\Q_p)}^{\GL_3(\Q_p)} \widetilde{\pi}_0^{\an})^{\an}/W$. By \ Lemma \ \ref{lem: hL-cent3} \ and \ (\ref{equ: lg2-GL3det2}), \ we \ deduce \ that \ $\Pi$ \ is \ also \ a \ subrepresentation \ of $(\Ind_{\overline{P}_1(\Q_p)}^{\GL_3(\Q_p)} \widetilde{\pi}^{\an})^{\an}/W$. Hence (\ref{equ: lg2-GL3comp}) induces $\GL_3(\Q_p)$-equivariant morphisms:
\begin{equation}\label{closetoend}
\widetilde\Pi^1(\alpha,\lambda, \psi_{\cL_1}) \hooklongrightarrow \Pi \xlongrightarrow{\tilde{f}} \widehat{S}(U^{\wp}, W^{\wp})_{\overline{\rho}}[\cI_v]^{\an}.
\end{equation}
As the composition in (\ref{closetoend}) restricts to $f_0$ via (\ref{equ: lg2-GL3simp}), we see it has image in $\widehat{S}(U^{\wp}, W^{\wp})_{\overline{\rho}}[\fm_{\rho}]$ (using that the analogue of (\ref{equ: lg2-GL3simp}) with $\widehat{S}(U^{\wp}, W^{\wp})_{\overline{\rho}}[\cI_v]$ instead of $\widehat{S}(U^{\wp}, W^{\wp})_{\overline{\rho}}[\fm_{\rho}]$ is still an injection). If $\fm_{\rho}\Ima(\tilde{f})\ne 0$, we deduce that $\fm_{\rho} \Ima(\tilde{f})\cong v_{\overline{P}_2}^{\infty}(\lambda)\otimes\unr(\alpha)\circ\dett$ is a subrepresentation of $\widehat{S}(U^{\wp}, W^{\wp})_{\overline{\rho}}[\cI_v]^{\an}$, a contradiction. Thus we have $\fm_{\rho}\Ima(\tilde{f})=0$, i.e. $\tilde{f}$ also has image in $\widehat{S}(U^{\wp}, W^{\wp})_{\overline{\rho}}[\fm_{\rho}]^{\an}$. The map $f\mapsto \tilde{f}$ gives a section to (\ref{equ: lg2-GL3}), which concludes the proof.
\end{proof}

\noindent
We refer to \S~\ref{sec: hL-pI}, \S~\ref{linvariantgl3} for the definition of the subrepresentations $\Pi^r(\lambda, \psi_{\cL_r})_0$, $\Pi^r(\lambda, \psi_{\cL_r})$, $\Pi^r(\lambda, \psi_{\cL_r})^+$ of $\widetilde\Pi^r(\lambda, \psi_{\cL_r})$. We set  $\Pi^r(\alpha,\lambda, \psi_{\cL_r})_0\!:=\Pi^r(\lambda, \psi_{\cL_r})_0\otimes \unr(\alpha)\circ\dett$, $\Pi^r(\alpha,\lambda, \psi_{\cL_r}):=\Pi^r(\lambda, \psi_{\cL_r})\otimes \unr(\alpha)\circ\dett$, and $\Pi^r(\alpha,\lambda, \psi_{\cL_r})^+:=\Pi^r(\lambda, \psi_{\cL_r})^+\otimes \unr(\alpha)\circ\dett$.

\begin{corollary}\label{unicity}
Let $r\in \{1,2\}$.\\
\noindent
(1) Let $\psi \in \Hom(\Q_p^{\times}, E)$, an injection $f: \St_3^\infty\otimes (\chi_1 \circ \dett) \hookrightarrow \widehat{S}(U^{\wp}, W^{\wp})_{\overline{\rho}}[\fm_{\rho}]$ extends to $\tilde{f}_1: \Pi^r(\alpha,\lambda, \psi)^+\ra \widehat{S}(U^{\wp}, W^{\wp})_{\overline{\rho}}[\fm_{\rho}]$ if and only if $\psi\in E\psi_{\cL_r}$.\\
\noindent
(2) Let $s\in \{1,2\}$, $s\neq r$, and let $v\in \Ext^1_{\GL_3(\Q_p)}(v_{\overline{P}_s}^{\infty}(\alpha,\lambda), \Pi^r(\alpha,\lambda, \psi_{\cL_r})^+)$. An injection $\Pi^r(\alpha,\lambda, \psi_{\cL_r})^+ \hookrightarrow \widehat{S}(U^{\wp}, W^{\wp})_{\overline{\rho}}[\fm_{\rho}]$ extends to:
\begin{equation*}
\sE\big(\Pi^r(\alpha,\lambda, \psi_{\cL_r})^+, v_{\overline{P}_s}^{\infty}(\alpha,\lambda), Ev\big)\longrightarrow \widehat{S}(U^{\wp}, W^{\wp})_{\overline{\rho}}[\fm_{\rho}]
\end{equation*}
if and only if $v\in \cL_{\aut}(D: D_r^{r+1})$.
\end{corollary}
\begin{proof}
For $i\in \{1,2\}$, denote by $\pi^i(\alpha,\lambda):=\St_3^{\infty}\otimes (\chi_1\circ\dett)-v_{\overline{P}_i}^{\infty}(\alpha,\lambda)$ the unique nonsplit locally algebraic extension of $v_{\overline{P}_i}^{\infty}(\alpha,\lambda)$ by $\St_3^{\infty}\otimes (\chi_1\circ\dett)$.\\
\noindent
(1) By Theorem \ref{thm: lg2-GL3}, $f$ extends to $\tilde{f}_0: \Pi^r(\alpha,\lambda, \psi_{\cL_r})^+\hookrightarrow \widehat{S}(U^{\wp}, W^{\wp})_{\overline{\rho}}[\fm_{\rho}]$, the ``if" part follows. If $\psi\notin E\psi_{\cL_r}$, we have $E\psi+ E\psi_{\cL_r}=\Hom(\Q_p^{\times}, E)$ and an injection induced by $\tilde{f}_0$, $\tilde{f}_1$ (where $S_{s,0}$ is defined as in \S~\ref{sec: hL-pI} with $\lambda=0$) :
\begin{equation*}
\Pi^r(\alpha,\lambda, \psi_{\cL_r})_0 \oplus_{S_{s,0}\otimes (\chi_1\circ\dett)} \Pi^r(\alpha,\lambda, \psi)_0\hooklongrightarrow \widehat{S}(U^{\wp}, W^{\wp})_{\overline{\rho}}[\fm_{\rho}].
\end{equation*}
Since $\val_p\in E\psi+ E\psi_{\cL_r}$, we easily deduce that the left hand side contains $\pi^r(\alpha,\lambda)$ as a subrepresentation. However $\pi^r(\alpha,\lambda)$ is not a subrepresentation of $\widehat{S}(U^{\wp}, W^{\wp})_{\overline{\rho}}[\fm_{\rho}]$ by (\ref{equ: lg2-GL3lalg}), a contradiction.\\
\noindent
(2) follows by the same argument, noting that if $v\notin \cL_{\aut}(D: D_r^{r+1})$, then one easily deduces from Lemma \ref{lem: hL-L3} that $\pi^s(\alpha,\lambda)$ is a subrepresentation of $$\sE(\Pi^r(\alpha,\lambda, \psi_{\cL_r})^+, v_{\overline{P}_s}^{\infty}(\alpha,\lambda), Ev) \oplus_{\Pi^r(\alpha,\lambda, \psi_{\cL_r})^+} \Pi^r(D)^-$$ (where $\Pi^r(D)^-\subseteq \widetilde \Pi^r(D)^-$ is defined in (\ref{equ: L3-Pi}) and (\ref{piD2}) modulo the twist by $\unr(\alpha)\circ\dett$), a contradiction.
\end{proof}

\noindent
We can now state our main result. We fix $\wp\vert p$, $\widetilde{\wp}\vert \wp$, $U^\wp=\prod_{v\ne \wp}U_v$ and $W^\wp$ as in \S~\ref{prelprel}.

\begin{corollary}\label{main}
Assume $n=3$, $F^+_\wp\cong F_{\widetilde{\wp}}=\Q_p$, $p\geq 5$ and $U_v$ maximal if $v\vert p$, $v\ne \wp$. Let $\rho: \Gal_F\ra \GL_3(E)$ be a continuous representation which is unramified at the places of $D(U^p)$ and such that:
\begin{itemize}
\item $\overline\rho$ is absolutely irreducible
\item $\widehat{S}(U^{\wp}, W^{\wp})[\fm_{\rho}]^{\lalg}\neq 0$
\item $\rho_{\widetilde{\wp}}$ is semi-stable with consecutive Hodge-Tate weights and $N^2\ne 0$ on $D_{\st}(\rho_{\widetilde{\wp}})$
\item any dimension $2$ subquotient of $\overline \rho_{\widetilde{\wp}}=\overline{\rho}|_{\Gal_{F_{\widetilde{\wp}}}}$ is nonsplit.
\end{itemize}
\noindent
Then we have the following results.\\
(1) The statement in Conjecture \ref{THEconjecture} is true, i.e. the restriction morphism is bijective:
\begin{equation*}
 \Hom_{\GL_3(\Q_p)}\big(\Pi(\rho_{\widetilde{\wp}}),\widehat{S}(U^{\wp}, W^{\wp})[\fm_{\rho}]\big)
 \xlongrightarrow{\sim} \Hom_{\GL_3(\Q_p)}\big(\Pi(\rho_{\widetilde{\wp}})^{\lalg}, \widehat{S}(U^{\wp}, W^{\wp})[\fm_{\rho}]\big).
\end{equation*}
\noindent
(2) The representation $\rho_{\widetilde{\wp}}$ of $\Gal_{\Q_p}$ is determined by the locally analytic representation $\widehat{S}(U^{\wp}, W^{\wp})[\fm_{\rho}]^{\rm an}$ of $\GL_3(\Q_p)$ (hence also by the continuous representation $\widehat{S}(U^{\wp}, W^{\wp})[\fm_{\rho}]$).
\end{corollary}
\begin{proof}
(1) By the same argument as in \cite[\S~6.2~\'Etape 1]{Br16}, we can assume that $U^p$ is sufficiently small. Define $\Pi(D)^-$ as at the end of \S~\ref{linvariantgl3}, then it follows from Theorem \ref{thm: lg2-GL3} and (a) in its proof (and arguing e.g. as in \cite[\S~6.2~\'Etape 2]{Br16}) that the statement holds with $\Pi(D)^-$ instead of $\Pi(\rho_{\widetilde{\wp}})=\Pi(D)$. By \cite[\S~6.4\ Cas\ $i\geq 3$]{Br16}, we have:
$$\Hom_{\GL_3(\Q_p)}\big(\Pi(D),\widehat{S}(U^{\wp}, W^{\wp})[\fm_{\rho}]\big)
 \xlongrightarrow{\sim} \Hom_{\GL_3(\Q_p)}\big(\Pi(D)^-, \widehat{S}(U^{\wp}, W^{\wp})[\fm_{\rho}]\big)$$
 and (1) follows. (2) is a direct consequence of Corollary \ref{unicity} (which {\it a fortiori} still holds when $U^p$ is not sufficiently small).
\end{proof}

\appendix
\renewcommand*{\thesection}{\Alph{section}}

\section{Appendix}

\noindent
The aim of this appendix is to give a complete proof of Proposition \ref{prop: hL-gl2pLL2}, for which we couldn't find precise references in the existing literature.

\subsection{Notation and preliminaries}\label{notprel}

\noindent
We recall some notation and results of Emerton and Colmez.\\

\noindent
For $G$ a topological group which is locally pro-$p$ and $A\in \Comp(\co_E)$ (see the beginning of \S~\ref{sec: ord-1}), we denote by $\Mod_{G}^{\sm}(A)$ the category of smooth representations of $G$ over $A$ in the sense of \cite[\S~2.2]{EOrd1}, $\Mod_{G}^{\fin}(A)$ the full subcategory of smooth representations of finite length and $\Mod_{G}^{\lfin}(A)$ the full subcategory of smooth representations locally of finite length (i.e. the subrepresentation generated by $v$ is of finite length for any vector $v$). We denote by $(\cdot)^\vee:=\Hom_{\co_E}(\cdot,E/\co_E)=$ Pontryagin duality.\\

\noindent
We let $\Mod_{G}^{\pro\ \!\!\!\aug}(A)$ be the category of profinite augmented representations of $G$ over $A$ in the sense of \cite[Def. 2.1.6]{EOrd1}. By \cite[(2.2.8)]{EOrd1}, the functor $\pi\mapsto \pi^\vee$ induces an anti-equivalence of categories:
\begin{equation}\label{equ: app-dual}
\Mod_{G}^{\sm}(A) \xlongrightarrow{\sim} \Mod_{G}^{\pro\ \!\!\!\aug}(A).
\end{equation}
As in \cite[\S~2.1]{EOrd1}, we denote by $\Mod_G^{\fg\ \!\!\!\aug}(A)$ the full subcategory of $\Mod_G^{\pro\ \!\!\!\aug}(A)$ consisting of augmented $G$-representations that are finitely generated over $A[[H]]$ for some (equivalently any) compact open subgroup $H$ of $G$. We denote by $\Mod_G^{\ortho}(A)$ the category of orthonormalizable admissible representations of $G$ over $A$ in the sense of \cite[Def. 3.1.11]{Em4}. By \cite[Prop. 3.1.12]{Em4}, the functor $\pi\mapsto \Hom_A(\pi,A)$ induces an anti-equivalence of categories between $\Mod_G^{\ortho}(A)$ and the full subcategory of $\Mod_G^{\fg\ \!\!\!\aug}(A)$ consisting of $G$-representations which are moreover pro-free $A$-modules.\\

\noindent
We denote by $\Rep^{\fin}_{\Gal_{\Q_p}}(\co_E)$ the category of continuous representations of $\Gal_{\Q_p}$ on finite length (hence torsion) $\co_E$-modules equipped with the discrete topology. Recall that Colmez defined a covariant exact functor (called Colmez's functor, see \cite{Colm10a}):
\begin{equation*}
\hV: \Mod^{\fin}_{\GL_2(\Q_p)}(\co_E) \lra \Rep^{\fin}_{\Gal_{\Q_p}}(\co_E).
\end{equation*}
For a continuous character $\zeta: \Q_p^{\times} \ra \co_E^{\times}$ (which we view as a continuous character of $\Gal_{\Q_p}$), we denote by $\hV_{\zeta}$ the functor $\pi \mapsto \hV(\pi) \otimes \zeta$. As in \cite[\S~3.2]{Em4}, for $A\in \Comp(\co_E)$, $\hV$ (resp. $\hV_{\zeta}$) extends to a covariant and exact functor, still denoted by $\hV$ (resp. by $\hV_{\zeta}$), from the full subcategory of $\Mod_{\GL_2(\Q_p)}^{\ortho}(A)$ consisting of $A$-representations $\pi$ such that $\pi\otimes_A A/\fm_A\in \Mod^{\fin}_{\GL_2(\Q_p)}(k_E)$ to the category of continuous $\Gal_{\Q_p}$-representations on finite rank free $A$-modules.\\

\noindent

\subsection{Deformations I}\label{sec: app-def}
\noindent The main results of this section are Corollary \ref{coro: hL-pLL1} and Corollary \ref{coro: app-NR} below.\\

\noindent
We keep the notation of \S~\ref{notprel}. We fix $\overline{\rho}: \Gal_{\Q_p}\ra \GL_2(k_E)$ a continuous representation and let $\pi(\overline{\rho})$ be the smooth representation of $\GL_2(\Q_p)$ over $k_E$ associated to $\overline{\rho}$ by the mod $p$ Langlands correspondence normalized so that $\hV_{\varepsilon^{-1}}(\pi(\overline{\rho})) \cong \overline{\rho}$ (this is the normalization of \cite[\S~3.1]{Br11b}). We assume:
\begin{equation}\label{hypo: hL-modp}
 \overline{\rho}\ncong \begin{pmatrix}
 1 & * \\
 0 & \overline{\varepsilon}
 \end{pmatrix} \text{ up to twist by a character (with $*$ zero or not)}.
\end{equation}
Note that the assumption implies that $\pi(\overline{\rho})$ has length $\leq 3$.\\

\noindent
We denote by $\Def_{\overline{\rho}}$ the groupoid over $\Comp(\co_E)$ of deformations of $\overline{\rho}$ (see \cite[Def.~3.3.6]{Em4}) and by $\Def_{\pi(\overline{\rho}), \ortho}$ the groupoid over $\Comp(\co_E)$ of orthonormalizable admissible deformations of $\pi(\overline{\rho})$ (see \cite[Def. 3.3.7]{Em4}). Following \cite[Def. 3.3.9]{Em4} we denote by $\Def_{\pi(\overline{\rho}), \ortho}^* \subseteq \Def_{\pi(\overline{\rho}), \ortho}$ the subgroupoid of deformations $\pi$ such that the center of $G$ acts on $\pi$ by the character $\dett(\hV_{\varepsilon^{-1}}(\pi))\varepsilon$. The following theorem follows from work of Kisin and Pa{\v{s}}k{\=u}nas (see \cite[Thm. 3.3.13~\&~Rem. 3.3.14]{Em4}).

\begin{theorem}\label{thm: hL-pLL}
The functor $\hV_{\varepsilon^{-1}}$ induces an isomorphism of groupoids:
\begin{equation}\label{equ: hL-pLL}
\Def_{\pi(\overline{\rho}), \ortho}^* \xlongrightarrow{\sim} \Def_{\overline{\rho}}.
\end{equation}
\end{theorem}

\noindent
Let $\xi=(\rho_{\xi}^0, \iota_{\xi})\in \Def(\overline{\rho})(\co_E)$ and $\rho_{\xi}:=\rho_{\xi}^0\otimes_{\co_E} E$ (recall $\iota_{\xi}$ is a $\Gal_{\Q_p}$-equivariant isomorphism $\iota_\xi:\rho_{\xi}^0\otimes_{\co_E}k_E\buildrel\sim\over\rightarrow \overline{\rho}$). We still denote by $\xi:=(\pi_{\xi}^0, \iota_{\xi}')\in \Def^*(\pi(\overline{\rho}))(\co_E)$ the inverse image of $\xi$ via the isomorphism (\ref{equ: hL-pLL}) and set $\widehat\pi(\rho_{\xi}):=\pi_{\xi}^0 \otimes_{\co_E} E$. The map $\rho_{\xi}\mapsto \widehat\pi(\rho_{\xi})$ is the $p$-adic local Langlands correspondence for $\GL_2(\Q_p)$ (normalized as in \cite[\S~3.1]{Br11b}).

\begin{corollary}\label{coro: hL-pLL1}
The functor $\hV_{\varepsilon^{-1}}$ induces a natural surjection:
\begin{equation}\label{equ: hL-pLL3}
\Ext^1_{\GL_2(\Q_p)}(\widehat{\pi}(\rho_{\xi}), \widehat{\pi}(\rho_{\xi})) \twoheadlongrightarrow \Ext^1_{\Gal_{\Q_p}}(\rho_{\xi}, \rho_{\xi})
\end{equation}
where the extension group on the left is in the category of (admissible) unitary Banach representations of $\GL_2(\Q_p)$.
\end{corollary}
\begin{proof}
Let $\widetilde{\pi}\in \Ext^1_{\GL_2(\Q_p)}(\widehat{\pi}(\rho_{\xi}), \widehat{\pi}(\rho_{\xi}))$ and $\widetilde{\pi}_0$ a $\GL_2(\Q_p)$-invariant open lattice. Using that two open lattices in the Banach space $\widetilde{\pi}$ are commensurable and the exactness $\hV_{\varepsilon^{-1}}$, one easily checks that $\hV_{\varepsilon^{-1}}(\widetilde{\pi}_0)[1/p]$ is in $\Ext^1_{\Gal_{\Q_p}}(\rho_{\xi}, \rho_{\xi})$ and doesn't depend on the choice of $\widetilde{\pi}_0$. This defines the morphism (\ref{equ: hL-pLL3}). We prove (\ref{equ: hL-pLL3}) is surjective. Let $\widetilde{\rho}_{\xi}$ be a deformation of $\rho_{\xi}$ over $E[\epsilon]/\epsilon^2$. By the proof of \cite[Prop. 2.3.5]{Kis09}, one can find a finite $\co_E$-subalgebra $A\subseteq E[\epsilon]/\epsilon^2$ such that $A[1/p]\cong E[\epsilon]/\epsilon^2$ and a deformation $\rho_{A,\xi}$ of $\overline{\rho}$ over $A$ such that $\rho_{A,\xi}\otimes_{A} \co_E\cong \rho_{\xi}^0$ (via the natural surjection $A\twoheadrightarrow \co_E$ induced by $E[\epsilon]/\epsilon^2\twoheadrightarrow E$) and $\rho_{A,\xi}\otimes_A E[\epsilon]/\epsilon^2\cong \widetilde{\rho}_{\xi}$. By (\ref{equ: hL-pLL}), there exists a deformation $\widehat{\pi}_A$ of $\pi(\overline{\rho})$ over $A$ such that $\hV_{\varepsilon^{-1}}(\widehat{\pi}_A)\cong \rho_{A,\xi}$. It is straightforward to check that $\widehat{\pi}_A[1/p]\in \Ext^1_{\GL_2(\Q_p)}(\widehat{\pi}(\rho_{\xi}), \widehat{\pi}(\rho_{\xi}))$ (using (\ref{equ: hL-pLL}) again) and that $\widehat{\pi}_A[1/p]$ is sent to $\widetilde{\rho}_{\xi}$ via (\ref{equ: hL-pLL3}).
\end{proof}

\noindent
From now on, we assume moreover $\End_{\Gal_{\Q_p}}(\overline{\rho})\cong k_E$. By \cite[Lem. 2.1.2]{Kis10}, we also have $\End_{\GL_2(\Q_p)}(\pi(\overline{\rho}))\cong k_E$. We now still denote by $\Def_{\overline{\rho}}$ (resp. $\Def_{\pi(\overline{\rho}), \ortho}^*$, $\Def_{\pi(\overline{\rho}), \ortho}$) the (usual) deformation functor (e.g. as in \S~\ref{sec: ord-galoisdef}) attached to the groupoid $\Def_{\overline{\rho}}$ (resp. $\Def_{\pi(\overline{\rho}), \ortho}^*$, $\Def_{\pi(\overline{\rho}), \ortho}$). We know that $\Def_{\overline{\rho}}$ is representable, hence so is $\Def_{\pi(\overline{\rho}), \ortho}^*$ by Theorem \ref{thm: hL-pLL}.\\

\noindent
Let $\overline{\zeta} :=\wedge^2_{k_E} \overline{\rho}$ be the determinant of $\overline{\rho}$ and recall that any element in $ \Ext^1_{\Gal_{\Q_p}}(\overline{\rho}, \overline{\rho})$ (resp. $\Ext^1_{\GL_2(\Q_p)}(\pi(\overline{\rho}), \pi(\overline{\rho}))$) can be viewed as a deformation $\widetilde{\rho}$ (resp. $\widetilde{\pi}$) of $\overline{\rho}$ (resp. $\pi(\overline{\rho})$) over $k_E[\epsilon]/\epsilon^2$. In particular we have a $k_E$-linear morphism:
\begin{equation}\label{equ: olg-det}
\Ext^1_{\Gal_{\Q_p}}(\overline{\rho},\overline{\rho}) \lra \Hom(\Gal_{\Q_p}, k_E)\cong \Hom(\Q_p^{\times}, k_E)
\end{equation}
(= group homomorphisms to the additive group $k_E$) sending $\widetilde{\rho}$ to $\big(\overline{\zeta}'\overline{\zeta}^{-1}-1\big)/\epsilon$ where $\overline{\zeta}':=\wedge^2_{k_E[\epsilon]/\epsilon^2} \widetilde{\rho}$. We define $\Ext^1_{\Gal_{\Q_p}, \overline{\zeta}}(\overline{\rho}, \overline{\rho})$ as the kernel of (\ref{equ: olg-det}). By the assumptions on $\overline{\rho}$, each irreducible constituent $\pi$ of $\pi(\overline{\rho})$ has multiplicity one in $\pi(\overline{\rho})$. Using the same arguments as in the proof of Lemma \ref{lem: hL-cent2}, we can then show that there exists $\overline{\zeta}': \Q_p^{\times} \ra (k_E[\epsilon]/\epsilon^2)^{\times}$ such that the center $Z(\Q_p)\cong \Q_p^\times$ acts on $\widetilde{\pi}$ by $\overline{\zeta}'\overline\varepsilon$. We thus deduce another $k_E$-linear morphism:
\begin{equation}\label{equ: gl2-det2}
\Ext^1_{\GL_2(\Q_p)}(\pi(\overline{\rho}), \pi(\overline{\rho})) \lra \Hom(\Q_p^{\times}, k_E), \ \widetilde{\pi} \longmapsto \big(\overline{\zeta}'\overline{\zeta}^{-1}-1\big)/\epsilon
\end{equation}
and we define $\Ext^1_{\GL_2(\Q_p), \overline{\zeta}\overline\varepsilon}(\pi(\overline{\rho}), \pi(\overline{\rho}))$ as the kernel of (\ref{equ: gl2-det2}), which is the $k_E$-vector subspace of extensions with central character $\overline{\zeta}\overline\varepsilon$.

\begin{lemma}\label{lem: olg-cc}
We have short exact sequences of $k_E$-vector spaces:
\begin{equation*}
0 \lra \Ext^1_{\Gal_{\Q_p}, \overline{\zeta}}(\overline{\rho}, \overline{\rho}) \lra \Ext^1_{\Gal_{\Q_p}}(\overline{\rho},\overline{\rho}) \xlongrightarrow{(\ref{equ: olg-det})} \Hom(\Q_p^{\times}, E) \lra 0
\end{equation*}
\begin{equation*}
0 \lra \Ext^1_{\GL_2(\Q_p), \overline{\zeta}\overline\varepsilon}(\pi(\overline{\rho}), \pi(\overline{\rho})) \lra \Ext^1_{\GL_2(\Q_p)}(\pi(\overline{\rho}), \pi(\overline{\rho}))\xlongrightarrow{(\ref{equ: gl2-det2})} \Hom(\Q_p^{\times}, k_E) \lra 0.
\end{equation*}
\end{lemma}
\begin{proof}
It is enough to prove that (\ref{equ: olg-det}) (resp. (\ref{equ: gl2-det2})) is surjective. The map $\psi\mapsto \overline{\rho} \otimes (1+\psi/2 \epsilon)$ (resp. $\psi \mapsto \pi(\overline{\rho}) \otimes (1+\psi/2\epsilon)\circ\dett$) gives a section of (\ref{equ: olg-det}) (resp. of (\ref{equ: gl2-det2})).
\end{proof}

\noindent
As in \cite{Pas13}, we call $\overline{\rho}$ \emph{generic} if either $\overline{\rho}$ is irreducible or $\overline{\rho}\cong \begin{pmatrix}\delta_1 & * \\ 0 &\delta_2 \end{pmatrix}$ for $\delta_1\delta_2^{-1}\notin \{\overline{\varepsilon}, 1\}$ and we call $\overline{\rho}$ \emph{nongeneric} if $\overline{\rho}\cong \begin{pmatrix} \delta \overline{\varepsilon} & * \\ 0 & \delta\end{pmatrix}$ for some $\delta: \Gal_{\Q_p} \ra k_E^{\times}$ (recall we have $*\ne 0$ since $\End_{\Gal_{\Q_p}}(\overline{\rho})\cong k_E$).

\begin{proposition}\label{prop: gl2-tang}
We have:
\begin{equation*}
\dim_{k_E} \Ext^1_{\Gal_{\Q_p}, \overline{\zeta}}(\overline{\rho}, \overline{\rho})=\dim_{k_E} \Ext^1_{\GL_2(\Q_p),\overline{\zeta}\overline\varepsilon}(\pi(\overline{\rho}), \pi(\overline{\rho}))=3,
\end{equation*}
\begin{equation*}
\dim_{k_E} \Ext^1_{\Gal_{\Q_p}}(\overline{\rho}, \overline{\rho})=\dim_{k_E} \Ext^1_{\GL_2(\Q_p)}(\pi(\overline{\rho}), \pi(\overline{\rho}))=5.
\end{equation*}
\end{proposition}
\begin{proof}
By Lemma \ref{lem: olg-cc}, it is enough to prove the result for $\Ext^1_{\Gal_{\Q_p}, \overline{\zeta}}$ and $\Ext^1_{\GL_2(\Q_p),\overline{\zeta}\overline\varepsilon}$. By our assumptions on $\overline{\rho}$, we easily check that $\dim_{k_E} \Ext^1_{\Gal_{\Q_p}, \overline{\zeta}}(\overline{\rho}, \overline{\rho})=3$. The result for $\Ext^1_{\GL_2(\Q_p), \overline{\zeta}\overline\varepsilon}$ follows from \cite[Prop. 6.1]{Pas13} (in the supersingular case), \cite[Cor. 8.5]{Pas13} (in the generic nonsupersingular case) and \cite[Thm. 6.10]{Pas15} together with $\dim_{k_E} \Ext^1_{\Gal_{\Q_p}, \overline{\zeta}}(\overline{\rho}, \overline{\rho})=3$ (in the nongeneric case).
\end{proof}

\noindent
Since $\End_{\GL_2(\Q_p)}(\pi(\overline{\rho}))\cong k_E$ and $\dim_{k_E} \Ext^1_{\GL_2(\Q_p)}(\pi(\overline{\rho}), \pi(\overline{\rho}))<\infty$ by the last equality in Proposition \ref{prop: gl2-tang}, it follows from Schlessinger's criterion that the functor $\Def_{\pi(\overline{\rho}),\ortho}$ is representable. Using Theorem \ref{thm: hL-pLL}, the third equality in Proposition \ref{prop: gl2-tang} and \cite[Lem.~2.1]{Fer} (and the representability of $\Def_{\overline{\rho}}$, $\Def_{\pi(\overline{\rho}), \ortho}^*$, $\Def_{\pi(\overline{\rho}), \ortho}$), we easily deduce that we have in fact isomorphisms:
\begin{equation}\label{isodef}
\Def_{\pi(\overline{\rho}), \ortho}^*\buildrel{\sim}\over\longrightarrow \Def_{\pi(\overline{\rho}), \ortho}\buildrel{{\sim}}\over\longrightarrow \Def_{\overline{\rho}}.
\end{equation}
Recall $\Art(\co_E)$ is the category of local artinian $\co_E$-algebras with residue field $k_E$ and let $\cC(\co_E)$ be the subcategory of $\Mod_{\GL_2(\Q_p)}^{\pro\ \!\!\!\aug}(\co_E)$ dual to $\Mod_{\GL_2(\Q_p)}^{\lfin}(\co_E)$ via (\ref{equ: app-dual}). Denote by $\Def_{\pi(\overline{\rho})^{\vee}, \cC(\co_E)}$ the functor from $\Art(\co_E)$ to (isomorphism classes of) deformations of $\pi(\overline{\rho})^{\vee}$ in the category $\cC(\co_E)$ in the sense of \cite[Def. 3.21]{Pas13} (since we only deal with commutative rings here, we drop the subscript ``ab'' of \cite[\S~3.1]{Pas13}). As $\Hom_{\cC(\co_E)}(\pi(\overline{\rho})^{\vee}, \pi(\overline{\rho})^{\vee})=k_E$ and $\dim_{k_E}\Ext^1_{\cC(\co_E)}(\pi(\overline{\rho})^{\vee}, \pi(\overline{\rho})^{\vee})<\infty $, Schlessinger's criterion again implies that $\Def_{\pi(\overline{\rho})^{\vee}, \cC(\co_E)}$ is pro-representable by a complete local noetherian $\co_E$-algebra $R_{\pi(\overline{\rho})^{\vee}}$ of residue field $k_E$.\\

\noindent
When considering a deformation, we now do not write anymore the reduction morphism $\iota$ (which is understated).

\begin{lemma}\label{tecnicA}
(1) Let $A$ in $\Art(\co_E)$ and $M_A\in \Def_{\pi(\overline{\rho})^{\vee}, \cC(\co_E)}(A)$, then $M_A\in \Mod_{\GL_2(\Q_p)}^{\fg\ \!\!\!\aug}(A)$ and $M_A$ is a pro-free $A$-module.\\
(2) Let $A$ in $\Art(\co_E)$ and $\pi_A\in \Def_{\pi(\overline{\rho}), \ortho}(A)$, then $\Hom_A(\pi_A, A) \in \Def_{\pi(\overline{\rho})^{\vee}, \cC(\co_E)}(A)$.
\end{lemma}
\begin{proof}
(1) Since $M_A$ is in $\cC(\co_E)$, it is profinite. By definition (see \cite[Def. 3.21]{Pas13}), $M_A$ is a flat $A$-module and by \cite[Exp. VII$_B$(0.3.8)]{SGA3}, the second part of (1) follows. It is straightforward to see $M_A$ is in $\Mod_{\GL_2(\Q_p)}^{\pro\ \!\!\!\aug}(A)$. Let $H$ be a pro-$p$ compact open subgroup of $\GL_2(\Q_p)$, the algebra $A[[H]]$ is (noncommutative) local. Since $\pi(\overline{\rho})$ is admissible, we know $M_A\otimes_{A[[H]]} k_E \cong \pi(\overline{\rho})^\vee\otimes_{k_E[[H]]} k_E$ is a finite dimensional $k_E$-vector space. By Nakayama's lemma (see e.g. \cite[Lem.~4.22]{Lam}), we deduce $M_A$ is finitely generated over $A[[H]]$.\\
\noindent
(2) By \cite[Prop. 3.1.12]{Em4} and its proof, we have that $M_A:=\Hom_A(\pi_A,A)$ is flat over $A$ and $M_A\otimes_A k_E\cong \pi(\overline{\rho})^{\vee}$. Since $\pi_A$ is admissible and $A$ is artinian, $\pi_A$ is locally finite by \cite[Thm.~2.3.8]{EOrd1}. The lemma follows by definition of $\cC(\co_E)$.
\end{proof}

\begin{proposition}\label{pro: gl2-defo2}
We have an isomorphism of deformation functors:
\begin{equation*}
\Def_{\pi(\overline{\rho}), \ortho} \xlongrightarrow{\sim} \Def_{\pi(\overline{\rho})^{\vee}, \cC(\co_E)}, \ [A\mapsto \{\pi_A\}_{/\sim}]\longmapsto [A\mapsto \{M_A=\Hom_A(\pi_A, A)\}_{/\sim}].
\end{equation*}
\end{proposition}
\begin{proof}
This follows from \cite[Prop. 3.1.12]{Em4} and Lemma \ref{tecnicA}.
\end{proof}

\noindent
Proposition \ref{pro: gl2-defo2} together with (\ref{equ: hL-pLL}) and (\ref{isodef}) imply an isomorphism of deformation functors:
\begin{equation}\label{equ: funteq}
 \Def_{\pi(\overline{\rho})^{\vee}, \cC(\co_E)} \xlongrightarrow{\sim} \Def_{\overline{\rho}}
\end{equation}
and hence $R_{\pi(\overline{\rho})^{\vee}}\cong R_{\overline{\rho}}$.
Let $\rho^{\univ}$ be the universal deformation of $\overline{\rho}$ over $R_{\overline{\rho}}$ (for $\Def_{\overline{\rho}}$), $\cN \in \cC(\co_E)$ the universal deformation of $\pi(\overline{\rho})^{\vee}$ over $R_{\overline{\rho}}$ (for $\Def_{\pi(\overline{\rho})^{\vee}, \cC(\co_E)}$) and $\pi^{\univ}(\overline{\rho})\in \Mod^{\ortho}_{\GL_2(\Q_p)}(R_{\overline{\rho}})$ the universal deformation of $\pi(\overline{\rho})$ (for $\Def_{\pi(\overline{\rho}), \ortho}$).

\begin{corollary}\label{coro: app-NR}
We have $\cN\cong \Hom_{R_{\overline{\rho}}}(\pi^{\univ}(\overline{\rho}), R_{\overline{\rho}})$.
\end{corollary}
\begin{proof}
This easily follows from Proposition \ref{pro: gl2-defo2}.
\end{proof}

\begin{corollary}\label{coro: app-NRI}
Let $\cI$ be an ideal of $R_{\overline{\rho}}$, then we have:
\begin{equation*}
\cN \otimes_{R_{\overline{\rho}}} R_{\overline{\rho}}/\cI \cong \Hom_{R_{\overline{\rho}}/\cI}\big(\pi^{\univ}(\overline{\rho})\otimes_{R_{\overline{\rho}}} R_{\overline{\rho}}/\cI, R_{\overline{\rho}}/\cI\big).
\end{equation*}
\end{corollary}
\begin{proof}
This follows from the isomorphism in Corollary \ref{coro: app-NR} and \cite[Lem. B.7]{Em4}.
\end{proof}

\begin{remark}\label{topotopo}
{\rm Recall the isomorphism in Corollary \ref{coro: app-NR} and the isomorphism in Corollary \ref{coro: app-NRI} are topological isomorphisms where the left hand side is equipped with the profinite topology and the right hand side with the topology of pointwise convergence (see \cite[Prop.~B.11(2)]{Em4}).}
\end{remark}

\noindent
For any $\zeta:\Q_p^\times\rightarrow \co_E^\times$ we denote by $\Mod^{\lfin}_{\GL_2(\Q_p), \zeta}(\co_E)$ the full subcategory of $\Mod^{\lfin}_{\GL_2(\Q_p)}(\co_E)$ of representations on which $Z(\Q_p)$ acts by $\zeta$, and by $\cC_{\zeta}(\co_E)$ the full subcategory of $\cC(\co_E)$ dual to $\Mod^{\lfin}_{\GL_2(\Q_p), \zeta}(\co_E)$ via (\ref{equ: app-dual}). For any $\zeta:\Q_p^\times\rightarrow \co_E^\times$ such that $\zeta\equiv \overline{\zeta}$ mod $\varpi_E$, we denote by $\Def_{\overline{\rho}}^{\zeta}$ the subfunctor of $\Def_{\overline{\rho}}$ of deformations with fixed determinant $\zeta$ and by $R_{\overline{\rho}}^{\zeta}$ the universal deformation ring for $\Def_{\overline{\rho}}^{\zeta}$. We denote by $\Def_{\pi(\overline{\rho})^\vee, \cC_{\zeta\varepsilon}(\co_E)}$ the deformation functor on $\Art(\co_E)$ defined in the same way as $\Def_{\pi(\overline{\rho})^\vee, \cC(\co_E)}$ replacing $\cC(\co_E)$ by the subcategory $\cC_{\zeta\varepsilon}(\co_E)$. By the second equality in Proposition \ref{prop: gl2-tang} and Schlessinger's criterion, $\Def_{\pi(\overline{\rho})^\vee, \cC_{\zeta\varepsilon}(\co_E)}$ is pro-representable by a complete local noetherian $\co_E$-algebra $R_{\pi(\overline{\rho})^{\vee}}^{\zeta\varepsilon}$ of residue field $k_E$. It is not difficult to see that the isomorphism in (\ref{equ: funteq}) induces a natural isomorphism (so that $R_{\overline{\rho}}^{\zeta}\buildrel\sim\over\rightarrow R_{\pi(\overline{\rho})^{\vee}}^{\zeta\varepsilon}$):
\begin{equation}\label{equ: GL2-centC}
\Def_{\pi(\overline{\rho})^\vee, \cC_{\zeta\varepsilon}(\co_E)} \buildrel\sim\over \lra \Def_{\overline{\rho}}^{\zeta}.
\end{equation}

\noindent
We denote by $\cN^{\zeta\varepsilon}$ the universal deformation of $\pi(\overline{\rho})^{\vee}$ over $R_{\pi(\overline{\rho})^{\vee}}^{\zeta\varepsilon}\cong R_{\overline{\rho}}^{\zeta}$ for $\Def_{\pi(\overline{\rho})^\vee, \cC_{\zeta\varepsilon}(\co_E)}$ (note that $\cN^{\zeta\varepsilon}$ is denoted by $N$ in \cite{HP}) and by $\rho^{\univ, \zeta}$ the universal deformation of $\overline{\rho}$ over $R_{\overline{\rho}}^{\zeta}$. Let $\Lambda$ be the universal deformation ring of the trivial $1$-dimensional representation of $\Gal_{\Q_p}$ over $k_E$ and $1^{\univ}$ the corresponding universal deformation (which is thus a free $\Lambda$-module of rank $1$). We have $R_{\overline{\rho}}\cong R_{\overline{\rho}}^{\zeta}\widehat\otimes_{\co_E} \Lambda$ and $\rho^{\univ}\cong \rho^{\univ, \zeta}\widehat\otimes_{\co_E} 1^{\univ}$ where $\widehat{\otimes}$ denotes the $\varpi_E$-adic completion of the usual tensor product. We equip $1^{\univ}$ with a natural action of $\GL_2(\Q_p)$ via $\dett: \GL_2(\Q_p) \ra \Q_p^{\times}$. One easily sees $1^{\univ}\in \cC(\co_E)$.

\begin{proposition}\label{prop: GL2-NN}
We have $\cN\cong \cN^{\zeta\varepsilon}\widehat{\otimes}_{\co_E} 1^{\univ}$.
\end{proposition}
\begin{proof}
We have that $\cN^{\zeta\varepsilon}\widehat{\otimes}_{\co_E} 1^{\univ}$ is a deformation of $\pi(\overline{\rho})^{\vee}$ over $R_{\overline{\rho}}^{\zeta} \widehat{\otimes}_{\co_E} \Lambda$ in $\cC(\co_E)$, from which we deduce a morphism of local $\co_E$-algebras $R_{\overline{\rho}} \lra R_{\overline{\rho}}^{\zeta} \widehat{\otimes}_{\co_E} \Lambda$, which is easily checked to be an isomorphism (e.g. by proving the tangent map is bijective). The proposition follows.
\end{proof}

\subsection{Deformations II}

\noindent
We prove here a key projectivity property of $\cN$.\\

\noindent
We keep the previous notation and assumption (in particular $\overline\rho$ satisfies (\ref{hypo: hL-modp}) and is such that $\End_{\Gal_{\Q_p}}(\overline{\rho})\cong k_E$). We assume moreover $p\geq 5$ if $\overline{\rho}$ is nongeneric.

\begin{proposition}\label{prop: GL2-proj}
There exist $x$, $y\in R_{\overline{\rho}}$ such that $S:=\co_E[[x,y]]$ is a subring of $R_{\overline{\rho}}$ and $\cN$ is a finitely generated projective $S[[\GL_2(\Z_p)]]$-module.
\end{proposition}
\begin{proof}
We fix $K$ a pro-$p$ compact open subgroup of $\GL_2(\Z_p)$ such that $K\cong K/Z_0\times Z_0$ with $Z_0:=K\cap Z(\Q_p)$ isomorphic to $\Z_p$. If $R$ is a (noncommutative) ring, we denote by $\Mod^{\fg}_R$ the category of finitely generated $R$-modules. It is enough to prove the statement with $\GL_2(\Z_p)$ replaced by $K$.\\
\noindent
(a) By \cite[Thm. 3.3]{HP} (in the generic case) and \cite[Thm. 3.5]{HP} and its proof (in the nongeneric case), there exists $x$ in the maximal ideal of $R_{\overline{\rho}}^{\zeta}$ such that $\cN^{\zeta\varepsilon}\otimes_{R_{\overline{\rho}}^{\zeta}}R_{\overline{\rho}}^{\zeta}/x$ is a finitely generated $\co_E[[K]]$-module which is projective in the category $\Mod_{\co_E[[K]], \zeta\varepsilon}^{\fg}$ := the full subcategory of $\Mod_{\co_E[[K]]}^{\fg}$ on which $Z_0$ acts by $\zeta\varepsilon$. In particular $\cN^{\zeta\varepsilon}$ is a finitely generated $S_1[[K]]$-module for $S_1:=\co_E[[x]]$. We first want to prove that $\cN^{\zeta\varepsilon}$ is moreover projective in $\Mod_{S_1[[K]], \zeta\varepsilon}^{\fg}$ (with obvious notation, as $\cN^{\zeta\varepsilon}$ is a $R_{\overline{\rho}}^{\zeta}$-module note this will also imply $S_1\hookrightarrow R_{\overline{\rho}}^{\zeta}$). Let $\chi: Z_0\ra \co_E^{\times}$ such that $\chi^2=\zeta$ (enlarging $E$ if necessary and using $Z_0\cong \Z_p$), we deduce an isomorphism of $\co_E[[K/Z_0]]$-modules (using that $\co_E[[K/Z_0]]$ is a local ring):
$$(\cN^{\zeta\varepsilon}\otimes \chi^{-1}\circ \dett)\otimes_{R_{\overline{\rho}}^{\zeta}}R_{\overline{\rho}}^{\zeta}/x \cong \co_E[[K/Z_0]]^{\oplus r}$$
and it is enough to prove that $\cN_1:=\cN^{\zeta\varepsilon}\otimes \chi^{-1}\circ \dett$ is projective in $\Mod_{S_1[[K/Z_0]]}^{\fg}$.\\
\noindent
(b) As at the beginning of \cite[\S~2.5]{Pas15}, it is enough to prove $\Tor^1_{S_1[[K/Z_0]]}(\cN_1, k_E)=0$ where $\Tor^i_{S_1[[K/Z_0]]}(-,k_E)$ denotes the $i$-th derived functor of $(\cdot){\otimes}_{S_1[[K/Z_0]]} k_E$ in $\Mod_{S_1[[K/Z_0]]}^{\fg}$ (recall $S_1[[K/Z_0]]$ is a local ring of residue field $k_E$). Indeed, let $\cP$ be a projective envelope of $\cN_1$ in $\Mod_{S_1[[K/Z_0]]}^{\fg}$ (whose existence follows from \cite[\S~23~\&~Prop.~24.12]{Lam}), and consider a short exact sequence $0 \ra \cM_1 \ra \cP \ra \cN_1 \ra 0$ in $\Mod_{S_1[[K/Z_0]]}^{\fg}$. If $\Tor^1_{S_1[[K/Z_0]]}(\cN_1, k_E)=0$, we get:
\begin{equation*}
0 \lra \cM_1{\otimes}_{S_1[[K/Z_0]]} k_E \lra \cP {\otimes}_{S_1[[K/Z_0]]} k_E \lra \cN_1{\otimes}_{S_1[[K/Z_0]]} k_E \lra 0.
\end{equation*}
Since $\cP$ is the projective envelope of $\cN_1$, we have $\cP {\otimes}_{S_1[[K/Z_0]]} k_E \xrightarrow{\sim} \cN_1{\otimes}_{S_1[[K/Z_0]]} k_E$, whence $\cM_1{\otimes}_{S_1[[K/Z_0]]} k_E=0$, and $\cM_1=0$ by Nakayama's lemma (\cite[Lem.~4.22]{Lam}). Now, the exact sequence $0 \lra \cN_1 \xlongrightarrow{x} \cN_1 \lra \cN_1/x \lra 0$ (recall $\cN^{\zeta\varepsilon}$ is flat over $R_{\overline{\rho}}^{\zeta}$) induces:
\begin{multline}\label{equ: gl2-NxTor}
\Tor^1_{S_1[[K/Z_0]]}(\cN_1, k_E) \xlongrightarrow{x} \Tor^1_{S_1[[K/Z_0]]}(\cN_1, k_E) \lra \Tor^1_{S_1[[K/Z_0]]}(\cN_1/x, k_E) \\ \lra \cN_1{\otimes}_{S_1[[K/Z_0]]} k_E
\lra \cN_1{\otimes}_{S_1[[K/Z_0]]} k_E \lra (\cN_1/x){\otimes}_{S_1[[K/Z_0]]} k_E \lra 0.
\end{multline}
Let $r:=\dim_{k_E} \cN_1{\otimes}_{S_1[[K/Z_0]]} k_E=\dim_{k_E} (\cN_1/x){\otimes}_{S_1[[K/Z_0]]} k_E$. By the argument as at end of the proof of \cite[Prop. 2.34]{Pas15}, $\Tor^1_{S_1[[K/Z_0]]}(\cN_1, k_E)$ is a finitely generated $S_1$-module (even a finite dimensional $k_E$-vector space). Using the exact sequence:
\begin{equation*}
0 \lra S_1[[K/Z_0]]^{\oplus r} \xlongrightarrow{x} S_1[[K/Z_0]]^{\oplus r} \lra \co_E[[K/Z_0]]^{\oplus r}(\cong \cN_1/x)\lra 0,
\end{equation*}
we easily deduce $\Tor^1_{S_1[[K/Z_0]]}(\cN_1/x, k_E) \buildrel\sim\over\lra k_E^{\oplus r}$, which implies with (\ref{equ: gl2-NxTor}) that the morphism $\Tor^1_{S_1[[K/Z_0]]}(\cN_1, k_E) \xlongrightarrow{x} \Tor^1_{S_1[[K/Z_0]]}(\cN_1, k_E)$ is surjective. But since $x\mapsto 0\in k_E$, we deduce $\Tor^1_{S_1[[K/Z_0]]}(\cN_1,k_E)=0$ and hence $\cN_1$ is projective, and even isomorphic to $S_1[[K/Z_0]]^{\oplus r}$.\\
\noindent
(c) We now finish the proof. Let $\Gamma:=1+p\Z_p$, the pro-$p$ completion of $\Q_p^{\times}$ is isomorphic to $\Gamma\times \Z_p$, from which we deduce $\Lambda \cong \co_E[[\Gamma \times \Z_p]]$. There exists thus $y\in \Lambda$ such that $\Lambda\cong S_2[[\Gamma]]$ with $S_2:=\co_E[[y]]$. Since $Z_0$ is a subgroup of finite index of $\Gamma$, we deduce $\Lambda$ is finite \'etale over $S_2[[Z_0]]$, and hence $1^{\univ}$ is a finite projective $S_2[[Z_0]]$-module. Together with (b), Proposition \ref{prop: GL2-NN} and $K\cong K/Z_0\times Z_0$, we obtain that $\cN\cong \cN^{\zeta\varepsilon}\widehat{\otimes}_{\co_E} 1^{\univ}$ is a finitely generated projective $S[[K]]$-module with $S:=\co_E[[x,y]]$. This concludes the proof.
\end{proof}

\subsection{Proof of Proposition \ref{prop: hL-gl2pLL2}} \label{sec: app-tri}

\noindent
We finally prove Proposition \ref{prop: hL-gl2pLL2}.\\

\noindent
We keep the previous notation. We assume $p\geq 5$ and fix $\rho:\Gal_{\Q_p}\rightarrow \GL_2(E)$ as in Proposition \ref{prop: hL-gl2pLL}, so that we have $D_{\rig}(\rho)\cong D(\alpha, \lambda, \psi)$ with $D(\alpha, \lambda, \psi)$ as in Lemma \ref{lem: hL-etale}. It is enough to prove the proposition with $D(p,\lambda, \psi)$, $\pi(\lambda^{\flat}, \psi)$ replaced by $D(\alpha, \lambda, \psi)$, $\pi(p^{-1}\alpha, \lambda^{\flat}, \psi)$ respectively (as in the proof of Proposition \ref{prop: hL-gl2pLL}). We fix a mod $p$ reduction $\overline{\rho}$ of $\rho$ satisfying (\ref{hypo: hL-modp}) and $\End_{\Gal_{\Q_p}}(\overline{\rho})=k_E$, and we define using Corollary \ref{coro: app-NR} and Remark \ref{topotopo}:
\begin{equation*}
\Pi:=\Hom_{\co_E}^{\cts}\big(\cN, \co_E\big)\otimes_{\co_E} E \cong \Hom_{\co_E}^{\cts}\big(\Hom_{R_{\overline{\rho}}}(\pi^{\univ}(\overline{\rho}), R_{\overline{\rho}}), \co_E\big)\otimes_{\co_E} E
\end{equation*}
where ``$\cts$" means the continuous morphisms. It follows from \cite[Thm.~1.2]{ST} and Proposition \ref{prop: GL2-proj} that the Banach space $\Pi$ (equipped with the supremum norm) is an $R_{\overline{\rho}}$-admissible continuous representation of $\GL_2(\Q_p)$ in the sense of \cite[D\'ef.~3.1]{BHS1}.

\begin{lemma}\label{chiant}
We have an isomorphism of Banach spaces:
$$\Pi\cong \Hom_{\co_E}^{\cts}(R_{\overline{\rho}},\co_E)\widehat \otimes_{R_{\overline{\rho}}} \pi^{\univ}(\overline{\rho})[1/p]$$
where $R_{\overline{\rho}}$ (in $\Hom_{\co_E}^{\cts}(R_{\overline{\rho}},\co_E)$) is equipped with its $\fm_{R_{\overline{\rho}}}$-adic topology and $\widehat \otimes$ is the $\varpi_E$-adic completion of the usual tensor product.
\end{lemma}
\begin{proof}
Note that $\Hom_{\co_E}^{\cts}(R_{\overline{\rho}},\co_E)$ is a cofinitely generated $R_{\overline{\rho}}$-module by \cite[Prop.~C.5]{Em4}. By \cite[Thm.~1.2]{ST}, it is enough to prove $\Hom_{\co_E}(\Hom_{\co_E}^{\cts}(R_{\overline{\rho}},\co_E)\widehat \otimes_{R_{\overline{\rho}}} \pi^{\univ}(\overline{\rho}),\co_E)\cong \cN$. But:
\begin{align}
\MoveEqLeft[18]  {\Hom_{\co_E}\big(\Hom_{\co_E}^{\cts}(R_{\overline{\rho}},\co_E)\widehat \otimes_{R_{\overline{\rho}}} \pi^{\univ}(\overline{\rho}),\co_E\big)}& \nonumber \\
\cong{} & \Hom_{\co_E}\big(\Hom_{\co_E}^{\cts}(R_{\overline{\rho}},\co_E)\otimes_{R_{\overline{\rho}}} \pi^{\univ}(\overline{\rho}),\co_E\big) \nonumber\\
\cong{} &\Hom_{R_{\overline{\rho}}}\big(\pi^{\univ}(\overline{\rho}),\Hom_{\co_E}(\Hom_{\co_E}^{\cts}(R_{\overline{\rho}},\co_E),\co_E)\big) \nonumber\\
\cong{} & \Hom_{R_{\overline{\rho}}}\big(\pi^{\univ}(\overline{\rho}),R_{\overline{\rho}}\big) \cong \cN \nonumber
\end{align}
where the first two isomorphisms are easy, the third one comes from \cite[Prop.~C.5]{Em4} and the last one from Corollary \ref{coro: app-NR}.
\end{proof}

\noindent
Any $\widetilde{\rho}\in \Ext^1_{\Gal_{\Q_p}}(\rho, \rho)$ gives rise to an $E[\epsilon]/\epsilon^2$-valued point of $R_{\overline{\rho}}$, hence to an ideal $\cI_{\widetilde{\rho}}\subseteq R_{\overline{\rho}}$ with $R_{\overline{\rho}}/\cI_{\widetilde{\rho}}\cong \co_E[\epsilon]/\epsilon^2$.

\begin{lemma}\label{equ: app-isot}
Let $\pi(\widetilde{\rho})^{\an}$ be the image of $\widetilde{\rho}$ via (\ref{equ: hL-pLLa}), then we have an isomorphism $\pi(\widetilde{\rho})^{\an}\cong \Pi[\cI_{\widetilde{\rho}}]^{\an}$ of locally analytic representations of $\GL_2(\Q_p)$ over $E$.
\end{lemma}
\begin{proof}
By the same proof as for Lemma \ref{chiant} using $\Pi[\cI_{\widetilde{\rho}}]\cong \Hom_{\co_E}^{\cts}(\cN/\cI_{\widetilde{\rho}}, \co_E)\otimes_{\co_E}E$ and Corollary \ref{coro: app-NRI}, we deduce $\Pi[\cI_{\widetilde{\rho}}]\cong \Hom_{\co_E}^{\cts}(R_{\overline{\rho}}/\cI_{\widetilde{\rho}},\co_E)\widehat \otimes_{R_{\overline{\rho}}/\cI_{\widetilde{\rho}}} (\pi^{\univ}(\overline{\rho})/\cI_{\widetilde{\rho}})[1/p]$. The result follows then from Remark \ref{rem: hL-pLL2}(2) and the fact $\Hom_{\co_E}^{\cts}(R_{\overline{\rho}}/\cI_{\widetilde{\rho}},\co_E)$ is free of rank one over $R_{\overline{\rho}}/\cI_{\widetilde{\rho}}\cong \co_E[\epsilon]/\epsilon^2$.
\end{proof}

\noindent
As in \cite[D\'ef.~3.2]{BHS1}, we denote by $\Pi^{R_{\overline{\rho}}-\an}$ the subspace of locally $R_{\overline{\rho}}$-analytic vectors of $\Pi$ and consider the locally analytic $T(\Q_p)$-representation $J_B(\Pi^{R_{\overline{\rho}}-\an})$ ($T$, $B$ as in \S~\ref{sec: hL-ext1}). As in \cite[\S~3.2]{BHS1}, the strong dual $J_B(\Pi^{R_{\overline{\rho}}-\an})^{\vee}$ corresponds to a coherent sheaf $\cM$ over $(\Spf R_{\overline{\rho}})^{\rig}\times \cT$ ($\cT$ as in \S~\ref{pordinary}) and we let $X$ denote the schematic support of $\cM$. A point $x=(\rho_x, \delta_x)\in (\Spf R_{\overline{\rho}})^{\rig}\times \cT$ lies in $X$ if and only if there is a $T(\Q_p)$-embedding $\delta_x \hookrightarrow J_B(\Pi^{R_{\overline{\rho}}-\an}[\fp_{\rho_x}])=J_B(\widehat{\pi}(\rho_x)^{\an})$ where $\fp_{\rho_x}\subseteq R_{\overline{\rho}}$ is the prime ideal attached to $\rho_x$ and the isomorphism $\widehat\pi({\rho}_x)^{\an}\cong \Pi^{R_{\overline{\rho}}-\an}[\fp_{\rho_x}]= \Pi[\fp_{\rho_x}]^{\an}$ is proven as for the one in Lemma \ref{equ: app-isot}.\\

\noindent
Consider the Zariski-closed trianguline variety $X_{\rm tri}(\overline\rho)$ of $(\Spf R_{\overline{\rho}})^{\rig}\times \cT$ defined exactly as in \cite[\S~2.2]{BHS1} (for $K=\Q_p$ and $n=2$) but without the framing, i.e. replacing $R^\square_{\overline\rho}$ by $R_{\overline\rho}$. As in \cite[Thm.~2.6]{BHS1} the rigid variety $X_{\rm tri}(\overline\rho)$ is equidimensional of dimension $4$ and contains a Zariski-open Zariski-dense subspace $U_{\tri}(\overline{\rho})^{\reg}$ which we define in the same way but removing the framing. Arguing inside $(\Spf R_{\overline{\rho}})^{\rig}\times \cT$, it easily follows from \cite{Colm14}, \cite{LXZ} (and the above characterization of points of $X$) that there is an embedding $U^{\tri}(\overline{\rho})^{\reg}\hookrightarrow X$ (be careful that there is a shift on the $\cT$-part between the two sides analogous to (the inverse of) \cite[(3.2)]{BHS2}), and hence we deduce a closed embedding (note that $X_{\rm tri}(\overline\rho)$ is reduced) $j: X_{\rm tri}(\overline\rho) \hooklongrightarrow X$. The pull-back $\cM_1:=j^* \cM$ is thus a coherent sheaf on $X_{\rm tri}(\overline\rho)$.\\

\noindent
It follows from Proposition \ref{prop: GL2-proj} that $\cN$ is finitely generated and \emph{projective} as $S[[\GL_2(\Z_p)]]$-module where $S=\co_E[[x,y]]\hookrightarrow R_{\overline{\rho}}$. In this case, by the same argument as in \cite[\S\S~3.3,~3.4~\&~3.5]{BHS1} (see especially \cite[Thm.~3.19]{BHS1}), one can prove that the set $Z$ of points $(\rho,\delta)\in X$ such that $\rho$ is crystalline generic (see before Lemma \ref{lem: lg1-classici2}) and $\delta$ is noncritical (see before Lemma \ref{lem: lg1-noncrit1}) is Zariski-dense and accumulation in $X$. Since such points are in $U_{\tri}(\overline{\rho})^{\reg}$ (modulo the aforementioned shift) we deduce that $j$ induces an isomorphism $X_{\rm tri}(\overline{\rho}) \xlongrightarrow{\sim} X_{\red}$. In particular, the noncritical point $x:=(\rho, \delta_{\lambda}(|\cdot|\otimes 1)\unr(\alpha)\circ\dett)$ is in $X_{\rm tri}(\overline{\rho})$ (indeed, as $j$ is an isomorphism, all the trianguline representations with mod $p$ reduction isomorphic to $\overline{\rho}$ appear on $X_{\rm tri}(\overline{\rho})$ since they do on $X$ using \cite{Colm14}, \cite{LXZ}).\\

\noindent
Using the isomorphism $X_{\rm tri}(\overline{\rho}) \xrightarrow{\sim} X_{\red}$ and the above characterization of points of $X$ together with \cite{Colm14}, \cite{LXZ} and \cite[Ex.~5.1.9]{Em2}, it easily follows that there exists a sufficiently small affinoid neighborhood $U\subseteq X_{\tri}(\overline{\rho})$ of $x$ such that the special fiber of the coherent sheaf $\cM_1$ at each point $x'\in U$ is one dimensional over the residue field of $x'$. Since $U$ is reduced, we deduce $\cM_1$ is locally free of rank $1$ over $U$ by \cite[Ex.~II.5.8(c)]{Harts} (which is there in the scheme setting, but the rigid setting is analogous). We denote by $V_x$ the tangent space of $X_{\rm tri}(\overline{\rho})$ at $x$ and we identify the tangent space of $\cT$ at $\delta_x:=\delta_{\lambda}(|\cdot|\otimes 1)\unr(\alpha)\circ\dett$ with $\Hom(T(\Q_p),E)$. By the global triangulation theory (\cite{KPX}, \cite{Liu}) and using similar arguments as in \cite[\S 4.1]{BHS2}, we have the following facts:
\begin{itemize}
\item the morphism $X_{\rm tri}(\overline{\rho}) \lra (\Spf R_{\overline{\rho}})^{\rig}$ induces an isomorphism:
\begin{equation}\label{equ: hL-tang1}
 j_x: V_x \xlongrightarrow{\sim} \Ext^1_{\tri}(\rho, \rho)
\end{equation}
\item for $v\in V_x$, denote by $\Psi_v=(\psi_{v,1}, \psi_{v,2})\in \Hom(T(\Q_p),E)$ the image of $v$ in the tangent space of $\cT$ at $\delta_x$ induced by $X_{\rm tri}(\overline{\rho})\ra \cT$, then the $\Gal_{\Q_p}$-representation $j_x(v)$ is trianguline of parameter $(x^{k_1}|\cdot|(1+\psi_{v,1}\epsilon)\unr(\alpha), x^{k_2}(1+\psi_{v,2}\epsilon)\unr(\alpha))$.
\end{itemize}
Now let $0\neq v:\Spec E[\epsilon]/\epsilon^2 \hookrightarrow U$ be a nonzero element in $V_x$. Since $\cM_1$ is locally free at $x$, we have that $W_v:=v^* \cM_1$ is a free $E[\epsilon]/\epsilon^2$-module of rank $1$. The action of $R_{\overline{\rho}}$ on $W_v$ is induced by $v: \Spec E[\epsilon]/\epsilon^2 \hookrightarrow U \ra (\Spf R_{\overline{\rho}})^{\rig}$ and we denote as usual by $\cI_v$ the corresponding ideal of $R_{\overline{\rho}}$. Moreover $T(\Q_p)$ acts on the $E$-dual of $W_v$ by $\delta_{\lambda}(|\cdot|\otimes 1)(1+\Psi_v\epsilon)\unr(\alpha)\circ\dett$. Note that it is possible that $\Psi_v=0$, but we always have $\cI_v\neq \fm_{\rho}$ by (\ref{equ: hL-tang1}). Since the rigid space $(\Spf R_{\overline{\rho}})^{\rig}\times \cT$ is nested (\cite[Def.~7.2.10]{Bch}), so are its closed subspaces $X$ and $X_{\rm tri}(\overline{\rho})$, and it follows that the composition:
\begin{equation*}
 v: \Spec E[\epsilon]/\epsilon^2 \hooklongrightarrow U \hooklongrightarrow X_{\rm tri}(\overline{\rho})\buildrel\sim\over\lra X
\end{equation*}
induces a surjection $\Gamma(X, \cM) \twoheadrightarrow v^* \cM\cong W_v$ (using that the image of the composition $\Gamma(X, \cM)\!\rightarrow \Gamma(U, \cM)\!\twoheadrightarrow v^* \cM\cong W_v$ is dense as a composition of continuous maps with dense images, hence is surjective since $W_v$ is finite dimensional). Taking duals and keeping track of the shift, we obtain an $R_{\overline{\rho}}\times T(\Q_p)$-equivariant injection $\delta_{\lambda^\flat}(|\cdot|\otimes |\cdot|^{-1})(1+\Psi_v\epsilon)\unr(p^{-1}\alpha)\circ\dett \hooklongrightarrow J_B(\Pi^{R_{\overline{\rho}}-\an})$. Since the $E$-dual of $W_v$ is killed by $\cI_v$, we see that this map factors through an $E[\epsilon]/\epsilon^2$-linear embedding of locally analytic representations of $T(\Q_p)$:
\begin{equation}\label{criteretri}
\delta_{\lambda^\flat}(|\cdot|\otimes |\cdot|^{-1})(1+\Psi_v\epsilon)\unr(p^{-1}\alpha)\hooklongrightarrow J_B(\Pi^{R_{\overline{\rho}}-\an}[\cI_v])
\end{equation}
(note that the left hand side of (\ref{criteretri}) always has dimension $2$ over $E$ even if $\Psi_v=0$).\\

\noindent
We can now finally prove Proposition \ref{prop: hL-gl2pLL2}. By Proposition \ref{prop: hL-gl2pLL}, Lemma \ref{lem: hL-tri3} and Lemma \ref{dimtri}(1), it is enough to prove that (\ref{equ: hL-pLLa}) maps $\Ext^1_{\tri}$ to $\Ext^1_{\tri}$ in such a way that Hypothesis \ref{hypo: hL-pLL0}(3) holds (up to twist by $\unr(p^{-1}\alpha)$ on both sides). Fix an extension in $\Ext^1_{\tri}(\rho, \rho)$, i.e. a trianguline deformation $\widetilde{\rho}$ of $\rho$ over $E[\epsilon]/\epsilon^2$, by (\ref{equ: hL-tang1}) and what is below (\ref{equ: hL-tang1}), we have that $(x^{k_1}|\cdot|(1+\psi_{v,1}\epsilon)\unr(\alpha), x^{k_2}(1+\psi_{v,2}\epsilon)\unr(\alpha))$ is a parameter for $D_{\rig}(\widetilde\rho)$ where $v\in V_x$ is the associated vector. Let $\pi(\widetilde{\rho})^{\an}$ be the image of $\widetilde{\rho}$ via (\ref{equ: hL-pLLa}), by Lemma \ref{equ: app-isot} we have $\pi(\widetilde{\rho})^{\an}\cong \Pi[\cI_v]^{\an}=\Pi^{R_{\overline{\rho}}-\an}[\cI_v]$ and by (\ref{criteretri}) together with Lemma \ref{dimtri}(2), we finally deduce the result.

\end{document}